\documentclass[11pt,reqno,names]{amsart}
\usepackage{amssymb,amsmath,amsthm,epsfig,times}
\usepackage{mathrsfs}
\usepackage{epstopdf}
\usepackage{pdfpages}
\numberwithin{equation}{section} 
\usepackage{verbatim}
\usepackage[pdftex]{hyperref}
\allowdisplaybreaks[4]

\def\C{{\mathbb C}}

\def\P{{\mathbb P}}

\def\R{{\mathbb R}}
\def\*S{{\mathbb S}}

\def\Z{{\mathbb Z}}

\def\d{{\text{d}}}

\DeclareMathOperator{\supp}{supp}

\def\interior{
\begin{picture}(6,6)
\put (0,0){\line(1,0){6}}
\put (6,0){\line(0,1){6}}
\end{picture} \hspace{0.04in}}
\newcommand{\dis}{\displaystyle}
\newcommand{\cal}{\mathcal}

\newtheorem{theorem}{Theorem}
\newtheorem{proposition}{Proposition}[section]
\newtheorem{lemma}[proposition]{Lemma}

\begin{document}
\title[Inverse conductivity problem on a Riemann surface.]
{Inverse conductivity problem on a Riemann surface.}

\author{Peter L. Polyakov}
\email{polyakov@uwyo.edu}



\begin{abstract}
We present an application of the {\it Faddeev-Henkin exponential ansatz} and of the
{\it $\partial$-to-$\bar\partial$ map} on the boundary to inverse conductivity problem
on a bordered Riemann surface in $\C\P^2$.
In our approach we use integral formulas for operator $\bar\partial$ developed in
\cite{12}$\div$\cite{15}, integral formulas for holomorphic functions on Riemann surfaces from \cite{22}, and introduce  the {\it $\partial$-to-$\bar\partial$ map} as an important tool in the inverse conductivity problem.
\end{abstract}

\maketitle


\section{\bf Introduction.}

\indent
Let ${\cal C}$ be a Riemann surface
\begin{equation}\label{ProjectiveCurve}
{\cal C}=\left\{z\in \C\P^2:\ P(z)=0\right\}
\end{equation}
defined by the polynomial $P$ of degree $d$, let
\begin{equation}\label{InteriorCurve}
\mathring{\cal C}\subset{\cal C}={\cal C}\setminus{\cal C}_{\infty},
\end{equation}
where
\begin{equation}\label{InfinityPoints}
{\cal C}_{\infty}={\cal C}\cap \left\{z\in\C\P^2: z_0=0\right\}=\left\{z^{(0)},\dots,z^{(d-1)}\right\}
\end{equation}
is the set of points of ${\cal C}$ at infinity, and let
\begin{equation}\label{VCurve}
V=\left\{z\in {\cal C}: \varrho(z)<0\right\}={\cal C}\setminus\bigcup_{r=1}^m V_r,
\end{equation}
be a subdomain in ${\cal C}$, where $\varrho$ is a smooth function on ${\cal C}$,
and $\left\{V_r\right\}_{r=1}^m$, $(m\leq d)$
is a collection of disjoint neighborhoods in ${\cal C}$ of the points at infinity.

\indent
We consider on $V$ the equation
\begin{equation}\label{OriginalEquation}
\d\sigma\d^c U=0,
\end{equation}
where $\d=\partial+\bar\partial$, $\d^c=i(\bar\partial-\partial)$, and $\sigma>0$ is the conductivity function satisfying the equality
\begin{equation}\label{Identity}
\sigma(z)\equiv 1\ \text{for}\ z\in {\cal C}\setminus V.
\end{equation}
The specific problem that is considered
in the present article is a solution of the inverse conductivity problem on a bordered
Riemann surface $V$, in which the conductivity function has to be reconstructed from the
\lq\lq boundary measurements\rq\rq. As such, the problem was posed by I.M. Gelfand in \cite{8},
in contrast with the one-dimensional inverse conductivity problem in \cite{9},
where the conductivity is reconstructed from the spectral data. Later, the problem was made
significantly more precise by A.P. Calderon \cite{4}, as the reconstruction of conductivity
$\sigma$ on $V$ from the {\it Dirichlet-to-Neumann map} on its boundary $bV$. The bibliography on this subject is quite extensive, and we refer the reader to articles by G.M. Henkin and R. G. Novikov \cite{11}, A.L. Bukhgeim \cite{3}, and C. Guillarmou and L. Tzou \cite{10} for a representative bibliography.\\
\indent
In the present article we suggest a modification of the {\it Gelfand-Calderon problem} of the reconstruction of conductivity on a bordered Riemann surface with the usage of the
{\it Faddeev-Henkin exponential ansatz} (see \eqref{fAnsatz}) and of the
{\it $\partial$-to-$\bar\partial$ map} (see \eqref{chiDefinition}) on its boundary.

\indent
The main goal of the article is the proof of the following theorem.
\begin{theorem}\label{Main}Let the form
${\dis q=\frac{\partial\bar\partial\sqrt{\sigma}}{\sqrt{\sigma}}\in C_{(1,1)}(V) }$
on a Riemann surface $V\subset \C\P^2$ as in \eqref{VCurve} be such that $0$ is not a Dirichlet eigenvalue for the operator $\partial\bar\partial f-f\cdot q$. Then the form $q$ can be determined from
a system of two Fredholm-type integral equations on $V$ with the usage of the Faddeev-Henkin exponential ansatz and of the $\partial$-to-$\bar\partial$ map on its boundary $bV$.
\end{theorem}

\indent
The main technical tools in the article are developed for the H$\ddot{\mathrm{\bf o}}$lder-type norms $\|\cdot\|_{\Lambda^{\delta,\alpha}(U)}$, where $U\subset \mathring{\cal C}$
is a relatively compact subdomain $U\subset \mathring{\cal C}$ and $\alpha\in (0,1)$.
Those H$\ddot{\mathrm{\bf o}}$lder-type spaces of functions and forms on $U$ are defined as follows:
\begin{equation*}
\begin{aligned}
&\Lambda^{\delta,\alpha}(U)=\left\{f\in C(U):
\left|\frac{f(z^{(1)})-f(z^{(2)})}{|z^{(1)}-z^{(2)}|^{\alpha}}\right|\leq C\right\}\ \text{for}\
z^{(1)}\neq z^{(2)} \in U,\ |z^{(1)}-z^{(2)}|\leq\delta,\\
&\|f\|_{\Lambda^{\delta,\alpha}(U)}
=\sup_{|z^{(1)}-z^{(2)}|\leq\delta}
\left|\frac{f(z^{(1)})-f(z^{(2)})}{|z^{(1)}-z^{(2)}|^{\alpha}}\right|+\sup_{z\in U}|f(z)|,\\
&\Lambda^{\delta,\alpha}_{(p,q)}(U)=
\left\{g\in C_{(p,q)}(U):\ \text{coefficients of the form $g$ are in}\
\Lambda^{\delta,\alpha}(U)\right\}.
\end{aligned}
\end{equation*}
\indent
The choice of the H$\ddot{\mathrm{\bf o}}$lder-type spaces was motivated by the application of such spaces in the proof of solvability of the Beltrami equation
by L. Lichtenstein and A. Korn in respectively \cite{19, 17},
as presented in the book of R. Courant \cite{5}, and by the discussion of solvability
of the Beltrami equation by L. Ahlfors in his book \cite{1}.\\
\indent
The author would like to thank Roman Novikov for useful discussions of the topics in the article, specifically the issues of: 1) keeping the value of $\lambda$ in the estimates of section \ref{sec:Solving} large, but finite, and 2) (possibly related to 1)) the issue of exponential instability of the reconstruction with the usage of the Dirichlet-to-Neumann map discussed in the articles
\cite{2, 20, 24}. These articles, though becoming known to the author after completing the present article, support the author's intention to look for an alternative to the A. Calderon's approach. In contrast with the reconstruction of potential $q$ based on the Dirichlet-to-Neumann map, in the described below $\partial$-to-$\bar\partial$ approach we start with a solution of equation \eqref{fEquation} that exists by Proposition~\ref{Solvability}, with the goal of construction of integral equations \eqref{PEquation} and \eqref{hEquation} that do not contain potential $q$. The scope of application of the $\partial$-to-$\bar\partial$ map in this approach is relatively narrow, since it is used only to determine the value of $\bar\partial f$ on the boundary from the known value of $\partial f$ on the boundary.

Because of the size of the article we shortly describe below its content.
\begin{itemize}
\item
In section \ref{sec:transformation}, following \cite{11}, we describe the transformation of equation
\eqref{OriginalEquation} into a more manageable equation \eqref{fEquation}
with respect to function $f=U\sqrt{\sigma}$.
\item
In sections \ref{sec:formula} and \ref{sec:operator} we describe the construction
of the integral operator $R_{\lambda}$ for solving the direct problem for equation
\eqref{fEquation} based on the one hand on the
\textit{Faddeev-Henkin exponential ansatz} (see \eqref{fAnsatz}), and on the other hand
on the integral formulas from \cite{12}$\div$\cite{15}.
\item
In section \ref{sec:REstimates} we prove necessary estimates for the operator $R_{\lambda}$,
and then, in Propositions~\ref{RlambdaEstimate} and \ref{Solvability}, prove the existence
of a solution of the direct problem for equation \ref{fEquation} for $\lambda$ large enough
and $q\in C_{(1,1)}(V)$ being a form with compact support.
\item
In section \ref{sec:fEquation} using Proposition~\ref{Solvability} we obtain equality \ref{fFormula}, which is similar to the formula obtained in \cite{11}, but is using the operator $R_{\lambda}$, introduced in section \ref{sec:operator}. Then, we prove Lemma~\ref{ZeroLimit} that allows the reformulation of equality \ref{fFormula} into the integral equation \eqref{hIntegralEquation}.  Integral equation \eqref{hIntegralEquation} is not a standard integral equation with respect to a function
or a form, but it is transformed into the integral equation \eqref{PEquation}
with respect to the form $\partial h$, where
$h(z,\lambda)=f(z,\lambda)e^{-\overline{\left\langle\lambda,z/z_0\right\rangle}}$,
and $f$ is the solution of equation \eqref{fEquation} from section \ref{sec:transformation}.
\item
In sections \ref{sec:BoundaryEstimates} and \ref{sec:DomainEstimates} we prove estimates
for the boundary (section \ref{sec:BoundaryEstimates}) and the domain
(section \ref{sec:DomainEstimates})
integrals in the right-hand side of the integral equation \eqref{hIntegralEquation}. Several
auxiliary estimates in the beginning of section \ref{sec:BoundaryEstimates} are formulated
so that they can be used for both cases.
\item
In section \ref{sec:Solving} we complete the first step in the construction of solution of the inverse conductivity problem. Namely, in Proposition~\ref{PEquationFredholm} using
Lemmas~\ref{FHolder} and \ref{QLemma} we introduce additional conditions on
$\alpha$ and $\delta$, and using estimates from sections \ref{sec:BoundaryEstimates} and \ref{sec:DomainEstimates} obtain the Fredholm property for the
integral operator in equation \eqref{PEquation}. The validity of transformation of equality
\eqref{hIntegralEquation} into the integral equation \eqref{PEquation} is proved in Lemma~\ref{PartialzLemma}.
\item
In section \ref{sec:Solvability} we complete the second step in the construction of solution of the
inverse conductivity problem by obtaining the Fredholm-type integral equation
\eqref{hEquation} on the space $\Lambda^{\alpha}(V)$. The role of the
$\partial$-to-$\bar\partial$ map $\chi$ defined in \eqref{chiDefinition} becomes clear
in this section, where we use it for a solution $f$ of equation \eqref{fEquation}. Several equalities
and estimates in this section are the analogues of the similar equalities and estimates developed in \cite{22}. 
\end{itemize}

\section{\bf Transformation of equation \eqref{OriginalEquation}.}\label{sec:transformation}

\indent
Following \cite{11} we make the substitution $U=f/\sqrt{\sigma}$ in the equation \eqref{OriginalEquation} and obtain
\begin{multline}\label{Transformation}
\d\sigma\d^cU=\d\left[\frac{\sigma}{\sqrt{\sigma}}i\left[\bar\partial f-\partial f\right]
-\frac{\sigma f}{2}\frac{i\left[\bar\partial\sigma-\partial\sigma\right]}{\sigma^{3/2}}\right]\\
=2i\sqrt{\sigma}\partial\bar\partial f+\frac{i}{2\sqrt{\sigma}}
\left(\bar\partial\sigma+\partial\sigma\right)\wedge\left(\bar\partial f-\partial f\right)
-\frac{i}{2\sqrt{\sigma}}\left(\bar\partial f+\partial f\right)
\wedge\left(\bar\partial\sigma-\partial\sigma\right)\\
-i\frac{f}{\sqrt{\sigma}}\partial\bar\partial\sigma
+i\frac{f}{4}\frac{\left(\bar\partial\sigma+\partial\sigma\right)
\wedge\left(\bar\partial\sigma-\partial\sigma\right)}{\sigma^{3/2}}\\
=2i\sqrt{\sigma}\partial\bar\partial f-if\left(\frac{\partial\bar\partial\sigma}{\sqrt{\sigma}}
-\frac{\partial\sigma\wedge\bar\partial\sigma}{2\sigma^{3/2}}\right),
\end{multline}
where in the last equality we used the equality
\begin{multline*}
\left(\bar\partial\sigma+\partial\sigma\right)\wedge\left(\bar\partial f-\partial f\right)
-\left(\bar\partial f+\partial f\right)\wedge\left(\bar\partial\sigma-\partial\sigma\right)\\
=\partial\sigma\wedge\bar\partial f-\bar\partial\sigma\wedge \partial f
-\partial f\wedge\bar\partial\sigma+\bar\partial f\wedge\partial\sigma=0.
\end{multline*}
\indent
From \eqref{Transformation} we obtain that substitution $U=f/\sqrt{\sigma}$
transforms equation \eqref{OriginalEquation} into the following equation with respect to $f$:
\begin{equation*}
\partial\bar\partial f-f\left(\frac{\partial\bar\partial\sigma}{2\sigma}
-\frac{\partial\sigma\wedge\bar\partial\sigma}{4\sigma^2}\right)=0,
\end{equation*}
or, using equality
\begin{equation*}
\frac{\partial\bar\partial\sqrt{\sigma}}{\sqrt{\sigma}}
=\frac{1}{\sqrt{\sigma}}\left(\partial\left[\frac{\bar\partial\sigma}
{2\sqrt{\sigma}}\right]\right)
=\frac{\partial\bar\partial{\sigma}}{2\sigma}
-\frac{\partial\sigma\wedge\bar\partial\sigma}{4\sigma^2},
\end{equation*}
into equation
\begin{equation}\label{fEquation}
\partial\bar\partial f-f\cdot q=0
\end{equation}
with ${\dis q=\frac{\partial\bar\partial\sqrt{\sigma}}{\sqrt{\sigma}} }$.

\indent
On the next step we consider the form $q$ in \eqref{fEquation} as a residual current on $\C\P^2$
in the sense of \cite{12, 13} and apply the integral formulas to this current with support
on the Riemann surface ${\cal C}\subset \C\P^2$, which requires a \lq\lq lift\rq\rq\ of the form
to a subdomain in $\C^3\setminus\left\{0\right\}$ with some homogeneity $\ell$.
In the lemma below we address the lifting problem.
\begin{lemma}\label{LiftingForm}
Let $\phi^{(0,1)}\in C_{(0,1)}(V)$ be a differential form with compact support on the
open Riemann surface
$V\subset \C\P^2\setminus \C\P^1_{\infty}$ as in \eqref{VCurve}, let $\ell\in \Z$, and let
$\tau: \C^3\setminus\left\{0\right\}\to \C\P^2$ be the standard projection.
Then there exist an open domain
$U\subset \C\P^2$, such that $V\subset U$, and differential forms: $\psi^{(0,1)}\in C_{(0,1)}(U)$ with compact support, and $\Psi^{(0,1)}=\tau^*(\psi)$
on $\tau^{-1}(U)\subset \C^3\setminus\left\{0\right\}$ satisfying
\begin{equation}\label{ExtensionExistence}
\begin{aligned}
&\psi|_{V}=z_0^{\ell}\cdot\phi,\\
&\Psi(\lambda z)=\lambda^{\ell}\cdot \Psi(z).
\end{aligned}
\end{equation}
\end{lemma}
\begin{proof}
We assume that the restriction of the map $\tau: \C^3\setminus\left\{0\right\}\to \C\P^2$
to $\C^2=\C\P^2\setminus\C\P^1_{\infty}$ is defined as
\begin{equation*}
\tau(z_0,z_1,z_2)=(z_1/z_0,z_2/z_0),
\end{equation*}
so that points of the form $(\lambda z_0,\lambda z_1,\lambda z_2)$ are mapped into the same point.
For the differential form
\begin{equation*}
\phi^{(0,1)}(z)=\phi_1\left(\frac{z_1}{z_0},\frac{z_2}{z_0}\right)d\left(\frac{{\bar z}_1}{{\bar z}_0}\right)
+\phi_2\left(\frac{z_1}{z_0},\frac{z_2}{z_0}\right)d\left(\frac{{\bar z}_2}{{\bar z}_0}\right).
\end{equation*}
we construct the necessary lift in two steps.
On the first step
we use the triviality of the holomorphic normal bundle of $V$ in
$\C^2$ (see \cite{7}) and extend the coefficients of $\phi^{(0,1)}$ identically along
the fibers of this bundle to obtain a form with compact support in some neighborhood of $V$ in $\C^2$,
which we also denote by $\phi^{(0,1)}$:
\begin{multline*}
\phi^{(0,1)}=\phi_1\left(\frac{z_1}{z_0},\frac{z_2}{z_0}\right)d\left(\frac{{\bar z}_1}{{\bar z}_0}\right)
+\phi_2\left(\frac{z_1}{z_0},\frac{z_2}{z_0}\right)d\left(\frac{{\bar z}_2}{{\bar z}_0}\right)\\
=\phi_1\left(\frac{z_1}{z_0},\frac{z_2}{z_0}\right)\cdot\left(\frac{d{\bar z}_1}{{\bar z}_0}
-\frac{{\bar z}_1}{{\bar z}_0^2}d{\bar z}_0\right)
+\phi_2\left(\frac{z_1}{z_0},\frac{z_2}{z_0}\right)\cdot\left(\frac{d{\bar z}_2}{{\bar z}_0}
-\frac{{\bar z}_2}{{\bar z}_0^2}d{\bar z}_0\right)\\
=\phi_1\left(\frac{z_1}{z_0},\frac{z_2}{z_0}\right)\cdot\frac{1}{{\bar z}_0}d{\bar z}_1
+\phi_2\left(\frac{z_1}{z_0},\frac{z_2}{z_0}\right)\cdot\frac{1}{{\bar z}_0}d{\bar z}_2
-\frac{1}{{\bar z}_0^2}\left(\phi_1\left(\frac{z_1}{z_0},\frac{z_2}{z_0}\right){\bar z}_1
+\phi_2\left(\frac{z_1}{z_0},\frac{z_2}{z_0}\right){\bar z}_2\right)d{\bar z}_0\\
={\bar z}_0^{-1}\sum_{i=1}^2\phi_i\left(\frac{z_1}{z_0},\frac{z_2}{z_0}\right)d{\bar z}_i
-\sum_{i=1}^2\frac{{\bar z}_i}{{\bar z}_0}\phi_i\left(\frac{z_1}{z_0},\frac{z_2}{z_0}\right)
\frac{d{\bar z}_0}{{\bar z}_0},
\end{multline*}
where the last form satisfies conditions of Proposition 1.1 from \cite{12} for
$\phi=\sum_{i=0}^2 \phi_i d{\bar z}^i$
\begin{equation*}
\left\{\begin{aligned}
&\phi_i(t\cdot z) = t^{\ell=0}\cdot {\bar t}^{-1} \phi_i(z),\\
&L_1 \phi =\sum_{i=0}^2{\bar z}^i\phi_i=0.
\end{aligned}\right.
\end{equation*}

\indent
Then, on the second step we define $\psi^{(0,1)}=z_0^{\ell}\cdot \phi^{(0,1)}$,
and $\Psi^{(0,1)}=\tau^*(\psi)$.
\end{proof}

\section{\bf Formula for a solution of \eqref{fEquation}.}\label{sec:formula}

\indent
To construct a solution of equation \eqref{fEquation} we use the {\it Faddeev-Henkin exponential ansatz} and consider solutions of the following form
\begin{equation}\label{fAnsatz}
f(\zeta,\lambda)=\mu(\zeta,\lambda)e^{\left\langle\lambda,\zeta/\zeta_0\right\rangle},
\end{equation}
where $\lambda\in \C$ is a parameter, whose role will become clear later in the estimates of the obtained solution, and
$\left\langle\lambda,\zeta/\zeta_0\right\rangle=\lambda\left(\frac{\zeta_1}{\zeta_0}
+\frac{\zeta_2}{\zeta_0}\right)$.
The exponential ansatz \eqref{fAnsatz} is a multidimensional generalization for
Riemann surfaces of the ansatz introduced and used by L. Faddeev for $\C^1$ in \cite{6},
probably partially motivated by the work of T. Kato \cite{16}. It was then introduced and used by
R. Novikov and G. Henkin in a multi-dimensional case in the articles \cite{19, 10}.
Those articles motivated the application of the exponential ansatz in the present article.\\
\indent
Substituting $f$ from \eqref{fAnsatz} into equation \eqref{fEquation}
we obtain the following equation for the function $\mu$
\begin{multline*}
\partial\bar\partial f-f\cdot q=\partial\bar\partial\left(\mu(\zeta,\lambda)
e^{\left\langle\lambda,\zeta/\zeta_0\right\rangle}\right)
-e^{\left\langle\lambda,\zeta/\zeta_0\right\rangle}\mu\cdot q\\
=e^{\left\langle\lambda,\zeta/\zeta_0\right\rangle} \partial\bar\partial \mu
+e^{\left\langle\lambda,\zeta/\zeta_0\right\rangle}
\left\langle\lambda,\d\left(\zeta/\zeta_0\right)\right\rangle\wedge\bar\partial\mu
-e^{\left\langle\lambda,\zeta/\zeta_0\right\rangle}\mu\cdot q=0,
\end{multline*}
which we rewrite as
\begin{equation}\label{muEquation}
\partial\bar\partial\mu
+\left\langle\lambda, \d\left(\zeta/\zeta_0\right)\right\rangle
\wedge\bar\partial\mu=\mu\cdot q.
\end{equation}

\indent
Below we consider the solvability of equation
\begin{equation}\label{partial-lambda}
\bar\partial\left[\partial u+u\cdot\left\langle\lambda, \d\left(z/z_0\right)\right\rangle\right]
=\phi_1^{(0,1)}(z)\d\left(\frac{z_1}{z_0}\right)
+\phi_2^{(0,1)}(z)\d\left(\frac{z_2}{z_0}\right)
\end{equation}
on the affine curve
$\mathring{\mathcal C}={\mathcal C}\setminus\left({\cal C}\cap\C\P^1_{\infty}\right)$, where
\begin{equation}\label{phiForm}
\phi^{(1,1)}(z)=\phi_1^{(0,1)}(z)\d\left(\frac{z_1}{z_0}\right)
+\phi_2^{(0,1)}(z)\d\left(\frac{z_2}{z_0}\right)
\end{equation}
is a form on ${\cal C}$ of homogeneity zero with compact support outside
of a neighborhood of the line
$$\C\P^1_{\infty}=\left\{z\in \C\P^2: z_0=0\right\}.$$

\indent
Since equation \eqref{partial-lambda} consists of two operators:
${\dis \partial+\cdot\left\langle\lambda, \d\left(\zeta/\zeta_0\right)\right\rangle }$
and $\bar\partial$, consecutively applied to a function
$u$, we are seeking solutions of two equations for operators $\bar\partial$ and $\partial$.
Our tool in solving this two first order equations will be the following proposition for
projective curves in $\C\P^2$, which is a corollary of Theorem 1 from \cite{15}.

\begin{proposition}\label{dbarSolvability}Let ${\cal C}$ and $V\subset{\cal C}$ be as in
\eqref{ProjectiveCurve} and \eqref{VCurve} respectively.
Let $\phi\in Z_R^{(0,1)}\left({\cal C}, {\cal O}(\ell)\right)$ be a residual current on ${\cal C}$
with coefficients in ${\cal O}(\ell)$ with $\ell>d-3$.
Then equation
\begin{equation}\label{HomogeneousHomotopy}
\bar\partial\psi=\phi
\end{equation}
is satisfied by the residual current $\psi\in C\left({\cal C}, {\cal O}(\ell)\right)$
defined by the formula
\begin{multline}\label{OriginalFormula}
\psi(z)=I_1[\phi](z)\\
\stackrel{\text{def}}{=}
\frac{2}{(2\pi i)^3}\lim_{\epsilon\to 0}\int_{\Gamma^{\epsilon}_{\zeta}\times\Delta^2}\phi(\zeta)
\wedge\omega_0^{\prime}\left((1-\lambda-\mu)\frac{\bar{z}}{B^*(\zeta,z)}
+\lambda\frac{\bar\zeta}{B(\zeta,z)}+\mu\frac{Q(\zeta,z)}{P(\zeta)}\right)\wedge\d\zeta,
\end{multline}
where
$$\Gamma^{\epsilon}_{\zeta}=\left\{|\zeta|=1, |P(\zeta)|=\epsilon\right\},\
\Delta^2=\left\{(\lambda,\mu\geq 0):\lambda+\mu\leq 1\right\},$$
functions $\left\{Q^i(\zeta,z)\right\}$ for $i=0\div 2$ satisfy
\begin{equation}\label{QFunctions}
\left\{\begin{array}{ll}
P(\zeta)-P(z)=\sum_{i=0}^2Q^i(\zeta,z)\cdot\left(\zeta_i-z_i\right),\vspace{0.1in}\\
Q^i(\lambda\zeta,\lambda z)=\lambda^{d-1}\cdot Q^i(\zeta,z)\
\mbox{for}\ \lambda\in\C,
\end{array}\right.
\end{equation}
\begin{equation}\label{BDefinition}
B^*(\zeta,z)=\sum_{j=0}^2{\bar z}_j\cdot\left(\zeta_j-z_j\right)
=-1+\sum_{j=0}^2{\bar z}_j\zeta_j,
\hspace{0.1in}
B(\zeta,z)=\sum_{j=0}^2{\bar\zeta}_j\cdot\left(\zeta_j-z_j\right)=1-\sum_{j=0}^2\bar\zeta_j z_j,
\end{equation}
and
\begin{equation}\label{OmegaPrimeForm}
\begin{array}{ll}
\omega_0^{\prime}(\eta_0,\eta_1,\eta_2)
=\eta_0\d_{\lambda,\mu}\eta_1\wedge \d_{\lambda,\mu}\eta_2
-\eta_1\d_{\lambda,\mu}\eta_0\wedge \d_{\lambda,\mu}\eta_2
+\eta_2\d_{\lambda,\mu}\eta_0\wedge \d_{\lambda,\mu}\eta_1,\\
\vspace{0.02in}\\
\d\zeta=\d\zeta_0\wedge\d\zeta_1\wedge\d\zeta_2.
\end{array}
\end{equation}
\qed
\end{proposition}

\indent
In the next lemma we transform formula \eqref{OriginalFormula} into the \lq\lq determinantal form\rq\rq\, which will be more convenient in the formulas and estimates below.

\begin{lemma}\label{OriginalTransformation}
Under the conditions of Proposition~\ref{dbarSolvability} the following equality holds:
\begin{multline}\label{LOperatorCP2}
\frac{2}{(2\pi i)^3}\int_{\Gamma^{\epsilon}_{\zeta}\times\Delta^2}
\phi^{(0,1)}(\zeta)\wedge 
\omega^{\prime}_0\left((1-\lambda-\mu)
\frac{\bar z}{B^*(\zeta,z)}
+\lambda\frac{\bar\zeta}{B(\zeta,z)}+\mu\frac{Q(\zeta,z)}{P(\zeta)}\right)\wedge\d\zeta\\
=\frac{1}{3(2\pi i)^3}\int_{\Gamma^{\epsilon}_{\zeta}}\phi^{(0,1)}(\zeta)\wedge K(\zeta,z)\d\zeta,
\end{multline}
where
\begin{equation}\label{KDefinition}
K(\zeta,z)=\det\left[\frac{\bar z}{B^*(\zeta,z)}\ \frac{\bar\zeta}{B(\zeta,z)}\ Q(\zeta,z)\right]
\stackrel{def}{=}\det\left[\begin{tabular}{ccc}
${\dis \frac{{\bar z}_0}{B^*(\zeta,z)} }$&${\dis \frac{{\bar \zeta}_0}{B(\zeta,z)} }$&
$Q^0(\zeta,z)$\vspace{0.05in}\\
${\dis \frac{{\bar z}_1}{B^*(\zeta,z)} }$&${\dis \frac{{\bar \zeta}_1}{B(\zeta,z)} }$&
$Q^1(\zeta,z)$\vspace{0.05in}\\
${\dis \frac{{\bar z}_2}{B^*(\zeta,z)} }$&${\dis \frac{{\bar \zeta}_2}{B(\zeta,z)} }$&
$Q^2(\zeta,z)$
\end{tabular}\right].
\end{equation}
\end{lemma}
\begin{proof}
\indent
We use equality
\begin{multline*}
\omega^{\prime}_0(\eta)=\eta_0\d_{\lambda,\mu}\eta_1\wedge \d_{\lambda,\mu}\eta_2
-\eta_1\d_{\lambda,\mu}\eta_0\wedge \d_{\lambda,\mu}\eta_2
+\eta_2\d_{\lambda,\mu}\eta_0\wedge \d_{\lambda,\mu}\eta_1\\
=\frac{1}{2}\det\left[\begin{tabular}{ccc}
$\eta_0$&$\d_{\lambda,\mu}\eta_0$&$\d_{\lambda,\mu}\eta_0$
\vspace{0.05in}\\
$\eta_1$&$\d_{\lambda,\mu}\eta_1$&$\d_{\lambda,\mu}\eta_1$\vspace{0.05in}\\
$\eta_2$&$\d_{\lambda,\mu}\eta_2$&$\d_{\lambda,\mu}\eta_2$
\end{tabular}\right]
=\frac{1}{2}\det\left[\begin{tabular}{ccc}
$\eta_0$&$(\d_{\lambda}\eta_0+\d_{\mu}\eta_0)$&$(\d_{\lambda}\eta_0+\d_{\mu}\eta_0)$
\vspace{0.05in}\\
$\eta_1$&$(\d_{\lambda}\eta_1+\d_{\mu}\eta_1)$&$(\d_{\lambda}\eta_1+\d_{\mu}\eta_1)$
\vspace{0.05in}\\
$\eta_2$&$(\d_{\lambda}\eta_2+\d_{\mu}\eta_2)$&$(\d_{\lambda}\eta_2+\d_{\mu}\eta_2)$
\end{tabular}\right]\\
=\frac{1}{2}\left(\det\left[\begin{tabular}{ccc}
$\eta_0$&$\d_{\lambda}\eta_0$&$\d_{\mu}\eta_0$
\vspace{0.05in}\\
$\eta_1$&$\d_{\lambda}\eta_1$&$\d_{\mu}\eta_1$
\vspace{0.05in}\\
$\eta_2$&$\d_{\lambda}\eta_2$&$\d_{\mu}\eta_2$
\end{tabular}\right]
+\det\left[\begin{tabular}{ccc}
$\eta_0$&$\d_{\mu}\eta_0$&$\d_{\lambda}\eta_0$
\vspace{0.05in}\\
$\eta_1$&$\d_{\mu}\eta_1$&$\d_{\lambda}\eta_1$
\vspace{0.05in}\\
$\eta_2$&$\d_{\mu}\eta_2$&$\d_{\lambda}\eta_2$
\end{tabular}\right]\right)
=\det\left[\begin{tabular}{ccc}
$\eta_0$&$\d_{\lambda}\eta_0$&$\d_{\mu}\eta_0$
\vspace{0.05in}\\
$\eta_1$&$\d_{\lambda}\eta_1$&$\d_{\mu}\eta_1$
\vspace{0.05in}\\
$\eta_2$&$\d_{\lambda}\eta_2$&$\d_{\mu}\eta_2$
\end{tabular}\right].
\end{multline*}

\indent
Then, substituting the values of $\eta_0,\ \eta_1,\ \eta_2$ into the right-hand side of equality above,
we obtain
\begin{multline}\label{DetFormula}
\omega_0^{\prime}(\eta)
=\det\left[\left((1-\lambda-\mu)
\frac{\bar z}{B^*(\zeta,z)}+\lambda\frac{\bar\zeta}{B(\zeta,z)}+\mu\frac{Q(\zeta,z)}{P(\zeta)}\right)\
\frac{\bar\zeta}{B(\zeta,z)}\ \frac{Q(\zeta,z)}{P(\zeta)}\right]d\lambda\wedge d\mu\\
=\det\left[(1-\lambda-\mu)\frac{\bar z}{B^*(\zeta,z)}\
\frac{\bar\zeta}{B(\zeta,z)}\ \frac{Q(\zeta,z)}{P(\zeta)}\right]d\lambda\wedge d\mu\\
=\det\left[\frac{\bar z}{B^*(\zeta,z)}\ \frac{\bar\zeta}{B(\zeta,z)}\ \frac{Q(\zeta,z)}{P(\zeta)}\right]
(1-\lambda-\mu)d\lambda\wedge d\mu.
\end{multline}
Computing integral
\begin{multline*}
\int_{\Delta^2}(1-\lambda-\mu)d\lambda\wedge d\mu
=\int_0^1d\lambda\int_0^{1-\lambda}(1-\lambda-\mu)d\mu
=\int_0^1d\lambda\left[(1-\lambda)^2-\frac{(1-\lambda)^2}{2}\right]\\
=\int_0^1\frac{(1-\lambda)^2}{2}d\lambda=-\frac{1}{2}\int_1^0t^2dt=\frac{1}{6},
\end{multline*}
we obtain from \eqref{DetFormula} equality \eqref{LOperatorCP2}.
\end{proof}

\indent
We will construct a solution of equation \eqref{partial-lambda} in two steps. On the first step we construct a solution of the $\bar\partial$-equation on some neighborhood of $V$. 
Assuming that $d>2$, denoting
\begin{equation}\label{C1}
{\bf{C}}=\frac{1}{3(2\pi i)^3},
\end{equation}
and applying Proposition~\ref{dbarSolvability} with $\ell>d-3$ to
$\bar\partial$-closed residual currents
$\left\{w^{\ell}\Phi_j(w)\right\}_{j=1,2}=\{w_i^{\ell}\cdot \phi^{(i)}_j(w)\}_{j=1,2}$
for $i=0,1,2$ with compactly supported $(0,1)$-forms $\phi^{(i)}_j$ from \eqref{phiForm}
we obtain that functions
\begin{equation}\label{psiFormula}
\psi_j(\zeta)=\zeta_0^{-\ell}\cdot I_1[w_0^{\ell}\cdot \phi^{(0)}_j(w)](\zeta)\\
={\bf{C}}\cdot\zeta_0^{-\ell}\lim_{\tau\to 0}\int_{\Gamma^{\tau}_w}
w_0^{\ell}\phi_j^{(0)}(w)\frac{K(w,\zeta)}{P(w)}\wedge\d w,
\end{equation}
satisfy equalities
\begin{equation}\label{psiSolutions}
\bar\partial_{\zeta}\psi_j(\zeta)=\phi_j(\zeta)
\end{equation}
on the affine curve $\mathring{\mathcal C}$.\\

\indent
Combining equalities \eqref{psiSolutions} for $j=1,2$ we obtain that the form
\begin{multline}\label{GeneralpsiForm}
\psi^{(1,0)}(\zeta)=\vartheta(\zeta)\left(\psi_1(\zeta)\d\left(\frac{\zeta_1}{\zeta_0}\right)
+\psi_2(\zeta)\d\left(\frac{\zeta_2}{\zeta_0}\right)\right)\\
={\bf{C}}\cdot\vartheta(\zeta)\zeta_0^{-\ell}\cdot\left[\left(\lim_{\tau\to 0}
\int_{\Gamma^{\tau}_w}w_0^{\ell}\cdot\phi_1^{(0,1)}(w)\frac{K(w,\zeta)}{P(w)}
\wedge\d w\right)\d\left(\frac{\zeta_1}{\zeta_0}\right)\right.\\
+\left.\left(\lim_{\tau\to 0}\int_{\Gamma^{\tau}_w}w_0^{\ell}
\cdot\phi_2^{(0,1)}(w)\frac{K(w,\zeta)}{P(w)}\wedge\d w\right)
\d\left(\frac{\zeta_2}{\zeta_0}\right)\right],
\end{multline}
where $\ell>d-3$, and $\vartheta(\zeta)$ is a smooth function on $\C\P^2$ satisfying
\begin{equation*}
\left\{\begin{aligned}
&\vartheta(\zeta)\equiv 1\ \text{for}\ \left\{\zeta: |P(\zeta)|<\epsilon, \varrho(z)<\epsilon\right\},\\
&\vartheta(\zeta)\equiv 0\ \text{for}\ \left\{\zeta: |P(\zeta)|>2\epsilon,\ \text{or}\ \varrho(z)>2\epsilon\right\},
\end{aligned}\right.
\end{equation*}
for some $\epsilon>0$, satisfies the equality
\begin{equation}\label{PsiEquality}
\bar\partial_{\zeta}\psi^{(1,0)}(\zeta)=\bar\partial\left(\psi_1(\zeta)\d\left(\frac{\zeta_1}{\zeta_0}\right)
+\psi_2(\zeta)\d\left(\frac{\zeta_2}{\zeta_0}\right)\right)=\phi^{(1,1)}(\zeta)
\end{equation}
on $V$.

\indent
On the next step we construct a solution of the equation
\begin{equation}\label{dlambdaEquation}
\partial u+u\cdot\left\langle\lambda, \d\left(z/z_0\right)\right\rangle=\psi^{(1,0)}
\end{equation}
on $V$. To find a solution of this equation we consider the complex conjugate of the operator $I_1$:
\begin{equation}\label{Ibar}
\bar{I}_1[\psi^{(1,0)}](z)=\overline{\bf{C}}\lim_{\epsilon\to 0}
\int_{\Gamma^{\epsilon}_{\zeta}}
\psi^{(1,0)}(\zeta)\frac{K(\bar{\zeta},{\bar z})}{P(\bar{\zeta})}\wedge\d\bar\zeta
\end{equation}
and for $\ell>d-3$ consider the function
\begin{multline}\label{uFunction}
u(z,\lambda)=\bar{z}_0^{-\ell}\cdot e^{-\langle\lambda,z/z_0\rangle+\overline{\langle\lambda,z/z_0\rangle}}
\cdot\bar{I}_1[\bar\zeta_0^{\ell}
\cdot e^{\langle\lambda,\zeta/\zeta_0\rangle-\overline{\langle\lambda,\zeta/\zeta_0\rangle}}
\psi^{(1,0)}(\zeta)]\\
=\overline{\bf{C}}\bar{z}_0^{-\ell}
\cdot e^{-\langle\lambda,z/z_0\rangle+\overline{\langle\lambda,z/z_0\rangle}}
\lim_{\epsilon\to 0}\int_{\Gamma^{\epsilon}_{\zeta}}
\bar\zeta_0^{\ell}\cdot e^{\langle\lambda,\zeta/\zeta_0\rangle
-\overline{\langle\lambda,\zeta/\zeta_0\rangle}}
\psi^{(1,0)}(\zeta)\frac{K(\bar{\zeta},{\bar z})}{P(\bar{\zeta})}\wedge\d\bar\zeta.
\end{multline}
Then, using the complex conjugate version of Proposition~\ref{dbarSolvability} we obtain that
function $u$ satisfies on $V$ the equality
\begin{multline*}
\partial_{z}u(z,\lambda)
=-\bar{z}_0^{-\ell}\cdot e^{-\langle\lambda,z/z_0\rangle+\overline{\langle\lambda,z/z_0\rangle}}
\cdot\bar{I}_1[\bar\zeta_0^{\ell}
\cdot e^{\langle\lambda,\zeta/\zeta_0\rangle-\overline{\langle\lambda,\zeta/\zeta_0\rangle}}
\psi^{(1,0)}(\zeta)]\left\langle\lambda, \d\left(z/z_0\right)\right\rangle\\
+\bar{z}_0^{-\ell}\cdot e^{-\langle\lambda,z/z_0\rangle+\overline{\langle\lambda,z/z_0\rangle}}
\partial_z\bar{I}_1[\bar\zeta_0^{\ell}
\cdot e^{\langle\lambda,\zeta/\zeta_0\rangle-\overline{\langle\lambda,\zeta/\zeta_0\rangle}}
\psi^{(1,0)}(\zeta)]\\
=-u(z,\lambda)\cdot\left\langle\lambda, \d\left(z/z_0\right)\right\rangle
+\bar{z}_0^{-\ell}\cdot e^{-\langle\lambda,z/z_0\rangle+\overline{\langle\lambda,z/z_0\rangle}}
\bar{z}_0^{\ell}e^{\langle\lambda,z/z_0\rangle-\overline{\langle\lambda,z/z_0\rangle}}
\psi^{(1,0)}(z)]\\
=-u(z,\lambda)\cdot\left\langle\lambda, \d\left(z/z_0\right)\right\rangle+\psi^{(1,0)}(z),\\
\end{multline*}
which we rewrite as
\begin{equation}\label{partialSolution}
\partial_{z}u(z,\lambda)+u(z,\lambda)\cdot\left\langle\lambda, \d\left(z/z_0\right)\right\rangle
=\psi^{(1,0)}(z).
\end{equation}
Combining equalities \eqref{PsiEquality} and \eqref{partialSolution} we obtain that function
$u$ in \eqref{uFunction} is a solution of the equation
\begin{multline}\label{uphiEquation}
\bar\partial\left[\partial u+u\cdot
\left\langle\lambda, \d\left(z/z_0\right)\right\rangle\right]
=\bar\partial\partial u+\bar\partial u\wedge
\left\langle\lambda, \d\left(z/z_0\right)\right\rangle=\bar\partial\psi^{(1,0)}=\phi^{(1,1)}\\
=\phi_1^{(0,1)}(z)\d\left(\frac{z_1}{z_0}\right)
+\phi_2^{(0,1)}(z)\d\left(\frac{z_2}{z_0}\right)
\end{multline}
on $V$.\\
\indent
Then, using equality
$$q^{(1,1)}(\zeta)=\d\left(\frac{\zeta_1}{\zeta_0}\right)\wedge q_1^{(0,1)}(\zeta)
+\d\left(\frac{\zeta_2}{\zeta_0}\right)\wedge q_2^{(0,1)}(\zeta),$$
considering
\begin{equation}\label{muq}
\phi^{(1,1)}=-\mu\cdot q^{(1,1)}(\zeta)
=\mu\cdot \left(q_1^{(0,1)}(\zeta)\wedge\d\left(\frac{\zeta_1}{\zeta_0}\right)
+q_2^{(0,1)}(\zeta)\wedge\d\left(\frac{\zeta_2}{\zeta_0}\right)\right),
\end{equation}
and using equalities \eqref{GeneralpsiForm} and \eqref{PsiEquality}
we obtain from equality \eqref{uphiEquation} that for
\begin{multline}\label{muEquality}
\mu(z,\lambda)=\bar{z}_0^{-\ell}\cdot e^{-\langle\lambda,z/z_0\rangle+\overline{\langle\lambda,z/z_0\rangle}}
\cdot\bar{I}_1[\bar\zeta_0^{\ell}
\cdot e^{\langle\lambda,\zeta/\zeta_0\rangle-\overline{\langle\lambda,\zeta/\zeta_0\rangle}}
\psi^{(1,0)}(\zeta)]\\
=\overline{{\bf{C}}}\bar{z}_0^{-\ell}
\cdot e^{-\langle\lambda,z/z_0\rangle+\overline{\langle\lambda,z/z_0\rangle}}
\lim_{\epsilon\to 0}\int_{\Gamma^{\epsilon}_{\zeta}}
\bar\zeta_0^{\ell}\cdot e^{\langle\lambda,\zeta/\zeta_0\rangle
-\overline{\langle\lambda,\zeta/\zeta_0\rangle}}
\psi^{(1,0)}(\zeta)\frac{K(\bar{\zeta},{\bar z})}{P(\bar{\zeta})}\wedge\d\bar\zeta
\end{multline}
with
\begin{multline}\label{psiForm}
\psi^{(1,0)}(\zeta)={\bf{C}}\cdot\vartheta(\zeta)\zeta_0^{-\ell}
\left[\left(\lim_{\tau\to 0}\int_{\Gamma^{\tau}_w}
w_0^{\ell}\cdot \mu q_1^{(0,1)}(w)\frac{K(w,\zeta)}{P(w)}\wedge\d w\right)
\d\left(\frac{\zeta_1}{\zeta_0}\right)\right.\\
+\left.\left(\lim_{\tau\to 0}\int_{\Gamma^{\tau}_w}w_0^{\ell}
\cdot \mu q_2^{(0,1)}(w)\frac{K(w,\zeta)}{P(w)}\wedge\d w\right)
\d\left(\frac{\zeta_2}{\zeta_0}\right)\right],
\end{multline}
the following equality holds on $V$
\begin{equation*}
\bar\partial\left(\partial+\left\langle\lambda, \d\left(z/z_0\right)\right\rangle\right)\mu(z)
=-\mu(z)\cdot q(z).
\end{equation*}

\indent
Reversing the order of differentiation in the left-hand side and the sign in the right-hand side of the
equality above we obtain equality \eqref{muEquation}.

\section{\bf Integral Operator $R_{\lambda}$.}\label{sec:operator}

\indent
From the formulas \eqref{muEquality} and \eqref{psiForm} we obtain the following integral
operator
\begin{multline}\label{IntegralOperator}
R_{\lambda}[\mu\cdot q](z,\lambda)=|{\bf{C}}|^2\cdot\bar{z}_0^{-\ell}
\cdot e^{-\langle\lambda,z/z_0\rangle+\overline{\langle\lambda,z/z_0\rangle}}
\lim_{\epsilon\to 0}\int_{\Gamma^{\epsilon}_{\zeta}}
\bar\zeta_0^{\ell}\zeta_0^{-\ell}\vartheta(\zeta)
e^{\langle\lambda,\zeta/\zeta_0\rangle-\overline{\langle\lambda,\zeta/\zeta_0\rangle}}\\
\times\Bigg[\left(\lim_{\tau\to 0}\int_{\Gamma^{\tau}_{w}}w_0^{\ell}
\cdot \mu(w,\lambda)q_1^{(0,1)}(w)\frac{K(w,\zeta)}{P(w)}
\wedge\d w\right)\d\left(\frac{\zeta_1}{\zeta_0}\right)\\
+\left(\lim_{\tau\to 0}\int_{\Gamma^{\tau}_{w}}w_0^{\ell}
\cdot \mu(w,\lambda)q_2^{(0,1)}(w)\frac{K(w,\zeta)}{P(w)}\wedge\d w\right)
\d\left(\frac{\zeta_2}{\zeta_0}\right)\Bigg]\frac{K(\bar\zeta,\bar{z})}{P(\bar\zeta)}
\wedge\d\bar\zeta.
\end{multline}
\indent
The importance of the operator $R_{\lambda}[\cdot q]$ follows from the lemma below, which
reduces the solvability of the differential equation \eqref{fEquation} to the solvability of an integral equation, and is the main point of the Faddeev-Henkin exponential ansatz.

\begin{lemma}\label{RImportance}
If a function $\mu(z,\lambda)$ satisfies equality 
\begin{equation}\label{IntegralEquation}
\left(I-R_{\lambda}[\cdot q]\right)\mu=1,
\end{equation}
then $\mu(z,\lambda)$ satisfies equality \eqref{muEquation}.
\end{lemma}
\begin{proof}
If a function $\mu(z,\lambda)$ satisfies equality \eqref{IntegralEquation}, which is equivalent to
equality
\begin{equation*}
\mu=1+R_{\lambda}[\mu\cdot q],
\end{equation*}
then we have
\begin{equation*}
\left(\partial+\left\langle\lambda,\d\left(z/z_0\right)\right\rangle\right)\mu
=\left\langle\lambda,\d\left(z/z_0\right)\right\rangle
+\left(\partial+\left\langle\lambda,\d\left(z/z_0\right)\right\rangle\right)R_{\lambda}[\mu\cdot q],
\end{equation*}
and, using equality \eqref{uphiEquation}, we obtain that
\begin{equation}\label{muSolution}
\bar\partial\left(\partial+\left\langle\lambda,\d\left(z/z_0\right)\right\rangle\right)\mu
=\bar\partial\left(\partial+\left\langle\lambda,\d\left(z/z_0\right)\right\rangle\right)
R_{\lambda}[\mu\cdot q]=-\mu\cdot q,
\end{equation}
i.e. $\mu$ satisfies equation \eqref{muEquation}, and
$f(\zeta)=\mu(\zeta,\lambda)e^{\left\langle\lambda,\zeta/\zeta_0\right\rangle}$
satisfies equation \eqref{fEquation}.
\end{proof}

\indent
In the rest of this section and in the next section we transform the integral operator $R_{\lambda}[\mu\cdot q]$
into a more convenient form and prove some estimates.

\indent
To estimate the interior integrals in the right-hand side of \eqref{IntegralOperator} we use the following lemma.
\begin{lemma}\label{InteriorEstimate}
Let $U\supset {\cal C}$ be a neighborhood of ${\cal C}$ in $\C\P^2$, and let
\begin{equation}\label{piNotation}
\pi:\*S^5(1)\to \C\P^2
\end{equation}
denote the restriction to $\*S^5(1)$ of the natural projection $\C^3\setminus{0}\to \C\P^2$.
Then, there exists a constant $C$ such that the following
estimate holds for an arbitrary form $\psi^{(0,1)}\in C_{(0,1)}\left(U\right)$,
points $\zeta^{(1)}, \zeta^{(2)}\in \pi^{-1}(U)$, and $\epsilon$ small enough
\begin{equation}\label{InteriorRIntegralEstimate}
\left|J_{\tau}(\zeta^{(1)})-J_{\tau}(\zeta^{(2)})\right|
\leq C\sqrt{\delta}\cdot\|\psi\|_{C_{(0,1)}(V)},
\end{equation}
where
\begin{equation}\label{JFunction}
J_{\tau}(\zeta)=\int_{\Gamma^{\tau}_{w}}\psi^{(0,1)}(w)K(w,\zeta)
\wedge\frac{\d w}{P(w)},
\end{equation}
$K(w,\zeta)$ is the form from \eqref{KDefinition}, and $\delta=|\zeta^{(1)}-\zeta^{(2)}|$.
\end{lemma}
\begin{proof}
We start with the estimate of the integral in \eqref{JFunction} in a small neighborhood about
$\zeta\in U$. Then, we have:
\begin{multline}\label{FirstSmall}
\left|\int_{\Gamma^{\tau}_{w}\cap\left\{|B(w,\zeta)|<\delta\right\}}
\psi^{(0,1)}(w)K(w,\zeta)\wedge\frac{\d w}{P(w)}\right|\\
=\left|\int_{\Gamma^{\tau}_{w}\cap\left\{|B(w,\zeta)|<\delta\right\}}
\psi^{(0,1)}(w)\det\left[\frac{\bar\zeta}{B^*(w,\zeta)}\ \frac{\bar w}{B(w,\zeta)}\
Q(w,\zeta)\right]\wedge\frac{\d w}{P(w)}\right|\\
\leq C\|\psi\|_{C_{(0,1)}(V)}\cdot\int_0^{2\pi}\d\theta\int_0^{\delta}\d t
\int_0^{\sqrt{\delta}}\frac{r^2dr}{\left(t+r^2\right)^2}
\leq C\|\psi\|_{C_{(0,1)}(V)}\cdot\int_0^{\sqrt{\delta}}\d r
\leq C\sqrt{\delta}\cdot\|\psi\|_{C_{(0,1)}(V)},
\end{multline}
where we used the coordinates
\begin{equation}\label{SmallCoordinates}
\left\{\begin{aligned}
&t=\text{Im}B(w,\zeta)=\text{Im}B^*(w,\zeta),\\
&\text{Re}B(w,\zeta)=(1/2)\cdot\sum_{j=0}^2(w_j-\zeta_j)({\bar w}_j-\bar\zeta_j)=r^2/2,\\
&\rho=\left|P(w)\right|,\\
&\theta=\text{Arg}P(w)
\end{aligned}\right.
\end{equation}
and estimates
\begin{equation*}
\left\{\begin{aligned}
&\psi^{(0,1)}\wedge\d w\Big|_{\Gamma^{\tau}}
\sim \d\theta\wedge \d t\wedge \d x\wedge \d y=r\cdot \d\theta\wedge \d t\wedge \d r,\\
&\left|{\bar w}\wedge \bar{\zeta}\right|\sim |w-\zeta|.
\end{aligned}\right.
\end{equation*}
Then, for arbitrary $\zeta^{(1)}, \zeta^{(2)}\in \pi^{-1}(U)$ such that
$|\zeta^{(1)}-\zeta^{(2)}|\leq\delta$, using equality
\begin{equation*}
B(\zeta^{(2)},w)-B(\zeta^{(1)},w)=\sum_{j=0}^2{\bar\zeta}^{(2)}_j\left(\zeta^{(2)}_j-w_j\right)
-\sum_{j=0}^2{\bar\zeta}^{(1)}_j\left(\zeta^{(1)}_j-w_j\right)
=\sum_{j=0}^2 w_j\left({\bar\zeta}_j^{(1)}-{\bar\zeta}_j^{(2)}\right),
\end{equation*}
we obtain
\begin{multline}\label{SecondOutside}
\Bigg|\int_{\Gamma^{\tau}_{w}\cap\left\{|B(\zeta^{(1)},w)|>\delta,
|B(\zeta^{(2)},w)|>\delta\right\}}\psi^{(0,1)}(w)\wedge
\Bigg(\det\left[\frac{{\bar\zeta}^{(1)}}{B^*(w,\zeta^{(1)})}\ \frac{\bar w}{B(w,\zeta^{(1)})}\
Q(w,\zeta^{(1)})\right]\\
-\det\left[\frac{{\bar\zeta}^{(2)}}{B^*(w,\zeta^{(2)})}\ \frac{\bar w}{B(w,\zeta^{(2)})}\
Q(w,\zeta^{(2)})\right]\Bigg)\wedge\frac{\d w}{P(w)}\Bigg|\\
\leq C\cdot\Bigg|\int_{\Gamma^{\tau}_{w}\cap\left\{|B(\zeta^{(1)},w)|>\delta,
|B(\zeta^{(2)},w)|>\delta\right\}}\psi^{(0,1)}(w)\\
\bigwedge\Bigg(\det\left[\frac{{\bar\zeta}^{(1)}-{\bar\zeta}^{(2)}}
{B^*(w,\zeta^{(1)})}\ \frac{\bar w}{B(w,\zeta^{(1)})}\ Q(w,\zeta^{(1)})\right]\\
+\det\left[\frac{{\bar\zeta}^{(2)}}{B^*(w,\zeta^{(1)})}\ \frac{\bar w}{B(w,\zeta^{(1)})}\
Q(w,\zeta^{(1)})\right]
-\det\left[\frac{{\bar\zeta}^{(2)}}{B^*(w,\zeta^{(2)})}\ \frac{\bar w}{B(w,\zeta^{(1)})}\
Q(w,\zeta^{(1)})\right]\\
+\det\left[\frac{{\bar\zeta}^{(2)}}{B^*(w,\zeta^{(2)})}\ \frac{\bar w}{B(w,\zeta^{(1)})}\
Q(w,\zeta^{(1)})\right]
-\det\left[\frac{{\bar\zeta}^{(2)}}{B^*(w,\zeta^{(2)})}\ \frac{\bar w}{B(w,\zeta^{(2)})}\
Q(w,\zeta^{(1)})\right]\\
+\det\left[\frac{{\bar\zeta}^{(2)}}{B^*(w,\zeta^{(2)})}\ \frac{\bar w}{B(w,\zeta^{(2)})}\
\left(Q(w,\zeta^{(1)})-Q(w,\zeta^{(2)})\right)\right]\Bigg)\Bigg|\\
\leq C\cdot\|\psi\|_{C_{(0,1)}(V)}\Bigg(\delta\int_0^{2\pi}\d\theta\int_0^{A}dt
\int_0^{A}\frac{r\d r}{\left(\delta+t+r^2\right)^2}
+\delta\int_0^{2\pi}\d\theta\int_0^{A}dt\int_0^{A}\frac{r^2\d r}
{\left(\delta+t+r^2\right)^2(\delta+r^2)}\Bigg)\\
\leq  C\cdot\|\psi\|_{C_{(0,1)}(V)}\delta\int_0^A\frac{\d r}{\delta+r^2}
\leq C\sqrt{\delta}\cdot\|\psi\|_{C_{(0,1)}(V)}.
\end{multline}
\indent
Combining estimates \eqref{FirstSmall} and \eqref{SecondOutside} we obtain estimate
\eqref{InteriorRIntegralEstimate}.
\end{proof}

\indent
Now, using Lemma~\ref{InteriorEstimate} we transform the integral operator from formula
\eqref{IntegralOperator} into a more convenient form. We start with the lemma on changing of the order of integration.

\begin{lemma}\label{IntegrationChange}
Let
\begin{equation}\label{q-wForm}
q(\zeta,w)=q_1^{(0,1)}(w)\d\left(\frac{\zeta_1}{\zeta_0}\right)
+q_2^{(0,1)}(w)\d\left(\frac{\zeta_2}{\zeta_0}\right)
\end{equation}
be a differential form with forms $q_i^{(0,1)}(w)\in C_{(0,1)}(V)\ (i=1,2)$ from \eqref{muq} having
compact support. Then, the following equality holds
\begin{multline}\label{Change-of-Order}
\lim_{\epsilon\to 0}\int_{\Gamma^{\epsilon}_{\zeta}}
\left(\lim_{\tau\to 0}\int_{\Gamma^{\tau}_w}
q(\zeta,w)K(w,\zeta)\wedge\frac{\d w}{P(w)}\right)
K(\bar\zeta,\bar{z})\wedge\frac{\d\bar\zeta}{P(\bar\zeta)}\\
=\lim_{\nu\to 0}\lim_{\tau\to 0}
\int_{\Gamma^{\tau}_w\cap\left\{|B(z,w)|>\nu\right\}}\\
\left(\lim_{\delta\to 0}\lim_{\epsilon\to 0}
\int_{\Gamma^{\epsilon}_{\zeta}\cap\left\{|B(\zeta,z)|>\delta, |B(\zeta,w)|>\delta\right\}}
q(\zeta,w)\wedge K(w,\zeta)K(\bar\zeta,\bar{z})\frac{\d\bar\zeta}{P(\bar\zeta)}\right)
\frac{\d w}{P(w)}.
\end{multline}
\end{lemma}
\begin{proof}
For the interior integral in the left-hand side of \eqref{Change-of-Order} we have from the
estimate \eqref{FirstSmall}:
\begin{equation}\label{Gamma-w}
\left|\int_{\Gamma^{\tau}_w\cap\left\{|B(w,\zeta)|<\nu\right\}}
q(\zeta,w)K(w,\zeta)\wedge\frac{\d w}{P(w)}\right|
\leq C\sqrt{\nu}\cdot\|q\|_{C_{(0,1)}(V)}\\
\end{equation}
uniformly with respect to $\zeta$ and $\tau$. From the estimate \eqref{InteriorRIntegralEstimate}
we obtain that the functions
\begin{equation*}
F_i(\zeta)=\lim_{\tau\to 0}\int_{\Gamma^{\tau}_w}q_i(w)K(w,\zeta)\wedge\frac{\d w}{P(w)}
\end{equation*}
are uniformly continuous functions of $\zeta$ on ${\cal C}$, and therefore, from the estimate \eqref{FirstSmall}
we obtain
\begin{equation}\label{Gamma-zeta}
\left|\int_{\Gamma^{\epsilon}_{\zeta}\cap\left\{|B(\zeta,z)|<\delta\right\}}
\left(\lim_{\tau\to 0}\int_{\Gamma^{\tau}_w}
q(\zeta,w)K(w,\zeta)\wedge\frac{\d w}{P(w)}\right)K(\bar\zeta,\bar{z})
\wedge\frac{\d\bar\zeta}{P(\bar\zeta)}\right|
\leq C\sqrt{\delta}\cdot\|F\|_{C_{(1,0)}(V)},
\end{equation}
where $\|F\|_{C_{(1,0)}({\cal C})}=\max_{i=1,2}\|F_i\|_{C({\cal C})}$.\\
\indent
From the estimates \eqref{Gamma-w} and \eqref{Gamma-zeta} we obtain the equality
\begin{multline*}
\lim_{\epsilon\to 0}\int_{\Gamma^{\epsilon}_{\zeta}}
\left(\lim_{\tau\to 0}\int_{\Gamma^{\tau}_w}
q(\zeta,w)K(w,\zeta)\wedge\frac{\d w}{P(w)}\right)K(\bar\zeta,\bar{z})
\wedge\frac{\d\bar\zeta}{P(\bar\zeta)}\\
=\lim_{\epsilon\to 0}
\int_{\Gamma^{\epsilon}_{\zeta}\cap\left\{|B(\zeta,z)|>\delta\right\}}
\left(\lim_{\tau\to 0}\int_{\Gamma^{\tau}_w\cap\left\{|B(\zeta,w)|>\delta\right\}}
q(\zeta,w)K(w,\zeta)\wedge\frac{\d w}{P(w)}\right)
K(\bar\zeta,\bar{z})\wedge\frac{\d\bar\zeta}{P(\bar\zeta)}\\
+\|q\|_{C_{(0,1)}}\cdot{\cal{O}}(\sqrt{\delta})\\
=\lim_{\tau\to 0}\int_{\Gamma^{\tau}_w}\left(\lim_{\epsilon\to 0}
\int_{\Gamma^{\epsilon}_{\zeta}\cap\left\{|B(\zeta,z)|>\delta, |B(\zeta,w)|>\delta\right\}}
q(\zeta,w)\wedge K(w,\zeta)K(\bar\zeta,\bar{z})\frac{\d\bar\zeta}{P(\bar\zeta)}\right)
\frac{\d w}{P(w)}\\
+\|q\|_{C_{(0,1)}}\cdot{\cal{O}}(\sqrt{\delta}),
\end{multline*}
and its corollary
\begin{multline}\label{KernelTransformation}
\lim_{\epsilon\to 0}\int_{\Gamma^{\epsilon}_{\zeta}}
\left(\lim_{\tau\to 0}\int_{\Gamma^{\tau}_w}
q(\zeta,w)K(w,\zeta)\wedge\frac{\d w}{P(w)}\right)K(\bar\zeta,\bar{z})
\wedge\frac{\d\bar\zeta}{P(\bar\zeta)}\\
=\lim_{\nu\to 0}\lim_{\tau\to 0}\int_{\Gamma^{\tau}_w\cap\left\{|B(z,w)|>\nu\right\}}\\
\left(\lim_{\delta\to 0}\lim_{\epsilon\to 0}
\int_{\Gamma^{\epsilon}_{\zeta}\cap\left\{|B(\zeta,z)|>\delta, |B(\zeta,w)|>\delta\right\}}
q(\zeta,w)\wedge K(w,\zeta)K(\bar\zeta,\bar{z})\frac{\d\bar\zeta}{P(\bar\zeta)}\right)
\frac{\d w}{P(w)}.
\end{multline}
\end{proof}
\indent
Using Lemma~\ref{IntegrationChange}
we represent the integral operator $R_{\lambda}:C^{\infty}_{(1,1)}\to C$ as follows
\begin{multline}\label{Integral-R}
R_{\lambda}[\mu\cdot q](z,\lambda)=|{\bf{C}}|^2\cdot\bar{z}_0^{-\ell}
\cdot e^{-\langle\lambda,z/z_0\rangle+\overline{\langle\lambda,z/z_0\rangle}}
\lim_{\delta\to 0}\lim_{\epsilon\to 0}
\int_{\Gamma^{\epsilon}_{\zeta}\cap\left\{|B(\zeta,z)|>\delta\right\}}
\bar\zeta_0^{\ell}\zeta_0^{-\ell}\vartheta(\zeta)
e^{\langle\lambda,\zeta/\zeta_0\rangle-\overline{\langle\lambda,\zeta/\zeta_0\rangle}}\\
\times\Bigg[\left(\lim_{\delta\to 0}\lim_{\tau\to 0}
\int_{\Gamma^{\tau}_{w}\cap\left\{|B(\zeta,w)|>\delta\right\}}w_0^{\ell}
\cdot \mu(w,\lambda)q_1^{(0,1)}(w)
K(w,\zeta)\wedge\frac{\d w}{P(w)}\right)\d\left(\frac{\zeta_1}{\zeta_0}\right)\\
+\left(\lim_{\delta\to 0}\lim_{\tau\to 0}
\int_{\Gamma^{\epsilon}_{w}\cap\left\{|B(\zeta,w)|>\delta\right\}}w_0^{\ell}
\cdot \mu(w,\lambda)q_2^{(0,1)}(w)
K(w,\zeta)\wedge\frac{\d w}{P(w)}\right)\d\left(\frac{\zeta_2}{\zeta_0}\right)\Bigg]
K(\bar\zeta,\bar{z})\wedge\frac{\d\bar\zeta}{P(\bar\zeta)}\\
=|{\bf{C}}|^2\cdot\bar{z}_0^{-\ell}
\cdot e^{-\langle\lambda,z/z_0\rangle+\overline{\langle\lambda,z/z_0\rangle}}
\lim_{\tau\to 0}\int_{\Gamma^{\tau}_{w}}w_0^{\ell}
\cdot \mu(w,\lambda)\frac{\d w}{P(w)}\\
\times \Bigg[\lim_{\epsilon\to 0}\int_{\Gamma^{\epsilon}_{\zeta}}
\bar\zeta_0^{\ell}\zeta_0^{-\ell}\vartheta(\zeta)
e^{\langle\lambda,\zeta/\zeta_0\rangle-\overline{\langle\lambda,\zeta/\zeta_0\rangle}}
q(w,\zeta)K(w,\zeta)
K(\bar\zeta,\bar{z})\wedge\frac{\d\bar\zeta}{P(\bar\zeta)}\Bigg]\\
=|{\bf{C}}|^2\cdot\bar{z}_0^{-\ell}
\cdot e^{-\langle\lambda,z/z_0\rangle+\overline{\langle\lambda,z/z_0\rangle}}
\lim_{\nu\to 0}\lim_{\tau\to 0}
\int_{\Gamma^{\tau}_{w}\cap\{|B(z,w)|>\nu\}}w_0^{\ell}
\cdot \mu(w,\lambda)q_1^{(0,1)}(w)\wedge\d w\\
\times \lim_{\delta\to 0}\lim_{\epsilon\to 0}
\int_{\Gamma^{\epsilon}_{\zeta}\cap\{|B(w,\zeta)|>\delta, |B(z,\zeta)|>\delta\}}
\bar\zeta_0^{\ell}\zeta_0^{-\ell}\vartheta(\zeta)
e^{\langle\lambda,\zeta/\zeta_0\rangle-\overline{\langle\lambda,\zeta/\zeta_0\rangle}}
\frac{K(w,\zeta)}{P(w)}\frac{K(\bar\zeta,\bar{z})}{P(\bar\zeta)}\d\left(\frac{\zeta_1}{\zeta_0}\right)\wedge\d\bar\zeta\\
+|{\bf{C}}|^2\cdot\bar{z}_0^{-\ell}
\cdot e^{-\langle\lambda,z/z_0\rangle+\overline{\langle\lambda,z/z_0\rangle}}
\lim_{\nu\to 0}\lim_{\tau\to 0}
\int_{\Gamma^{\tau}_{w}\cap\{|B(z,w)|>\nu\} }w_0^{\ell}
\cdot \mu(w,\lambda)q_2^{(0,1)}(w)\wedge\d w\\
\times \lim_{\delta\to 0}\lim_{\epsilon\to 0}
\int_{\Gamma^{\epsilon}_{\zeta}\cap\{|B(w,\zeta)|>\delta, |B(z,\zeta)|>\delta\}}
\bar\zeta_0^{\ell}\zeta_0^{-\ell}\vartheta(\zeta)
e^{\langle\lambda,\zeta/\zeta_0\rangle-\overline{\langle\lambda,\zeta/\zeta_0\rangle}}
\frac{K(w,\zeta)}{P(w)}\frac{K(\bar\zeta,\bar{z})}{P(\bar\zeta)}
\d\left(\frac{\zeta_2}{\zeta_0}\right)\wedge\d\bar\zeta.\\
\end{multline}
\indent
To further transform the right-hand side of \eqref{Integral-R} we introduce the forms
\begin{equation}\label{omegaforms}
\omega_1(w)=-w_0^2\cdot\d\left(\frac{w_2}{w_0}\right)\wedge\d w_0,\hspace{0.3in}
\omega_2(w)=w_0^2\cdot\d\left(\frac{w_1}{w_0}\right)\wedge\d w_0,
\end{equation}
satisfying equalities
\begin{multline*}
\d w_0\wedge\d w_1\wedge\d w_2=
w_0^2\cdot\d\left(\frac{w_1}{w_0}\right)\wedge\d\left(\frac{w_2}{w_0}\right)\wedge\d w_0\\
=-\d\left(\frac{w_1}{w_0}\right)\wedge\left[-w_0^2\cdot\d \left(\frac{w_2}{w_0}\right)
\wedge\d w_0\right]
=-\d\left(\frac{w_2}{w_0}\right)\wedge\left[w_0^2\cdot\d\left(\frac{w_1}{w_0}\right)
\wedge\d w_0\right].
\end{multline*}
\indent
Then, denoting
\begin{multline}\label{hFunctions}
h_j(z,w,\lambda)=|{\bf{C}}|^2\cdot\bar{z}_0^{-\ell}
\cdot e^{-\langle\lambda,z/z_0\rangle+\overline{\langle\lambda,z/z_0\rangle}}\frac{w_0^{\ell}}{P(w)}\\
\times\lim_{\delta\to 0}\lim_{\epsilon\to 0}
\int_{\Gamma^{\epsilon}_{\zeta}\cap\left\{|B(\zeta,w)|>\delta,|B(\zeta,z)|>\delta\right\}}
\bar\zeta_0^{\ell}\zeta_0^{-\ell}\vartheta(\zeta)
e^{\langle\lambda,\zeta/\zeta_0\rangle-\overline{\langle\lambda,\zeta/\zeta_0\rangle}}
K(w,\zeta)\frac{K(\bar\zeta,\bar{z})}{P(\bar\zeta)}\d\left(\frac{\zeta_j}{\zeta_0}\right)
\wedge\d\bar\zeta
\end{multline}
for $j=1,2$, we obtain from equality \eqref{Integral-R}
\begin{multline}\label{RFormula}
R_{\lambda}[\mu\cdot q](z,\lambda)
=\lim_{\nu\to 0}\lim_{\tau\to 0}\int_{\Gamma^{\tau}_w\cap\left\{|B(z,w)|>\nu\right\}}
\mu(w,\lambda)q_1^{(0,1)}(w)h_1(z,w,\lambda)\wedge\d w\\
+\lim_{\nu\to 0}\lim_{\tau\to 0}\int_{\Gamma^{\tau}_w\cap\left\{|B(z,w)|>\nu\right\}}
\mu(w,\lambda)q_2^{(0,1)}(w)h_2(z,w,\lambda)\wedge\d w\\
=-\lim_{\nu\to 0}\lim_{\tau\to 0}\int_{\Gamma^{\tau}_w\cap\left\{|B(z,w)|>\nu\right\}}
\mu(w,\lambda)q_1^{(0,1)}(w)\d\left(\frac{w_1}{w_0}\right)
h_1(z,w,\lambda)\wedge\omega_1(w)\\
-\lim_{\nu\to 0}\lim_{\tau\to 0}\int_{\Gamma^{\tau}_w\cap\left\{|B(z,w)|>\nu\right\}}
\mu(w,\lambda)q_2^{(0,1)}(w)\d\left(\frac{w_2}{w_0}\right)
h_2(z,w,\lambda)\wedge\omega_2(w)\\
=\lim_{\nu\to 0}\lim_{\tau\to 0}\int_{\Gamma^{\tau}_w\cap\left\{|B(z,w)|>\nu\right\}}
\mu(w,\lambda)q^{(1,1)}(w)\wedge\big(h_1(z,w,\lambda)\omega_1(w)
+h_2(z,w,\lambda)\omega_2(w)\big)\\
=\lim_{\nu\to 0}\lim_{\tau\to 0}\int_{\Gamma^{\tau}_w\cap\left\{|B(z,w)|>\nu\right\}}
\mu(w,\lambda)q^{(1,1)}(w)\wedge H^{(2,0)}(z,w,\lambda),
\end{multline}
where
$$q^{(1,1)}(w)=\d\left(\frac{w_1}{w_0}\right)\wedge q_1^{(0,1)}(w)
+\d\left(\frac{w_2}{w_0}\right)\wedge q_2^{(0,1)}(w),$$
and
\begin{equation}\label{HKernel}
H^{(2,0)}(z,w,\lambda)=h_1(z,w,\lambda)\omega^{(2,0)}_1(w)
+h_2(z,w,\lambda)\omega^{(2,0)}_2(w).
\end{equation}

\section{\bf Estimates for Integral Operator $R_{\lambda}$.}\label{sec:REstimates}

\indent
We proceed now with the estimates for the integral operator $R_{\lambda}$. We rewrite formula \eqref{RFormula} as follows
\begin{multline}\label{CompositIntegral}
R_{\lambda}[\mu\cdot q](z)
=\lim_{\tau\to 0}\int_{\Gamma^{\tau}_w}
\mu(w,\lambda)q^{(1,1)}(w)\wedge H(z,w,\lambda)\\
=|{\bf{C}}|^2\cdot\bar{z}_0^{-\ell}
\cdot e^{-\langle\lambda,z/z_0\rangle+\overline{\langle\lambda,z/z_0\rangle}}
\lim_{\tau\to 0}\int_{\Gamma^{\tau}_w}w_0^{\ell}\cdot\mu(w,\lambda)q^{(1,1)}(w)\\
\bigwedge\sum_{j=1,2}\Bigg(\lim_{\delta\to 0}\lim_{\epsilon\to 0}
\int_{\Gamma^{\epsilon}_{\zeta}\cap\left\{|B(\zeta,z)|>\delta,|B(w,\zeta)|>\delta\right\}}
\bar\zeta_0^{\ell}\zeta_0^{-\ell}\vartheta(\zeta)
e^{\langle\lambda,\zeta/\zeta_0\rangle-\overline{\langle\lambda,\zeta/\zeta_0\rangle}}\\
\det\left[\frac{\bar\zeta}{B^*(w,\zeta)}\ \frac{\bar{w}}{B(w,\zeta)}\ Q(w,\zeta)\right]
\times\det\left[\frac{z}{B^*(\bar\zeta,\bar{z})}\ \frac{\zeta}
{B(\bar\zeta,\bar{z})}\ Q(\bar\zeta,\bar{z})\right]\d\left(\frac{\zeta_j}{\zeta_0}\right)
\wedge\frac{\d\bar\zeta}{P(\bar\zeta)}\Bigg)
\wedge\frac{\omega_j(w)}{P(w)},
\end{multline}
and start with the estimate of the interior integral in the right-hand side of \eqref{CompositIntegral}.
In this estimate we will use a couple of assumptions. Since $\supp\{q\}$ is compact, we may assume that $q(w)=0$ for $|w_0|>c_0$ for some $c_0>0$. Also, we will be interested in the values of $f$
for $z\in\supp\{q\}$, and, therefore, will assume that in the formulas below $|z_0|>c_0$.
In the next lemma we describe the neighborhoods of points $z, w$ that will be used in the estimates.

\begin{lemma}\label{GammaNeighborhoods}
Let $z,w,\zeta$ be points in $\*S^5(1)$ such that $|B(w,z)|=\gamma$. Then the following inequalities hold:
\begin{equation}\label{GammaInequalities}
|B(w,\zeta)|\leq \gamma/9\Rightarrow\
\left\{\begin{aligned}
&\frac{2}{9}\gamma\leq|B(z,\zeta)|\leq\frac{16}{9}\gamma,\\
&|\zeta-z|\leq\frac{4\sqrt{2}}{3}\sqrt{\gamma},
\end{aligned}\right.\\
\end{equation}
\end{lemma}
\begin{proof}
To prove inequalities \eqref{GammaInequalities} we use definitions \eqref{BDefinition} to obtain:
\begin{multline}\label{BEquality}
B(z,\zeta)=\sum_{j=0}^2{\bar z}_j\cdot\left(z_j-\zeta_j\right)
=\sum_{j=0}^2{\bar w}_j\cdot\left(z_j-\zeta_j\right)
+\sum_{j=0}^2\left({\bar z}_j-{\bar w}_j\right)\cdot\left(z_j-\zeta_j\right)\\
=\sum_{j=0}^2{\bar w}_j\cdot\left(w_j-\zeta_j\right)-\sum_{j=0}^2{\bar w}_j\cdot\left(w_j-z_j\right)
+\sum_{j=0}^2\left({\bar z}_j-{\bar w}_j\right)\cdot\left(z_j-\zeta_j\right)\\
=B(w,\zeta)-B(w,z)+\sum_{j=0}^2\left({\bar z}_j-{\bar w}_j\right)\cdot\left(z_j-w_j\right)
+\sum_{j=0}^2\left({\bar z}_j-{\bar w}_j\right)\cdot\left(w_j-\zeta_j\right).
\end{multline}
\indent
Then, using inequalities
\begin{equation*}
\begin{aligned}
&\text{Re}B(w,z)=\sum_{j=0}^2\frac{(w_j-z_j)({\bar w}_j-{\bar z}_j)}{2}\leq\gamma\Rightarrow
\sum_{j=0}^2|w_j-z_j|^2\leq 2\gamma,\\
&\text{Re}B(w,\zeta)=\sum_{j=0}^2\frac{(w_j-\zeta_j)({\bar w}_j-{\bar\zeta}_j)}{2}\leq\frac{\gamma}{9}
\Rightarrow \sum_{j=0}^2|w_j-\zeta_j|^2\leq\frac{2}{9}\gamma,
\end{aligned}
\end{equation*}
\indent
we obtain for $\zeta$ such that $|B(w,\zeta)|\leq \gamma/9$:
\begin{equation*}
|B(z,\zeta)|\geq\left|-B(w,z)+2\text{Re}B(z,w)\right|-\frac{\gamma}{9}-\sqrt{2\gamma}
\cdot\frac{\sqrt{2\gamma}}{3}
=\gamma-\frac{\gamma}{9}-\frac{2\gamma}{3}=\frac{2}{9}\gamma,
\end{equation*}
and
\begin{equation*}
|B(z,\zeta)|\leq\gamma+\frac{\gamma}{9}+\frac{2\gamma}{3}=\frac{16}{9}\gamma \Rightarrow
|\zeta-z|\leq\frac{4\sqrt{2}}{3}\sqrt{\gamma}.
\end{equation*}
\end{proof}

\indent
Before proving estimates for the operator $R_{\lambda}[\mu\cdot q]$ we denote for fixed $z$
and $w$ a number $\gamma=|B(z,w)|$, and for an arbitrary $\eta\leq\gamma$ consider two domains
\begin{equation}\label{U-neighborhoods}
U^{\eta}_{\zeta}(w)=\left\{\zeta\in \Gamma^{\epsilon}: |B(w,\zeta)|<\eta/9\right\},\hspace{0.1in}
U^{\eta}_{\zeta}(z)=\left\{\zeta\in \Gamma^{\epsilon}: |B(z,\zeta)|<\eta/9\right\},
\end{equation}
for which, with $\eta$ small enough, we would have, according to
Lemma~\ref{GammaNeighborhoods} the estimates \eqref{GammaInequalities}:
\begin{equation*}
\zeta\in U^{\eta}_{\zeta}(w)\Rightarrow\ 
\begin{aligned}
&|B(z,\zeta)|>c_1\cdot\eta,\\
&|\zeta-z|<c_2\cdot\sqrt{\eta},
\end{aligned}\hspace{0.3in}
\zeta\in U^{\eta}_{\zeta}(z)\Rightarrow\
\begin{aligned}
&|B(w,\zeta)|>c_1\cdot\eta,\\
&|\zeta-w|<c_2\cdot\sqrt{\eta},
\end{aligned}
\end{equation*}
with some constants $c_1, c_2 >0$.\\
\indent
In the lemma below we prove an estimate of the interior integrals in the right-hand side of \eqref{CompositIntegral}, which we denote by
\begin{multline}\label{RIntegrals}
R_j(z,w,\lambda)
=\lim_{\delta\to 0}\lim_{\epsilon\to 0}\int_{\Gamma^{\epsilon}_{\zeta}
\cap\left\{|B(z,\zeta)|>\delta, |B(w,\zeta|>\delta\right\}}
\bar\zeta_0^{\ell}\zeta_0^{-\ell}\vartheta(\zeta)
e^{\langle\lambda,\zeta/\zeta_0\rangle-\overline{\langle\lambda,\zeta/\zeta_0\rangle}}\\
\times\det\left[\frac{z}{B^*(\bar\zeta,\bar{z})}\ \frac{\zeta}
{B(\bar\zeta,\bar{z})}\ Q(\bar\zeta,\bar{z})\right]
\det\left[\frac{\bar\zeta}{B^*(w,\zeta)}\ \frac{\bar{w}}{B(w,\zeta)}\ Q(w,\zeta)\right]
\d\left(\frac{\zeta_j}{\zeta_0}\right)
\wedge\frac{\d\bar\zeta}{P(\bar\zeta)}\hspace{0.1in}(j=1,2).
\end{multline}

\begin{lemma}\label{InteriorIntegral}
There exist constants $C(\lambda)$,
satisfying $\lim_{\lambda\to\infty}C(\lambda)=0$, such that for any $\eta>0$
the following estimates hold
\begin{itemize}
\item[(i)] if the points $z, w\in V$ are such that $\gamma=|B(z,w)|<\eta$,
then
\begin{equation}\label{REtaGammaEstimate}
\left|B(z,w)\cdot R^{\eta}_j(z,w,\lambda)\right|\leq C(\lambda),
\end{equation}
where
\begin{multline}\label{REtaGammaIntegral}
R^{\eta}_j(z,w,\lambda)=
\lim_{\epsilon\to 0}\int_{\Gamma^{\epsilon}_{\zeta}\cap U^{\eta}_{\zeta}(z,w)}
\bar\zeta_0^{\ell}\zeta_0^{-\ell}\vartheta(\zeta)
e^{\langle\lambda,\zeta/\zeta_0\rangle-\overline{\langle\lambda,\zeta/\zeta_0\rangle}}\\
\times\det\left[\frac{\bar\zeta}{B^*(w,\zeta)}\ \frac{\bar{w}}{B(w,\zeta)}\ Q(w,\zeta)\right]
\det\left[\frac{z}{B^*(\bar\zeta,\bar{z})}\ \frac{\zeta}
{B(\bar\zeta,\bar{z})}\ Q(\bar\zeta,\bar{z})\right]
\d\left(\frac{\zeta_j}{\zeta_0}\right)\wedge\frac{\d\bar\zeta}{P(\bar\zeta)};\\
\end{multline}
\item[(ii)] if the points $z, w\in V$ are such that $\gamma=|B(z,w)|>\eta$, then
\begin{equation}\label{REtaEstimate}
\left|R_j(z,w,\lambda)\right|\leq \frac{C(\lambda)}{\eta}.
\end{equation}
for $R_j$ in \eqref{RIntegrals}.
\end{itemize}
\end{lemma}
\begin{proof}
\indent
To prove estimate \eqref{REtaGammaEstimate} for fixed $z,w$, and $\eta$ such that
$\eta>\gamma=|B(z,w)|$ we estimate the integral in \eqref{REtaGammaIntegral} separately in
the domains $U^{\gamma}_{\zeta}(w)$, $U^{\gamma}_{\zeta}(z)$, and
\begin{equation}\label{UDifference}
U^{\eta}_{\zeta}(z,w)\setminus
\left\{U^{\gamma}_{\zeta}(z)\cup U^{\gamma}_{\zeta}(w)\right\}
=\left\{\zeta: \eta>|B(z,\zeta)|, |B(w,\zeta)|>\gamma/9\right\}.
\end{equation}
Using estimates \eqref{GammaInequalities} and notation \eqref{piNotation}
$\pi: \*S^5(1)\to \C\P^2$ for the restriction
of the map $\tau: \C^3\to \C\P^2$ we obtain for $U^{\gamma}_{\zeta}(w)$:
\begin{multline}\label{RUGammaEstimate}
\Bigg|\lim_{\delta\to 0}\lim_{\epsilon\to 0}B(z,w)^{3/2}\cdot
\int_{\Gamma^{\epsilon}_{\zeta}\cap\left\{U^{\gamma}_{\zeta}(w)
\setminus U^{\delta}_{\zeta}(w)\right\}}\bar\zeta_0^{\ell}\zeta_0^{-\ell}\vartheta(\zeta)
e^{\langle\lambda,\zeta/\zeta_0\rangle-\overline{\langle\lambda,\zeta/\zeta_0\rangle}}\\
\times\det\left[\frac{\bar\zeta}{B^*(w,\zeta)}\ \frac{\bar{w}}{B(w,\zeta)}\ Q(w,\zeta)\right]
\det\left[\frac{z}{B^*(\bar\zeta,\bar{z})}\ \frac{\zeta}
{B(\bar\zeta,\bar{z})}\ Q(\bar\zeta,\bar{z})\right]
\d\left(\frac{\zeta_i}{\zeta_0}\right)\wedge\frac{\d\bar\zeta}{P(\bar\zeta)}\Bigg|\\
\leq C\cdot\gamma^{3/2}\int_{\pi^{-1}({\cal C})\cap U^{\gamma}_{\zeta}(w)}
\frac{e^{\langle\lambda,\zeta/\zeta_0\rangle-\overline{\langle\lambda,\zeta/\zeta_0\rangle}}
|\zeta-z||{\bar\zeta}-{\bar w}|}{|B^*({\bar\zeta},{\bar z})||B({\bar\zeta},{\bar z})|
|B^*(w,\zeta)||B(w,\zeta)|}
\d\left(\frac{\zeta_i}{\zeta_0}\right)\wedge\left(\d{\bar P}(\zeta)\interior\d\bar\zeta\right)\\
\leq C\cdot\gamma^{3/2}\int_{\pi^{-1}({\cal C})\cap U^{\gamma}_{\zeta}(w)}
\frac{e^{\langle\lambda,\zeta/\zeta_0\rangle-\overline{\langle\lambda,\zeta/\zeta_0\rangle}}
|\zeta-w|}{|B(\zeta,z)|^{3/2}|B(w,\zeta)|^2}
\d\left(\frac{\zeta_i}{\zeta_0}\right)\wedge\left(\d{\bar P}(\zeta)\interior\d\bar\zeta\right)\\
\leq C(\lambda)\cdot\gamma^{3/2}\int_0^\gamma\d t\int_0^{\sqrt{\gamma}}\frac{r^2\d r}
{(\gamma+r^2)^{3/2}(t+r^2)^2}
\leq C(\lambda)\cdot\int_0^{\sqrt{\gamma}}\d r\leq C(\lambda)\cdot\sqrt{\gamma},
\end{multline}
where we used notations and estimates
\begin{equation*}
\left\{\begin{aligned}
&t=\text{Im}B(w,\zeta),\ \text{Re}B(w,\zeta)\sim r^2,\\
&|\zeta-z|\leq C\cdot|B(\zeta,z)|^{1/2},\\
&|B(\zeta,z)|\geq C\cdot(\gamma+r^2)\ \text{for}\ \zeta\in U^{\gamma}_{\zeta}(w),\\
\end{aligned}\right.
\end{equation*}
and the Riemann-Lebesgue Lemma (see \cite{23}) for the integral in
\eqref{RUGammaEstimate}.\\

{\bf Remark.} {\it We notice that expression $B(z,w)^{3/2}$ in \eqref{RUGammaEstimate} is well defined for the set
\begin{equation*}
\left\{(z,w): |z|=1, |w|=1, z\neq w\right\},
\end{equation*}
since
\begin{equation*}
\text{Re}B(w,z)=\text{Re}\left(\sum_{j=0}^2{\bar w}_j\cdot\left(w_j-z_j\right)\right)
=1-\text{Re}\left(\sum_{j=0}^2{\bar w}_j\cdot z_j\right)>0\ \text{for such}\ z,w.
\end{equation*}
}

\indent
We obtain the same estimate for the integral over the domain
$\Gamma^{\epsilon}_{\zeta}\cap U^{\gamma}_{\zeta}(z)$, and, as a corollary, estimate
\eqref{REtaGammaEstimate} for the integrals in \eqref{REtaGammaIntegral} over the domains $U^{\gamma}_{\zeta}(w)$ and $U^{\gamma}_{\zeta}(z)$.\\
\indent
For the integral over $U^{\eta}_{\zeta}(z,w)\setminus U^{\gamma}_{\zeta}(z,w)$ we have
\begin{multline}\label{RUEtaGammaEstimate}
\Bigg|\lim_{\epsilon\to 0}\int_{\Gamma^{\epsilon}_{\zeta}\cap
U^{\eta}_{\zeta}(z,w)\setminus U^{\gamma}_{\zeta}(z,w)}
B(z,w)^{3/2}e^{\langle\lambda,\zeta/\zeta_0\rangle-\overline{\langle\lambda,\zeta/\zeta_0\rangle}}\\
\det\left[\frac{\bar\zeta}{B^*(w,\zeta)}\ \frac{\bar{w}}{B(w,\zeta)}\ Q(w,\zeta)\right]
\wedge\det\left[\frac{z}{B^*(\bar\zeta,\bar{z})}\
\frac{\zeta}{B(\bar\zeta,\bar{z})}\ Q(\bar\zeta,\bar{z})\right]
\wedge\d\left(\frac{\zeta_j}{\zeta_0}\right)
\wedge\frac{\d\bar\zeta}{P(\bar\zeta)}\Bigg|\\
=\Bigg|\lim_{\epsilon\to 0}\int_{\Gamma^{\epsilon}_{\zeta}\cap
U^{\eta}_{\zeta}(z,w)\setminus U^{\gamma}_{\zeta}(z,w)}B(z,w)^{3/2}
\frac{e^{\langle\lambda,\zeta/\zeta_0\rangle-\overline{\langle\lambda,\zeta/\zeta_0\rangle}}
|\bar\zeta-{\bar w}\cdot|\zeta-z|}
{B^*(w,\zeta)B(w,\zeta)B^*(\bar\zeta,\bar{z})B(\bar\zeta,\bar{z})}
\d\left(\frac{\zeta_j}{\zeta_0}\right)\wedge\frac{\d\bar\zeta}{P(\bar\zeta)}\Bigg|\\
\leq C(\lambda)\cdot\gamma^{3/2}\int_{\gamma}^{\eta}\d t\int_{\sqrt{\gamma}}^{\sqrt{\eta}}
\frac{r\d r}{(\gamma+t+r^2)^{3/2}(\gamma+r^2)^{3/2}}
\leq C(\lambda)\cdot\gamma^{3/2}\int_{\sqrt{\gamma}}^{\sqrt{\eta}}\frac{r\d r}{(\gamma+r^2)^2}\\
\leq C(\lambda)\cdot\gamma^{3/2}\int_{\gamma}^{\eta}\frac{\d u}{(\gamma+u)^2}
\leq C(\lambda)\cdot\gamma^{1/2},
\end{multline}
where we used notations and estimates from the proof of \eqref{RUGammaEstimate}.\\
\indent
Combining estimates \eqref{RUGammaEstimate} for $U^{\gamma}_{\zeta}(w)$ and $U^{\gamma}_{\zeta}(z)$,
and estimate \eqref{RUEtaGammaEstimate} for $U^{\eta}_{\zeta}(z,w)\setminus U^{\gamma}_{\zeta}(z,w)$ we
obtain estimate \eqref{REtaGammaEstimate}.\\
\indent
To prove estimate \eqref{REtaEstimate} we, similarly to \eqref{RUGammaEstimate}, obtain
for $U^{\eta}_{\zeta}(w)$
\begin{multline}\label{RUEtaEstimate}
\Bigg|\lim_{\delta\to 0}\lim_{\epsilon\to 0}\eta^{3/2}\cdot
\int_{\Gamma^{\epsilon}_{\zeta}\cap\left\{U^{\eta}_{\zeta}(w)
\setminus U^{\delta}_{\zeta}(w)\right\}}\bar\zeta_0^{\ell}\zeta_0^{-\ell}\vartheta(\zeta)
e^{\langle\lambda,\zeta/\zeta_0\rangle-\overline{\langle\lambda,\zeta/\zeta_0\rangle}}\\
\times\det\left[\frac{\bar\zeta}{B^*(w,\zeta)}\ \frac{\bar{w}}{B(w,\zeta)}\ Q(w,\zeta)\right]
\det\left[\frac{z}{B^*(\bar\zeta,\bar{z})}\ \frac{\zeta}
{B(\bar\zeta,\bar{z})}\ Q(\bar\zeta,\bar{z})\right]
\d\left(\frac{\zeta_i}{\zeta_0}\right)\wedge\frac{\d\bar\zeta}{P(\bar\zeta)}\Bigg|\\
\leq C\cdot\eta^{3/2}\int_{\pi^{-1}({\cal C})\cap U^{\eta}_{\zeta}(w)}
\frac{e^{\langle\lambda,\zeta/\zeta_0\rangle-\overline{\langle\lambda,\zeta/\zeta_0\rangle}}
|\zeta-z||{\bar\zeta}-{\bar w}|}{|B^*({\bar\zeta},{\bar z})||B({\bar\zeta},{\bar z})|
|B^*(w,\zeta)||B(w,\zeta)|}
\d\left(\frac{\zeta_i}{\zeta_0}\right)\wedge\left(\d{\bar P}(\zeta)\interior\d\bar\zeta\right)\\
\leq C\cdot\eta^{3/2}\int_{\pi^{-1}({\cal C})\cap U^{\eta}_{\zeta}(w)}
\frac{e^{\langle\lambda,\zeta/\zeta_0\rangle-\overline{\langle\lambda,\zeta/\zeta_0\rangle}}
|\zeta-w|}{|B(\zeta,z)|^{3/2}|B(w,\zeta)|^2}
\d\left(\frac{\zeta_i}{\zeta_0}\right)\wedge\left(\d{\bar P}(\zeta)\interior\d\bar\zeta\right)\\
\leq C(\lambda)\cdot\eta^{3/2}\int_0^\eta\d t\int_0^{\sqrt{\eta}}\frac{r^2\d r}
{(\eta+r^2)^{3/2}(t+r^2)^2}
\leq C(\lambda)\cdot\int_0^{\sqrt{\eta}}\d r\leq C(\lambda)\cdot\sqrt{\eta},
\end{multline}
\indent
Similar estimate holds for $U^{\eta}_{\zeta}(z)$.

\indent
For the integral over $U^{\tau}_{\zeta}(z,w)\setminus U^{\eta}_{\zeta}(z,w)$ for arbitrary
$\tau>\eta$ similarly to \eqref{RUEtaGammaEstimate} we have
\begin{multline}\label{RUTauEtaEstimate}
\Bigg|\lim_{\epsilon\to 0}\int_{\Gamma^{\epsilon}_{\zeta}\cap
U^{\tau}_{\zeta}(z,w)\setminus U^{\eta}_{\zeta}(z,w)}
\eta^{3/2}e^{\langle\lambda,\zeta/\zeta_0\rangle-\overline{\langle\lambda,\zeta/\zeta_0\rangle}}\\
\det\left[\frac{\bar\zeta}{B^*(w,\zeta)}\ \frac{\bar{w}}{B(w,\zeta)}\ Q(w,\zeta)\right]
\wedge\det\left[\frac{z}{B^*(\bar\zeta,\bar{z})}\
\frac{\zeta}{B(\bar\zeta,\bar{z})}\ Q(\bar\zeta,\bar{z})\right]
\wedge\d\left(\frac{\zeta_j}{\zeta_0}\right)
\wedge\frac{\d\bar\zeta}{P(\bar\zeta)}\Bigg|\\
=\Bigg|\lim_{\epsilon\to 0}\int_{\Gamma^{\epsilon}_{\zeta}\cap
U^{\tau}_{\zeta}(z,w)\setminus U^{\eta}_{\zeta}(z,w)}B(z,w)^{3/2}
\frac{e^{\langle\lambda,\zeta/\zeta_0\rangle-\overline{\langle\lambda,\zeta/\zeta_0\rangle}}
|\bar\zeta-{\bar w}\cdot|\zeta-z|}
{B^*(w,\zeta)B(w,\zeta)B^*(\bar\zeta,\bar{z})B(\bar\zeta,\bar{z})}
\d\left(\frac{\zeta_j}{\zeta_0}\right)\wedge\frac{\d\bar\zeta}{P(\bar\zeta)}\Bigg|\\
\leq C(\lambda)\cdot\eta^{3/2}\int_{\eta}^{\tau}\d t\int_{\sqrt{\eta}}^{\sqrt{\tau}}
\frac{r\d r}{(\eta+t+r^2)^{3/2}(\eta+r^2)^{3/2}}
\leq C(\lambda)\cdot\eta^{3/2}\int_{\sqrt{\eta}}^{\sqrt{\tau}}\frac{r\d r}{(\eta+r^2)^2}\\
\leq C(\lambda)\cdot\eta^{3/2}\int_{\eta}^{\tau}\frac{\d u}{(\eta+u)^2}
\leq C(\lambda)\cdot\eta^{1/2}.
\end{multline}
\indent
Then, from estimates \eqref{RUEtaEstimate} and \eqref{RUTauEtaEstimate} we obtain estimate
\eqref{REtaEstimate}.
\end{proof}

\indent
In the proposition below we prove an estimate for the operator $R_{\lambda}[\cdot q]$ for
large $\lambda$.
\begin{proposition}\label{RlambdaEstimate}
Let $q\in C_{(1,1)}(V)$ be a form with compact support
$S\subset\left\{|w_0|>c_0\right\}$. Then the following estimate
\begin{equation}\label{RmuEstimate}
\left\|R_{\lambda}[\mu\cdot q]\right\|_{C(S)}\leq \frac{1}{2}\cdot\|\mu\|_{C(S)}
\end{equation}
holds for $|\lambda|$ large enough.
\end{proposition}
\begin{proof}
For a fixed $\eta>0$ we divide the formula \eqref{CompositIntegral}
for $R_{\lambda}[\mu\cdot q]$ into two parts as follows
\begin{multline}\label{RzwEstimate}
R_{\lambda}[\mu\cdot q](z)=|{\bf{C}}|^2\cdot\bar{z}_0^{-\ell}
\cdot e^{-\langle\lambda,z/z_0\rangle+\overline{\langle\lambda,z/z_0\rangle}}
\lim_{\nu\to 0}\lim_{\tau\to 0}\int_{\Gamma^{\tau}_w\cap\left\{|B(z,w)|>\nu\right\}}
w_0^{\ell}\mu(w,\lambda)q^{(1,1)}(w)\\
\bigwedge\sum_{j=1,2}
\Bigg(\lim_{\delta\to 0}\lim_{\epsilon\to 0}\int_{\Gamma^{\epsilon}_{\zeta}
\cap\left\{|B(\zeta,z)|>\delta,|B(w,\zeta)|>\delta\right\}}
\bar\zeta_0^{\ell}\zeta_0^{-\ell}\vartheta(\zeta)
e^{\langle\lambda,\zeta/\zeta_0\rangle-\overline{\langle\lambda,\zeta/\zeta_0\rangle}}\\
\times\det\left[\frac{\bar\zeta}{B^*(w,\zeta)}\ \frac{\bar{w}}{B(w,\zeta)}\ Q(w,\zeta)\right]
\det\left[\frac{z}{B^*(\bar\zeta,\bar{z})}\ \frac{\zeta}
{B(\bar\zeta,\bar{z})}\ Q(\bar\zeta,\bar{z})\right]\d\left(\frac{\zeta_j}{\zeta_0}\right)
\wedge\frac{\d\bar\zeta}{P(\bar\zeta)}\Bigg)
\wedge\frac{\omega_j(w)}{P(w)}\\
=|{\bf{C}}|^2\cdot\bar{z}_0^{-\ell}
\cdot e^{-\langle\lambda,z/z_0\rangle+\overline{\langle\lambda,z/z_0\rangle}}
\lim_{\nu\to 0}\lim_{\tau\to 0}\int_{\Gamma^{\tau}_w\cap\left\{\eta>|B(z,w)|>\nu\right\}}
\frac{w_0^{\ell}\mu(w,\lambda)q^{(1,1)}(w)}{B(z,w)}\\
\bigwedge\sum_{j=1,2}B(z,w)\cdot R_j(z,w,\lambda)\wedge\frac{\omega_j(w)}{P(w)}\\
+|{\bf{C}}|^2\cdot\bar{z}_0^{-\ell}
\cdot e^{-\langle\lambda,z/z_0\rangle+\overline{\langle\lambda,z/z_0\rangle}}
\lim_{\tau\to 0}\int_{\Gamma^{\tau}_w\cap\left\{|B(z,w)|>\eta\right\}}
w_0^{\ell}\mu(w,\lambda)q^{(1,1)}(w)\\
\bigwedge\sum_{j=1,2}R_j(z,w,\lambda)\wedge\frac{\omega_j(w)}{P(w)},\\
\end{multline}
where $R_j(z,w,\lambda)$ are the functions from formula \eqref{RIntegrals}.\\

\indent
For the first integral in the right-hand side of \eqref{RzwEstimate} using estimate
\eqref{RUTauEtaEstimate} we obtain the following estimate of interior integrals
\begin{multline}\label{lambdaEstimate1}
\Bigg|\lim_{\epsilon\to 0}
\int_{\Gamma^{\epsilon}_{\zeta}\setminus U_{\zeta}^{\eta}(z,w)}
\bar\zeta_0^{\ell}\zeta_0^{-\ell}\vartheta(\zeta)
e^{\langle\lambda,\zeta/\zeta_0\rangle-\overline{\langle\lambda,\zeta/\zeta_0\rangle}}\\
\times\det\left[\frac{\bar\zeta}{B^*(w,\zeta)}\ \frac{\bar{w}}{B(w,\zeta)}\ Q(w,\zeta)\right]
\det\left[\frac{z}{B^*(\bar\zeta,\bar{z})}\ \frac{\zeta}
{B(\bar\zeta,\bar{z})}\ Q(\bar\zeta,\bar{z})\right]
\d\left(\frac{\zeta_j}{\zeta_0}\right)\wedge\frac{\d\bar\zeta}{P(\bar\zeta)}\Bigg|
\leq \frac{C(\lambda)}{\eta}.
\end{multline}

\indent
Combining estimate \eqref{lambdaEstimate1} with estimate \eqref{REtaGammaEstimate}
we obtain for the first integral in the right-hand side of \eqref{RzwEstimate} for $\eta$ small enough and $\lambda>\lambda(\eta)$
\begin{multline}\label{FirstREstimate}
\Bigg|\lim_{\nu\to 0}\lim_{\tau\to 0}\int_{\Gamma^{\tau}_w\cap\left\{|B(z,w)|>\nu\right\}}
w_0^{\ell}\mu(w,\lambda)q^{(1,1)}(w)
\wedge\sum_{j=1,2}
\Bigg(\lim_{\epsilon\to 0}
\int_{\Gamma^{\epsilon}_{\zeta}\cap U_{\zeta}^{\eta}(z,w)}
\bar\zeta_0^{\ell}\zeta_0^{-\ell}\vartheta(\zeta)
e^{\langle\lambda,\zeta/\zeta_0\rangle-\overline{\langle\lambda,\zeta/\zeta_0\rangle}}\\
\times\det\left[\frac{\bar\zeta}{B^*(w,\zeta)}\ \frac{\bar{w}}{B(w,\zeta)}\ Q(w,\zeta)\right]
\det\left[\frac{z}{B^*(\bar\zeta,\bar{z})}\ \frac{\zeta}
{B(\bar\zeta,\bar{z})}\ Q(\bar\zeta,\bar{z})\right]\d\left(\frac{\zeta_j}{\zeta_0}\right)
\wedge\frac{\d\bar\zeta}{P(\bar\zeta)}\Bigg)
\wedge\frac{\omega_j(w)}{P(w)}\Bigg|\\
+\Bigg|\lim_{\nu\to 0}\lim_{\tau\to 0}\int_{\Gamma^{\tau}_w\cap\left\{|B(z,w)|>\nu\right\}}
w_0^{\ell}\mu(w,\lambda)q^{(1,1)}(w)\wedge\sum_{j=1,2}
\Bigg(\lim_{\epsilon\to 0}\int_{\Gamma^{\epsilon}_{\zeta}\setminus U_{\zeta}^{\eta}(z,w)}
\bar\zeta_0^{\ell}\zeta_0^{-\ell}\vartheta(\zeta)
e^{\langle\lambda,\zeta/\zeta_0\rangle-\overline{\langle\lambda,\zeta/\zeta_0\rangle}}\\
\times\det\left[\frac{\bar\zeta}{B^*(w,\zeta)}\ \frac{\bar{w}}{B(w,\zeta)}\ Q(w,\zeta)\right]
\det\left[\frac{z}{B^*(\bar\zeta,\bar{z})}\ \frac{\zeta}
{B(\bar\zeta,\bar{z})}\ Q(\bar\zeta,\bar{z})\right]\d\left(\frac{\zeta_j}{\zeta_0}\right)
\wedge\frac{\d\bar\zeta}{P(\bar\zeta)}\Bigg)
\wedge\frac{\omega_j(w)}{P(w)}\Bigg|\\
\leq C\cdot\Bigg|\lim_{\nu\to 0}\lim_{\tau\to 0}
\int_{\Gamma^{\tau}_w\cap\left\{\eta>|B(z,w)|>\nu\right\}}
\frac{w_0^{\ell}\mu(w,\lambda)q^{(1,1)}(w)}{B(z,w)}
\wedge\sum_{j=1,2}B(z,w)\cdot R^{\eta}_j(z,w,\lambda)\wedge\frac{\omega_j(w)}{P(w)}\Bigg|\\
+C\cdot\Bigg|\lim_{\nu\to 0}\lim_{\tau\to 0}\int_{\Gamma^{\tau}_w\cap\left\{|B(z,w)|>\eta\right\}}
w_0^{\ell}\mu(w,\lambda)q^{(1,1)}(w)
\wedge\sum_{j=1,2}
\Bigg(\lim_{\epsilon\to 0}\int_{\Gamma^{\epsilon}_{\zeta}\setminus U_{\zeta}^{\eta}(z,w)}
e^{\langle\lambda,\zeta/\zeta_0\rangle-\overline{\langle\lambda,\zeta/\zeta_0\rangle}}\\
\times\det\left[\frac{\bar\zeta}{B^*(w,\zeta)}\ \frac{\bar{w}}{B(w,\zeta)}\ Q(w,\zeta)\right]
\det\left[\frac{z}{B^*(\bar\zeta,\bar{z})}\ \frac{\zeta}
{B(\bar\zeta,\bar{z})}\ Q(\bar\zeta,\bar{z})\right]\d\left(\frac{\zeta_j}{\zeta_0}\right)
\wedge\frac{\d\bar\zeta}{P(\bar\zeta)}\Bigg)
\wedge\frac{\omega_j(w)}{P(w)}\Bigg|\\
\leq C(\lambda)\cdot\Bigg(\left|\lim_{\nu\to 0}\lim_{\tau\to 0}
\int_{\Gamma^{\tau}_w\cap\left\{\eta>|B(z,w)|>\nu\right\}}
\frac{w_0^{\ell}\mu(w,\lambda)q^{(1,1)}(w)}{B(z,w)}\wedge\frac{\omega_j(w)}{P(w)}\right|
+\frac{1}{\eta}\cdot\|q\|_{C(V)}\|\mu\|_{C(V)}\Bigg)\\
\leq C(\lambda)\cdot\|q\|_{C(V)}\|\mu\|_{C(V)}\left(\int_0^{\eta}\d t\int_0^{\sqrt{\eta}}
\frac{r\d r}{(t+r^2)}
+\frac{1}{\eta}\right)\\
\leq C(\lambda)\cdot\left(\eta\log{\eta}+\frac{1}{\eta}\right)\cdot\|q\|_{C(V)}\|\mu\|_{C(V)}.
\end{multline}

\indent
For the second integral in the right-hand side of \eqref{RzwEstimate} using estimate
\eqref{REtaEstimate} we obtain the estimate
\begin{equation}\label{SecondREstimate}
\left|\lim_{\tau\to 0}\int_{\Gamma^{\tau}_w\cap\left\{|B(z,w)|>\eta\right\}}
w_0^{\ell}\mu(w,\lambda)q^{(1,1)}(w)
\bigwedge\sum_{j=1,2}R_j(z,w,\lambda)\wedge\frac{\omega_j(w)}{P(w)}\right|
\leq \frac{C(\lambda)}{\eta}\cdot\|q\|_{C(V)}\|\mu\|_{C(V)}.
\end{equation}

\indent
Combining estimates \eqref{FirstREstimate} and \eqref{SecondREstimate} for some fixed $\eta$
and for large enough $\lambda$ we obtain estimate \eqref{RmuEstimate} for the operator $R_{\lambda}[\cdot q]$.
\end{proof}

\indent
As it was shown in Lemma~\ref{RImportance} the following proposition is a corollary of
Proposition~\ref{RlambdaEstimate}.
\begin{proposition}\label{Solvability}
Let  $q\in C_{(1,1)}(V)$ be a form with compact support in $\left\{z\in V:|z_0|>c_0\right\}$.
Then, for $\lambda$ large enough, equation \eqref{muEquation} has a solution $\mu$ that satisfies
equation
\begin{equation*}
\bar\partial\left(\partial+\left\langle\lambda,\d\left(\frac{z}{z_0}\right)\right\rangle\right)\mu
=-\mu\cdot q,
\end{equation*}
with function $f(z)=\mu(z,\lambda)e^{\left\langle\lambda, z/z_0\right\rangle}$
satisfying equation \eqref{fEquation}.\\
\end{proposition}
\begin{proof}
From the Lemma~\ref{RImportance} it follows that the only thing that needs to be
proved is the existence of a solution of equation \eqref{IntegralEquation}
\begin{equation*}
\left(I-R_{\lambda}[\cdot q]\right)\mu=1.
\end{equation*}
But the existence of a solution $\mu$ follows from the convergence of the perturbation series
for operator $\left(I-R_{\lambda}[\cdot q]\right)^{-1}$ if the operator $R_{\lambda}[\cdot q]$ has a small enough norm, which has been proved in Proposition~\ref{RlambdaEstimate} for $\lambda$ large enough.
\end{proof}

\section{\bf Integral equation for $f$ on $V$.}\label{sec:fEquation}

\indent
From the Proposition~\ref{Solvability} we obtain that for large enough $\lambda$ integral equation
\eqref{IntegralEquation} has a unique solution and the function
$f(z,\lambda)=\mu(z,\lambda)e^{\left\langle\lambda, z/z_0\right\rangle}$, defined in
\eqref{fAnsatz}, satisfies equation \eqref{fEquation}.

\indent
Using equation \eqref{fEquation} we consider the resulting formula for the function $f$ on $V$:
\begin{multline}\label{fFormula}
f(z,\lambda)=e^{\left\langle\lambda,z/z_0\right\rangle}\mu(z,\lambda)
=e^{\left\langle\lambda,z/z_0\right\rangle}\left(1+R_{\lambda}[\mu\cdot q]\right)\\
=e^{\left\langle\lambda,z/z_0\right\rangle}
+\lim_{\tau\to 0}\int_{\Gamma^{\tau}_w}
\mu(w,\lambda)q(w)e^{\left\langle\lambda,z/z_0\right\rangle}H(z,w,\lambda)\\
=e^{\left\langle\lambda,z/z_0\right\rangle}
+\lim_{\tau\to 0}\int_{\Gamma^{\tau}_w}
f(w,\lambda)q(w)e^{\left\langle\lambda,\left(z/z_0-w/w_0\right)\right\rangle}H(z,w,\lambda)\\
=e^{\left\langle\lambda,z/z_0\right\rangle}
+\lim_{\tau\to 0}\int_{\Gamma^{\tau}_w}
\partial\bar\partial f(w,\lambda)e^{\left\langle\lambda,\left(z/z_0-w/w_0\right)\right\rangle}
H(z,w,\lambda)\\
=e^{\left\langle\lambda,z/z_0\right\rangle}
+\lim_{\tau\to 0}\int_{\Gamma^{\tau}_w}
\partial\bar\partial f(w,\lambda)
e^{\overline{\langle\lambda,z/z_0\rangle}}e^{-\langle\lambda,w/w_0\rangle}
G(z,w,\lambda),
\end{multline}
where the integration in the right-hand side of \eqref{fFormula} is over a neighborhood of $V$, because of equality \eqref{fEquation} and equality $q\big|_{{\cal C}\setminus V}=0$.

\indent
The form
\begin{multline}\label{HForm}
H^{(2,0)}(z,w,\lambda)=|{\bf{C}}|^2\cdot\bar{z}_0^{-\ell}
\cdot e^{-\langle\lambda,z/z_0\rangle+\overline{\langle\lambda,z/z_0\rangle}}w_0^{\ell}\\
\times\sum_{j=1,2}\Bigg(\lim_{\delta\to 0}\lim_{\epsilon\to 0}
\int_{\Gamma^{\epsilon}_{\zeta}\cap\left\{|B(\zeta,w)|>\delta,|B(\zeta,z)|>\delta\right\}}
\bar\zeta_0^{\ell}\zeta_0^{-\ell}\vartheta(\zeta)
e^{\langle\lambda,\zeta/\zeta_0\rangle-\overline{\langle\lambda,\zeta/\zeta_0\rangle}}
K(w,\zeta)K(\bar\zeta,\bar{z})\\
\d\left(\frac{\zeta_j}{\zeta_0}\right)
\wedge\frac{\d\bar\zeta}{P(\bar\zeta)}\Bigg)\wedge\frac{\omega_j(w)}{P(w)}\\
=|{\bf{C}}|^2\cdot e^{-\langle\lambda,z/z_0\rangle+\overline{\langle\lambda,z/z_0\rangle}}\\
\times \bar{z}_0^{-\ell}w_0^{\ell}\sum_{j=1,2}\Bigg(\lim_{\delta\to 0}\lim_{\epsilon\to 0}
\int_{\Gamma^{\epsilon}_{\zeta}\cap\left\{|B(\zeta,w)|>\delta,|B(\zeta,z)|>\delta\right\}}
\bar\zeta_0^{\ell}\zeta_0^{-\ell}\vartheta(\zeta)
e^{\langle\lambda,\zeta/\zeta_0\rangle-\overline{\langle\lambda,\zeta/\zeta_0\rangle}}\\
\times\det\left[\frac{\bar\zeta}{B^*(w,\zeta)}\ \frac{\bar{w}}{B(w,\zeta)}\ Q(w,\zeta)\right]
\det\left[\frac{z}{B^*(\bar\zeta,\bar{z})}\ \frac{\zeta}
{B(\bar\zeta,\bar{z})}\ Q(\bar\zeta,\bar{z})\right]\d\left(\frac{\zeta_j}{\zeta_0}\right)
\wedge\frac{\d\bar\zeta}{P(\bar\zeta)}\Bigg)
\wedge\frac{\omega_j(w)}{P(w)}\\
=e^{-\langle\lambda,z/z_0\rangle+\overline{\langle\lambda,z/z_0\rangle}}\cdot G(z,w,\lambda)
\end{multline}
is the form from \eqref{HKernel}, and
\begin{multline}\label{GForm}
G(z,w,\lambda)=|{\bf{C}}|^2\cdot\bar{z}_0^{-\ell}w_0^{\ell}\sum_{j=1,2}\Bigg(\lim_{\delta\to 0}
\lim_{\epsilon\to 0}\int_{\Gamma^{\epsilon}_{\zeta}
\cap\left\{|B(z,\zeta)|>\delta,|B(w,\zeta)|>\delta\right\}}
\bar\zeta_0^{\ell}\zeta_0^{-\ell}\vartheta(\zeta)
e^{\langle\lambda,\zeta/\zeta_0\rangle-\overline{\langle\lambda,\zeta/\zeta_0\rangle}}\\
\times\det\left[\frac{\bar\zeta}{B^*(w,\zeta)}\ \frac{\bar{w}}{B(w,\zeta)}\ Q(w,\zeta)\right]
\det\left[\frac{z}{B^*(\bar\zeta,\bar{z})}\ \frac{\zeta}
{B(\bar\zeta,\bar{z})}\ Q(\bar\zeta,\bar{z})\right]\d\left(\frac{\zeta_j}{\zeta_0}\right)
\wedge\frac{\d\bar\zeta}{P(\bar\zeta)}\Bigg)
\wedge\frac{\omega^{(2,0)}_j(w)}{P(w)},
\end{multline}
is also the form of the type $(2,0)$ with respect to $w$.

\indent
Formula \eqref{fFormula} is motivated by and is similar to the formula in the article of G. Henkin and R. Novikov \cite{11}. The difference is the result of the difference in integral operators
used in the present article and in \cite{11}.

\indent
The goal of this section is the transformation of the right-hand side of formula \eqref{fFormula} into
an integral containing only the form $\partial f$ for the unknown function on $V$.
To transform equality \eqref{fFormula} for $z\in V$ we use the Stokes' theorem in \eqref{fFormula} and equality
\begin{equation*}
\lim_{\tau\to 0}\int_{\Gamma^{\tau}_w}\partial_w f
\wedge\partial_w G(z,w,\lambda)=0
\end{equation*}
to obtain the following equality for $z\in V$
\begin{multline}\label{InV-ByParts}
\lim_{\eta\to 0}\lim_{\tau\to 0}\int_{\Gamma^{\tau}_w\cap\{|B(z,w)|>\eta\}}
\bar\partial\partial f(w,\lambda)\wedge e^{\overline{\langle\lambda,z/z_0\rangle}}
e^{-\langle\lambda,w/w_0\rangle}G^{(2,0)}(z,w,\lambda)\\
=\lim_{\tau\to 0}\int_{\Gamma^{\tau}_w\cap \pi^{-1}(bV)}
\partial f(w,\lambda)\wedge e^{\overline{\langle\lambda,z/z_0\rangle}}
e^{-\langle\lambda,w/w_0\rangle}G^{(2,0)}(z,w,\lambda)\\
-\lim_{\eta\to 0}\lim_{\tau\to 0}\int_{\Gamma^{\tau}_w\cap\{|B(z,w)|=\eta\}}
\partial f(w,\lambda)\wedge e^{\overline{\langle\lambda,z/z_0\rangle}}
e^{-\langle\lambda,w/w_0\rangle}G^{(2,0)}(z,w,\lambda)\\
+\lim_{\eta\to 0}\lim_{\tau\to 0}\int_{\Gamma^{\tau}_w\cap\{|B(z,w)|>\eta\}}
\partial f(w,\lambda)\wedge e^{\overline{\langle\lambda,z/z_0\rangle}}
e^{-\langle\lambda,w/w_0\rangle}\bar\partial_w G^{(2,0)}(z,w,\lambda).
\end{multline}
The limit of the last integral does not a priori exist, but from the estimates below it will
be clear that it is well defined for $f$ such that $\partial f\in C_{(1,0)}(V)$.

\indent
Then, we prove the lemma that will be used in various estimates below. It provides an estimate 
of the area of the domain $\left\{|w|=1,P(w)=0,\ |B(z,w)|=\delta\right\}$ as a function of $\delta$.

\begin{lemma}\label{DeltaArea}
There exists a constant $C$ such that the following estimate
\begin{equation}\label{DeltaAreaEstimate}
\text{Area}\left({\cal D}(\delta)\right)\sim C\cdot\delta^{3/2}
\end{equation}
holds for the area of the domain
$${\cal D}(\delta)=\left\{|w|=1,P(w)=0,\ |B(z,w)|=\delta\right\}.$$
\end{lemma}
\begin{proof}
We consider the following coordinates in a neighborhood of a fixed point $z\in V$: 
\begin{equation*}
t=\text{Im}B(z,w),\ u_1=\text{Re}F(z,w), u_2=\text{Im}F(z,w), \text{so that}\ B(z,w)=it+u_1^2+u_2^2,
\end{equation*}
where $F(z,w)=u_1+iu_2$ is a local holomorphic coordinate at $z$ along the curve
$\left\{w\in \C\P^2: P(w)=0\right\}$.

Then domain ${\cal D}(\delta)$ can be described as the 2-dimensional boundary of the domain in
$\R^3$
\begin{equation*}
{\cal D}(\delta)=\left\{\sqrt{t^2+\left(u_1^2+u_2^2\right)^2}=\delta\right\}\subset\R^3,
\end{equation*}
or
\begin{equation*}
t^2+r^4=\delta^2,\ \text{or}\ r=\sqrt[4]{\delta^2-t^2}\
\text{where}\ r\ \text{is the radius of the circle in}\ \R^2(u_1,u_2).
\end{equation*}
Then, for the area $A(\delta)$ we have
\begin{equation}\label{AreaEstimate}
A(\delta)\sim C\cdot\int_0^{\delta}\left(\sqrt[4]{\delta^2-t^2}\right)\d t
\sim C\cdot\sqrt[4]{\delta}\int_0^{\delta}\sqrt[4]{\delta-t}\ \d t
=-C\cdot\sqrt[4]{\delta}\left(u^{5/4}\Bigg|_{\delta}^0\right)
=C\cdot\delta^{3/2}.
\end{equation}
\end{proof}

\indent
In the next lemma, using estimate \eqref{DeltaAreaEstimate}, we eliminate the second term in the right-hand side of equality \eqref{InV-ByParts}.

\begin{lemma}\label{ZeroLimit}
Let the form $g^{(1,0)}(w,\lambda)\in C^{(1,0)}(V)$. Then the following equality holds
\begin{equation}\label{ZeroLimitIntegral}
\lim_{\eta\to 0}\lim_{\tau\to 0}\int_{\Gamma^{\tau}_w\cap\{|B(z,w)|=\eta\}}
g^{(1,0)}(w,\lambda)\wedge e^{\overline{\langle\lambda,z/z_0\rangle}}
e^{-\langle\lambda,w/w_0\rangle}G(z,w,\lambda)=0.
\end{equation}
\end{lemma}
\begin{proof}
\indent
We transform the integral from \eqref{ZeroLimitIntegral} into
\begin{multline}\label{ZeroSecondIntegral}
\lim_{\eta\to 0}\lim_{\tau\to 0}\int_{\Gamma^{\tau}_w\cap\{|B(z,w)|=\eta\}}
g^{(1,0)}(w,\lambda)\wedge G(z,w,\lambda)\\
=|{\bf{C}}|^2\cdot\bar{z}_0^{-\ell}\cdot e^{\overline{\langle\lambda,z/z_0\rangle}}
\lim_{\eta\to 0}\lim_{\tau\to 0}\int_{\Gamma^{\tau}_w\cap\{|B(z,w)|=\eta\}}
\frac{g^{(1,0)}(w,\lambda)\cdot e^{-\left\langle\lambda,w/w_0\right\rangle}}
{B(z,w)\cdot P(w)}\\
\bigwedge\sum_{j=1,2}w_0^{\ell}\cdot B(z,w)
\Bigg(\lim_{\delta\to 0}\lim_{\epsilon\to 0}\int_{\Gamma^{\epsilon}_{\zeta}
\cap\left\{|B(\zeta,z)|>\delta,|B(w,\zeta)|>\delta\right\}}
\bar\zeta_0^{\ell}\zeta_0^{-\ell}\vartheta(\zeta)
e^{\langle\lambda,\zeta/\zeta_0\rangle-\overline{\langle\lambda,\zeta/\zeta_0\rangle}}\\
\times\det\left[\frac{\bar\zeta}{B^*(w,\zeta)}\ \frac{\bar{w}}{B(w,\zeta)}\ Q(w,\zeta)\right]
\det\left[\frac{z}{B^*(\bar\zeta,\bar{z})}\ \frac{\zeta}
{B(\bar\zeta,\bar{z})}\ Q(\bar\zeta,\bar{z})\right]
\d\left(\frac{\zeta_j}{\zeta_0}\right)\wedge\frac{\d\bar\zeta}{P(\bar\zeta)}\Bigg)
\wedge\omega_j(w).\\
\end{multline}
\indent
To estimate the interior integrals in the right-hand side of \eqref{ZeroSecondIntegral} we consider
the following two domains as in \eqref{U-neighborhoods}
\begin{equation*}
\begin{aligned}
&U^{\eta}_{\zeta}(w)=\left\{\zeta\in \Gamma^{\epsilon}_{\zeta}: |B(w,\zeta)|<\eta/9\right\},\\
&U^{\eta}_{\zeta}(z)=\left\{\zeta\in \Gamma^{\epsilon}_{\zeta}: |B(z,\zeta)|<\eta/9\right\}.
\end{aligned}
\end{equation*}
\indent
Since, according to \eqref{GammaInequalities}, for $\zeta$ such that $|B(z,\zeta)|\leq \eta/9$ we have
$2\eta/9<|B(w,\zeta)|\leq 16\eta/9$, we have the following equalities
\begin{equation*}
\begin{aligned}
&U^{\eta}_{\zeta}(w)\cap U^{\eta}_{\zeta}(z)=\emptyset,\\
&U^{\eta}_{\zeta}(w)\cup U^{\eta}_{\zeta}(z)\cup\left\{\Gamma^{\epsilon}_{\zeta}
\setminus \left(U^{\eta}_{\zeta}(w)\cup U^{\eta}_{\zeta}(z)\right)\right\}=\Gamma^{\epsilon}_{\zeta}.
\end{aligned}
\end{equation*}

\indent
Then, for $U^{\eta}_{\zeta}(w)$ using estimates \eqref{GammaInequalities} we obtain the following estimate:
\begin{multline}\label{UBigDeltawEstimate}
\Bigg|B(z,w)\cdot\lim_{\epsilon\to 0}\int_{\Gamma^{\epsilon}_{\zeta}\cap U^{\eta}_{\zeta}(w)}
\bar\zeta_0^{\ell}\zeta_0^{-\ell}\vartheta(\zeta)
e^{\langle\lambda,\zeta/\zeta_0\rangle-\overline{\langle\lambda,\zeta/\zeta_0\rangle}}\\
\times\det\left[\frac{\bar\zeta}{B^*(w,\zeta)}\ \frac{\bar{w}}{B(w,\zeta)}\ Q(w,\zeta)\right]
\det\left[\frac{z}{B^*(\bar\zeta,\bar{z})}\ \frac{\zeta}
{B(\bar\zeta,\bar{z})}\ Q(\bar\zeta,\bar{z})\right]
\d\left(\frac{\zeta_j}{\zeta_0}\right)\wedge\frac{\d\bar\zeta}{P(\bar\zeta)}\Bigg|\\
\leq C\cdot\eta\int_{V\cap U^{\eta}_{\zeta}(w)}
\frac{|\zeta-z||{\bar\zeta}-{\bar w}|}{|B^*({\bar\zeta},{\bar z})||B({\bar\zeta},{\bar z})|
|B^*(w,\zeta)\cdot B(w,\zeta)|}
\d\left(\frac{\zeta_j}{\zeta_0}\right)\wedge\left(\d{\bar P}(\zeta)\interior\d\bar\zeta\right)\\
\leq C\cdot\int_{V\cap U^{\eta}_{\zeta}(w)}
\frac{|\zeta-z||{\bar\zeta}-{\bar w}|}{|B^*({\bar\zeta},{\bar z})|
|B^*(w,\zeta)||B(w,\zeta)|}
\d\left(\frac{\zeta_j}{\zeta_0}\right)\wedge\left(\d{\bar P}(\zeta)\interior\d\bar\zeta\right)\\
\leq \frac{C}{\sqrt{\eta}}\cdot\int_0^{\eta}\d t\int_0^{\sqrt{\eta}}\frac{r^2\d r}{(t+r^2)^2}
\leq\frac{C}{\sqrt{\eta}}\cdot\int_0^{\sqrt{\eta}}\frac{r^2\d r}{r^2}\leq C,\\
\end{multline}
where we used the corollaries of estimates \eqref{GammaInequalities}:
\begin{equation}\label{B-Corollaries}
\left\{
\begin{aligned}
&|B({\bar\zeta},{\bar z})|\geq c_1\cdot\eta,\\
&|B({\bar\zeta},{\bar z})|\geq c_2\cdot\sqrt{\eta}|\zeta-z|,
\end{aligned}\right.
\end{equation}
and variables $ r=|\zeta-w|$, $t=\text{Im}B^*(w,\zeta)$.

\indent
Similarly, we obtain the same estimate for $U^{\eta}_{\zeta}(z)$:
\begin{multline}\label{UBigDeltazEstimate}
\Bigg|B(z,w)\cdot\int_{\Gamma^{\epsilon}_{\zeta}\cap U^{\eta}_{\zeta}(z)}
\bar\zeta_0^{\ell}\zeta_0^{-\ell}\vartheta(\zeta)
e^{\langle\lambda,\zeta/\zeta_0\rangle-\overline{\langle\lambda,\zeta/\zeta_0\rangle}}\\
\times\det\left[\frac{\bar\zeta}{B^*(w,\zeta)}\ \frac{\bar{w}}{B(w,\zeta)}\ Q(w,\zeta)\right]
\det\left[\frac{z}{B^*(\bar\zeta,\bar{z})}\ \frac{\zeta}
{B(\bar\zeta,\bar{z})}\ Q(\bar\zeta,\bar{z})\right]
\d\left(\frac{\zeta_j}{\zeta_0}\right)\wedge\frac{\d\bar\zeta}{P(\bar\zeta)}\Bigg|\leq C.
\end{multline}

\indent
Below we estimate the same integral over the domain
$\Gamma^{\epsilon}_{\zeta}
\setminus\left(U^{\eta}_{\zeta}(w)\cup U^{\eta}_{\zeta}(z)\right)$
for sufficiently small $\eta$:
\begin{multline}\label{UNoDeltaEstimate}
\Bigg|B(z,w)\cdot\lim_{\epsilon\to 0}\int_{\Gamma^{\epsilon}_{\zeta}
\setminus \left(U^{\eta}_{\zeta}(w)\cup U^{\eta}_{\zeta}(z)\right)}
\bar\zeta_0^{\ell}\zeta_0^{-\ell}\vartheta(\zeta)
e^{\langle\lambda,\zeta/\zeta_0\rangle-\overline{\langle\lambda,\zeta/\zeta_0\rangle}}\\
\times\det\left[\frac{\bar\zeta}{B^*(w,\zeta)}\ \frac{\bar{w}}{B(w,\zeta)}\ Q(w,\zeta)\right]
\det\left[\frac{z}{B^*(\bar\zeta,\bar{z})}\ \frac{\zeta}
{B(\bar\zeta,\bar{z})}\ Q(\bar\zeta,\bar{z})\right]
d\left(\frac{\zeta_j}{\zeta_0}\right)\wedge\frac{\d\bar\zeta}{P(\bar\zeta)}\Bigg|\\
\leq C\cdot\eta\int_{V\setminus \left(U^{\eta}_{\zeta}(w)\cup U^{\eta}_{\zeta}(z)\right)}
\frac{|\zeta-z||{\bar\zeta}-{\bar w}|}{|B^*({\bar\zeta},{\bar z})||B({\bar\zeta},{\bar z})|
|B^*(w,\zeta)||B(w,\zeta)|}
d\left(\frac{\zeta_j}{\zeta_0}\right)\wedge\left(d{\bar P}(\zeta)\interior\d\bar\zeta\right)\\
\leq C\cdot\eta^{1/2}\int_{\eta}^A\d t
\int_{\sqrt{\eta}}^B\frac{r^2\d r}{(\eta+t+r^2)^2(\eta+r^2)}
\leq C\cdot\eta^{1/2}\int_{\sqrt{\eta}}^B\frac{\d r}{\eta+r^2}\\
\leq C\cdot \int_1^{\infty}\frac{du}{1+u^2}\leq C,
\end{multline}
where again we used inequalities \eqref{B-Corollaries}
and variables $ r=|\zeta-w|$, $t=\text{Im}B^*(w,\zeta)$.\\
\indent
Estimates \eqref{UBigDeltawEstimate}, \eqref{UBigDeltazEstimate},
and \eqref{UNoDeltaEstimate} show
that the interior integrals in \eqref{ZeroSecondIntegral}
\begin{multline*}
G_j(z,w,\lambda)\big|_{j=1,2}=\lim_{\epsilon\to 0}\int_{\Gamma^{\epsilon}_{\zeta}}
\bar\zeta_0^{\ell}\zeta_0^{-\ell}\vartheta(\zeta)
e^{\langle\lambda,\zeta/\zeta_0\rangle-\overline{\langle\lambda,\zeta/\zeta_0\rangle}}\\
\times\det\left[\frac{\bar\zeta}{B^*(w,\zeta)}\ \frac{\bar{w}}{B(w,\zeta)}\ Q(w,\zeta)\right]
\det\left[\frac{z}{B^*(\bar\zeta,\bar{z})}\ \frac{\zeta}
{B(\bar\zeta,\bar{z})}\ Q(\bar\zeta,\bar{z})\right]
\d\left(\frac{\zeta_j}{\zeta_0}\right)\wedge\frac{\d\bar\zeta}{P(\bar\zeta)}
\end{multline*}
satisfy the following estimate as functions of variables $z,w$:
\begin{equation}\label{GJEstimate}
\left|B(z,w)\cdot G_j(z,w,\lambda)\right|\leq C.
\end{equation}

\indent
Then, we can rewrite the integral in \eqref{ZeroSecondIntegral} as follows
\begin{multline}\label{GSum}
\lim_{\eta\to 0}\lim_{\tau\to 0}\int_{\Gamma^{\tau}_w\cap\{|B(z,w)|=\eta\}}
\frac{g^{(1,0)}(w,\lambda)\cdot e^{-\left\langle\lambda,w/w_0\right\rangle}}
{B(z,w)\cdot P(w)}\\
\bigwedge\sum_{j=1,2}w_0^{\ell}B(z,w)
\Bigg(\lim_{\epsilon\to 0}\int_{\Gamma^{\epsilon}_{\zeta}}
\bar\zeta_0^{\ell}\zeta_0^{-\ell}\vartheta(\zeta)
e^{\langle\lambda,\zeta/\zeta_0\rangle-\overline{\langle\lambda,\zeta/\zeta_0\rangle}}\\
\times\det\left[\frac{\bar\zeta}{B^*(w,\zeta)}\ \frac{\bar{w}}{B(w,\zeta)}\ Q(w,\zeta)\right]
\det\left[\frac{z}{B^*(\bar\zeta,\bar{z})}\ \frac{\zeta}
{B(\bar\zeta,\bar{z})}\ Q(\bar\zeta,\bar{z})\right]
\d\left(\frac{\zeta_j}{\zeta_0}\right)\wedge\frac{\d\bar\zeta}{P(\bar\zeta)}\Bigg)
\wedge\omega_j(w)\\
=\lim_{\eta\to 0}\lim_{\tau\to 0}\int_{\Gamma^{\tau}_w\cap\{|B(z,w)|=\eta\}}
\frac{g^{(1,0)}(w,\lambda)\cdot e^{-\left\langle\lambda,w/w_0\right\rangle}}{B(z,w)\cdot P(w)}
\wedge w_0^{\ell}\sum_{j=1,2}B(z,w)\cdot G_j(z,w,\lambda)\omega_j(w).\\
\end{multline}

\indent
Using the continuity of $g^{(1,0)}(w,\lambda)$, estimate of the functions $G_j(z,w,\lambda)$
in \eqref{GJEstimate}, and the area estimate from \eqref{DeltaAreaEstimate} in equality \eqref{GSum} we obtain
\begin{multline}\label{BoundaryIntegralG}
\Bigg|\lim_{\tau\to 0}\int_{\Gamma^{\tau}_w\cap\{|B(z,w)|=\eta\}}
\frac{g^{(1,0)}(w,\lambda)\cdot e^{-\left\langle\lambda,w/w_0\right\rangle}}{B(z,w)\cdot P(w)}
\wedge w_0^{\ell}\sum_{j=1,2}B(z,w)\cdot G_j(z,w,\lambda)\omega_j(w)\Bigg|\\
\leq C\cdot \|f\|_{C^1}\frac{\eta^{3/2}}{\eta} \to 0,\ \text{as}\ \eta\to 0,
\end{multline}
which implies equality \eqref{ZeroLimitIntegral}.
\end{proof}

\indent
Using equalities \eqref{InV-ByParts} and \eqref{ZeroLimitIntegral} in \eqref{fFormula} on $V$
we obtain that equality
\begin{multline}\label{fIntegralEquation}
f(z,\lambda)=e^{\left\langle\lambda,z/z_0\right\rangle}\\
-\lim_{\eta\to 0}\Bigg(\lim_{\tau\to 0}
\int_{\Gamma^{\tau}_w\cap \pi^{-1}(bV)\cap\{|B(z,w)|>\eta\}}
\partial_w f(w,\lambda)\wedge e^{\overline{\langle\lambda,z/z_0\rangle}}
e^{-\langle\lambda,w/w_0\rangle}G(z,w,\lambda)\\
+\lim_{\tau\to 0}\int_{\Gamma^{\tau}_w\cap\{|B(z,w)|>\eta\}}
\partial_w f(w,\lambda)\wedge e^{\overline{\langle\lambda,z/z_0\rangle}}
e^{-\langle\lambda,w/w_0\rangle}\bar\partial_w G(z,w,\lambda)\Bigg)
\end{multline}
should be satisfied by any solution $f(w,\lambda)$ of \eqref{fFormula} with
$\partial_w f(w,\lambda)\in C_{(1,0)}(V)$. To see that this condition is satisfied for
$q\in C_{(1,1)}(V)$ by the solution constructed in Proposition~\ref{Solvability} we notice that this solution is continuous. Therefore, we can consider
equation \eqref{fEquation} as an equation of the form
$$\partial\bar\partial f=u$$
with a continuous right-hand side, which implies the inclusion
$\partial_w f(w,\lambda)\in C_{(1,0)}(V)$.

\indent
We further transform equality \eqref{fIntegralEquation} by multiplying both sides
by $e^{-\overline{\left\langle\lambda,z/z_0\right\rangle}}$.
Then, we obtain the equality
\begin{multline*}
f(z,\lambda)e^{-\overline{\left\langle\lambda,z/z_0\right\rangle}}
=e^{\left\langle\lambda,z/z_0\right\rangle}e^{-\overline{\left\langle\lambda,z/z_0\right\rangle}}\\
-\lim_{\eta\to 0}\Bigg(\lim_{\tau\to 0}
\int_{\Gamma^{\tau}_w\cap \pi^{-1}(bV)\cap\{|B(z,w)|>\eta\}}
\partial_w f(w,\lambda)\wedge e^{-\langle\lambda,w/w_0\rangle}G(z,w,\lambda)\\
+\lim_{\tau\to 0}\int_{\Gamma^{\tau}_w\cap\{|B(z,w)|>\eta\}}
\partial_w f(w,\lambda)\wedge
e^{-\langle\lambda,w/w_0\rangle}\bar\partial_w G(z,w,\lambda)\Bigg),\\
\end{multline*}
which, after denoting
$h(z,\lambda)=f(z,\lambda)e^{-\overline{\left\langle\lambda,z/z_0\right\rangle}}$,
we rewrite as equality
\begin{multline}\label{hIntegralEquation}
h(z,\lambda)
=e^{\left\langle\lambda,z/z_0\right\rangle-\overline{\left\langle\lambda,z/z_0\right\rangle}}\\
-\lim_{\eta\to 0}\Bigg(\lim_{\tau\to 0}
\int_{\Gamma^{\tau}_w\cap \pi^{-1}(bV)\cap\{|B(z,w)|>\eta\}}
\partial_w h(w,\lambda)\cdot e^{\overline{\left\langle\lambda,w/w_0\right\rangle}
-\langle\lambda,w/w_0\rangle}G(z,w,\lambda)\\
+\lim_{\tau\to 0}\int_{\Gamma^{\tau}_w\cap\{|B(z,w)|>\eta\}}
\partial_w h(w,\lambda)\wedge
e^{\overline{\left\langle\lambda,w/w_0\right\rangle}
-\langle\lambda,w/w_0\rangle}\bar\partial_w G(z,w,\lambda)\Bigg).
\end{multline}

\indent
The importance of equality \eqref{hIntegralEquation} is in the fact that it does not explicitly depend on the potential $q$. Therefore, the derivative of this equality -- equation \eqref{FormIntegralEquation} and its reformulation in \eqref{PEquation} --
can be solved with respect to $\partial h$ without the knowledge of $q$, and then its solution
can be used together with the knowledge of the $\partial$-to-$\bar\partial$ map from \eqref{chiDefinition} to find a function $h(z,\lambda)$ and a solution
\begin{equation}\label{fhEquality}
f(z,\lambda)=e^{\overline{\left\langle\lambda,z/z_0\right\rangle}}h(z,\lambda)
\end{equation}
of equation \eqref{fEquation}.

\section{\bf Estimates for the boundary integral in \eqref{hIntegralEquation}.}\label{sec:BoundaryEstimates}

\indent
In this section we prove some estimates for the $\partial_z$- derivatives of the boundary
integral in \eqref{hIntegralEquation} to be used in section \ref{sec:Solving}.
We start with the following lemma, in which we slightly modify the formula for the interior integral in the definition
of $\partial_z G(z,w,\lambda)$ to make it easier to estimate this not absolutely converging integral.

\begin{lemma}\label{ChangeOneDeterminant}
The following equality holds for arbitrary fixed $\delta>0$
\begin{multline}\label{NewDeterminant}
\lim_{\epsilon\to 0}\int_{\Gamma_{\zeta}^{\epsilon}\setminus\left(U^{\delta}_{\zeta}(w)\cup U^{\delta}_{\zeta}(z)\right)}
\bar\zeta_0^{\ell}\zeta_0^{-\ell}\vartheta(\zeta)
e^{\langle\lambda,\zeta/\zeta_0\rangle-\overline{\langle\lambda,\zeta/\zeta_0\rangle}}\\
\times\det\left[\frac{\bar\zeta}{B^*(w,\zeta)}\ \frac{\bar{w}}{B(w,\zeta)}\ Q(w,\zeta)\right]
\det\left[\partial_z\left(\frac{z}{B^*(\bar\zeta,\bar{z})}\right)\ \frac{\zeta}{B(\bar\zeta,\bar{z})}\
\frac{Q(\bar\zeta,{\bar z})}{P(\bar\zeta)})\right]
\d\left(\frac{\zeta_j}{\zeta_0}\right)\wedge\d\bar\zeta\\
=\lim_{\epsilon\to 0}\int_{\Gamma_{\zeta}^{\epsilon}\setminus\left(U^{\delta}_{\zeta}(w)
\cup U^{\delta}_{\zeta}(z)\right)}\bar\zeta_0^{\ell}\zeta_0^{-\ell}\vartheta(\zeta)
e^{\langle\lambda,\zeta/\zeta_0\rangle-\overline{\langle\lambda,\zeta/\zeta_0\rangle}}\\
\times\det\left[\frac{\bar\zeta}{B^*(w,\zeta)}\ \frac{\bar{w}}{B(w,\zeta)}\ Q(w,\zeta)\right]
\det\left[\frac{\d z}{B^*(\bar\zeta,\bar{z})}\ \frac{z}{B^*(\bar\zeta,\bar{z})}\
\frac{Q(\bar\zeta,\bar{z})}{P(\bar\zeta)}\right]
\d\left(\frac{\zeta_j}{\zeta_0}\right)\wedge\d\bar\zeta.
\end{multline}
\end{lemma}
\begin{proof}
To prove the lemma we use equality
\begin{multline}\label{DetEquality}
\det\left[\partial_z\left(\frac{z}{B^*(\bar\zeta,\bar{z})}\right)\ \frac{\zeta}{B(\bar\zeta,\bar{z})}\
\frac{Q(\bar\zeta,\bar{z})}{P(\bar\zeta)}\right]
=\det\left[\partial_z\left(\frac{z}{B^*(\bar\zeta,\bar{z})}\right)\ \frac{z}{B^*(\bar\zeta,\bar{z})}\
\frac{Q(\bar\zeta,\bar{z})}{P(\bar\zeta)}\right]\\
-\det\left[\partial_z\left(\frac{z}{B^*(\bar\zeta,\bar{z})}\right)\ \frac{z}{B^*(\bar\zeta,\bar{z})}\
\frac{\zeta}{B(\bar\zeta,\bar{z})}\right],
\end{multline}
which is a corollary of the following equality for the determinant of the $4\times 4$ matrix:
\begin{equation*}
\det\left[\begin{tabular}{cccc}
0&1&1&1\vspace{0.05in}\\
${\dis \partial_z\left(\frac{z}{B^*(\bar\zeta,\bar{z})}\right) }$&${\dis \frac{z}{B^*(\bar\zeta,\bar{z})} }$&
${\dis \frac{\zeta}{B(\bar\zeta,\bar{z})} }$&
${\dis \frac{Q(\bar\zeta,\bar{z})}{P(\bar\zeta)} }$
\end{tabular}\right]=0.
\end{equation*}
The last equality follows from the fact that the first row of the matrix is a linear combination of the lower rows with
coefficients $(\bar\zeta_j-{\bar z}_j)\ (j=0,1,2)$.\\
\indent
Using equality \eqref{DetEquality} we obtain that
\begin{multline*}
\lim_{\epsilon\to 0}\int_{\Gamma_{\zeta}^{\epsilon}\setminus
\left(U^{\delta}_{\zeta}(w)\cup U^{\delta}_{\zeta}(z)\right)}
\bar\zeta_0^{\ell}\zeta_0^{-\ell}\vartheta(\zeta)
e^{\langle\lambda,\zeta/\zeta_0\rangle-\overline{\langle\lambda,\zeta/\zeta_0\rangle}}\\
\times\det\left[\frac{\bar\zeta}{B^*(w,\zeta)}\ \frac{\bar{w}}{B(w,\zeta)}\ Q(w,\zeta)\right]
\det\left[\partial_z\left(\frac{z}{B^*(\bar\zeta,\bar{z})}\right)\ \frac{\zeta}{B(\bar\zeta,\bar{z})}\
\frac{Q(\bar\zeta,{\bar z})}{P(\bar\zeta)})\right]
\d\left(\frac{\zeta_j}{\zeta_0}\right)\wedge\d\bar\zeta\\
=\lim_{\epsilon\to 0}\int_{\Gamma_{\zeta}^{\epsilon}
\setminus\left(U^{\delta}_{\zeta}(w)\cup U^{\delta}_{\zeta}(z)\right)}
\bar\zeta_0^{\ell}\zeta_0^{-\ell}\vartheta(\zeta)
e^{\langle\lambda,\zeta/\zeta_0\rangle-\overline{\langle\lambda,\zeta/\zeta_0\rangle}}\\
\times\det\left[\frac{\bar\zeta}{B^*(w,\zeta)}\ \frac{\bar{w}}{B(w,\zeta)}\ Q(w,\zeta)\right]
\det\left[\partial_z\left(\frac{z}{B^*(\bar\zeta,\bar{z})}\right)\ \frac{z}{B^*(\bar\zeta,\bar{z})}\
\frac{Q(\bar\zeta,\bar{z})}{P(\bar\zeta)}\right]
\d\left(\frac{\zeta_j}{\zeta_0}\right)\wedge\d\bar\zeta\\
-\lim_{\epsilon\to 0}\int_{\Gamma_{\zeta}^{\epsilon}\setminus\left(U^{\delta}_{\zeta}(w)\cup U^{\delta}_{\zeta}(z)\right)}
\bar\zeta_0^{\ell}\zeta_0^{-\ell}\vartheta(\zeta)
e^{\langle\lambda,\zeta/\zeta_0\rangle-\overline{\langle\lambda,\zeta/\zeta_0\rangle}}\\
\times\det\left[\frac{\bar\zeta}{B^*(w,\zeta)}\ \frac{\bar{w}}{B(w,\zeta)}\ Q(w,\zeta)\right]
\det\left[\partial_z\left(\frac{z}{B^*(\bar\zeta,\bar{z})}\right)\ \frac{z}{B^*(\bar\zeta,\bar{z})}\
\frac{\zeta}{B(\bar\zeta,\bar{z})}\right]
\d\left(\frac{\zeta_j}{\zeta_0}\right)\wedge\d\bar\zeta\\
=\lim_{\epsilon\to 0}\int_{\Gamma_{\zeta}^{\epsilon}\setminus\left(U^{\delta}_{\zeta}(w)\cup U^{\delta}_{\zeta}(z)\right)}
\bar\zeta_0^{\ell}\zeta_0^{-\ell}\vartheta(\zeta)
e^{\langle\lambda,\zeta/\zeta_0\rangle-\overline{\langle\lambda,\zeta/\zeta_0\rangle}}\\
\times\det\left[\frac{\bar\zeta}{B^*(w,\zeta)}\ \frac{\bar{w}}{B(w,\zeta)}\ Q(w,\zeta)\right]
\det\left[\frac{\d z}{B^*(\bar\zeta,\bar{z})}\ \frac{z}{B^*(\bar\zeta,\bar{z})}\
\frac{Q(\bar\zeta,\bar{z})}{P(\bar\zeta)}\right]
\d\left(\frac{\zeta_j}{\zeta_0}\right)\wedge\d\bar\zeta,\\
\end{multline*}
where in the last equality we used the absence of the factor $1/P(\bar\zeta)$ in
\begin{equation*}
\det\left[\frac{\bar\zeta}{B^*(w,\zeta)}\ \frac{\bar{w}}{B(w,\zeta)}\ Q(w,\zeta)\right]
\det\left[\frac{\d z}{B^*(\bar\zeta,\bar{z})}\ \frac{z}{B^*(\bar\zeta,\bar{z})}\
\frac{\zeta}{B(\bar\zeta,\bar{z})}\right]
\d\left(\frac{\zeta_j}{\zeta_0}\right)\wedge\d\bar\zeta
\end{equation*}
and the fact that
$\text{Volume}\left\{\Gamma_{\zeta}^{\epsilon}
\setminus\left(U^{\delta}_{\zeta}(w)\cup U^{\delta}_{\zeta}(z)\right)\right\}\to 0$ as
$\epsilon\to 0$.
\end{proof}

\indent
Using Lemma~\ref{ChangeOneDeterminant} and equalities
\begin{equation*}
\begin{aligned}
&\partial_z B(\bar\zeta,\bar{z})=\partial_z\left(1-\sum_{j=0}^2\zeta_j{\bar z}_j\right)=0,\\
&\partial_z B^*(\bar\zeta,{\bar z})
=\partial_z\left(\sum_{j=0}^2\bar\zeta_jz_j-1\right)=\sum_{j=0}^2\bar\zeta_j\d z_j,
\end{aligned}
\end{equation*}
we obtain the following equality
\begin{multline}\label{PartialzG}
\partial_z G(z,w,\lambda)=\partial_z\Bigg\{
|{\bf{C}}|^2\cdot\bar{z}_0^{-\ell}w_0^{\ell}\\
\times\sum_{j=1,2}\Bigg(\lim_{\delta\to 0}
\lim_{\epsilon\to 0}\int_{\Gamma^{\epsilon}_{\zeta}
\cap\left\{|B(z,\zeta)|>\delta,|B(w,\zeta)|>\delta\right\}}
\bar\zeta_0^{\ell}\zeta_0^{-\ell}\vartheta(\zeta)
e^{\langle\lambda,\zeta/\zeta_0\rangle-\overline{\langle\lambda,\zeta/\zeta_0\rangle}}\\
\times\det\left[\frac{\bar\zeta}{B^*(w,\zeta)}\ \frac{\bar{w}}{B(w,\zeta)}\ Q(w,\zeta)\right]
\det\left[\frac{z}{B^*(\bar\zeta,\bar{z})}\ \frac{\zeta}
{B(\bar\zeta,\bar{z})}\ Q(\bar\zeta,\bar{z})\right]\d\left(\frac{\zeta_j}{\zeta_0}\right)
\wedge\frac{\d\bar\zeta}{P(\bar\zeta)}\Bigg)
\wedge\frac{\omega^{(2,0)}_j(w)}{P(w)}\Bigg\}\\
=|{\bf{C}}|^2\cdot\bar{z}_0^{-\ell}w_0^{\ell}
\sum_{j=1,2}\Bigg(\lim_{\delta\to 0}\lim_{\epsilon\to 0}
\int_{\Gamma_{\zeta}^{\epsilon}\setminus\left(U^{\delta}_{\zeta}(w)\cup U^{\delta}_{\zeta}(z)\right)}
\bar\zeta_0^{\ell}\zeta_0^{-\ell}\vartheta(\zeta)
e^{\langle\lambda,\zeta/\zeta_0\rangle-\overline{\langle\lambda,\zeta/\zeta_0\rangle}}\\
\wedge\det\left[\frac{\bar\zeta}{B^*(w,\zeta)}\ \frac{\bar{w}}{B(w,\zeta)}\ Q(w,\zeta)\right]
\det\left[\partial_z\left(\frac{z}{B^*(\bar\zeta,\bar{z})}\right)\
\frac{\zeta}{B(\bar\zeta,\bar{z})}\ Q(\bar\zeta,\bar{z})\right]\\
\wedge\d\left(\frac{\zeta_j}{\zeta_0}\right)
\wedge\frac{\d\bar\zeta}{P(\bar\zeta)}\Bigg)
\wedge\frac{\omega^{(2,0)}_j(w)}{P(w)}\\
=|{\bf{C}}|^2\cdot\bar{z}_0^{-\ell}w_0^{\ell}
\sum_{j=1,2}\Bigg(\lim_{\delta\to 0}\lim_{\epsilon\to 0}
\int_{\Gamma_{\zeta}^{\epsilon}
\setminus\left(U^{\delta}_{\zeta}(w)\cup U^{\delta}_{\zeta}(z)\right)}
\bar\zeta_0^{\ell}\zeta_0^{-\ell}\vartheta(\zeta)
e^{\langle\lambda,\zeta/\zeta_0\rangle-\overline{\langle\lambda,\zeta/\zeta_0\rangle}}\\
\wedge\det\left[\frac{\bar\zeta}{B^*(w,\zeta)}\ \frac{\bar{w}}{B(w,\zeta)}\ Q(w,\zeta)\right]
\det\left[\frac{\d z}{B^*(\bar\zeta,\bar{z})}\
\frac{z}{B^*(\bar\zeta,\bar{z})}\ Q(\bar\zeta,\bar{z})\right]
\wedge\d\left(\frac{\zeta_j}{\zeta_0}\right)
\wedge\frac{\d\bar\zeta}{P(\bar\zeta)}\Bigg)
\wedge\frac{\omega^{(2,0)}_j(w)}{P(w)},
\end{multline}
where $\omega_1(w), \omega_2(w)$ are the forms from \eqref{omegaforms}.

\indent
In estimating the form $\partial_z G(z,w,\lambda)$ we will use equality \eqref{PartialzG}, and will be
proving estimates for the forms
\begin{multline}\label{NForm}
N_j(z,w,\lambda)=w_0^{\ell}\cdot\lim_{\delta\to 0}\lim_{\epsilon\to 0}
\int_{\Gamma_{\zeta}^{\epsilon}
\setminus\left(U^{\delta}_{\zeta}(w)\cup U^{\delta}_{\zeta}(z)\right)}
\bar\zeta_0^{\ell}\zeta_0^{-\ell}\vartheta(\zeta)
e^{\langle\lambda,\zeta/\zeta_0\rangle-\overline{\langle\lambda,\zeta/\zeta_0\rangle}}\\
\times\det\left[\frac{\bar\zeta}{B^*(w,\zeta)}\ \frac{\bar{w}}{B(w,\zeta)}\ Q(w,\zeta)\right]
\det\left[\frac{\d z}{B^*(\bar\zeta,\bar{z})}\
\frac{z}{B^*(\bar\zeta,\bar{z})}\ Q(\bar\zeta,\bar{z})\right]
\wedge\d\left(\frac{\zeta_j}{\zeta_0}\right)
\wedge\frac{\d\bar\zeta}{P(\bar\zeta)}.
\end{multline}

{\bf Remark.} {\it In several lemmas below we use the Stokes' theorem for the form
$$\frac{\Psi(z,w,\lambda,\zeta)}{B^*({\bar\zeta},{\bar z})^2}\wedge\text{volume form}.$$
Application of the Stokes' theorem on the domain
$\Gamma^{\epsilon}_{\zeta}=\{\zeta: |\zeta|=1, |P(\zeta)|=\epsilon\}$ is based on the following
representation of the volume form on $\Gamma^{\epsilon}_{\zeta}$
\begin{equation}\label{GammaVolumeForm}
\d\left(\frac{\zeta_j}{\zeta_0}\right)\wedge\d\bar\zeta\sim
\left[\left(\partial_{\zeta}P(\zeta)\wedge\left\langle\bar\zeta,\d\zeta\right\rangle\right)\interior\d\zeta\right]
\wedge\d\bar\zeta,
\end{equation}
and on equalities
\begin{equation*}
\begin{aligned}
&\partial_{\zeta}B^*(\bar\zeta,{\bar z})=0,\
\d_{\zeta}B^*(\bar\zeta,{\bar z})=\bar\partial_{\zeta}B^*(\bar\zeta,{\bar z})
=\sum_{j=0}^2z_j\d\bar\zeta_j=\left\langle z,\d\bar\zeta\right\rangle,\\
&\d\bar\zeta=\left\langle z,\d\bar\zeta\right\rangle
\wedge\left(\left\langle z,\d\bar\zeta\right\rangle\interior\d\bar\zeta\right).
\end{aligned}
\end{equation*}
}
\indent
In the next two lemmas we obtain the auxiliary estimates that will be used in the estimates, respectively of the integral
\begin{multline*}
N_j(z,w,\lambda)=w_0^{\ell}\cdot\lim_{\epsilon\to 0}
\int_{\Gamma^{\epsilon}_{\zeta}}\bar\zeta_0^{\ell}\zeta_0^{-\ell}\vartheta(\zeta)
e^{\langle\lambda,\zeta/\zeta_0\rangle-\overline{\langle\lambda,\zeta/\zeta_0\rangle}}\\
\times\det\left[\frac{\bar\zeta}{B^*(w,\zeta)}\ \frac{\bar{w}}{B(w,\zeta)}\
Q(w,\zeta)\right]
\det\left[\frac{\d z}{B^*(\bar\zeta,\bar{z})}\ \frac{z}{B^*(\bar\zeta,\bar{z})}\ Q(\bar\zeta,\bar{z})\right]
\d\left(\frac{\zeta_j}{\zeta_0}\right)\wedge\frac{\d\bar\zeta}{P(\bar\zeta)}
\end{multline*}
in Lemmas~\ref{NTauDomain} $\div$ \ref{NwDifferenceLemma} and of the integral
\begin{multline*}
L_j(z,w,\lambda)=w_0^{\ell}\cdot\lim_{\epsilon\to 0}
\int_{\Gamma_{\zeta}^{\epsilon}}\bar\zeta_0^{\ell}\zeta_0^{-\ell}\vartheta(\zeta)
e^{\langle\lambda,\zeta/\zeta_0\rangle-\overline{\langle\lambda,\zeta/\zeta_0\rangle}}\\
\times\det\left[\frac{\d\bar{w}}{B(w,\zeta)}\ \frac{\bar{w}}{B(w,\zeta)}\ Q(w,\zeta)\right]
\wedge\det\left[\frac{\d z}{B^*(\bar\zeta,\bar{z})}\ \frac{z}{B^*(\bar\zeta,\bar{z})}\
Q(\bar\zeta,\bar{z})\right]
\d\left(\frac{\zeta_j}{\zeta_0}\right)\wedge\frac{\d\bar\zeta}{P(\bar\zeta)}
\end{multline*}
in Lemmas~\ref{LTauDomain} $\div$ \ref{LwDifferenceLemma}.

Lemma~\ref{NLzTauDomain} below covers the estimates of those integrals over the neighborhoods
of the type $U_{\zeta}(z)$ and Lemma~\ref{NLwKappaDomain} covers the neighborhoods
of the type $U_{\zeta}(w)$.

\begin{lemma}\label{NLzTauDomain}
There exist constants $C(\lambda)$, satisfying $\lim_{\lambda\to\infty}C(\lambda)=0$,
such that the following estimates hold:
\begin{equation}\label{NLzKappaEstimate}
\Bigg|\lim_{\eta\to 0}\int_{\pi^{-1}({\cal C}) \cap\left\{\eta<|B(z,\zeta|<\kappa\right\}}
\frac{\Phi^{(p)}(z,w,\lambda,\zeta)\d(\zeta_j/\zeta_0)
\wedge\left(\d{\bar P}(\zeta)\interior\d\bar\zeta\right)}{B^*({\bar\zeta},{\bar z})^2}\Bigg|
\leq C(\lambda)\cdot\sqrt{\frac{\kappa}{\gamma}},
\end{equation}
where $\kappa\leq\gamma=|B(z,w)|$,
and
\begin{equation}\label{NLzTauDomainEstimate}
\left|\int_{\pi^{-1}({\cal C}) \cap\left\{U^{\tau}_{\zeta}(z,w)\setminus U^{\gamma}_{\zeta}(z,w)\right\}}
\frac{\Phi^{(p)}(z,w,\lambda,\zeta)\d(\zeta_j/\zeta_0)
\wedge\left(\d{\bar P}(\zeta)\interior\d\bar\zeta\right)}
{B^*({\bar\zeta},{\bar z})^2}\right|\leq C(\lambda),
\end{equation}
where $\gamma=|B(z,w)|<\tau$, for
$$\Phi^{(p)}(z,w,\lambda,\zeta)=\psi(z,w,\lambda,\zeta)
\frac{B(z,w)^{(3+p)/2}(\bar\zeta-{\bar w})^{1-p}}
{B^*(w,\zeta)^{1-p}B(w,\zeta)^{1+p}}\hspace{0.1in}\text{with}\ p=0,1,$$
$$\psi(z,w,\lambda,\zeta)=\bar\zeta_0^{\ell}\zeta_0^{-\ell}\vartheta(\zeta)
e^{\langle\lambda,\zeta/\zeta_0\rangle-\overline{\langle\lambda,\zeta/\zeta_0\rangle}}
S(w,\zeta)\det\left[\d z\ z\ Q(\bar\zeta,\bar{z})\right],$$
and $S(w,\zeta)$ is a smooth function.
\end{lemma}
\begin{proof} To prove estimate \eqref{NLzKappaEstimate} we use the following representation
in the neighborhood of the point $z\in V$:
\begin{equation}\label{baromega}
\d\bar\zeta=\d{\bar P}(\zeta)\wedge\d_{\zeta}{\bar F}(z,\zeta)
\wedge\left(\sum_{j=0}^2z_j\d{\bar\zeta}_j\right),
\end{equation}
where
$$\d_{\zeta}B^*(\bar\zeta,{\bar z})=\d_{\zeta}\sum_{j=0}^2z_j(\bar\zeta_j-{\bar z}_j)
=\sum_{j=0}^2z_j\d\bar\zeta_j,$$
and $F(z,\zeta)$ is a local complex analytic coordinate along the curve ${\cal C}\subset\C\P^2$
at the point $z$.\\

\indent
Then we obtain
\begin{multline*}
\int_{\pi^{-1}({\cal C}) \cap\left\{|B(z,\zeta|<\kappa\right\}}
\frac{\Phi^{(p)}(z,w,\lambda,\zeta)\d(\zeta_j/\zeta_0)
\wedge\left(\d{\bar P}(\zeta)\interior\d\bar\zeta\right)}
{B^*({\bar\zeta},{\bar z})^2}\\
=\int_{\pi^{-1}({\cal C}) \cap\left\{|B(z,\zeta|<\kappa\right\}}
\Phi^{(p)}(z,w,\lambda,\zeta)\d(\zeta_j/\zeta_0)
\wedge\left(\left(\bar\partial_{\zeta}B^*(\bar\zeta,{\bar z})
\wedge\d{\bar P}(\zeta)\right)\interior\d\bar\zeta\right)
\wedge\d_{\zeta}\left(\frac{1}{B^*({\bar\zeta},{\bar z})}\right).
\end{multline*}
\indent
Using the Stokes' theorem in the last equality we obtain
\begin{multline}\label{NLzKappaEquality}
\int_{\pi^{-1}({\cal C}) \cap\left\{|B(z,\zeta|<\kappa\right\}}
\frac{\Phi^{(p)}(z,w,\lambda,\zeta)\d(\zeta_j/\zeta_0)
\wedge\left(\d{\bar P}(\zeta)\interior\d\bar\zeta\right)}
{B^*({\bar\zeta},{\bar z})^2}\\
=\int_{\pi^{-1}({\cal C}) \cap\left\{|B(z,\zeta|=\kappa\right\}}
\Phi^{(p)}(z,w,\lambda,\zeta)\frac{\d(\zeta_j/\zeta_0)\wedge\left(\left(\bar\partial_{\zeta}B^*(\bar\zeta,{\bar z})
\wedge\d{\bar P}(\zeta)\right)\interior\d\bar\zeta\right)}
{B^*({\bar\zeta},{\bar z})}\\
-\lim_{\eta\to 0}\int_{\pi^{-1}({\cal C}) \cap\left\{|B(z,\zeta|=\eta\right\}}
\Phi^{(p)}(z,w,\lambda,\zeta)\frac{\d(\zeta_j/\zeta_0)\wedge\left(\left(\bar\partial_{\zeta}B^*(\bar\zeta,{\bar z})
\wedge\d{\bar P}(\zeta)\right)\interior\d\bar\zeta\right)}
{B^*({\bar\zeta},{\bar z})}\\
-\lim_{\eta\to 0}\int_{\pi^{-1}({\cal C}) \cap\left\{\eta<|B(z,\zeta)|<\kappa\right\}}
\d_{\zeta}\Phi^{(p)}(z,w,\lambda,\zeta)
\wedge\frac{\left(\d(\zeta_j/\zeta_0)\wedge\left(\bar\partial_{\zeta}B^*(\bar\zeta,{\bar z})
\wedge\d{\bar P}(\zeta)\right)\interior\d\bar\zeta\right)}
{B^*({\bar\zeta},{\bar z})}.
\end{multline}

\indent
For the first integral in the right-hand side of \eqref{NLzKappaEquality} we use estimate \eqref{DeltaAreaEstimate}
\begin{equation*}
A(\kappa)=C\cdot\kappa^{3/2}
\end{equation*}
of the area of integration $\left\{\pi^{-1}({\cal C})\cap|B(z,\zeta)|=\kappa\right\}$ in this integral
and the boundedness of the functions $\Phi^{(p)}(z,w,\lambda,\zeta)$ on
$\left\{\eta<|B(z,\zeta|<\kappa\right\}$, which follows from the estimates
\eqref{GammaInequalities}. Then, we obtain the following estimate
\begin{multline}\label{NLzKappaEstimateFirst}
\Bigg|\int_{\pi^{-1}({\cal C}) \cap\left\{|B(z,\zeta|=\kappa\right\}}
\Phi^{(p)}(z,w,\lambda,\zeta)\frac{\d(\zeta_j/\zeta_0)\wedge\left(\left(\bar\partial_{\zeta}B^*(\bar\zeta,{\bar z})
\wedge\d{\bar P}(\zeta)\right)\interior\d\bar\zeta\right)}
{B^*({\bar\zeta},{\bar z})}\Bigg|\leq C(\lambda)\cdot\frac{\kappa^{3/2}}{\kappa}\\
\leq C(\lambda)\cdot\kappa^{1/2},\\
\end{multline}
where the property $\lim_{\lambda\to\infty}C(\lambda)=0$ follows from the Riemann-Lebesgue Lemma (see \cite{23}).\\
\indent
Using similar estimate for the second integral in the right-hand side of \eqref{NLzKappaEquality}
we obtain
\begin{multline}\label{NLzKappaEstimateSecond}
\Bigg|\int_{\pi^{-1}({\cal C}) \cap\left\{|B(z,\zeta|=\eta\right\}}
\Phi^{(p)}(z,w,\lambda,\zeta)\frac{\d(\zeta_j/\zeta_0)\wedge\left(\left(\bar\partial_{\zeta}B^*(\bar\zeta,{\bar z})
\wedge\d{\bar P}(\zeta)\right)\interior\d\bar\zeta\right)}
{B^*({\bar\zeta},{\bar z})}\Bigg|\to 0\\
\end{multline}
as $\eta$ goes to zero.

For the third integral in the right-hand side of \eqref{NLzKappaEquality},
according to \eqref{baromega}, we have the following equality
$$\left(\bar\partial_{\zeta}B^*(\bar\zeta,{\bar z})
\wedge\d{\bar P}(\zeta)\right)\interior\d\bar\zeta\sim \d_{\zeta}{\bar F}(z,\zeta),$$
which is the only anti-holomorphic differential available on $\pi^{-1}(\cal{C})$ lifted from $\cal{C}$.
Then, the form
$\d(\zeta_j/\zeta_0)\wedge\left(\left(\bar\partial_{\zeta}B^*(\bar\zeta,{\bar z})
\wedge\d{\bar P}(\zeta)\right)\interior\d\bar\zeta\right)$ contains all available complex
differentials lifted from $\cal{C}$, and therefore
\begin{equation}\label{ZeroDifferentials}
\d(\zeta_k/\zeta_0)\wedge\d(\zeta_j/\zeta_0)\wedge\left(\left(\bar\partial_{\zeta}B^*(\bar\zeta,{\bar z})\wedge\d{\bar P}(\zeta)\right)\interior\d\bar\zeta\right)\Bigg|_{\pi^{-1}({\cal C})}=0.
\end{equation}
Therefore, using equality \eqref{ZeroDifferentials} and equalities
\begin{equation}\label{DExponent}
\begin{aligned}
&\d e^{\langle\lambda,\zeta/\zeta_0\rangle-\overline{\langle\lambda,\zeta/\zeta_0\rangle}}
=e^{\langle\lambda,\zeta/\zeta_0\rangle-\overline{\langle\lambda,\zeta/\zeta_0\rangle}}
\left(\lambda\sum_{j=1,2}\d\frac{\zeta_j}{\zeta_0}
-\bar\lambda\sum_{j=1,2}\d\frac{\bar\zeta_j}{\bar\zeta_0}\right),\\
&\left(\sum_{j=1,2}\d\frac{\bar\zeta_j}{\bar\zeta_0}\right)
\wedge\d{\bar P}(\zeta)\wedge\d_{\zeta}{\bar F}(z,\zeta)\Bigg|_{\pi^{-1}({\cal C})}=0,\\
&\left(\sum_{j=1,2}\d\frac{\zeta_j}{\zeta_0}\right)\wedge
\d\frac{\zeta_k}{\zeta_0}
\wedge\left(\left(\bar\partial_{\zeta}B^*(\bar\zeta,{\bar z})
\wedge\d{\bar P}(\zeta)\right)\interior\d\bar\zeta\right)
\Bigg|_{\pi^{-1}({\cal C})}=0,
\end{aligned}
\end{equation}
we conclude that the differential of the function
$e^{\langle\lambda,\zeta/\zeta_0\rangle-\overline{\langle\lambda,\zeta/\zeta_0\rangle}}$
can be neglected in the estimate.\\
Then, in the case $p=0$, using the estimate
\begin{multline}\label{dPhi}
\d_{\zeta}\Phi^{(0)}(z,w,\lambda,\zeta)
=\d_{\zeta}\left(\psi(z,w,\lambda,\zeta)
\frac{B(z,w)^{3/2}(\bar\zeta-{\bar w})}{B^*(w,\zeta)B(w,\zeta)}\right)\\
\sim B(z,w)^{3/2}\Bigg(\frac{\d\bar\zeta}{B^*(w,\zeta)B(w,\zeta)}
+\frac{(\bar\zeta-{\bar w})\d_{\zeta}B^*(w,\zeta)}{B^*(w,\zeta)^2B(w,\zeta)}
+\frac{(\bar\zeta-{\bar w})\d_{\zeta}B(w,\zeta)}{B^*(w,\zeta)B(w,\zeta)^2}\Bigg),
\end{multline}
coordinates $|\zeta-z|=r$ and $B({\bar\zeta},{\bar z})=it+r^2$, and inequality
$|\zeta-w|<C\sqrt{\gamma}$ for
$\zeta\in \{|B(z,\zeta|<\gamma\}$ from \eqref{GammaInequalities}, we obtain
\begin{multline}\label{NLzKappaEstimateThirdZero}
\Bigg|\lim_{\eta\to 0}\int_{\pi^{-1}({\cal C}) \cap\left\{\eta<|B(z,\zeta|<\kappa\right\}}
\d_{\zeta}\Phi^{(0)}(z,w,\lambda,\zeta)
\wedge\frac{\left(\d(\zeta_j/\zeta_0)\wedge\left(\bar\partial_{\zeta}B^*(\bar\zeta,{\bar z})
\wedge\d{\bar P}(\zeta)\right)\interior\d\bar\zeta\right)}
{B^*({\bar\zeta},{\bar z})}\Bigg|\\
\leq C(\lambda)\cdot\gamma^{3/2}\left(\int_0^{\kappa}\d t\int_0^{\sqrt{\kappa}}
\frac{r\d r}{(\gamma+r^2)^2(t+r^2)}
+\int_0^{\kappa}\d t\int_0^{\sqrt{\kappa}}
\frac{r\gamma^{1/2}\d r}{(\gamma+r^2)^3(t+r^2)}\right)\\
\leq C(\lambda)\cdot\gamma^{3/2}\int_0^{\kappa}\d t\int_0^{\sqrt{\kappa}}\frac{\d r}
{(\gamma+r^2)^2(t+r^2)}
\leq C(\lambda)\cdot\gamma^{-1/2}\int_0^{\kappa}\frac{\d t}{\sqrt{t}}\int_0^{\sqrt{\kappa/t}}\frac{\d u}{1+u^2}
\leq C(\lambda)\cdot\sqrt{\frac{\kappa}{\gamma}},
\end{multline}
where the property $\lim_{\lambda\to\infty}C(\lambda)=0$ follows from the Riemann-Lebesgue
Lemma (see \cite{23}).

\indent
For the third integral in the right-hand side of \eqref{NLzKappaEquality} in the case $p=1$ using
equalities \eqref{DExponent}, estimate
\begin{equation*}
\d_{\zeta}\Phi^{(1)}(z,w,\lambda,\zeta)
=\d_{\zeta}\left(\psi(z,w,\lambda,\zeta)
\frac{B(z,w)^2}{B(w,\zeta)^2}\right)\\
\sim B(z,w)^2\frac{\d_{\zeta}B(w,\zeta)}{B(w,\zeta)^3},
\end{equation*}
coordinates $|\zeta-z|=r$ and $B({\bar\zeta},{\bar z})=it+r^2$, and inequality
$r\sqrt{\gamma}<\gamma+r^2$, we obtain
\begin{multline}\label{NLzKappaEstimateThirdOne}
\Bigg|\lim_{\eta\to 0}\int_{\pi^{-1}({\cal C}) \cap\left\{\eta<|B(z,\zeta|<\kappa\right\}}
\d_{\zeta}\Phi^{(1)}(z,w,\lambda,\zeta)
\wedge\frac{\left(\d(\zeta_j/\zeta_0)\wedge\left(\bar\partial_{\zeta}B^*(\bar\zeta,{\bar z})
\wedge\d{\bar P}(\zeta)\right)\interior\d\bar\zeta\right)}
{B^*({\bar\zeta},{\bar z})}\Bigg|\\
\leq C(\lambda)\cdot\gamma^2\int_0^{\kappa}\d t\int_0^{\sqrt{\kappa}}
\frac{r\d r}{(\gamma+r^2)^3(t+r^2)}
\leq C(\lambda)\cdot\gamma^{3/2}\int_0^{\kappa}\d t\int_0^{\sqrt{\kappa}}\frac{\d r}
{(\gamma+r^2)^2(t+r^2)}\\
\leq C(\lambda)\cdot\gamma^{-1/2}\int_0^{\kappa}\frac{\d t}{\sqrt{t}}\int_0^{\sqrt{\kappa/t}}\frac{\d u}{1+u^2}
\leq C(\lambda)\cdot\sqrt{\frac{\kappa}{\gamma}},
\end{multline}
where the property $\lim_{\lambda\to\infty}C(\lambda)=0$ follows from the Riemann-Lebesgue
Lemma (see \cite{23}).\\
\indent
Combining estimates \eqref{NLzKappaEstimateFirst}, \eqref{NLzKappaEstimateSecond},
\eqref{NLzKappaEstimateThirdZero}, and \eqref{NLzKappaEstimateThirdOne}
we obtain estimate \eqref{NLzKappaEstimate}.\\
\indent
To prove estimate \eqref{NLzTauDomainEstimate} we use notations
$$t=\text{Im}B(w,\zeta),\ r=|\zeta-w|\Rightarrow\ \text{Re}B(w,\zeta)\sim r^2$$
to obtain the following estimates for $p=0,1$ and the domain in \eqref{UDifference}
\begin{equation*}
U^{\tau}_{\zeta}(z,w)\setminus
\left\{U^{\gamma}_{\zeta}(z)\cup U^{\gamma}_{\zeta}(w)\right\}
=\left\{\zeta: \tau>|B(z,\zeta)|, |B(w,\zeta)|>\gamma/9\right\}
\end{equation*}
\begin{multline}\label{NLzTauDomainEstimatePreliminary}
\left|\int_{\pi^{-1}({\cal C}) \cap\left\{U^{\tau}_{\zeta}(z,w)
\setminus U^{\gamma}_{\zeta}(z,w)\right\}}
\frac{\Phi^{(0)}(z,w,\lambda,\zeta)\d(\zeta_j/\zeta_0)
\wedge\left(\d{\bar P}(\zeta)\interior\d\bar\zeta\right)}
{B^*({\bar\zeta},{\bar z})^2}\right|\\
\leq C(\lambda)\cdot\Bigg|\lim_{\epsilon\to 0}\int_{\Gamma^{\epsilon}_{\zeta}\cap
U^{\tau}_{\zeta}(z,w)\setminus U^{\gamma}_{\zeta}(z,w)}
\frac{B(z,w)^{3/2}(\bar\zeta-{\bar w})}
{B^*(w,\zeta)B(w,\zeta)B^*({\bar\zeta},{\bar z})^2}
\d\left(\frac{\zeta_j}{\zeta_0}\right)\wedge\frac{\d\bar\zeta}{P(\bar\zeta)}\Bigg|\\
\leq C(\lambda)\cdot\gamma^{3/2}\int_{\gamma}^{\tau}\d t\int_{\sqrt{\gamma}}^{\sqrt{\tau}}\frac{r^2\d r}{(\gamma+t+r^2)^2(\gamma+r^2)^2}
\leq C(\lambda)\cdot\gamma^{3/2}\int_{\sqrt{\gamma}}^{\sqrt{\tau}}
\frac{\d r}{(\gamma+r^2)^2}\leq C(\lambda),
\end{multline}
and
\begin{multline*}
\left|\int_{\pi^{-1}({\cal C}) \cap\left\{U^{\tau}_{\zeta}(z,w)
\setminus U^{\gamma}_{\zeta}(z,w)\right\}}
\frac{\Phi^{(1)}(z,w,\lambda,\zeta)\d(\zeta_j/\zeta_0)
\wedge\left(\d{\bar P}(\zeta)\interior\d\bar\zeta\right)}
{B^*({\bar\zeta},{\bar z})^2}\right|\\
\leq C(\lambda)\cdot\gamma^2\int_{\gamma}^{\tau}\d t\int_{\sqrt{\gamma}}^{\sqrt{\tau}}\frac{r\d r}{(\gamma+t+r^2)^2(\gamma+r^2)^2}
\leq C(\lambda)\cdot\gamma^2\int_{\sqrt{\gamma}}^{\sqrt{\tau}}
\frac{r\d r}{(\gamma+r^2)^3}\leq C(\lambda),
\end{multline*}
where the property $\lim_{\lambda\to\infty}C(\lambda)=0$ follows from the Riemann-Lebesgue
Lemma (see \cite{23}).
\end{proof}
\indent
Next lemma for $p=0$ will be used in Lemma~\ref{NTauSmallDomains}
and for $p=1$ in Lemma~\ref{LTauSmallDomains} for estimates of integrals over
$U^{\kappa}_{\zeta}(w)$.

\begin{lemma}\label{NLwKappaDomain}
There exist constants $C(\lambda)$, satisfying $\lim_{\lambda\to\infty}C(\lambda)=0$,
such that the following estimate holds for
$\kappa\leq\gamma=|B(z,w)|$
\begin{equation}\label{NLwKappaDomainEstimate}
\left|\int_{\pi^{-1}({\cal C}) \cap\left\{|B(w,\zeta|<\kappa\right\}}
\frac{\Phi^{(p)}(z,w,\lambda,\zeta)\d(\zeta_j/\zeta_0)
\wedge\left(\d{\bar P}(\zeta)\interior\d\bar\zeta\right)}
{B(w,\zeta)^2}\right|\leq C(\lambda)\cdot\sqrt{\frac{\kappa}{\gamma}},
\end{equation}
where
$$\Phi^{(p)}(z,w,\lambda,\zeta)=\psi(z,w,\lambda,\zeta)
\frac{B(z,w)^{(3+p)/2}(\bar\zeta-{\bar w})^{1-p}B(w,\zeta)^{1-p}}
{B^*({\bar\zeta},{\bar z})^2B^*(w,\zeta)^{1-p}}\hspace{0.1in}\text{for}\ p=0,1,$$
$$\psi(z,w,\lambda,\zeta)=\bar\zeta_0^{\ell}\zeta_0^{-\ell}\vartheta(\zeta)
e^{\langle\lambda,\zeta/\zeta_0\rangle-\overline{\langle\lambda,\zeta/\zeta_0\rangle}}
S(w,\zeta)\det\left[\d z\ z\ Q(\bar\zeta,\bar{z})\right],$$
and $S(w,\zeta)$ is a smooth function.
\end{lemma}
\begin{proof}
\indent
We prove the cases $p=0$ and $p=1$ separately.
For the case $p=0$ we have the following estimate
\begin{multline*}
\Bigg|\int_{\pi^{-1}({\cal C}) \cap\left\{|B(w,\zeta|<\kappa\right\}}
\frac{B(z,w)^{3/2}(\bar\zeta-{\bar w})\d(\zeta_j/\zeta_0)
\wedge\left(\d{\bar P}(\zeta)\interior\d\bar\zeta\right)}
{B^*({\bar\zeta},{\bar z})^2B(w,\zeta)B^*(w,\zeta)}\Bigg|\\
\leq C(\lambda)\cdot\gamma^{3/2}\int_0^{\kappa}\d t \int_0^{\sqrt{\kappa}}
\frac{r^2\d r}{(\gamma+r^2)^2(t+r^2)^2}
\leq C(\lambda)\cdot\gamma^{3/2}\int_0^{\sqrt{\kappa}}\frac{\d r}{(\gamma+r^2)^2}
\leq C(\lambda)\cdot\gamma^{1/2}\int_0^{\sqrt{\kappa}}\frac{\d r}{(\gamma+r^2)}\\
\leq C(\lambda)\cdot\int_0^{\sqrt{\kappa/\gamma}}\frac{\d u}{(1+u^2)}
\leq C(\lambda)\cdot\sqrt{\frac{\kappa}{\gamma}},
\end{multline*}
which is equivalent to estimate \eqref{NLwKappaDomainEstimate}.

\indent
For the case $p=1$ we use equality
\begin{multline*}
\d_{\zeta}B(w,\zeta)=\d_{\zeta}\left(1-\sum_{j=0}^2\bar{w}_j\zeta_j\right)
=\sum_{j=0}^2(\bar\zeta_j-\bar{w}_j)\d\zeta_j-\sum_{j=0}^2\bar\zeta_j\d\zeta_j\\
=\sum_{j=0}^2(\bar\zeta_j-\bar{w}_j)\d\zeta_j-\d|\zeta|^2
+\sum_{j=0}^2\zeta_j\d\bar\zeta_j,
\end{multline*}
to obtain the following representation on $\*S^5(1)$
\begin{multline}\label{pOneDecomposition}
\d\bar\zeta=\d{\bar P}(\zeta)\wedge\d_{\zeta}{\bar F}(w,\zeta)
\wedge\left(\sum_{j=0}^2\zeta_j\d{\bar\zeta}_j\right)\\
=\d{\bar P}(\zeta)\wedge\d_{\zeta}{\bar F}(w,\zeta)\wedge\d_{\zeta}B(w,\zeta)
-\d{\bar P}(\zeta)\wedge\d_{\zeta}{\bar F}(w,\zeta)
\wedge\left(\sum_{j=0}^2(\bar\zeta_j-\bar{w}_j)\d\zeta_j\right),
\end{multline}
where $F(w,\zeta)$ is a local holomorphic coordinate along ${\cal C}_{\zeta}$ in a neighborhood
of $w$.
For the second term of the right-hand side of \eqref{pOneDecomposition} we have
the following estimate
\begin{multline}\label{NLwKappaDomainEstimatepOne}
\Bigg|\int_{\pi^{-1}({\cal C}) \cap\left\{|B(w,\zeta|<\kappa\right\}}
\frac{B(z,w)^2\left(\sum_{j=0}^2(\bar\zeta_j-\bar{w}_j)\d\zeta_j\right)\d(\zeta_j/\zeta_0)
\wedge\left(\d{\bar P}(\zeta)\interior\d\bar\zeta\right)}
{B^*({\bar\zeta},{\bar z})^2B(w,\zeta)^2}\Bigg|\\
\leq C(\lambda)\cdot\gamma^2\int_0^{\kappa}\d t \int_0^{\sqrt{\kappa}}
\frac{r^2\d r}{(\gamma+r^2)^2(t+r^2)^2}
\leq C(\lambda)\cdot\gamma^2\int_0^{\sqrt{\kappa}}\frac{\d r}{(\gamma+r^2)^2}
\leq C(\lambda)\cdot\sqrt{\kappa},
\end{multline}
which implies estimate \eqref{NLwKappaDomainEstimate} for this term.

For the first term of the right-hand side of \eqref{pOneDecomposition} we
use the Stokes' theorem to obtain the following equality
\begin{multline}\label{NLwKappaDomainEquality}
\int_{\pi^{-1}({\cal C}) \cap\left\{|B(w,\zeta|<\kappa\right\}}
\frac{\Phi^{(1)}(z,w,\lambda,\zeta)\d(\zeta_j/\zeta_0)\wedge\left(\d{\bar P}(\zeta)\interior\d\bar\zeta\right)}{B(w,\zeta)^2}\\
=\int_{\pi^{-1}({\cal C}) \cap\left\{|B(w,\zeta|<\kappa\right\}}
\Phi^{(1)}(z,w,\zeta,\lambda) \Big(\d_{\zeta}B(w,\zeta)\wedge\d{\bar P}(\zeta)\Big)
\interior\Big(\d\left(\zeta_j/\zeta_0\right)\wedge\d\bar\zeta\Big)
\wedge\d_{\zeta}\left(\frac{1}{B(w,\zeta)}\right)\\
=\int_{\pi^{-1}({\cal C}) \cap\left\{|B(w,\zeta|=\kappa\right\}}
\Phi^{(1)}(z,w,\zeta,\lambda)\frac{\Big(\d_{\zeta}B(w,\zeta)\wedge\d{\bar P}(\zeta)\Big)
\interior\Big(\d\left(\zeta_j/\zeta_0\right)\wedge\d\bar\zeta\Big)}{B(w,\zeta)}\\
-\lim_{\eta\to 0}\int_{\pi^{-1}({\cal C}) \cap\left\{|B(w,\zeta|=\eta\right\}}
\Phi^{(1)}(z,w,\zeta,\lambda)\frac{\Big(\d_{\zeta}B(w,\zeta)\wedge\d{\bar P}(\zeta)\Big)
\interior\Big(\d\left(\zeta_j/\zeta_0\right)\wedge\d\bar\zeta\Big)}{B(w,\zeta)}\\
-\int_{\pi^{-1}({\cal C}) \cap\left\{|B(w,\zeta|<\kappa\right\}}
\d_{\zeta}\Phi^{(1)}(z,w,\zeta,\lambda)\wedge
\frac{\Big(\d_{\zeta}B(w,\zeta)\wedge\d{\bar P}(\zeta)\Big)
\interior\Big(\d\left(\zeta_j/\zeta_0\right)\wedge\d\bar\zeta\Big)}{B(w,\zeta)}.
\end{multline}

\indent
For the first integral in the right-hand side of \eqref{NLwKappaDomainEquality} we use estimate \eqref{DeltaAreaEstimate}
\begin{equation*}
A(\kappa)=C\cdot\kappa^{3/2}
\end{equation*}
of the area of integration $\left\{\pi^{-1}({\cal C})\cap|B(w,\zeta)|=\kappa\right\}$ in this integral
and the boundedness of the function $\Phi^{(1)}(z,w,\lambda,\zeta)$ on
$\left\{\eta<|B(w,\zeta|<\kappa\right\}$, which follows from the estimates
\eqref{GammaInequalities}. Then, we obtain the following estimate
\begin{equation}\label{NLwDomainEstimateFirstOne}
\Bigg|\int_{\pi^{-1}({\cal C}) \cap\left\{|B(w,\zeta|=\kappa\right\}}
\Phi^{(1)}(z,w,\zeta,\lambda)\frac{\Big(\d_{\zeta}B(w,\zeta)\wedge\d{\bar P}(\zeta)\Big)
\interior\Big(\d\left(\zeta_j/\zeta_0\right)\wedge\d\bar\zeta\Big)}{B(w,\zeta)}\Bigg|
\leq C(\lambda))\cdot\sqrt{\kappa}.
\end{equation}
\indent
For the second integral in the right-hand side of \eqref{NLwKappaDomainEquality} using estimates
\eqref{DeltaAreaEstimate} and \eqref{NLwDomainEstimateFirstOne} we obtain the estimate
\begin{equation}\label{NLwDomainEstimateSecondOne}
\Bigg|\int_{\pi^{-1}({\cal C}) \cap\left\{|B(w,\zeta|=\eta\right\}}
\Phi^{(1)}(z,w,\zeta,\lambda)\frac{\Big(\d_{\zeta}B(w,\zeta)\wedge\d{\bar P}(\zeta)\Big)
\interior\Big(\d\left(\zeta_j/\zeta_0\right)\wedge\d\bar\zeta\Big)}{B(w,\zeta)}\Bigg|
\leq C(\lambda)\cdot\sqrt{\eta}\to 0,
\end{equation}
as $\eta\to 0$.\\
\indent
For the third integral in the right-hand side of \eqref{NLwKappaDomainEquality}
using equalities \eqref{DExponent} we obtain the estimate
\begin{equation}\label{PhiOneDecomposition}
\d_{\zeta}\left(\Phi^{(1)}(z,w,\lambda,\zeta)=\psi(z,w,\lambda,\zeta)
\frac{B(z,w)^2}{B^*({\bar\zeta},{\bar z})^2}\right)\frac{1}{B(w,\zeta)}
\sim\frac{\gamma^2\d_{\zeta}B^*({\bar\zeta},{\bar z})}
{B^*({\bar\zeta},{\bar z})^3B(w,\zeta)},
\end{equation}
and then, using coordinates $t=Im B(w,\zeta),\ r=|w-\zeta|$, and inequalities
$$|B(z,\zeta)|>\gamma+|\zeta-w|^2,\ \gamma<|B(z,\zeta)|,\ 
\gamma^{1/2}r\leq \gamma+r^2$$
we obtain
\begin{multline}\label{NLwDomainEstimateThirdOne}
\Bigg|\lim_{\eta\to 0}\int_{\pi^{-1}({\cal C}) \cap\left\{|B(w,\zeta|<\kappa\right\}}
\d_{\zeta}\Phi^{(1)}(z,w,\zeta,\lambda)\wedge
\frac{\Big(\d_{\zeta}B(w,\zeta)\wedge\d{\bar P}(\zeta)\Big)
\interior\Big(\d\left(\zeta_j/\zeta_0\right)\wedge\d\bar\zeta\Big)}{B(w,\zeta)}\Bigg|\\
\leq C(\lambda)\cdot\gamma^2\int_0^{\kappa}\d t \int_0^{\sqrt{\kappa}}
\frac{r\d r}{(\gamma+r^2)^3(t+r^2)}
\leq C(\lambda)\cdot\gamma^{-1/2}\int_0^{\kappa}\d t \int_0^{\sqrt{\kappa}}
\frac{\d r}{t+r^2}\\
\leq C(\lambda)\cdot\gamma^{-1/2}\int_0^{\kappa}\frac{\d t}{\sqrt{t}}
\int_0^{\sqrt{\frac{\kappa}{t}}}\frac{\d u}{1+u^2}
\leq C(\lambda))\cdot\sqrt{\frac{\kappa}{\gamma}}.
\end{multline}
\indent
Combining estimates \eqref{NLwKappaDomainEstimatepOne},
\eqref{NLwDomainEstimateFirstOne}, \eqref{NLwDomainEstimateSecondOne}, and \eqref{NLwDomainEstimateThirdOne}
we obtain the estimate \eqref{NLwKappaDomainEstimate}. The property of $C(\lambda)$ that
$\lim_{\lambda\to\infty}C(\lambda)=0$ follows from the Riemann-Lebesgue
Lemma (see \cite{23}).
\end{proof}

\indent
In the next five lemmas: \ref{NTauDomain}$\div$\ref{NwDifferenceLemma}, using estimates
from Lemmas~\ref{NLzTauDomain} and \ref{NLwKappaDomain},
we establish necessary estimates for the kernels $N_j(z,w,\lambda)$.

\begin{lemma}\label{NTauDomain}
There exist constants $C(\lambda)$, satisfying $\lim_{\lambda\to\infty}C(\lambda)=0$, such that the following estimate
\begin{equation}\label{NzTauDomainEstimate}
\left|B(z,w)^{3/2}\cdot N^{\tau,\gamma}_j(z,w,\lambda)\right|\leq C(\lambda)
\end{equation}
holds for the integral
\begin{multline}\label{NTauIntegral}
N^{\tau,\gamma}_j(z,w,\lambda)=w_0^{\ell}\cdot\lim_{\epsilon\to 0}
\int_{\Gamma^{\epsilon}_{\zeta}\cap\left(U^{\tau}_{\zeta}(z)\setminus U^{\gamma}_{\zeta}(z,w)\right)}
\bar\zeta_0^{\ell}\zeta_0^{-\ell}\vartheta(\zeta)
e^{\langle\lambda,\zeta/\zeta_0\rangle-\overline{\langle\lambda,\zeta/\zeta_0\rangle}}\\
\times\det\left[\frac{\bar\zeta}{B^*(w,\zeta)}\ \frac{\bar{w}}{B(w,\zeta)}\
Q(w,\zeta)\right]
\det\left[\frac{\d z}{B^*(\bar\zeta,\bar{z})}\ \frac{z}{B^*(\bar\zeta,\bar{z})}\ Q(\bar\zeta,\bar{z})\right]
\d\left(\frac{\zeta_j}{\zeta_0}\right)\wedge\frac{\d\bar\zeta}{P(\bar\zeta)}
\end{multline}
over the domain $U^{\tau}_{\zeta}(z)\setminus U^{\gamma}_{\zeta}(z,w)$
for points $z,w$ such that $\gamma=|B(z,w)|<\tau$.
\end{lemma}
\begin{proof}
For the integral in \eqref{NTauIntegral}, denoting
$$\psi=(z,w,\lambda,\zeta)=\bar\zeta_0^{\ell}\zeta_0^{-\ell}\vartheta(\zeta)
e^{\langle\lambda,\zeta/\zeta_0\rangle-\overline{\langle\lambda,\zeta/\zeta_0\rangle}}
S(w,\zeta)\det\left[\d z\ z\ Q(\bar\zeta,\bar{z})\right],$$
defining the bounded form on $U^{\tau}_{\zeta}(z)$
$$\Psi(z,w,\lambda,\zeta)=\psi(z,w,\lambda,\zeta)
\frac{B(z,w)^{3/2}(\bar\zeta-{\bar w})}
{B^*(w,\zeta)B(w,\zeta)},$$
and using estimates \eqref{GammaInequalities}, we obtain the following inequality
\begin{multline}\label{NzTauDomainEquality}
\Bigg|\lim_{\epsilon\to 0}
\int_{\Gamma^{\epsilon}_{\zeta}\cap
U^{	\tau}_{\zeta}(z)\setminus U^{\gamma}_{\zeta}(z,w)}
B(z,w)^{3/2}\det\left[\frac{\bar\zeta}{B^*(w,\zeta)}\ \frac{\bar{w}}{B(w,\zeta)}\ Q(w,\zeta)\right]\\
\wedge\det\left[\frac{\d z}{B^*(\bar\zeta,\bar{z})}\
\frac{z}{B^*(\bar\zeta,\bar{z})}\ Q(\bar\zeta,\bar{z})\right]
\wedge\d\left(\frac{\zeta_j}{\zeta_0}\right)
\wedge\frac{\d\bar\zeta}{P(\bar\zeta)}\Bigg|\\
=\Bigg|\lim_{\epsilon\to 0}\int_{\Gamma^{\epsilon}_{\zeta}\cap
U^{\tau}_{\zeta}(z)\setminus U^{\gamma}_{\zeta}(z,w)}
B(z,w)^{3/2}\frac{\psi(z,w,\lambda,\zeta)(\bar\zeta-{\bar w})}
{B^*(w,\zeta)B(w,\zeta)B^*(\bar\zeta,\bar{z})^2}
\d\left(\frac{\zeta_j}{\zeta_0}\right)\wedge\frac{\d\bar\zeta}{P(\bar\zeta)}\Bigg|\\
=\Bigg|\int_{\pi^{-1}({\cal C})
\cap U^{\tau}_{\zeta}(z)\setminus U^{\gamma}_{\zeta}(z,w)}
B(z,w)^{3/2}\psi(z,w,\lambda,\zeta)\frac{({\bar\zeta}-{\bar w})
\d\left(\zeta_j/\zeta_0\right)\wedge\left(\d{\bar P}(\zeta)\interior\d\bar\zeta\right)}
{B^*(w,\zeta)B(w,\zeta)B^*({\bar\zeta},{\bar z})^2}\Bigg|\\
\leq C\cdot\Bigg|\int_{\pi^{-1}({\cal C}) \cap\left\{\gamma<|B(z,\zeta|<\tau\right\}}
\frac{\Psi(z,w,\lambda,\zeta)\d(\zeta_j/\zeta_0)\wedge\left(\d{\bar P}(\zeta)\interior\d\bar\zeta\right)}
{B^*({\bar\zeta},{\bar z})^2}\Bigg|.
\end{multline}
\indent
Then, using estimate \eqref{NLzTauDomainEstimate} from Lemma~\ref{NLzTauDomain} for $p=0$ we obtain estimate \eqref{NzTauDomainEstimate}.
\end{proof}

\indent
In the lemma below we prove the complementary estimate to \eqref{NzTauDomainEstimate} in Lemma~\ref{NTauDomain} for kernels $N_j(z,w,\lambda)$ considering
the domain of integration $\Gamma^{\epsilon}_{\zeta}
\cap\{U^{\gamma}_{\zeta}(w)\cup U^{\gamma}_{\zeta}(z)\}$.

\begin{lemma}\label{NTauSmallDomains}
There exist constants $C(\lambda)$, satisfying
$\lim_{\lambda\to\infty}C(\lambda)=0$, such that  the following estimates hold:
\begin{equation}\label{NGammaSmallDomainsEstimate}
\left|B(z,w)^{3/2}\cdot N^{\gamma}_j(z,w,\lambda)\right|\leq C(\lambda)
\end{equation}
for
\begin{multline}\label{NGamma}
N^{\gamma}_j(z,w,\lambda)=w_0^{\ell}\cdot\lim_{\epsilon\to 0}
\int_{\Gamma^{\epsilon}_{\zeta}\cap\{U^{\gamma}_{\zeta}(w)\cup U^{\gamma}_{\zeta}(z)\}}
\bar\zeta_0^{\ell}\zeta_0^{-\ell}\vartheta(\zeta)
e^{\langle\lambda,\zeta/\zeta_0\rangle-\overline{\langle\lambda,\zeta/\zeta_0\rangle}}\\
\times\det\left[\frac{\bar\zeta}{B^*(w,\zeta)}\ \frac{\bar{w}}{B(w,\zeta)}\
Q(w,\zeta)\right]
\det\left[\frac{\d z}{B^*(\bar\zeta,\bar{z})}\ \frac{z}{B^*(\bar\zeta,\bar{z})}\
\frac{Q(\bar\zeta,\bar{z})}{P(\bar\zeta)}\right]
\d\left(\frac{\zeta_j}{\zeta_0}\right)\wedge\d\bar\zeta
\end{multline}
over the domains $U^{\gamma}_{\zeta}(z)$ and $U^{\gamma}_{\zeta}(w)$ for $z,w$ such that
$\gamma=|B(z,w)|$.
\end{lemma}
\begin{proof}
In the proof below we will be proving a more general estimate
\begin{equation}\label{NKappaEstimate}
\left|B(z,w)^{3/2}\cdot N^{\kappa}_j(z,w,\lambda)\right|
\leq C(\lambda)\cdot\sqrt{\frac{\kappa}{\gamma}}
\end{equation}
for
\begin{multline}\label{NKappa}
N^{\kappa}_j(z,w,\lambda)=w_0^{\ell}\cdot\lim_{\epsilon\to 0}
\int_{\Gamma^{\epsilon}_{\zeta}\cap\{U^{\kappa}_{\zeta}(w)\cup U^{\kappa}_{\zeta}(z)\}}
\bar\zeta_0^{\ell}\zeta_0^{-\ell}\vartheta(\zeta)
e^{\langle\lambda,\zeta/\zeta_0\rangle-\overline{\langle\lambda,\zeta/\zeta_0\rangle}}\\
\times\det\left[\frac{\bar\zeta}{B^*(w,\zeta)}\ \frac{\bar{w}}{B(w,\zeta)}\
Q(w,\zeta)\right]
\det\left[\frac{\d z}{B^*(\bar\zeta,\bar{z})}\ \frac{z}{B^*(\bar\zeta,\bar{z})}\
\frac{Q(\bar\zeta,\bar{z})}{P(\bar\zeta)}\right]
\d\left(\frac{\zeta_j}{\zeta_0}\right)\wedge\d\bar\zeta
\end{multline}
with $\kappa\leq\gamma$.
For the integral in \eqref{NKappa} over $U^{\kappa}_{\zeta}(z)$, denoting
$$\psi(z,w,\lambda,\zeta)=\bar\zeta_0^{\ell}\zeta_0^{-\ell}\vartheta(\zeta)
e^{\langle\lambda,\zeta/\zeta_0\rangle-\overline{\langle\lambda,\zeta/\zeta_0\rangle}}
S(w,\zeta)\det\left[\d z\ z\ Q(\bar\zeta,\bar{z})\right]$$
and defining the bounded form on $U^{\kappa}_{\zeta}(z)$\\
$$\Psi(z,w,\lambda,\zeta)=\psi(z,w,\lambda,\zeta)
\frac{B(z,w)^{3/2}(\bar\zeta-{\bar w})}{B^*(w,\zeta)B(w,\zeta)}$$
we have the following equality
\begin{multline}\label{NzDomainEquality}
\Bigg|\lim_{\epsilon\to 0}
\int_{\Gamma_{\zeta}^{\epsilon}\cap U^{\kappa}_{\zeta}(z)}B(z,w)^{3/2}
\bar\zeta_0^{\ell}\zeta_0^{-\ell}\vartheta(\zeta)
e^{\langle\lambda,\zeta/\zeta_0\rangle-\overline{\langle\lambda,\zeta/\zeta_0\rangle}}
\det\left[\frac{\bar\zeta}{B^*(w,\zeta)}\ \frac{\bar{w}}{B(w,\zeta)}\ Q(w,\zeta)\right]\\
\wedge\det\left[\frac{\d z}{B^*(\bar\zeta,\bar{z})}\
\frac{z}{B^*(\bar\zeta,\bar{z})}\ Q(\bar\zeta,\bar{z})\right]
\wedge\d\left(\frac{\zeta_j}{\zeta_0}\right)
\wedge\frac{\d\bar\zeta}{P(\bar\zeta)}\Bigg|\\
=\Bigg|\lim_{\epsilon\to 0}\lim_{\eta\to 0}
\int_{\Gamma_{\zeta}^{\epsilon} \cap\left\{U^{\kappa}_{\zeta}(z),|B(z,\zeta|>\eta\right\}}
B(z,w)^{3/2}\psi(z,w,\lambda,\zeta)\frac{(\bar\zeta-{\bar w})}
{B^*(w,\zeta)B(w,\zeta)B^*(\bar\zeta,\bar{z})^2}
\d\left(\frac{\zeta_j}{\zeta_0}\right)\wedge\frac{\d\bar\zeta}{P(\bar\zeta)}\Bigg|\\
\leq C\cdot\Bigg|\lim_{\eta\to 0}\int_{\pi^{-1}({\cal C}) \cap\left\{U^{\kappa}_{\zeta}(z),|B(z,\zeta|>\eta\right\}}
B(z,w)^{3/2}\psi(z,w,\lambda,\zeta)\frac{({\bar\zeta}-{\bar w})
\d\left(\zeta_j/\zeta_0\right)\wedge\left(\d{\bar P}(\zeta)\interior\d\bar\zeta\right)}
{B^*(w,\zeta)B(w,\zeta)B^*({\bar\zeta},{\bar z})^2}\Bigg|\\
\leq C\cdot\Bigg|\lim_{\eta\to 0}\int_{\pi^{-1}({\cal C}) \cap\left\{\eta<|B(z,\zeta|<\kappa\right\}}
\frac{\Psi(z,w,\lambda,\zeta)\d(\zeta_j/\zeta_0)\wedge\left(\d{\bar P}(\zeta)\interior\d\bar\zeta\right)}
{B^*({\bar\zeta},{\bar z})^2}\Bigg|.\\
\end{multline}
\indent
Then, using estimate \eqref{NLzKappaEstimate} from Lemma~\ref{NLzTauDomain} for $p=0$ we
obtain the estimate
\begin{equation}\label{NzKappaEstimate}
\Bigg|\lim_{\eta\to 0}\int_{\pi^{-1}({\cal C}) \cap\left\{\eta<|B(z,\zeta|<\kappa\right\}}
\frac{\Psi(z,w,\lambda,\zeta)\d(\zeta_j/\zeta_0)\wedge\left(\d{\bar P}(\zeta)\interior\d\bar\zeta\right)}
{B^*({\bar\zeta},{\bar z})^2}\Bigg|
\leq C(\lambda)\cdot\sqrt{\frac{\kappa}{\gamma}}.
\end{equation}

\indent
Estimate \eqref{NGammaSmallDomainsEstimate} for the integral over $U^{\gamma}_{\zeta}(z)$ follows from the application of estimate \eqref{NzKappaEstimate} to the case $\kappa=\gamma$.

\indent
For the integral in \eqref{NKappa} over $U^{\kappa}_{\zeta}(w)$, denoting
$$\psi(z,w,\lambda,\zeta)=\bar\zeta_0^{\ell}\zeta_0^{-\ell}\vartheta(\zeta)
e^{\langle\lambda,\zeta/\zeta_0\rangle-\overline{\langle\lambda,\zeta/\zeta_0\rangle}}
S(w,\zeta)\det\left[\d z\ z\ Q(\bar\zeta,\bar{z})\right]$$
and defining the bounded form on $U^{\kappa}_{\zeta}(w)$
$$\Psi(z,w,\lambda,\zeta)=\psi(z,w,\lambda,\zeta)
\frac{B(z,w)^{3/2}(\bar\zeta-{\bar w})B(w,\zeta)}
{B^*({\bar\zeta},{\bar z})^2B^*(w,\zeta)}$$
we obtain the following estimate
\begin{multline}\label{NwDomainEquality}
\Bigg|\lim_{\epsilon\to 0}
\int_{\Gamma_{\zeta}^{\epsilon}\cap U^{\kappa}_{\zeta}(w)}B(z,w)^{3/2}
\det\left[\frac{\bar\zeta}{B^*(w,\zeta)}\ \frac{\bar{w}}{B(w,\zeta)}\ Q(w,\zeta)\right]
\wedge\det\left[\frac{\d z}{B^*(\bar\zeta,\bar{z})}\
\frac{z}{B^*(\bar\zeta,\bar{z})}\ Q(\bar\zeta,\bar{z})\right]\\
\wedge\d\left(\frac{\zeta_j}{\zeta_0}\right)
\wedge\frac{\d\bar\zeta}{P(\bar\zeta)}\Bigg|\\
\leq C\cdot\Bigg|\lim_{\epsilon\to 0}
\int_{\Gamma_{\zeta}^{\epsilon}\cap U^{\kappa}_{\zeta}(w)}\psi(z,w,\lambda,\zeta)
\frac{B(z,w)^{3/2}(\bar\zeta-{\bar w})}
{B^*(w,\zeta)B(w,\zeta)B^*(\bar\zeta,\bar{z})^2}
\wedge\d\left(\frac{\zeta_j}{\zeta_0}\right)
\wedge\frac{\d\bar\zeta}{P(\bar\zeta)}\Bigg|\\
\leq C\cdot\Bigg|\lim_{\epsilon\to 0}
\int_{\Gamma_{\zeta}^{\epsilon}\cap U^{\kappa}_{\zeta}(w)}\psi(z,w,\lambda,\zeta)
\frac{B(z,w)^{3/2}(\bar\zeta-{\bar w})B(w,\zeta)}
{B^*(\bar\zeta,\bar{z})^2B^*(w,\zeta)}\cdot\frac{1}{B(w,\zeta)^2}
\wedge\d\left(\frac{\zeta_j}{\zeta_0}\right)
\wedge\frac{\d\bar\zeta}{P(\bar\zeta)}\Bigg|\\
\leq C\cdot\Bigg|\lim_{\eta\to 0}\int_{\pi^{-1}({\cal C}) \cap\left\{\eta<|B(w,\zeta|<\kappa\right\}}
\frac{\Psi(z,w,\lambda,\zeta)\d(\zeta_j/\zeta_0)\wedge\left(\d{\bar P}(\zeta)\interior\d\bar\zeta\right)}
{B(w,\zeta)^2}\Bigg|.
\end{multline}
\indent
Then, using estimate \eqref{NLwKappaDomainEstimate} from Lemma~\ref{NLwKappaDomain} for $p=0$ we
obtain the estimate
\begin{equation}\label{NwKappaEstimate}
\Bigg|\lim_{\eta\to 0}\int_{\pi^{-1}({\cal C}) \cap\left\{\eta<|B(w,\zeta|<\kappa\right\}}
\frac{\Psi(z,w,\lambda,\zeta)\d(\zeta_j/\zeta_0)\wedge\left(\d{\bar P}(\zeta)\interior\d\bar\zeta\right)}
{B(w,\zeta)^2}\Bigg|
\leq C(\lambda)\cdot\sqrt{\frac{\kappa}{\gamma}},
\end{equation}
which implies estimate \eqref{NKappaEstimate} for the integral over $U^{\kappa}_{\zeta}(w)$.\\
\indent
Estimate \eqref{NGammaSmallDomainsEstimate} for the integral over $U^{\gamma}_{\zeta}(w)$ follows from the application of estimate \eqref{NwKappaEstimate} to the case
$\kappa=\gamma$.
\end{proof}

\indent
In the Lemmas~\ref{NDeltaDomain}$\div$\ref{NwDifferenceLemma} below we prove final
estimates for the integrals $N_j(z,w,\lambda)$ to be used in the estimates of the operator in \eqref{POperator}. The following equality will be used in these lemmas
\begin{multline}\label{B-Difference}
B(z^{(2)},w)-B(z^{(1)},w)=\sum_{j=0}^2{\bar z}^{(2)}_j\left(z^{(2)}_j-w_j\right)
-\sum_{j=0}^2{\bar z}^{(1)}_j\left(z^{(1)}_j-w_j\right)\\
=\sum_{j=0}^2\left({\bar z}^{(2)}_j-{\bar z}^{(1)}_j\right)\left(z^{(2)}_j-w_j\right)
+\sum_{j=0}^2{\bar z}^{(1)}_j\left(z^{(2)}_j-w_j\right)
-\sum_{j=0}^2{\bar z}^{(1)}_j\left(z^{(1)}_j-w_j\right)\\
=\sum_{j=0}^2\left({\bar z}^{(2)}_j-{\bar z}^{(1)}_j\right)\left(z^{(2)}_j-w_j\right)
+\sum_{j=0}^2{\bar z}^{(1)}_j\left(z^{(2)}_j-z^{(1)}_j\right)\\
=\sum_{j=0}^2\left({\bar z}^{(2)}_j-{\bar z}^{(1)}_j\right)\left(z^{(2)}_j-w_j\right)
-B(z^{(1)},z^{(2)}).\\
\end{multline}

\begin{lemma}\label{NDeltaDomain}
There exist constants $C(\lambda)$, satisfying $\lim_{\lambda\to\infty}C(\lambda)=0$, and
such that the following estimate holds for arbitrary $z,w \in V$
\begin{equation}\label{NDeltaEstimate}
\left|N_j(z,w,\lambda)\right|\leq \frac{C(\lambda)}{|B(z,w)|^{3/2}},
\end{equation}
where
\begin{multline}\label{NIntegral}
N_j(z,w,\lambda)=w_0^{\ell}\cdot\lim_{\epsilon\to 0}
\int_{\Gamma^{\epsilon}_{\zeta}}\bar\zeta_0^{\ell}\zeta_0^{-\ell}\vartheta(\zeta)
e^{\langle\lambda,\zeta/\zeta_0\rangle-\overline{\langle\lambda,\zeta/\zeta_0\rangle}}\\
\times\det\left[\frac{\bar\zeta}{B^*(w,\zeta)}\ \frac{\bar{w}}{B(w,\zeta)}\
Q(w,\zeta)\right]
\det\left[\frac{\d z}{B^*(\bar\zeta,\bar{z})}\ \frac{z}{B^*(\bar\zeta,\bar{z})}\
Q(\bar\zeta,\bar{z})\right]
\d\left(\frac{\zeta_j}{\zeta_0}\right)\wedge\frac{\d\bar\zeta}{P(\bar\zeta)}
\end{multline}
is the form from \eqref{NForm}.
\end{lemma}
\begin{proof}
\indent
From the estimate \eqref{NzTauDomainEstimate} in Lemma~\ref{NTauDomain} we obtain that there exist constants $C(\lambda)$ such that the following estimate holds
\begin{equation}\label{NDeltaEstimateFirst}
\left|B(z,w)^{3/2}\cdot N^{\tau,\gamma}_j(z,w,\lambda)\right|\leq C(\lambda)
\end{equation}
for the integral in \eqref{NIntegral} over
$U^{\tau}_{\zeta}(z)\setminus U^{\gamma}_{\zeta}(z,w)$
for $z,w$ such that $\gamma=|B(z,w)|<\tau$.\\
\indent
From the estimate \eqref{NGammaSmallDomainsEstimate} in Lemma~\ref{NTauSmallDomains}
we obtain the existence of a constant $C(\lambda)$ such that for $\tau>0$
the following estimate holds
\begin{equation}\label{NDeltaEstimateSecond}
\left|B(z,w)^{3/2}\cdot N^{\gamma}_j(z,w,\lambda)\right|\leq C(\lambda)
\end{equation}
for the integral in \eqref{NIntegral} over the domains
$U^{\gamma}_{\zeta}(z)$ and $U^{\gamma}_{\zeta}(w)$
for $z,w$ such that $\gamma=|B(z,w)|<\tau$.\\
\indent
Then, combining estimates \eqref{NDeltaEstimateFirst} and \eqref{NDeltaEstimateSecond} we obtain estimate \eqref{NDeltaEstimate}.
\end{proof}

\indent
In the lemmas below we will be using two inequalities for points $z^{(1)}, z^{(2)}, \zeta$
such that $|B(z^{(1)},z^{(2)})|=\delta^2$ and $|B(z^{(i)},\zeta)|>9\delta^2$.
The first one is the inequality
\begin{equation}\label{BDifferenceInequality}
\left|B(z^{(1)},\zeta)-B(z^{(2)},\zeta)\right|\leq \sqrt{2}\delta\cdot|z-\zeta|+\delta^2,
\end{equation}
which is the corollary of inequality \eqref{B-Difference} and inequality
$$\text{Re}B(z^{(1)},z^{(2)})=\frac{1}{2}|z^{(1)}-z^{(2)}|^2\leq \delta^2.$$
The second one for $\zeta, z^{(1)}, z^{(2)}$ such that $|B(z^{(i)},\zeta)|>9\delta^2$
\begin{multline}\label{B0Inequality}
|B(z^{(2)},\zeta)|\geq |B(z^{(1)},\zeta)|-\left|B(z^{(1)},\zeta)-B(z^{(2)},\zeta)\right|
\geq |B(z^{(1)},\zeta)|-\left(2\delta\sqrt{|B(z^{(1)},\zeta)|}+\delta^2\right)\\
\geq |B(z^{(1)},\zeta)|\cdot\left(1-\frac{2\delta}{\sqrt{|B(z^{(1)},\zeta)|}}
-\frac{\delta^2}{|B(z^{(1)},\zeta)|}\right)
\geq \frac{2}{9}|B(z^{(1)},\zeta)|
\end{multline}
is the corollary of \eqref{BDifferenceInequality} and of the inequality
\begin{equation}\label{ModulInequality}
|z^{(1)}-\zeta|\leq\sqrt{2}\cdot\sqrt{|B(z^{(1)},\zeta)|}.
\end{equation}

\begin{lemma}\label{NDifferenceLemma}
There exist constants $C(\lambda)$, satisfying $\lim_{\lambda\to\infty}C(\lambda)=0$,
such that the following estimate
\begin{equation}\label{NDifferenceEstimate}
\left|N_j(z^{(1)},w,\lambda)-N_j(z^{(2)},w,\lambda)\right|
\leq C(\lambda)\cdot\frac{\delta}{|B(z^{(1)},w)|^2}
\end{equation}
holds for an arbitrary fixed $\delta$,
any two points $z^{(1)},z^{(2)}\in V$, such that $|B(z^{(1)},z^{(2)})|=\delta^2$,
$w$ satisfying $\gamma=|B(z^{(1)},w)|>9\delta^2$,
and the forms
from \eqref{NForm}
\begin{multline*}
N_j(z,w,\lambda)=w_0^{\ell}\cdot\lim_{\epsilon\to 0}
\int_{\Gamma_{\zeta}^{\epsilon}}\bar\zeta_0^{\ell}\zeta_0^{-\ell}\vartheta(\zeta)
e^{\langle\lambda,\zeta/\zeta_0\rangle-\overline{\langle\lambda,\zeta/\zeta_0\rangle}}\\
\times\det\left[\frac{\bar\zeta}{B^*(w,\zeta)}\ \frac{\bar{w}}{B(w,\zeta)}\ Q(w,\zeta)\right]
\det\left[\frac{\d z}{B^*(\bar\zeta,\bar{z})}\
\frac{z}{B^*(\bar\zeta,\bar{z})}\ Q(\bar\zeta,\bar{z})\right]
\wedge\d\left(\frac{\zeta_j}{\zeta_0}\right)
\wedge\frac{\d\bar\zeta}{P(\bar\zeta)}.
\end{multline*}
\end{lemma}
\begin{proof}
We denote $\nu=9\delta^2$ and consider the following domains:
\begin{equation}\label{DifferenceDomains}
U^{\nu}_{\zeta}(z^{(i)})\ \text{for}\ i=1,2;\
U^{\gamma,\nu}_{\zeta}(z)=U^{\gamma}_{\zeta}(z^{(1)}, z^{(2)})
\setminus\left\{\cup_{i=1,2}U^{\nu}_{\zeta}(z^{(i)})\right\};\
U_{\zeta}^{\gamma}(w);\ U_{\zeta}^{\tau,\gamma}(z,w).
\end{equation}
Then, we estimate the difference of integrals in $N_j(z^{(i)},w,\lambda)$ in those domains.

\indent
For the first two domains: $U^{\nu}_{\zeta}(z^{(i)})$, using estimate \eqref{NKappaEstimate}
from Lemma~\ref{NTauSmallDomains} we obtain the estimate
\begin{multline}\label{NFirstDomains}
\left|B(z^{(i)},w)^{3/2}\cdot N^{\nu}_j(z^{(i)},w,\lambda)\right|\\
=\Bigg|B(z^{(i)},w)^{3/2}\cdot w_0^{\ell}\cdot\lim_{\epsilon\to 0}
\int_{\Gamma_{\zeta}^{\epsilon} \cap\left\{|B(z^{(i)},\zeta)|<\nu\right\}}
\bar\zeta_0^{\ell}\zeta_0^{-\ell}\vartheta(\zeta)
e^{\langle\lambda,\zeta/\zeta_0\rangle-\overline{\langle\lambda,\zeta/\zeta_0\rangle}}\\
\times\det\left[\frac{\bar\zeta}{B^*(w,\zeta)}\ \frac{\bar{w}}{B(w,\zeta)}\ Q(w,\zeta)\right]
\det\left[\frac{\d z}{B^*(\bar\zeta,\bar{z}^{(i)})}\
\frac{z^{(i)}}{B^*(\bar\zeta,\bar{z}^{(i)})}\ Q(\bar\zeta,\bar{z}^{(i)})\right]
\wedge\d\left(\frac{\zeta_j}{\zeta_0}\right)
\wedge\frac{\d\bar\zeta}{P(\bar\zeta)}\Bigg|\\
\leq C(\lambda)\cdot\sqrt{\frac{\nu}{\gamma}}
=C(\lambda)\cdot\frac{\delta}{\sqrt{|B(z^{(i)},w)|}},
\end{multline}
which is similar to estimate \eqref{NzKappaEstimate} and implies the estimate
\eqref{NDifferenceEstimate} for each term of the integral over $U^{\nu}_{\zeta}(z^{(i)})$
in the left-hand side of \eqref{NDifferenceEstimate}.

\indent
To estimate the difference of integrals in $N_j(z^{(i)},w,\lambda)$ over the rest of the domains in
\eqref{DifferenceDomains} we will use the notation
\begin{equation}\label{psiNotation}
\psi(z,w,\lambda,\zeta)=\bar\zeta_0^{\ell}\zeta_0^{-\ell}\vartheta(\zeta)
e^{\langle\lambda,\zeta/\zeta_0\rangle-\overline{\langle\lambda,\zeta/\zeta_0\rangle}}
S(w,\zeta)\det\left[\d z\ z\ Q(\bar\zeta,\bar{z})\right]
\wedge\left((\d{\bar P}(\zeta)\wedge\d_{\zeta}B^*(\bar\zeta,{\bar z}))
\interior\d\bar\zeta\right).
\end{equation}

Then, for the integrals over 
$U^{\gamma,\nu}_{\zeta}(z)=U^{\gamma}_{\zeta}(z^{(1)}, z^{(2)})
\setminus\left\{\cup_{i=1,2}U^{\nu}_{\zeta}(z^{(i)})\right\}$ we use the formulas from
Lemma~\ref{NLzTauDomain} to obtain
\begin{multline}\label{NSecondDifference}
\frac{B(z^{(1)},w)^{3/2}(\bar\zeta-{\bar w})}{B(w,\zeta)B^*(w,\zeta)}
\d(\zeta_j/\zeta_0)\wedge\frac{\d {\bar P}(\zeta)}{{\bar P}(\zeta)}\\
\wedge\left(\frac{\psi(z^{(1)},w,\lambda,\zeta)\wedge\d_{\zeta}B^*(\bar\zeta,{\bar z}^{(1)})}
{B^*({\bar\zeta},{\bar z}^{(1)})^2}
-\frac{\psi(z^{(2)},w,\lambda,\zeta)\wedge\d_{\zeta}B^*(\bar\zeta,{\bar z}^{(2)})}
{B^*({\bar\zeta},{\bar z}^{(2)})^2}\right)\\
=\frac{B(z^{(1)},w)^{3/2}(\bar\zeta-{\bar w})}{B(w,\zeta)B^*(w,\zeta)}
\d(\zeta_j/\zeta_0)\wedge\frac{\d {\bar P}(\zeta)}{{\bar P}(\zeta)}\\
\wedge\left(\psi(z^{(1)},w,\lambda,\zeta)\wedge
\d_{\zeta}\left(\frac{1}{B^*({\bar\zeta},{\bar z}^{(1)})}\right)
-\psi(z^{(2)},w,\lambda,\zeta)\wedge
\d_{\zeta}\left(\frac{1}{B^*({\bar\zeta},{\bar z}^{(2)})}\right)\right).
\end{multline}

\indent
As in \eqref{NLzKappaEquality} we apply the Stokes' theorem and obtain
\begin{multline}\label{NDifferenceStokes}
\lim_{\epsilon\to 0}
\int_{\Gamma_{\zeta}^{\epsilon} \cap\left\{\nu<|B(z^{(1)},\zeta)|<\gamma\right\}}
\frac{B(z^{(1)},w)^{3/2}(\bar\zeta-{\bar w})}{B(w,\zeta)B^*(w,\zeta)}\d(\zeta_j/\zeta_0)
\wedge\frac{\d {\bar P}(\zeta)}{{\bar P}(\zeta)}\\
\wedge\left(\psi(z^{(1)},w,\lambda,\zeta)\wedge
\d_{\zeta}\left(\frac{1}{B^*({\bar\zeta},{\bar z}^{(1)})}\right)
-\psi(z^{(2)},w,\lambda,\zeta)\wedge
\d_{\zeta}\left(\frac{1}{B^*({\bar\zeta},{\bar z}^{(2)})}\right)\right)\\
=B(z^{(1)},w)^{3/2}\int_{\pi^{-1}({\cal C}) \cap\left\{|B(z^{(1)},\zeta|=\nu\right\}}
\left(\frac{\Psi(z^{(1)},w,\lambda,\zeta)}{B^*({\bar\zeta},{\bar z}^{(1)})}
-\frac{\Psi(z^{(2)},w,\lambda,\zeta)}{B^*({\bar\zeta},{\bar z}^{(2)})}\right)\\
-B(z^{(1)},w)^{3/2}\int_{\pi^{-1}({\cal C}) \cap\left\{|B(z^{(1)},\zeta|=\gamma\right\}}
\left(\frac{\Psi(z^{(1)},w,\lambda,\zeta)}{B^*({\bar\zeta},{\bar z}^{(1)})}
-\frac{\Psi(z^{(2)},w,\lambda,\zeta)}{B^*({\bar\zeta},{\bar z}^{(2)})}\right)\\
-B(z^{(1)},w)^{3/2}\int_{\pi^{-1}({\cal C}) \cap\left\{\nu<|B(z^{(1)},\zeta|<\gamma\right\}}
\left(\frac{\d_{\zeta}\Psi(z^{(1)},w,\lambda,\zeta)}{B^*({\bar\zeta},{\bar z}^{(1)})}
-\frac{\d_{\zeta}\Psi(z^{(2)},w,\lambda,\zeta)}{B^*({\bar\zeta},{\bar z}^{(2)})}\right),
\end{multline}
where
$$\Psi(z^{(i)},w,\lambda,\zeta)=\frac{(\bar\zeta-{\bar w})}
{B(w,\zeta)B^*(w,\zeta)}
\d(\zeta_j/\zeta_0)\wedge\psi(z^{(i)},w,\lambda,\zeta).$$

\indent
For the first integral in the right-hand side of \eqref{NDifferenceStokes} we use estimate \eqref{DeltaAreaEstimate}
\begin{equation*}
A(\nu)=C\cdot\nu^{3/2}
\end{equation*}
of the area of integration $\left\{\pi^{-1}({\cal C})\cap|B(z,\zeta)|=\nu\right\}$ in this integral
and the boundedness of the forms $B(z^{(1)},w)^{3/2}\Psi(z^{(i)},w,\lambda,\zeta)$ on
$\left\{\nu<|B(z,\zeta|<\gamma\right\}$, which follows from the estimates
\eqref{GammaInequalities}. Then, we obtain the following estimate
\begin{equation}\label{NDifferenceStokesFirst}
\Bigg|B(z^{(1)},w)^{3/2}\int_{\pi^{-1}({\cal C}) \cap\left\{|B(z^{(i)},\zeta|=\nu\right\}}
\frac{\Psi(z^{(i)},w,\lambda,\zeta)}{B^*({\bar\zeta},{\bar z}^{(i)})}\Bigg|
\leq C(\lambda)\cdot\frac{\nu^{3/2}}{\nu}\leq C(\lambda)\cdot\delta,\\
\end{equation}
where the property $\lim_{\lambda\to\infty}C(\lambda)=0$ follows from the Riemann-Lebesgue Lemma (see \cite{23}).\\
\indent
For the second integral in the right-hand side of \eqref{NDifferenceStokes} we use the equality
\begin{multline}\label{NDifferenceStokesSecondRepresentation}
B(z^{(1)},w)^{3/2}\int_{\pi^{-1}({\cal C}) \cap\left\{|B(z^{(1)},\zeta|=\gamma\right\}}
\left(\frac{\Psi(z^{(1)},w,\lambda,\zeta)}{B^*({\bar\zeta},{\bar z}^{(1)})}
-\frac{\Psi(z^{(2)},w,\lambda,\zeta)}{B^*({\bar\zeta},{\bar z}^{(2)})}\right)\\
=B(z^{(1)},w)^{3/2}\int_{\pi^{-1}({\cal C}) \cap\left\{|B(z^{(1)},\zeta|=\gamma\right\}}
\Psi(z^{(1)},w,\lambda,\zeta)\left(\frac{B^*({\bar\zeta},{\bar z}^{(2)})
-B^*({\bar\zeta},{\bar z}^{(1)})}
{B^*({\bar\zeta},{\bar z}^{(1)})B^*({\bar\zeta},{\bar z}^{(2)})}\right)\\
+B(z^{(1)},w)^{3/2}\int_{\pi^{-1}({\cal C}) \cap\left\{|B(z^{(1)},\zeta|=\gamma\right\}}
\left(\frac{\Psi(z^{(1)},w,\lambda,\zeta)-\Psi(z^{(2)},w,\lambda,\zeta)}
{B^*({\bar\zeta},{\bar z}^{(2)})}\right).
\end{multline}
For the first term of the right-hand side of \eqref{NDifferenceStokesSecondRepresentation} we use
estimate \eqref{BDifferenceInequality} and the boundedness of the forms
$B(z^{(1)},w)^{3/2}\Psi(z^{(i)},w,\lambda,\zeta)$ on
$\left\{|B(z^{(1)},\zeta|=\gamma\right\}$.
Then, using coordinates $B(z,\zeta)= it+r^2$ and inequality $|B(w,\zeta)|\geq C\cdot(\gamma+r^2)$
we obtain the estimate
\begin{multline}\label{NDifferenceStokesSecondFirst}
\Bigg|B(z^{(1)},w)^{3/2}\int_{\pi^{-1}({\cal C}) \cap\left\{|B(z^{(1)},\zeta|=\gamma\right\}}
\Psi(z^{(1)},w,\lambda,\zeta)\left(\frac{B^*({\bar\zeta},{\bar z}^{(2)})
-B^*({\bar\zeta},{\bar z}^{(1)})}
{B^*({\bar\zeta},{\bar z}^{(1)})B^*({\bar\zeta},{\bar z}^{(2)})}\right)\Bigg|\\
\leq C(\lambda)\cdot\delta\gamma^{3/2}\int_{\nu}^{\gamma}\d t
\int_{\sqrt{\nu}}^{\sqrt{\gamma}}
\frac{\d r}{(\gamma+t+r^2)^{3/2}(\gamma+r^2)^{3/2}}
\leq C(\lambda)\cdot\delta.
\end{multline}
For the second term of the right-hand side of \eqref{NDifferenceStokesSecondRepresentation}
using estimate
$$\left|\psi(z^{(1)},w,\lambda,\zeta)-\psi(z^{(2)},w,\lambda,\zeta)\right|\leq C\cdot\delta$$
we obtain
\begin{multline}\label{NDifferenceStokesSecondSecond}
\Bigg|B(z^{(1)},w)^{3/2}\int_{\pi^{-1}({\cal C}) \cap\left\{|B(z^{(1)},\zeta|=\gamma\right\}}
\left(\frac{\Psi(z^{(1)},w,\lambda,\zeta)-\Psi(z^{(2)},w,\lambda,\zeta)}
{B^*({\bar\zeta},{\bar z}^{(2)})}\right)\Bigg|\\
\leq C(\lambda)\cdot\delta\gamma^{3/2}\int_{\nu}^{\gamma}\d t
\int_{\sqrt{\nu}}^{\sqrt{\gamma}}
\frac{\d r}{(\gamma+t+r^2)(\gamma+r^2)^{3/2}}
\leq C(\lambda)\cdot\delta.
\end{multline}
\indent
For the third integral in the right-hand side of \eqref{NDifferenceStokes} we have the equality
\begin{multline}\label{NDifferenceStokesThirdRepresentation}
B(z^{(1)},w)^{3/2}\int_{\pi^{-1}({\cal C}) \cap\left\{\nu<|B(z^{(1)},\zeta|<\gamma\right\}}
\left(\frac{\d_{\zeta}\Psi(z^{(1)},w,\lambda,\zeta)}{B^*({\bar\zeta},{\bar z}^{(1)})}
-\frac{\d_{\zeta}\Psi(z^{(2)},w,\lambda,\zeta)}{B^*({\bar\zeta},{\bar z}^{(2)})}\right)\\
=B(z^{(1)},w)^{3/2}\int_{\pi^{-1}({\cal C}) \cap\left\{\nu<|B(z^{(1)},\zeta|<\gamma\right\}}
\d_{\zeta}\Psi(z^{(1)},w,\lambda,\zeta)\left(\frac{B^*({\bar\zeta},{\bar z}^{(2)})
-B^*({\bar\zeta},{\bar z}^{(1)})}
{B^*({\bar\zeta},{\bar z}^{(1)})B^*({\bar\zeta},{\bar z}^{(2)})}\right)\\
+B(z^{(1)},w)^{3/2}\int_{\pi^{-1}({\cal C}) \cap\left\{\nu<|B(z^{(1)},\zeta|<\gamma\right\}}
\left(\frac{\d_{\zeta}\Psi(z^{(1)},w,\lambda,\zeta)-\d_{\zeta}\Psi(z^{(2)},w,\lambda,\zeta)}
{B^*({\bar\zeta},{\bar z}^{(2)})}\right).
\end{multline}

Then, for the first term of the right-hand side of \eqref{NDifferenceStokesThirdRepresentation}
using as in Lemma~\ref{NLzTauDomain} the following analogue of estimate \eqref{dPhi}:
\begin{multline}\label{NAnalogdPhi}
\d_{\zeta}\Psi(z^{(i)},w,\lambda,\zeta)
=\d_{\zeta}\left(\psi(z^{(i)},w,\lambda,\zeta)
\frac{(\bar\zeta-{\bar w})}{B^*(w,\zeta)B(w,\zeta)}\right)\\
\sim \d_{\zeta}\psi(z^{(i)},w,\lambda,\zeta)
\left(\frac{(\bar\zeta-{\bar w})}{B^*(w,\zeta)B(w,\zeta)}\right)\\
+\psi(z^{(i)},w,\lambda,\zeta)\Bigg(\frac{\d\bar\zeta}{B^*(w,\zeta)B(w,\zeta)}
+\frac{(\bar\zeta-{\bar w})\d_{\zeta}B^*(w,\zeta)}{B^*(w,\zeta)^2B(w,\zeta)}
+\frac{(\bar\zeta-{\bar w})\d_{\zeta}B(w,\zeta)}{B^*(w,\zeta)B(w,\zeta)^2}\Bigg),
\end{multline}
coordinates $|\zeta-z|=r$ and $B({\bar\zeta},{\bar z})=it+r^2$, and inequality
$|\zeta-w|<C\sqrt{\gamma}$ for
$\zeta\in \{|B(z,\zeta|<\gamma\}$ from \eqref{GammaInequalities}, and estimate
\eqref{BDifferenceInequality} we obtain
\begin{multline}\label{NDifferenceStokesThirdFirst}
\Bigg|B(z^{(1)},w)^{3/2}\int_{\pi^{-1}({\cal C}) \cap\left\{\nu<|B(z^{(1)},\zeta|<\gamma\right\}}
\d_{\zeta}\Psi(z^{(1)},w,\lambda,\zeta)\left(\frac{B^*({\bar\zeta},{\bar z}^{(2)})
-B^*({\bar\zeta},{\bar z}^{(1)})}
{B^*({\bar\zeta},{\bar z}^{(1)})B^*({\bar\zeta},{\bar z}^{(2)})}\right)\Bigg|\\
\leq C(\lambda)\cdot\delta\gamma^{3/2}\Bigg|\int_{\nu}^{\gamma}\d t
\int_{\sqrt{\nu}}^{\sqrt{\gamma}}
\frac{r\d r}{(\nu+t+r^2)^{3/2}(\gamma+r^2)^{5/2}}\Bigg|
\leq C(\lambda)\cdot\delta\gamma^{3/2}\int_{\sqrt{\nu}}^{\sqrt{\gamma}}
\frac{\d r}{(\gamma+r^2)^{5/2}}\\
\leq C(\lambda)\cdot \delta\gamma^{-1/2}.
\end{multline}
\indent
For the second term of the right-hand side of \eqref{NDifferenceStokesThirdRepresentation}
also using estimate \eqref{NAnalogdPhi} we obtain
\begin{multline}\label{NDifferenceStokesThirdSecond}
\Bigg|B(z^{(1)},w)^{3/2}\int_{\pi^{-1}({\cal C}) \cap\left\{\nu<|B(z^{(1)},\zeta|<\gamma\right\}}
\left(\frac{\d_{\zeta}\Psi(z^{(1)},w,\lambda,\zeta)-\d_{\zeta}\Psi(z^{(2)},w,\lambda,\zeta)}
{B^*({\bar\zeta},{\bar z}^{(2)})}\right)\Bigg|\\
\leq C(\lambda)\cdot\delta\gamma^{3/2}\Bigg|\int_{\nu}^{\gamma}\d t
\int_{\sqrt{\nu}}^{\sqrt{\gamma}}
\frac{r\d r}{(\nu+t+r^2)(\gamma+r^2)^{5/2}}\Bigg|
\leq  C(\lambda)\cdot\delta\gamma^{3/2}\int_{\sqrt{\nu}}^{\sqrt{\gamma}}
\frac{r(\log{r})\d r}{(\gamma+r^2)^{5/2}}\\
\leq C(\lambda)\cdot \delta\gamma^{-1/2}.
\end{multline}
\indent
Combining estimates \eqref{NDifferenceStokesFirst}, \eqref{NDifferenceStokesSecondFirst},
\eqref{NDifferenceStokesSecondSecond}, \eqref{NDifferenceStokesThirdFirst}, and \eqref{NDifferenceStokesThirdSecond} we obtain the estimate \eqref{NDifferenceEstimate}
for the integrals in $N_j(z^{(i)},w,\lambda)$ over
$U^{\gamma,\nu}_{\zeta}(z)=U^{\gamma}_{\zeta}(z^{(1)}, z^{(2)})
\setminus\left\{\cup_{i=1,2}U^{\nu}_{\zeta}(z^{(i)})\right\}$.

\indent
To estimate the differences of integrals over the last two domains in \eqref{DifferenceDomains}
we denote
$$\phi(z,w,\lambda,\zeta)
=\psi(z,w,\lambda,\zeta)\wedge\d_{\zeta}B^*(\bar\zeta,{\bar z}),$$
and using formula \eqref{NSecondDifference} obtain the equality
\begin{multline}\label{NThirdDifference}
\frac{B(z^{(1)},w)^{3/2}(\bar\zeta-{\bar w})}{B(w,\zeta)B^*(w,\zeta)}
\left(\frac{\phi(z^{(1)},w,\lambda,\zeta)}{B^*({\bar\zeta},{\bar z}^{(1)})^2}
-\frac{\phi(z^{(2)},w,\lambda,\zeta)}{B^*({\bar\zeta},{\bar z}^{(2)})^2}\right)
\d(\zeta_j/\zeta_0)\wedge\frac{\d{\bar P}(\zeta)}{{\bar P}(\zeta)}\\
=\frac{B(z^{(1)},w)^{3/2}(\bar\zeta-{\bar w})}{B(w,\zeta)B^*(w,\zeta)}
\phi(z^{(1)},w,\lambda,\zeta)\cdot\left(\frac{B^*({\bar\zeta},{\bar z}^{(2)})^2
-B^*({\bar\zeta},{\bar z}^{(1)})^2}
{B^*({\bar\zeta},{\bar z}^{(1)})^2B^*({\bar\zeta},{\bar z}^{(2)})^2}\right)
\d(\zeta_j/\zeta_0)\wedge\frac{\d{\bar P}(\zeta)}{{\bar P}(\zeta)}\\
+\frac{B(z^{(1)},w)^{3/2}(\bar\zeta-{\bar w})}{B(w,\zeta)B^*(w,\zeta)}
\cdot\left(\frac{\phi(z^{(1)},w,\lambda,\zeta)-\phi(z^{(2)},w,\lambda,\zeta)}
{B^*({\bar\zeta},{\bar z}^{(2)})^2}\right)
\d(\zeta_j/\zeta_0)\wedge\frac{\d{\bar P}(\zeta)}{{\bar P}(\zeta)}.
\end{multline}
Then, for the first term of the right-hand side of \eqref {NThirdDifference}
on $U_{\zeta}^{\gamma}(w)$ using coordinates $B(w,\zeta)=it+r^2$ we obtain
\begin{multline}\label{NThirdDifferenceFirst}
\Bigg|\int_{\pi^{-1}({\cal C}) \cap U_{\zeta}^{\gamma}(w)}
\frac{B(z^{(1)},w)^{3/2}(\bar\zeta-{\bar w})}{B(w,\zeta)B^*(w,\zeta)}
\phi(z^{(1)},w,\lambda,\zeta)\cdot\left(\frac{B^*({\bar\zeta},{\bar z}^{(2)})^2
-B^*({\bar\zeta},{\bar z}^{(1)})^2}
{B^*({\bar\zeta},{\bar z}^{(1)})^2B^*({\bar\zeta},{\bar z}^{(2)})^2}\right)
\wedge\d(\zeta_j/\zeta_0)\Bigg|\\
\leq C(\lambda)\cdot\delta\gamma^{3/2}\Bigg|\int_{0}^{\gamma}\d t
\int_{0}^{\sqrt{\gamma}}
\frac{r^2\d r}{(\nu+t+r^2)^2(\gamma+r^2)^{5/2}}\Bigg|
\leq C(\lambda)\cdot\delta\int_0^{\sqrt{\gamma}}\frac{\d r}{(\sqrt{\gamma}+r)^2}
\leq C(\lambda)\cdot\delta\gamma^{-1/2}.
\end{multline}
For the second term of the right-hand side of \eqref {NThirdDifference}
on $U_{\zeta}^{\gamma}(w)$ using coordinates $B(w,\zeta)=it+r^2$ we obtain
\begin{multline}\label{NThirdDifferenceSecond}
\Bigg|\int_{\pi^{-1}({\cal C}) \cap U_{\zeta}^{\gamma}(w)}
\frac{B(z^{(1)},w)^{3/2}(\bar\zeta-{\bar w})}{B(w,\zeta)B^*(w,\zeta)}
\cdot\left(\frac{\phi(z^{(1)},w,\lambda,\zeta)-\phi(z^{(2)},w,\lambda,\zeta)}
{B^*({\bar\zeta},{\bar z}^{(2)})^2}\right)
\d(\zeta_j/\zeta_0)\Bigg|\\
\leq C(\lambda)\cdot\delta\gamma^{3/2}\Bigg|\int_{0}^{\gamma}\d t
\int_{0}^{\sqrt{\gamma}}
\frac{r^2\d r}{(\nu+t+r^2)^2(\gamma+r^2)^2}\Bigg|
\leq C(\lambda)\cdot\delta.
\end{multline}
\indent
From the estimates \eqref{NThirdDifferenceFirst} and \eqref{NThirdDifferenceSecond} we obtain the
estimate \eqref{NDifferenceEstimate} for the integrals in $N_j(z^{(i)},w,\lambda)$ over the domain
$U_{\zeta}^{\gamma}(w)$.

\indent
For the last domain in \eqref{DifferenceDomains} --
$U_{\zeta}^{\tau,\gamma}(z,w)=U_{\zeta}^{\tau}(z)
\setminus\left\{U^{\gamma}_{\zeta}(z)\cup U^{\gamma}_{\zeta}(w)\right\}$
-- we use equality \eqref{NThirdDifference} and consider the integral of the first term of the
right-hand side of this equality on $U_{\zeta}^{\tau,\gamma}(z,w)$
in coordinates $B(w,\zeta)=it+r^2$, to obtain
\begin{multline}\label{NThirdDifferenceThird}
\Bigg|\int_{\pi^{-1}({\cal C}) \cap U_{\zeta}^{\tau,\gamma}(z,w)}
\frac{B(z^{(1)},w)^{3/2}(\bar\zeta-{\bar w})}{B(w,\zeta)B^*(w,\zeta)}
\phi(z^{(1)},w,\lambda,\zeta)\cdot\left(\frac{B^*({\bar\zeta},{\bar z}^{(2)})^2
-B^*({\bar\zeta},{\bar z}^{(1)})^2}
{B^*({\bar\zeta},{\bar z}^{(1)})^2B^*({\bar\zeta},{\bar z}^{(2)})^2}\right)
\d(\zeta_j/\zeta_0)\Bigg|\\
\leq C(\lambda)\cdot\delta\gamma^{3/2}\int_{\gamma}^{\tau}\d t
\int_{\sqrt{\gamma}}^{\sqrt{\tau}}
\frac{r^2\d r}{(\gamma+t+r^2)^2(\gamma+r^2)^{5/2}}
\leq C(\lambda)\cdot\delta\gamma^{3/2}
\int_{\sqrt{\gamma}}^{\sqrt{\tau}}
\frac{r^2\d r}{(\gamma+r^2)^{7/2}}\\
\leq C(\lambda)\cdot\delta\gamma^{1/2}
\int_{\sqrt{\gamma}}^{\sqrt{\tau}}
\frac{\d r}{(\gamma+r^2)^{3/2}}
\leq C(\lambda)\cdot\delta\gamma^{1/2}
\int_{\sqrt{\gamma}}^{\sqrt{\tau}}
\frac{\d r}{(\sqrt{\gamma}+r)^3}
\leq C(\lambda)\cdot\delta\gamma^{-1/2}.\\
\end{multline}
\indent
For the second term of the right-hand side of \eqref {NThirdDifference}
on $U_{\zeta}^{\tau,\gamma}(z,w)$ using coordinates $B(w,\zeta)=it+r^2$ we obtain
\begin{multline}\label{NThirdDifferenceFourth}
\Bigg|\int_{\pi^{-1}({\cal C}) \cap U_{\zeta}^{\tau,\gamma}(z,w)}
\frac{B(z^{(1)},w)^{3/2}(\bar\zeta-{\bar w})}{B(w,\zeta)B^*(w,\zeta)}
\cdot\left(\frac{\phi(z^{(1)},w,\lambda,\zeta)-\phi(z^{(2)},w,\lambda,\zeta)}
{B^*({\bar\zeta},{\bar z}^{(2)})^2}\right)\d(\zeta_j/\zeta_0)\Bigg|\\
\leq C(\lambda)\cdot\delta\gamma^{3/2}\int_{\gamma}^{\tau}\d t
\int_{\sqrt{\gamma}}^{\sqrt{\tau}}
\frac{r^2\d r}{(\gamma+t+r^2)^2(\gamma+r^2)^2}
\leq C(\lambda)\cdot\delta\gamma^{1/2}
\int_{\sqrt{\gamma}}^{\sqrt{\tau}}
\frac{\d r}{(\sqrt{\gamma}+r)^2}
\leq C(\lambda)\cdot\delta.
\end{multline}
\indent
Combining all estimates of the differences of integrals $N_j(z^{(1)},w,\lambda)$ and
$N_j(z^{(2)},w,\lambda)$ over the domains in \eqref{DifferenceDomains} we obtain
estimate \eqref{NDifferenceEstimate}.
\end{proof}

In the following lemma we estimate the difference $N_j(z,w^{(1)},\lambda)-N_j(z,w^{(2)},\lambda)$
to be used in Lemma~\ref{QLemma}.
\begin{lemma}\label{NwDifferenceLemma}
There exist constants $C(\lambda)$, satisfying $\lim_{\lambda\to\infty}C(\lambda)=0$,
such that the following estimate
\begin{equation}\label{NwDifferenceEstimate}
\left|N_j(z,w^{(1)},\lambda)-N_j(z,w^{(2)},\lambda)\right|
\leq C(\lambda)\cdot\frac{\delta\log{\delta}}{|B(z,w^{(1)})|^2}
\end{equation}
holds for an arbitrary fixed $\delta$,
any two points $w^{(1)},w^{(2)}\in V$, such that $|B(w^{(1)},w^{(2)})|=\delta^2$,
$w^{(1)}$ satisfying $\gamma=|B(z,w^{(1)})|>9\delta^2$,
and the forms
from \eqref{NForm}
\begin{multline*}
N_j(z,w,\lambda)=w_0^{\ell}\cdot\lim_{\epsilon\to 0}
\int_{\Gamma_{\zeta}^{\epsilon}}\bar\zeta_0^{\ell}\zeta_0^{-\ell}\vartheta(\zeta)
e^{\langle\lambda,\zeta/\zeta_0\rangle-\overline{\langle\lambda,\zeta/\zeta_0\rangle}}\\
\times\det\left[\frac{\bar\zeta}{B^*(w,\zeta)}\ \frac{\bar{w}}{B(w,\zeta)}\ Q(w,\zeta)\right]
\det\left[\frac{\d z}{B^*(\bar\zeta,\bar{z})}\
\frac{z}{B^*(\bar\zeta,\bar{z})}\ Q(\bar\zeta,\bar{z})\right]
\wedge\d\left(\frac{\zeta_j}{\zeta_0}\right)
\wedge\frac{\d\bar\zeta}{P(\bar\zeta)}.
\end{multline*}
\end{lemma}
\begin{proof}
As in the Lemma~\ref{NDifferenceLemma} we denote $\nu=9\delta^2$ and consider the following domains:
\begin{equation}\label{NwDifferenceDomains}
U^{\nu}_{\zeta}(w^{(i)})\ \text{for}\ i=1,2;\
U^{\gamma,\nu}_{\zeta}(w)=U^{\gamma}_{\zeta}(w^{(1)}, w^{(2)})
\setminus\left\{\cup_{i=1,2}U^{\nu}_{\zeta}(w^{(i)})\right\};\
U_{\zeta}^{\gamma}(z);\ U_{\zeta}^{\tau,\gamma}(z,w).
\end{equation}
Then, we estimate the difference of integrals in $N_j(z^{(i)},w,\lambda)$ in those domains.

\indent
For the first two domains: $U^{\nu}_{\zeta}(w^{(i)})$, using estimate \eqref{NKappaEstimate}
from Lemma~\ref{NTauSmallDomains} we obtain the estimate
\begin{multline}\label{NwFirstDomains}
\left|B(z,w^{(i)})^{3/2}\cdot N^{\nu}_j(z,w^{(i)},\lambda)\right|\\
=\Bigg|B(z,w^{(i)})^{3/2}\cdot w_0^{\ell}\cdot\lim_{\epsilon\to 0}
\int_{\Gamma_{\zeta}^{\epsilon} \cap\left\{|B(w^{(i)},\zeta)|<\nu\right\}}
\bar\zeta_0^{\ell}\zeta_0^{-\ell}\vartheta(\zeta)
e^{\langle\lambda,\zeta/\zeta_0\rangle-\overline{\langle\lambda,\zeta/\zeta_0\rangle}}\\
\times\det\left[\frac{\bar\zeta}{B^*(w^{(i)},\zeta)}\ \frac{\bar{w}^{(i)}}{B(w^{(i)},\zeta)}\
Q(w^{(i)},\zeta)\right]
\det\left[\frac{\d z}{B^*(\bar\zeta,\bar{z})}\
\frac{z}{B^*(\bar\zeta,\bar{z})}\ Q(\bar\zeta,\bar{z})\right]
\wedge\d\left(\frac{\zeta_j}{\zeta_0}\right)
\wedge\frac{\d\bar\zeta}{P(\bar\zeta)}\Bigg|\\
\leq C(\lambda)\cdot\sqrt{\frac{\nu}{\gamma}}
=C(\lambda)\cdot\frac{\delta}{\sqrt{|B(z,w^{(i)})|}},
\end{multline}
which is similar to estimate \eqref{NwKappaEstimate} and implies the estimate
\eqref{NwDifferenceEstimate} for both integrals over $U^{\nu}_{\zeta}(w^{(i)})$
in the left-hand side of \eqref{NwDifferenceEstimate}.

\indent
To estimate the difference from \eqref{NwDifferenceEstimate} in the rest of the domains in
\eqref{NwDifferenceDomains} we use the following formula
\begin{multline}\label{NwDifferenceFormula}
N_j(z,w^{(1)},\lambda)-N_j(z,w^{(2)},\lambda)=(w^{(1)}_0)^{\ell}\cdot\lim_{\epsilon\to 0}
\int_{\Gamma_{\zeta}^{\epsilon}}
\bar\zeta_0^{\ell}\zeta_0^{-\ell}\vartheta(\zeta)
e^{\langle\lambda,\zeta/\zeta_0\rangle-\overline{\langle\lambda,\zeta/\zeta_0\rangle}}\\
\times\det\left[\frac{\bar\zeta}{B^*(w^{(1)},\zeta)}\ \frac{\bar{w}^{(1)}}{B(w^{(1)},\zeta)}\
Q(w^{(1)},\zeta)\right]
\det\left[\frac{\d z}{B^*(\bar\zeta,\bar{z})}\
\frac{z}{B^*(\bar\zeta,\bar{z})}\ Q(\bar\zeta,\bar{z})\right]
\wedge\d\left(\frac{\zeta_j}{\zeta_0}\right)
\wedge\frac{\d\bar\zeta}{P(\bar\zeta)}\\
-(w^{(2)}_0)^{\ell}\cdot\lim_{\epsilon\to 0}
\int_{\Gamma_{\zeta}^{\epsilon}}
\bar\zeta_0^{\ell}\zeta_0^{-\ell}\vartheta(\zeta)
e^{\langle\lambda,\zeta/\zeta_0\rangle-\overline{\langle\lambda,\zeta/\zeta_0\rangle}}\\
\times\det\left[\frac{\bar\zeta}{B^*(w^{(2)},\zeta)}\ \frac{\bar{w}^{(2)}}{B(w^{(2)},\zeta)}\
Q(w^{(1)},\zeta)\right]
\det\left[\frac{\d z}{B^*(\bar\zeta,\bar{z})}\
\frac{z}{B^*(\bar\zeta,\bar{z})}\ Q(\bar\zeta,\bar{z})\right]
\wedge\d\left(\frac{\zeta_j}{\zeta_0}\right)
\wedge\frac{\d\bar\zeta}{P(\bar\zeta)}\\
=\left((w^{(1)}_0)^{\ell}-(w^{(2)}_0)^{\ell}\right)\cdot\lim_{\epsilon\to 0}
\int_{\Gamma_{\zeta}^{\epsilon}}
\bar\zeta_0^{\ell}\zeta_0^{-\ell}\vartheta(\zeta)
e^{\langle\lambda,\zeta/\zeta_0\rangle-\overline{\langle\lambda,\zeta/\zeta_0\rangle}}\\
\times\det\left[\frac{\bar\zeta}{B^*(w^{(1)},\zeta)}\ \frac{\bar{w}^{(1)}}{B(w^{(1)},\zeta)}\
Q(w^{(1)},\zeta)\right]
\det\left[\frac{\d z}{B^*(\bar\zeta,\bar{z})}\
\frac{z}{B^*(\bar\zeta,\bar{z})}\ Q(\bar\zeta,\bar{z})\right]
\wedge\d\left(\frac{\zeta_j}{\zeta_0}\right)
\wedge\frac{\d\bar\zeta}{P(\bar\zeta)}\\
+(w^{(2)}_0)^{\ell}\cdot\lim_{\epsilon\to 0}
\int_{\Gamma_{\zeta}^{\epsilon}}
\bar\zeta_0^{\ell}\zeta_0^{-\ell}\vartheta(\zeta)
e^{\langle\lambda,\zeta/\zeta_0\rangle-\overline{\langle\lambda,\zeta/\zeta_0\rangle}}\\
\times\det\left[\bar\zeta\cdot\left(\frac{1}{B^*(w^{(1)},\zeta)}
-\frac{1}{B^*(w^{(2)},\zeta)}\right)\ \frac{\bar{w}^{(1)}}{B(w^{(1)},\zeta)}\
Q(w^{(1)},\zeta)\right]\\
\wedge\det\left[\frac{\d z}{B^*(\bar\zeta,\bar{z})}\
\frac{z}{B^*(\bar\zeta,\bar{z})}\ Q(\bar\zeta,\bar{z})\right]
\wedge\d\left(\frac{\zeta_j}{\zeta_0}\right)
\wedge\frac{\d\bar\zeta}{P(\bar\zeta)}\\
+(w^{(2)}_0)^{\ell}\cdot\lim_{\epsilon\to 0}
\int_{\Gamma_{\zeta}^{\epsilon}}
\bar\zeta_0^{\ell}\zeta_0^{-\ell}\vartheta(\zeta)
e^{\langle\lambda,\zeta/\zeta_0\rangle-\overline{\langle\lambda,\zeta/\zeta_0\rangle}}\\
\times\det\left[\frac{\bar\zeta}{B^*(w^{(2)},\zeta)}\ 
\frac{\left(\bar{w}^{(1)}-\bar{w}^{(2)}\right)}{B(w^{(1)},\zeta)}\
Q(w^{(1)},\zeta)\right]
\wedge\det\left[\frac{\d z}{B^*(\bar\zeta,\bar{z})}\
\frac{z}{B^*(\bar\zeta,\bar{z})}\ Q(\bar\zeta,\bar{z})\right]
\wedge\d\left(\frac{\zeta_j}{\zeta_0}\right)
\wedge\frac{\d\bar\zeta}{P(\bar\zeta)}\\
+(w^{(2)}_0)^{\ell}\cdot\lim_{\epsilon\to 0}
\int_{\Gamma_{\zeta}^{\epsilon}}
\bar\zeta_0^{\ell}\zeta_0^{-\ell}\vartheta(\zeta)
e^{\langle\lambda,\zeta/\zeta_0\rangle-\overline{\langle\lambda,\zeta/\zeta_0\rangle}}\\
\times\det\left[\frac{\bar\zeta}{B^*(w^{(2)},\zeta)}\ 
\bar{w}^{(2)}\left(\frac{1}{B(w^{(1)},\zeta)}-\frac{1}{B(w^{(2)},\zeta)}\right)\
Q(w^{(1)},\zeta)\right]\\
\wedge\det\left[\frac{\d z}{B^*(\bar\zeta,\bar{z})}\
\frac{z}{B^*(\bar\zeta,\bar{z})}\ Q(\bar\zeta,\bar{z})\right]
\wedge\d\left(\frac{\zeta_j}{\zeta_0}\right)
\wedge\frac{\d\bar\zeta}{P(\bar\zeta)}\\
+(w^{(2)}_0)^{\ell}\cdot\lim_{\epsilon\to 0}
\int_{\Gamma_{\zeta}^{\epsilon}}
\bar\zeta_0^{\ell}\zeta_0^{-\ell}\vartheta(\zeta)
e^{\langle\lambda,\zeta/\zeta_0\rangle-\overline{\langle\lambda,\zeta/\zeta_0\rangle}}\\
\times\det\left[\frac{\bar\zeta}{B^*(w^{(2)},\zeta)}\ 
\frac{\bar{w}^{(2)}}{B(w^{(2)},\zeta)}\
\left(Q(w^{(1)},\zeta)-Q(w^{(2)},\zeta)\right)\right]\\
\wedge\det\left[\frac{\d z}{B^*(\bar\zeta,\bar{z})}\
\frac{z}{B^*(\bar\zeta,\bar{z})}\ Q(\bar\zeta,\bar{z})\right]
\wedge\d\left(\frac{\zeta_j}{\zeta_0}\right)
\wedge\frac{\d\bar\zeta}{P(\bar\zeta)}.
\end{multline}
\indent
For the first integral of the right-hand side of \eqref{NwDifferenceFormula} in
$U^{\gamma,\nu}_{\zeta}(w)=U^{\gamma}_{\zeta}(w^{(1)}, w^{(2)})
\setminus\left\{\cup_{i=1,2}U^{\nu}_{\zeta}(w^{(i)})\right\}$ using estimates
$$\left|(w^{(1)}_0)^{\ell}-(w^{(2)}_0)^{\ell}\right|\leq C|w^{(1)}-w^{(2)}|$$
and \eqref{NDeltaEstimate} from Lemma~\ref{NDeltaDomain} we obtain
estimate \ref{NwDifferenceEstimate} for this term.

\indent
For the determinant in the second integral of the right-hand side of \eqref{NwDifferenceFormula} in
$U^{\gamma,\nu}_{\zeta}(w)$ using inequalities
\eqref{BDifferenceInequality} and \eqref{ModulInequality} we obtain the inequality
\begin{equation}\label{wDifferenceInequality}
\left|\det\left[\bar\zeta\cdot\left(\frac{1}{B^*(w^{(1)},\zeta)}
-\frac{1}{B^*(w^{(2)},\zeta)}\right)\ \frac{\bar{w}^{(1)}}{B(w^{(1)},\zeta)}\
Q(w^{(1)},\zeta)\right]\right|\leq C\cdot\frac{\delta|w^{(1)}-\zeta|}{|B(w^{(1)},\zeta)|^{5/2}}
\end{equation}

Then we have the estimate for the second term
\begin{multline*}
\Bigg|\gamma^{3/2}\int_{\Gamma_{\zeta}^{\epsilon}
\cap\left\{\gamma>|B(w^{(1)},\zeta)|>\nu\right\}}
\bar\zeta_0^{\ell}\zeta_0^{-\ell}\vartheta(\zeta)
e^{\langle\lambda,\zeta/\zeta_0\rangle-\overline{\langle\lambda,\zeta/\zeta_0\rangle}}\\
\times\det\left[\bar\zeta\cdot\left(\frac{1}{B^*(w^{(1)},\zeta)}
-\frac{1}{B^*(w^{(2)},\zeta)}\right)\ \frac{\bar{w}^{(1)}}{B(w^{(1)},\zeta)}\
Q(w^{(1)},\zeta)\right]\\
\wedge\det\left[\frac{\d z}{B^*(\bar\zeta,\bar{z})}\
\frac{z}{B^*(\bar\zeta,\bar{z})}\ Q(\bar\zeta,\bar{z})\right]
\wedge\d\left(\frac{\zeta_j}{\zeta_0}\right)
\wedge\frac{\d\bar\zeta}{P(\bar\zeta)}\Bigg|\\
\leq C(\lambda)\gamma^{3/2}\int_{\delta^2}^{\gamma^2}\d t\int_{\delta}^{\gamma}\frac{\delta\cdot r^2\d r}{(t+r^2)^{5/2}(\gamma+r^2)^2}\leq C(\lambda)\gamma^{3/2}\delta
\int_{\delta}^{\gamma}\frac{\d r}{r(\gamma+r^2)^2}\\
\leq C(\lambda)\gamma^{-1/2}\delta\log{\delta},
\end{multline*}
which produces the inequality
\begin{multline}\label{NwDifferenceSecondDomainTwo}
\Bigg|(w^{(2)}_0)^{\ell}\cdot\lim_{\epsilon\to 0}
\int_{\Gamma_{\zeta}^{\epsilon}\cap\left\{\gamma>|B(w^{(1)},\zeta)|>\nu\right\}}
\bar\zeta_0^{\ell}\zeta_0^{-\ell}\vartheta(\zeta)
e^{\langle\lambda,\zeta/\zeta_0\rangle-\overline{\langle\lambda,\zeta/\zeta_0\rangle}}\\
\times\det\left[\bar\zeta\cdot\left(\frac{1}{B^*(w^{(1)},\zeta)}
-\frac{1}{B^*(w^{(2)},\zeta)}\right)\ \frac{\bar{w}^{(1)}}{B(w^{(1)},\zeta)}\
Q(w^{(1)},\zeta)\right]\\
\wedge\det\left[\frac{\d z}{B^*(\bar\zeta,\bar{z})}\
\frac{z}{B^*(\bar\zeta,\bar{z})}\ Q(\bar\zeta,\bar{z})\right]
\wedge\d\left(\frac{\zeta_j}{\zeta_0}\right)
\wedge\frac{\d\bar\zeta}{P(\bar\zeta)}\Bigg|\leq C(\lambda)\frac{\delta\log{\delta}}
{|B(z,w^{(1)})|^2}.
\end{multline}

\indent
For the third term of the right-hand side of \eqref{NwDifferenceFormula} we have
\begin{multline*}
\Bigg|\gamma^{3/2}\cdot(w^{(2)}_0)^{\ell}\cdot\lim_{\epsilon\to 0}
\int_{\Gamma_{\zeta}^{\epsilon}\cap\left\{\gamma>|B(w^{(1)},\zeta)|>\nu\right\}}
\bar\zeta_0^{\ell}\zeta_0^{-\ell}\vartheta(\zeta)
e^{\langle\lambda,\zeta/\zeta_0\rangle-\overline{\langle\lambda,\zeta/\zeta_0\rangle}}\\
\times\det\left[\frac{\bar\zeta}{B^*(w^{(2)},\zeta)}\ 
\frac{\left(\bar{w}^{(1)}-\bar{w}^{(2)}\right)}{B(w^{(1)},\zeta)}\
Q(w^{(1)},\zeta)\right]
\wedge\det\left[\frac{\d z}{B^*(\bar\zeta,\bar{z})}\
\frac{z}{B^*(\bar\zeta,\bar{z})}\ Q(\bar\zeta,\bar{z})\right]
\wedge\d\left(\frac{\zeta_j}{\zeta_0}\right)
\wedge\frac{\d\bar\zeta}{P(\bar\zeta)}\Bigg|\\
\leq C(\lambda)\gamma^{3/2}\int_{\delta^2}^{\gamma^2}\d t\int_{\delta}^{\gamma}\frac{\delta\cdot r\d r}{(t+r^2)^2(\gamma+r^2)^2}
\leq C(\lambda)\gamma^{3/2}\delta\int_{\delta}^{\gamma}\frac{\d r}{r(\gamma+r^2)^2}
\leq C(\lambda)\gamma^{-1/2}\delta\log{\delta},
\end{multline*}
which produces the estimate \eqref{NwDifferenceEstimate}.

For the fourth term of the right-hand side of \eqref{NwDifferenceFormula} similarly to \eqref{NwDifferenceSecondDomainTwo} we obtain
\begin{multline}\label{NwDifferenceSecondDomainFour}
(w^{(2)}_0)^{\ell}\cdot\lim_{\epsilon\to 0}
\int_{\Gamma_{\zeta}^{\epsilon}\cap\left\{\gamma>|B(w^{(1)},\zeta)|>\nu\right\}}
\bar\zeta_0^{\ell}\zeta_0^{-\ell}\vartheta(\zeta)
e^{\langle\lambda,\zeta/\zeta_0\rangle-\overline{\langle\lambda,\zeta/\zeta_0\rangle}}\\
\times\det\left[\frac{\bar\zeta}{B^*(w^{(2)},\zeta)}\ 
\bar{w}^{(2)}\left(\frac{1}{B(w^{(1)},\zeta)}-\frac{1}{B(w^{(2)},\zeta)}\right)\
Q(w^{(1)},\zeta)\right]\\
\wedge\det\left[\frac{\d z}{B^*(\bar\zeta,\bar{z})}\
\frac{z}{B^*(\bar\zeta,\bar{z})}\ Q(\bar\zeta,\bar{z})\right]
\wedge\d\left(\frac{\zeta_j}{\zeta_0}\right)
\wedge\frac{\d\bar\zeta}{P(\bar\zeta)}
\leq C(\lambda)\frac{\delta\log{\delta}}{|B(z,w^{(1)})|^2}.
\end{multline}

\indent
For the fifth term of the right-hand side of \eqref{NwDifferenceFormula} using estimates
$$\left|Q(w^{(1)},\zeta)-Q(w^{(2)},\zeta)\right|\leq C|w^{(1)}-w^{(2)}|$$
and \eqref{NDeltaEstimate} from Lemma~\ref{NDeltaDomain} similarly to the first term
of \eqref{NwDifferenceFormula} we obtain
estimate \eqref{NwDifferenceEstimate}.

\indent
From the estimates above we obtain the estimate \eqref{NwDifferenceEstimate} for the difference of integrals in the domain
$$U^{\gamma,\nu}_{\zeta}(w)=U^{\gamma}_{\zeta}(w^{(1)}, w^{(2)})
\setminus\left\{\cup_{i=1,2}U^{\nu}_{\zeta}(w^{(i)})\right\}.$$

\indent
For the domain $U_{\zeta}^{\gamma}(z)$ we again estimate the integrals of each term of
\eqref{NwDifferenceFormula} over this domain.
For the first and fifth terms of the decomposition \eqref{NwDifferenceFormula} using estimate \eqref{NDeltaEstimate} from the Lemma \ref{NDeltaDomain} we obtain the necessary estimate.

For the second term we use notation \eqref{psiNotation}
\begin{equation*}
\psi(z,w,\lambda,\zeta)=\bar\zeta_0^{\ell}\zeta_0^{-\ell}\vartheta(\zeta)
e^{\langle\lambda,\zeta/\zeta_0\rangle-\overline{\langle\lambda,\zeta/\zeta_0\rangle}}
S(w,\zeta)\det\left[\d z\ z\ Q(\bar\zeta,\bar{z})\right]
\wedge\left((\d{\bar P}(\zeta)\wedge\d_{\zeta}B^*(\bar\zeta,{\bar z}))
\interior\d\bar\zeta\right),
\end{equation*}
and denote
\begin{equation*}
D(w^{(1)},w^{(2)},\zeta)=\det\left[\bar\zeta\cdot\left(\frac{1}{B^*(w^{(1)},\zeta)}
-\frac{1}{B^*(w^{(2)},\zeta)}\right)\ \frac{\bar{w}^{(1)}}{B(w^{(1)},\zeta)}\
Q(w^{(1)},\zeta)\right].
\end{equation*}
Then, we obtain
\begin{multline*}
\Bigg|\gamma^{3/2}(w^{(2)}_0)^{\ell}\cdot\lim_{\epsilon\to 0}
\int_{\Gamma_{\zeta}^{\epsilon}\cap\left\{|B(z,\zeta)|<\gamma\right\}}
\bar\zeta_0^{\ell}\zeta_0^{-\ell}\vartheta(\zeta)
e^{\langle\lambda,\zeta/\zeta_0\rangle-\overline{\langle\lambda,\zeta/\zeta_0\rangle}}\\
\times\det\left[\bar\zeta\cdot\left(\frac{1}{B^*(w^{(1)},\zeta)}
-\frac{1}{B^*(w^{(2)},\zeta)}\right)\ \frac{\bar{w}^{(1)}}{B(w^{(1)},\zeta)}\
Q(w^{(1)},\zeta)\right]\\
\wedge\det\left[\frac{\d z}{B^*(\bar\zeta,\bar{z})}\
\frac{z}{B^*(\bar\zeta,\bar{z})}\ Q(\bar\zeta,\bar{z})\right]
\wedge\d\left(\frac{\zeta_j}{\zeta_0}\right)
\wedge\frac{\d\bar\zeta}{P(\bar\zeta)}\Bigg|\\
\leq\gamma^{3/2}\Bigg|\int_{\pi^{-1}({\cal C})\cap\left\{|B(z,\zeta)|<\gamma\right\}}D(w^{(1)},w^{(2)},\zeta)\psi(z,w,\lambda,\zeta)\wedge
\d_{\zeta}\left(\frac{1}{B^*({\bar\zeta},{\bar z})}\right)\Bigg|.
\end{multline*}

For the integral in the right-hand side of inequality above we apply the Stokes' theorem and obtain
\begin{multline}\label{NwStokes}
\gamma^{3/2}\int_{\pi^{-1}({\cal C})\cap\left\{|B(z,\zeta)|<\gamma\right\}}
D(w^{(1)},w^{(2)},\zeta)\psi(z,w,\lambda,\zeta)\wedge
\d_{\zeta}\left(\frac{1}{B^*({\bar\zeta},{\bar z})}\right)\\
=\gamma^{3/2}\int_{\pi^{-1}({\cal C})\cap\left\{|B(z,\zeta)|=\gamma\right\}}
D(w^{(1)},w^{(2)},\zeta)\psi(z,w,\lambda,\zeta)\times
\left(\frac{1}{B^*({\bar\zeta},{\bar z})}\right)\\
-\gamma^{3/2}\lim_{\eta\to 0}
\int_{\pi^{-1}({\cal C})\cap\left\{|B(z,\zeta)|=\eta\right\}}
D(w^{(1)},w^{(2)},\zeta)\psi(z,w,\lambda,\zeta)\times
\left(\frac{1}{B^*({\bar\zeta},{\bar z})}\right)\\
-\gamma^{3/2}\int_{\pi^{-1}({\cal C})\cap\left\{|B(z,\zeta)|<\gamma\right\}}
\d_{\zeta}\left[D(w^{(1)},w^{(2)},\zeta)\psi(z,w,\lambda,\zeta)\right]\times
\left(\frac{1}{B^*({\bar\zeta},{\bar z})}\right).
\end{multline}
For the first term of the right-hand side of \eqref{NwStokes}, using estimates
\eqref{wDifferenceInequality} and \eqref{DeltaAreaEstimate} we obtain
\begin{equation}\label{NwStokesOne}
\gamma^{3/2}\left|\int_{\pi^{-1}({\cal C})\cap\left\{|B(z,\zeta)|=\gamma\right\}}
D(w^{(1)},w^{(2)},\zeta)\psi(z,w,\lambda,\zeta)\times
\left(\frac{1}{B^*({\bar\zeta},{\bar z})}\right)\right|
\leq C(\lambda)\gamma^{3/2}\delta\frac{\gamma^{3/2}}{\gamma^3}
\leq C(\lambda)\delta.
\end{equation}
Similar estimate for the second term of the right-hand side of \eqref{NwStokes} using estimates
\eqref{wDifferenceInequality} and \eqref{DeltaAreaEstimate} gives
\begin{equation}\label{NwStokesTwo}
\gamma^{3/2}\left|\int_{\pi^{-1}({\cal C})\cap\left\{|B(z,\zeta)|=\eta\right\}}
D(w^{(1)},w^{(2)},\zeta)\psi(z,w,\lambda,\zeta)\times
\left(\frac{1}{B^*({\bar\zeta},{\bar z})}\right)\right|
\leq C(\lambda)\gamma^{3/2}\delta\frac{\eta^{3/2}}{\gamma^2\eta}\to 0,
\end{equation}
as $\eta\to 0$.

\indent
For the third term of the right-hand side of \eqref{NwStokes} using the estimate
\begin{equation*}
\left|\d_{\zeta}\left[D(w^{(1)},w^{(2)},\zeta)\psi(z,w,\lambda,\zeta)\right]\right|
\leq C\cdot\frac{\delta}{|B(w^{(1)},\zeta)|^{5/2}}
\end{equation*}
for $\zeta\in \left\{|B(z,\zeta)|<\gamma\right\}$, we obtain
\begin{multline}\label{NwStokesThree}
\gamma^{3/2}\left|\int_{\pi^{-1}({\cal C})\cap\left\{|B(z,\zeta)|<\gamma\right\}}
\d_{\zeta}\left[D(w^{(1)},w^{(2)},\zeta)\psi(z,w,\lambda,\zeta)\right]\times
\left(\frac{1}{B^*({\bar\zeta},{\bar z})}\right)\right|\\
\leq C(\lambda)\cdot\delta\gamma^{3/2}\int_0^{\gamma^2}\d t\int_0^{\gamma}
\frac{r\d r}{(\gamma+r^2)^{5/2}(t+r^2)}
\leq C(\lambda)\cdot\delta\gamma^{3/2}\int_0^{\gamma}
\frac{r(\log{r})\d r}{(\gamma+r^2)^{5/2}}\leq C(\lambda)\cdot\delta\gamma^{-1/2},
\end{multline}
which implies the estimate \eqref{NwDifferenceEstimate}.

\indent
Combining estimates \eqref{NwStokesOne}$\div$\eqref{NwStokesThree} we obtain estimate
\eqref{NwDifferenceEstimate} for the second term of \eqref{NwDifferenceFormula} in
$U_{\zeta}^{\gamma}(z)$.

\indent
Similarly we estimate the fourth term of \eqref{NwDifferenceFormula} because the estimates for the
determinants
$$\det\left[\bar\zeta\cdot\left(\frac{1}{B^*(w^{(1)},\zeta)}
-\frac{1}{B^*(w^{(2)},\zeta)}\right)\ \frac{\bar{w}^{(1)}}{B(w^{(1)},\zeta)}\
Q(w^{(1)},\zeta)\right]$$
and
$$\det\left[\frac{\bar\zeta}{B^*(w^{(2)},\zeta)}\ 
\bar{w}^{(2)}\left(\frac{1}{B(w^{(1)},\zeta)}-\frac{1}{B(w^{(2)},\zeta)}\right)\
Q(w^{(1)},\zeta)\right]$$
are similar.

For the third term of the right-hand side of \eqref{NwDifferenceFormula}, as above, we denote
\begin{equation*}
D(w^{(1)},w^{(2)},\zeta)=\det\left[\frac{\bar\zeta}{B^*(w^{(2)},\zeta)}\ 
\frac{\left(\bar{w}^{(1)}-\bar{w}^{(2)}\right)}{B(w^{(1)},\zeta)}\
Q(w^{(1)},\zeta)\right],
\end{equation*}
and use the inequality
\begin{multline*}
\Bigg|\gamma^{3/2}\cdot(w^{(2)}_0)^{\ell}\cdot\lim_{\epsilon\to 0}
\int_{\Gamma_{\zeta}^{\epsilon}\cap\left\{|B(z,\zeta)|<\gamma\right\}}
\bar\zeta_0^{\ell}\zeta_0^{-\ell}\vartheta(\zeta)
e^{\langle\lambda,\zeta/\zeta_0\rangle-\overline{\langle\lambda,\zeta/\zeta_0\rangle}}\\
\times\det\left[\frac{\bar\zeta}{B^*(w^{(2)},\zeta)}\ 
\frac{\left(\bar{w}^{(1)}-\bar{w}^{(2)}\right)}{B(w^{(1)},\zeta)}\
Q(w^{(1)},\zeta)\right]
\wedge\det\left[\frac{\d z}{B^*(\bar\zeta,\bar{z})}\
\frac{z}{B^*(\bar\zeta,\bar{z})}\ Q(\bar\zeta,\bar{z})\right]
\wedge\d\left(\frac{\zeta_j}{\zeta_0}\right)
\wedge\frac{\d\bar\zeta}{P(\bar\zeta)}\Bigg|\\
\leq\gamma^{3/2}\Bigg|\int_{\pi^{-1}({\cal C})\cap\left\{|B(z,\zeta)|<\gamma\right\}}D(w^{(1)},w^{(2)},\zeta)\psi(z,w,\lambda,\zeta)\wedge
\d_{\zeta}\left(\frac{1}{B^*({\bar\zeta},{\bar z})}\right)\Bigg|.
\end{multline*}
Then, using the Stokes' theorem and the estimate
\begin{equation*}
\left|\d_{\zeta}\left[D(w^{(1)},w^{(2)},\zeta)\psi(z,w,\lambda,\zeta)\right]\right|
\leq C\cdot\frac{\delta}{|B(w^{(1)},\zeta)|^3}
\end{equation*}
for $\zeta\in \left\{|B(z,\zeta)|<\gamma\right\}$, we obtain
\begin{multline}\label{NwStokes2}
\gamma^{3/2}\left|\int_{\pi^{-1}({\cal C})\cap\left\{|B(z,\zeta)|<\gamma\right\}}
\d_{\zeta}\left[D(w^{(1)},w^{(2)},\zeta)\psi(z,w,\lambda,\zeta)\right]\times
\left(\frac{1}{B^*({\bar\zeta},{\bar z})}\right)\right|\\
\leq C(\lambda)\cdot\delta\gamma^{3/2}\int_0^{\gamma^2}\d t\int_0^{\gamma}
\frac{r\d r}{(\gamma+r^2)^3(t+r^2)}
\leq C(\lambda)\cdot\delta\gamma\int_0^{\gamma}
\frac{(\log{r})\d r}{(\gamma+r^2)^2}\\
\leq C(\lambda)\cdot\delta\gamma^{-1}\int_0^{\gamma}(\log{r})\d r
\leq C(\lambda)\cdot\delta\log{\gamma},
\end{multline}
which implies the estimate \eqref{NwDifferenceEstimate}.

\indent
From the estimates \eqref{NwStokesOne}$\div$\eqref{NwStokes2} we obtain the necessary estimates for the integrals over the domain $U_{\zeta}^{\gamma}(z)$.

\indent
For the first and fifth terms of the decomposition \eqref{NwDifferenceFormula} in the last domain in \eqref{NwDifferenceDomains} -- $U_{\zeta}^{\tau,\gamma}(z,w)=U_{\zeta}^{\tau}(z)
\setminus\left\{U^{\gamma}_{\zeta}(z)\cup U^{\gamma}_{\zeta}(w)\right\}$
-- we again use estimate \eqref{NDeltaEstimate} from the Lemma \ref{NDeltaDomain} and obtain the necessary estimate.

\indent
For the second term of the decomposition \eqref{NwDifferenceFormula} in
$U_{\zeta}^{\tau,\gamma}(z,w)$ using coordinates $B(w,\zeta)=it+r^2$ we obtain
\begin{multline*}
\Bigg|\gamma^{3/2}\int_{\Gamma_{\zeta}^{\epsilon}\cap U_{\zeta}^{\tau,\gamma}(z,w)}
\bar\zeta_0^{\ell}\zeta_0^{-\ell}\vartheta(\zeta)
e^{\langle\lambda,\zeta/\zeta_0\rangle-\overline{\langle\lambda,\zeta/\zeta_0\rangle}}\\
\times\det\left[\bar\zeta\cdot\left(\frac{1}{B^*(w^{(1)},\zeta)}
-\frac{1}{B^*(w^{(2)},\zeta)}\right)\ \frac{\bar{w}^{(1)}}{B(w^{(1)},\zeta)}\
Q(w^{(1)},\zeta)\right]\\
\wedge\det\left[\frac{\d z}{B^*(\bar\zeta,\bar{z})}\
\frac{z}{B^*(\bar\zeta,\bar{z})}\ Q(\bar\zeta,\bar{z})\right]
\wedge\d\left(\frac{\zeta_j}{\zeta_0}\right)
\wedge\frac{\d\bar\zeta}{P(\bar\zeta)}\Bigg|\\
\leq C(\lambda)\gamma^{3/2}\int_{\gamma^2}^A\d t\int_{\gamma}^B
\frac{\delta\cdot r^2\d r}{(\gamma+t+r^2)^{5/2}(\gamma+r^2)^2}
\leq C(\lambda)\gamma^{3/2}\delta\int_{\gamma}^B\frac{\d r}{(\gamma+r^2)^{5/2}}\\
\leq C(\lambda)\gamma^{-1/2}\delta\int_{\sqrt{\gamma}}^{B/\sqrt{\gamma}}
\frac{\d u}{(1+u^2)^{5/2}}\leq C(\lambda)\gamma^{-1/2}\delta,
\end{multline*}
which implies the estimate \eqref{NwDifferenceEstimate}.

\indent
Similarly we obtain the necessary estimate for the fourth term of \eqref{NwDifferenceFormula}.

\indent
For the third term of the right-hand side of \eqref{NwDifferenceFormula} using coordinates
$B(w,\zeta)=it+r^2$ we obtain
\begin{multline*}
\Bigg|\gamma^{3/2}\cdot(w^{(2)}_0)^{\ell}\cdot\lim_{\epsilon\to 0}
\int_{\Gamma_{\zeta}^{\epsilon}\cap U_{\zeta}^{\tau,\gamma}(z,w)}
\bar\zeta_0^{\ell}\zeta_0^{-\ell}\vartheta(\zeta)
e^{\langle\lambda,\zeta/\zeta_0\rangle-\overline{\langle\lambda,\zeta/\zeta_0\rangle}}\\
\times\det\left[\frac{\bar\zeta}{B^*(w^{(2)},\zeta)}\ 
\frac{\left(\bar{w}^{(1)}-\bar{w}^{(2)}\right)}{B(w^{(1)},\zeta)}\
Q(w^{(1)},\zeta)\right]
\wedge\det\left[\frac{\d z}{B^*(\bar\zeta,\bar{z})}\
\frac{z}{B^*(\bar\zeta,\bar{z})}\ Q(\bar\zeta,\bar{z})\right]
\wedge\d\left(\frac{\zeta_j}{\zeta_0}\right)
\wedge\frac{\d\bar\zeta}{P(\bar\zeta)}\Bigg|\\
\leq C(\lambda)\gamma^{3/2}\int_{\gamma^2}^A\d t\int_{\gamma}^B
\frac{\delta\cdot r\d r}{(\gamma+t+r^2)^2(\gamma+r^2)^2}
\leq C(\lambda)\gamma^{3/2}\delta\int_{\gamma}^B\frac{r\d r}{(\gamma+r^2)^3}
\leq C(\lambda)\gamma^{-1/2}\delta,
\end{multline*}
which implies the estimate \eqref{NwDifferenceEstimate}.

This completes the proof of Lemma~\ref{NwDifferenceLemma}
\end{proof}

\section{\bf Estimates for the domain integral in \eqref{hIntegralEquation}.}\label{sec:DomainEstimates}

\indent
The goal of this section is the proof of estimates of the derivatives of the domain integral
in \eqref{hIntegralEquation}, similar to the estimates in section \ref{sec:BoundaryEstimates},
and to be used in section \ref{sec:Solving}.

\indent
We start with the following lemma, in which we simplify the expression for
$\partial_z\bar\partial_w G(z,w,\lambda)$.

\begin{lemma}\label{ChangeTwoDeterminants}
For an arbitrary form $g(w,\lambda)\in C_{(1,0)}(V)$ and any fixed $\nu,\delta>0$
the following equality holds
\begin{multline}\label{NewTwoDeterminants}
\lim_{\tau\to 0}\int_{\Gamma^{\tau}_w\cap\{|B(z,w)|>\nu\}}
g(w,\lambda)\wedge\partial_z\bar\partial_w G(z,w,\lambda)\\
=|{\bf{C}}|^2\cdot\lim_{\tau\to 0}\int_{\Gamma^{\tau}_w\cap\{|B(z,w)|>\nu\}}g(w,\lambda)
\bar{z}_0^{-\ell}\cdot w_0^{\ell}\\
\bigwedge\sum_{j=1,2}\Bigg(\lim_{\epsilon\to 0}\int_{\Gamma^{\epsilon}_{\zeta}
\cap\left\{|B(z,\zeta)|>\delta,|B(w,\zeta)|>\delta\right\}}
\bar\zeta_0^{\ell}\zeta_0^{-\ell}\vartheta(\zeta)
e^{\langle\lambda,\zeta/\zeta_0\rangle-\overline{\langle\lambda,\zeta/\zeta_0\rangle}}\\
\wedge\det\left[\frac{\bar\zeta}{B^*(w,\zeta)}\ \bar\partial_w\left(\frac{\bar{w}}
{B(w,\zeta)}\right)\ Q(w,\zeta)\right]
\wedge\det\left[\partial_z\left(\frac{z}{B^*(\bar\zeta,\bar{z})}\right)\
\frac{\zeta}{B(\bar\zeta,\bar{z})}\ Q(\bar\zeta,\bar{z})\right]\\
\wedge\d\left(\frac{\zeta_j}{\zeta_0}\right)
\wedge\frac{\d\bar\zeta}{P(\bar\zeta)}\Bigg)
\wedge\frac{\omega^{(2,0)}_j(w)}{P(w)}\\
=|{\bf{C}}|^2\cdot\lim_{\tau\to 0}\int_{\Gamma^{\tau}_w\cap\{|B(z,w)|>\nu\}}g(w,\lambda)
\bar{z}_0^{-\ell}\cdot w_0^{\ell}\\
\bigwedge\sum_{j=1,2}\Bigg(\lim_{\epsilon\to 0}\int_{\Gamma^{\epsilon}_{\zeta}
\cap\left\{|B(z,\zeta)|>\delta,|B(w,\zeta)|>\delta\right\}}
\bar\zeta_0^{\ell}\zeta_0^{-\ell}\vartheta(\zeta)
e^{\langle\lambda,\zeta/\zeta_0\rangle-\overline{\langle\lambda,\zeta/\zeta_0\rangle}}\\
\wedge\det\left[\frac{\bar{w}}{B(w,\zeta)}\ \frac{\d\bar{w}}{B(w,\zeta)}\
Q(w,\zeta)\right]
\wedge\det\left[\frac{\d z}{B^*(\bar\zeta,\bar{z})}\ \frac{z}{B^*(\bar\zeta,\bar{z})}\
Q(\bar\zeta,\bar{z})\right]\\
\wedge\d\left(\frac{\zeta_j}{\zeta_0}\right)
\wedge\frac{\d\bar\zeta}{P(\bar\zeta)}\Bigg)
\wedge\frac{\omega^{(2,0)}_j(w)}{P(w)}.
\end{multline}
\end{lemma}
\begin{proof}
The first equality in \eqref{ChangeTwoDeterminants} follows from the equalities
$$\bar\partial_w B^*(w,\zeta)=\bar\partial_w\left(-1+\sum_{j=0}^2w_j\bar\zeta_j\right)=0,\hspace{0.1in}
\partial_z B(\bar\zeta,{\bar z})=\partial_z\left(1-\sum_{j=0}^2\zeta_j{\bar z}_j\right)=0.$$
The second equality states that the replacement of determinants in the form
$\partial_z\bar\partial_w G(z,w,\lambda)$ doesn't change the value of the integral in
\eqref{ChangeTwoDeterminants}.
The validity of the replacement of
\begin{equation*}
\det\left[\partial_z\left(\frac{z}{B^*(\bar\zeta,\bar{z})}\right)\
\frac{\zeta}{B(\bar\zeta,\bar{z})}\ \frac{Q(\bar\zeta,\bar{z})}{P(\bar\zeta)}\right]\ \text{by}\
\det\left[\frac{\d z}{B^*(\bar\zeta,\bar{z})}\ \frac{z}{B^*(\bar\zeta,\bar{z})}\
\frac{Q(\bar\zeta,\bar{z})}{P(\bar\zeta)}\right]
\end{equation*}
has been proved in Lemma~\ref{ChangeOneDeterminant}. To prove the validity of the other
replacement we use equality
\begin{equation*}
\det\left[\begin{tabular}{cccc}
1&0&1&$P(w)$\vspace{0.05in}\\
${\dis \frac{\bar\zeta}{B^*(w,\zeta)} }$
&${\dis \bar\partial_w\left(\frac{\bar{w}}{B(w,\zeta)}\right) }$
&${\dis \frac{\bar{w}}{B(w,\zeta)} }$
&${\dis Q(w,\zeta) }$
\end{tabular}\right]=0,
\end{equation*}
which is a corollary of the fact that the first row of this 4x4 matrix is a linear
combination of three lower rows with coefficients $w_j-\zeta_j$.\\
\indent
Then we obtain the equality
\begin{multline*}
\det\left[\frac{\bar\zeta}{B^*(w,\zeta)}\ \bar\partial_w\left(\frac{\bar{w}}{B(w,\zeta)}\right)\
Q(w,\zeta)\right]\\
=-\det\left[\bar\partial_w\left(\frac{\bar{w}}{B(w,\zeta)}\right)\
\frac{\bar{w}}{B(w,\zeta)}\ Q(w,\zeta)\right]
+P(w)\cdot\det\left[\frac{\bar\zeta}{B^*(w,\zeta)}\ \bar\partial_w\left(\frac{\bar{w}}{B(w,\zeta)}\right)\
\frac{\bar{w}}{B(w,\zeta)}\right]\\
=\det\left[\frac{\bar{w}}{B(w,\zeta)}\ \bar\partial_w\left(\frac{\bar{w}}{B(w,\zeta)}\right)\
Q(w,\zeta)\right]
+P(w)\cdot\det\left[\frac{\bar\zeta}{B^*(w,\zeta)}\ \bar\partial_w\left(\frac{\bar{w}}{B(w,\zeta)}\right)\
\frac{\bar{w}}{B(w,\zeta)}\right].\\
\end{multline*}
\indent
As in the last step of the proof of Lemma~\ref{ChangeOneDeterminant} we notice that because
of the cancellation of the factor $1/P(w)$ in
\begin{multline*}
\bar\zeta_0^{\ell}\zeta_0^{-\ell}\vartheta(\zeta)
e^{\langle\lambda,\zeta/\zeta_0\rangle-\overline{\langle\lambda,\zeta/\zeta_0\rangle}}
\wedge\det\left[\frac{\bar\zeta}{B^*(w,\zeta)}\ \bar\partial_w\left(\frac{\bar{w}}{B(w,\zeta)}\right)\
\frac{\bar{w}}{B(w,\zeta)}\right]\\
\wedge\det\left[\frac{\d z}{B^*(\bar\zeta,\bar{z})}\ \frac{z}{B^*(\bar\zeta,\bar{z})}\
\frac{Q(\bar\zeta,\bar{z})}{P(\bar\zeta)}\right]
\d\left(\frac{\zeta_j}{\zeta_0}\right)\wedge\d\bar\zeta\wedge\omega^{(2,0)}_j(w)
\end{multline*}
and the fact that
$\text{Volume}\left(\Gamma^{\tau}_w\cap\{|B(z,w)|>\nu\}\right)\to 0$ as $\tau\to 0$,
the contribution of the term containing
$$P(w)\cdot\det\left[\frac{\bar\zeta}{B^*(w,\zeta)}\ \bar\partial_w\left(\frac{\bar{w}}{B(w,\zeta)}\right)\
\frac{\bar{w}}{B(w,\zeta)}\right]$$
is zero.
\end{proof}

\indent
In what follows we will be using the formula
\begin{multline}\label{PartialzPartialwG}
\partial_z\bar\partial_w G(z,w,\lambda)=
|{\bf{C}}|^2\cdot\bar{z}_0^{-\ell}w_0^{\ell}
\cdot\sum_{j=1,2}\Bigg(\lim_{\delta\to 0}\lim_{\epsilon\to 0}\int_{\Gamma^{\epsilon}_{\zeta}
\setminus\left(U^{\delta}_{\zeta}(w)\cup U^{\delta}_{\zeta}(z)\right)}
\bar\zeta_0^{\ell}\zeta_0^{-\ell}\vartheta(\zeta)
e^{\langle\lambda,\zeta/\zeta_0\rangle-\overline{\langle\lambda,\zeta/\zeta_0\rangle}}\\
\wedge\det\left[\frac{\bar{w}}{B(w,\zeta)}\ \frac{\d\bar{w}}{B(w,\zeta)}\
Q(w,\zeta)\right]
\wedge\det\left[\frac{\d z}{B^*(\bar\zeta,\bar{z})}\ \frac{z}{B^*(\bar\zeta,\bar{z})}\
Q(\bar\zeta,\bar{z})\right]\\
\wedge\d\left(\frac{\zeta_j}{\zeta_0}\right)
\wedge\frac{\d\bar\zeta}{P(\bar\zeta)}\Bigg)
\wedge\frac{\omega^{(2,0)}_j(w)}{P(w)},
\end{multline}
and will be proving necessary estimates for the integrals in the right-hand side of
equality above, which we denote by
\begin{multline}\label{LForm}
L_j(z,w,\lambda)=w_0^{\ell}\cdot\lim_{\delta\to 0}\lim_{\epsilon\to 0}
\int_{\Gamma_{\zeta}^{\epsilon}\setminus\left(U^{\delta}_{\zeta}(w)\cup U^{\delta}_{\zeta}(z)\right)}
\bar\zeta_0^{\ell}\zeta_0^{-\ell}\vartheta(\zeta)
e^{\langle\lambda,\zeta/\zeta_0\rangle-\overline{\langle\lambda,\zeta/\zeta_0\rangle}}\\
\times\det\left[\frac{\bar{w}}{B(w,\zeta)}\ \frac{\d\bar{w}}{B(w,\zeta)}\ Q(w,\zeta)\right]
\wedge\det\left[\frac{\d z}{B^*(\bar\zeta,\bar{z})}\ \frac{z}{B^*(\bar\zeta,\bar{z})}\
Q(\bar\zeta,\bar{z})\right]
\d\left(\frac{\zeta_j}{\zeta_0}\right)\wedge\frac{\d\bar\zeta}{P(\bar\zeta)}.
\end{multline}

\indent
In the next two lemmas, using estimates from Lemmas~\ref{NLzTauDomain}
and \ref{NLwKappaDomain}, we establish necessary estimates for the kernels
$L_j(z,w,\lambda)$.

\begin{lemma}\label{LTauDomain}
There exist constants $C(\lambda)$, satisfying $\lim_{\lambda\to\infty}C(\lambda)=0$,
such that the following estimate
\begin{equation}\label{LzTauDomainEstimate}
\left|B(z,w)^2\cdot L^{\tau,\gamma}_j(z,w,\lambda)\right|\leq C(\lambda)
\end{equation}
holds for the integral
\begin{multline}\label{LTauIntegral}
L^{\tau,\gamma}_j(z,w,\lambda)=w_0^{\ell}\cdot\lim_{\epsilon\to 0}
\int_{\Gamma_{\zeta}^{\epsilon}\cap\left(U^{\tau}_{\zeta}(z)\setminus U^{\gamma}_{\zeta}(z,w)\right)}
\bar\zeta_0^{\ell}\zeta_0^{-\ell}\vartheta(\zeta)
e^{\langle\lambda,\zeta/\zeta_0\rangle-\overline{\langle\lambda,\zeta/\zeta_0\rangle}}\\
\times\det\left[\frac{\bar{w}}{B(w,\zeta)}\ \frac{\d\bar{w}}{B(w,\zeta)}\ Q(w,\zeta)\right]
\wedge\det\left[\frac{\d z}{B^*(\bar\zeta,\bar{z})}\ \frac{z}{B^*(\bar\zeta,\bar{z})}\
Q(\bar\zeta,\bar{z})\right]
\d\left(\frac{\zeta_j}{\zeta_0}\right)\wedge\frac{\d\bar\zeta}{P(\bar\zeta)}
\end{multline}
over the domain $U^{\tau}_{\zeta}(z)\setminus U^{\gamma}_{\zeta}(z,w)$
for an arbitrary $\tau>0$ and points $z,w$ such that $\gamma=|B(z,w)|<\tau$.
\end{lemma}
\begin{proof}
For the integral in \eqref{LTauIntegral}, denoting
$$\psi(z,w,\lambda,\zeta)=\bar\zeta_0^{\ell}\zeta_0^{-\ell}\vartheta(\zeta)
e^{\langle\lambda,\zeta/\zeta_0\rangle-\overline{\langle\lambda,\zeta/\zeta_0\rangle}}
S(w,\zeta)\det\left[\d z\ z\ Q(\bar\zeta,\bar{z})\right]$$
and defining the form
$$\Psi(z,w,\lambda,\zeta)=\psi(z,w,\lambda,\zeta)
\frac{B(z,w)^2}{B(w,\zeta)^2},$$
which is bounded on $U^{\tau}_{\zeta}(z)\setminus U^{\gamma}_{\zeta}(z,w)$ by the estimate
\eqref{GammaInequalities},
we obtain the following estimate
\begin{multline}\label{LzTauDomainEquality}
\Bigg|\lim_{\epsilon\to 0}
\int_{\Gamma^{\epsilon}_{\zeta}\cap
U^{	\tau}_{\zeta}(z)\setminus U^{\gamma}_{\zeta}(z,w)}
B(z,w)^2
\det\left[\frac{\bar{w}}{B(w,\zeta)}\ \frac{\d\bar{w}}{B(w,\zeta)}\ Q(w,\zeta)\right]\\
\wedge\det\left[\frac{\d z}{B^*(\bar\zeta,\bar{z})}\
\frac{z}{B^*(\bar\zeta,\bar{z})}\ Q(\bar\zeta,\bar{z})\right]
\wedge\d\left(\frac{\zeta_j}{\zeta_0}\right)
\wedge\frac{\d\bar\zeta}{P(\bar\zeta)}\Bigg|\\
\leq C\cdot\Bigg|\lim_{\epsilon\to 0}\int_{\Gamma^{\epsilon}_{\zeta}\cap
U^{\tau}_{\zeta}(z)\setminus U^{\gamma}_{\zeta}(z,w)}
B(z,w)^2\frac{\psi(z,w,\lambda,\zeta)}
{B(w,\zeta)^2B^*(\bar\zeta,\bar{z})^2}
\d\left(\frac{\zeta_j}{\zeta_0}\right)\wedge\frac{\d\bar\zeta}{P(\bar\zeta)}\Bigg|\\
\leq C\cdot\Bigg|\int_{\pi^{-1}({\cal C})
\cap U^{\tau}_{\zeta}(z)\setminus U^{\gamma}_{\zeta}(z,w)}
B(z,w)^2\psi(z,w,\lambda,\zeta)\frac{\d\left(\zeta_j/\zeta_0\right)
\wedge\left(\d{\bar P}(\zeta)\interior\d\bar\zeta\right)}
{B(w,\zeta)^2B^*({\bar\zeta},{\bar z})^2}\Bigg|\\
\leq C\cdot\Bigg|\int_{\pi^{-1}({\cal C}) \cap\left\{\gamma<|B(z,\zeta|<\tau\right\}}
\frac{\Psi(z,w,\lambda,\zeta)\d(\zeta_j/\zeta_0)\wedge\left(\d{\bar P}(\zeta)\interior\d\bar\zeta\right)}
{B^*({\bar\zeta},{\bar z})^2}\Bigg|.
\end{multline}
\indent
Then, using estimate \eqref{NLzTauDomainEstimate} from Lemma~\ref{NLzTauDomain} for $p=1$ we
obtain the estimate
\begin{equation*}
\left|B(z,w)^2\cdot L^{\tau,\gamma}_j(z,w,\lambda)\right|\leq C(\lambda),
\end{equation*}
which is the estimate \eqref{LzTauDomainEstimate}.
\end{proof}

\indent
The next lemma is a domain analogue of the Lemma~\ref{NTauSmallDomains}. In this lemma
we prove the complementary estimate to the estimate \eqref{LzTauDomainEstimate} in Lemma~\ref{LTauDomain}, considering the domain of integration
$\Gamma^{\epsilon}_{\zeta}\cap\{U^{\gamma}_{\zeta}(w)\cup U^{\gamma}_{\zeta}(z)\}$.

\begin{lemma}\label{LTauSmallDomains}
There exist constants $C(\lambda)$, satisfying $\lim_{\lambda\to\infty}C(\lambda)=0$,
such that the following estimate holds:
\begin{equation}\label{LGammaSmallDomainsEstimate}
\left|B(z,w)^2\cdot L^{\gamma}_j(z,w,\lambda)\right|\leq C(\lambda)
\end{equation}
for
\begin{multline}\label{LGamma}
L^{\gamma}_j(z,w,\lambda)=w_0^{\ell}\cdot\lim_{\epsilon\to 0}
\int_{\Gamma_{\zeta}^{\epsilon}\cap\left(U^{\gamma}_{\zeta}(z)\cup U^{\gamma}_{\zeta}(w)\right)}
\bar\zeta_0^{\ell}\zeta_0^{-\ell}\vartheta(\zeta)
e^{\langle\lambda,\zeta/\zeta_0\rangle-\overline{\langle\lambda,\zeta/\zeta_0\rangle}}\\
\times\det\left[\frac{\bar{w}}{B(w,\zeta)}\ \frac{\d\bar{w}}{B(w,\zeta)}\ Q(w,\zeta)\right]
\wedge\det\left[\frac{\d z}{B^*(\bar\zeta,\bar{z})}\ \frac{z}{B^*(\bar\zeta,\bar{z})}\
Q(\bar\zeta,\bar{z})\right]
\d\left(\frac{\zeta_j}{\zeta_0}\right)\wedge\frac{\d\bar\zeta}{P(\bar\zeta)}
\end{multline}
over the domains $U^{\gamma}_{\zeta}(z)$ and $U^{\gamma}_{\zeta}(w)$ for $z,w$ such that
$\gamma=|B(z,w)|$.
\end{lemma}
\begin{proof}
In the proof of Lemma~\ref{LTauSmallDomains} we will be proving  a more general estimate
\begin{equation}\label{LKappaEstimate}
\left|B(z,w)^2\cdot L^{\kappa}_j(z,w,\lambda)\right|
\leq C(\lambda)\cdot\sqrt{\frac{\kappa}{\gamma}}
\end{equation}
for
\begin{multline}\label{LKappa}
L^{\kappa}_j(z,w,\lambda)=w_0^{\ell}\cdot\lim_{\epsilon\to 0}
\int_{\Gamma_{\zeta}^{\epsilon}\cap\left(U^{\kappa}_{\zeta}(z)\cup U^{\kappa}_{\zeta}(w)\right)}
\bar\zeta_0^{\ell}\zeta_0^{-\ell}\vartheta(\zeta)
e^{\langle\lambda,\zeta/\zeta_0\rangle-\overline{\langle\lambda,\zeta/\zeta_0\rangle}}\\
\times\det\left[\frac{\bar{w}}{B(w,\zeta)}\ \frac{\d\bar{w}}{B(w,\zeta)}\ Q(w,\zeta)\right]
\wedge\det\left[\frac{\d z}{B^*(\bar\zeta,\bar{z})}\ \frac{z}{B^*(\bar\zeta,\bar{z})}\
Q(\bar\zeta,\bar{z})\right]
\d\left(\frac{\zeta_j}{\zeta_0}\right)\wedge\frac{\d\bar\zeta}{P(\bar\zeta)}
\end{multline}
with $\kappa\leq\gamma$.
\indent
For the integral in \eqref{LKappa} over $U^{\kappa}_{\zeta}(z)$ we denote
$$\psi(z,w,\lambda,\zeta)=\bar\zeta_0^{\ell}\zeta_0^{-\ell}\vartheta(\zeta)
e^{\langle\lambda,\zeta/\zeta_0\rangle-\overline{\langle\lambda,\zeta/\zeta_0\rangle}}
S(w,\zeta)\det\left[\d z\ z\ Q(\bar\zeta,\bar{z})\right]$$
and using estimates \eqref{GammaInequalities} define the bounded form on
$U^{\kappa}_{\zeta}(z)$
$$\Psi(z,w,\lambda,\zeta)=\psi(z,w,\lambda,\zeta)
\frac{B(z,w)^2}{B(w,\zeta)^2}.$$
Then we obtain the following estimate
\begin{multline}\label{LzDomainEquality}
\Bigg|\lim_{\epsilon\to 0}
\int_{\Gamma^{\epsilon}_{\zeta}\cap U^{\kappa}_{\zeta}(z)}
B(z,w)^2\det\left[\frac{\bar{w}}{B(w,\zeta)}\ \frac{\d\bar{w}}{B(w,\zeta)}\ Q(w,\zeta)\right]\\
\wedge\det\left[\frac{\d z}{B^*(\bar\zeta,\bar{z})}\
\frac{z}{B^*(\bar\zeta,\bar{z})}\ Q(\bar\zeta,\bar{z})\right]
\wedge\d\left(\frac{\zeta_j}{\zeta_0}\right)
\wedge\frac{\d\bar\zeta}{P(\bar\zeta)}\Bigg|\\
\leq C\cdot\Bigg|\lim_{\nu\to 0}\lim_{\epsilon\to 0}
\int_{\Gamma^{\epsilon}_{\zeta}\cap\left\{U^{\kappa}_{\zeta}(z),|B(z,\zeta|>\nu\right\}}
B(z,w)^2\frac{\psi(z,w,\lambda,\zeta)}
{B(w,\zeta)^2B^*(\bar\zeta,\bar{z})^2}
\d\left(\frac{\zeta_j}{\zeta_0}\right)\wedge\frac{\d\bar\zeta}{P(\bar\zeta)}\Bigg|\\
\leq C\cdot\Bigg|\lim_{\nu\to 0}\int_{\pi^{-1}({\cal C}) \cap\left\{\nu<|B(z,\zeta|<\kappa\right\}}
\frac{\Psi(z,w,\lambda,\zeta)\d(\zeta_j/\zeta_0)\wedge\left(\d{\bar P}(\zeta)\interior\d\bar\zeta\right)}
{B^*({\bar\zeta},{\bar z})^2}\Bigg|.
\end{multline}

\indent
Using estimate \eqref{NLzKappaEstimate} from Lemma~\ref{NLzTauDomain} for $p=1$ we
obtain the estimate
\begin{equation}\label{LzKappaEstimate}
\Bigg|\lim_{\nu\to 0}\int_{\pi^{-1}({\cal C}) \cap\left\{\nu<|B(z,\zeta|<\kappa\right\}}
\frac{\Psi(z,w,\lambda,\zeta)\d(\zeta_j/\zeta_0)\wedge\left(\d{\bar P}(\zeta)\interior\d\bar\zeta\right)}
{B^*({\bar\zeta},{\bar z})^2}\Bigg|
\leq C(\lambda)\cdot\sqrt{\frac{\kappa}{\gamma}},
\end{equation}
which implies the estimate \eqref{LKappaEstimate} for the integral over
$U^{\kappa}_{\zeta}(z)$.\\
\indent
Estimate \eqref{LGammaSmallDomainsEstimate} for the integral over $U^{\gamma}_{\zeta}(z)$ follows from
the application of estimate \eqref{LzKappaEstimate} to the case $\kappa=\gamma$.

\indent
For the integral in \eqref{LKappa} over $U^{\kappa}_{\zeta}(w)$, denoting
$$\psi(z,w,\lambda,\zeta)=\bar\zeta_0^{\ell}\zeta_0^{-\ell}\vartheta(\zeta)
e^{\langle\lambda,\zeta/\zeta_0\rangle-\overline{\langle\lambda,\zeta/\zeta_0\rangle}}
S(w,\zeta)\det\left[\d z\ z\ Q(\bar\zeta,\bar{z})\right]$$
and defining the bounded form on $U^{\kappa}_{\zeta}(w)$
$$\Psi(z,w,\lambda,\zeta)=\psi(z,w,\lambda,\zeta)
\frac{B(z,w)^2}{B^*({\bar\zeta},{\bar z})^2}$$
we obtain the following estimate
\begin{multline}\label{LwDomainEquality}
\Bigg|\lim_{\epsilon\to 0}
\int_{\Gamma^{\epsilon}_{\zeta}\cap U^{\kappa}_{\zeta}(w)}
B(z,w)^2\det\left[\frac{\bar{w}}{B(w,\zeta)}\ \frac{\d\bar{w}}{B(w,\zeta)}\ Q(w,\zeta)\right]\\
\wedge\det\left[\frac{\d z}{B^*(\bar\zeta,\bar{z})}\
\frac{z}{B^*(\bar\zeta,\bar{z})}\ Q(\bar\zeta,\bar{z})\right]
\wedge\d\left(\frac{\zeta_j}{\zeta_0}\right)
\wedge\frac{\d\bar\zeta}{P(\bar\zeta)}\Bigg|\\
\leq C\cdot\Bigg|\lim_{\epsilon\to 0}
\int_{\Gamma_{\zeta}^{\epsilon}\cap U^{\kappa}_{\zeta}(w)}\psi(z,w,\lambda,\zeta)
\frac{B(z,w)^2}{B(w,\zeta)^2)B^*(\bar\zeta,\bar{z})^2}
\wedge\d\left(\frac{\zeta_j}{\zeta_0}\right)
\wedge\frac{\d\bar\zeta}{P(\bar\zeta)}\Bigg|\\
\leq C\cdot\Bigg|\lim_{\eta\to 0}\int_{\pi^{-1}({\cal C}) \cap\left\{\eta<|B(w,\zeta|<\kappa\right\}}
\frac{\Psi(z,w,\lambda,\zeta)\d(\zeta_j/\zeta_0)\wedge\left(\d{\bar P}(\zeta)\interior\d\bar\zeta\right)}
{B(w,\zeta)^2}\Bigg|.
\end{multline}

\indent
Then, using estimate \eqref{NLwKappaDomainEstimate} from Lemma~\ref{NLwKappaDomain} for $p=1$ we obtain the estimate
\begin{equation*}
\Bigg|\lim_{\eta\to 0}\int_{\pi^{-1}({\cal C}) \cap\left\{\eta<|B(w,\zeta|<\kappa\right\}}
\frac{\Psi(z,w,\lambda,\zeta)\d(\zeta_j/\zeta_0)\wedge\left(\d{\bar P}(\zeta)\interior\d\bar\zeta\right)}
{B(w,\zeta)^2}\Bigg|
\leq C(\lambda)\cdot\sqrt{\frac{\kappa}{\gamma}},
\end{equation*}
which implies the estimate \eqref{LKappaEstimate} for the integral over $U^{\kappa}_{\zeta}(w)$.\\
\indent
Estimate \eqref{LGammaSmallDomainsEstimate} for the integral over $U^{\gamma}_{\zeta}(w)$ follows from the application of estimate \eqref{LKappaEstimate} to the case $\kappa=\gamma$.
\end{proof}

\indent
In the next two lemmas we prove estimates similar to those in Lemmas~\ref{NDeltaDomain}
and \ref{NDifferenceLemma}, but for the forms $L_j(z,w,\lambda)$ from \eqref{LForm} to be used in estimates of the operator in \eqref{POperator}.

\begin{lemma}\label{LDeltaDomain}
There exist constants $C(\lambda)$, satisfying $\lim_{\lambda\to\infty}C(\lambda)=0$,
such that the following estimate holds for an arbitrary $z,w \in V$
\begin{equation}\label{LDeltaEstimate}
\left|L_j(z,w,\lambda)\right|\leq \frac{C(\lambda)}{|B(z,w)|^2},
\end{equation}
where
\begin{multline}\label{LDeltaIntegral}
L_j(z,w,\lambda)=w_0^{\ell}\cdot\lim_{\delta\to 0}\lim_{\epsilon\to 0}
\int_{\Gamma_{\zeta}^{\epsilon}\setminus\left(U^{\delta}_{\zeta}(w)\cup U^{\delta}_{\zeta}(z)\right)}
\bar\zeta_0^{\ell}\zeta_0^{-\ell}\vartheta(\zeta)
e^{\langle\lambda,\zeta/\zeta_0\rangle-\overline{\langle\lambda,\zeta/\zeta_0\rangle}}\\
\times\det\left[\frac{\d\bar{w}}{B(w,\zeta)}\ \frac{\bar{w}}{B(w,\zeta)}\ Q(w,\zeta)\right]
\wedge\det\left[\frac{\d z}{B^*(\bar\zeta,\bar{z})}\ \frac{z}{B^*(\bar\zeta,\bar{z})}\
Q(\bar\zeta,\bar{z})\right]
\d\left(\frac{\zeta_j}{\zeta_0}\right)\wedge\frac{\d\bar\zeta}{P(\bar\zeta)}.
\end{multline}
\end{lemma}
\begin{proof}
\indent
From the estimate \eqref{LzTauDomainEstimate} in Lemma~\ref{LTauDomain} we obtain the
existence of  a constant $C(\lambda)$ satisfying the condition of the lemma and
such that the following estimate holds for an arbitrary $\tau>0$
\begin{equation}\label{LDeltaEstimateFirst}
\left|B(z,w)^2\cdot L^{\tau,\gamma}_j(z,w,\lambda)\right|\leq C(\lambda)
\end{equation}
for the integral in \eqref{LDeltaIntegral} over
$U^{\tau}_{\zeta}(z)\setminus U^{\gamma}_{\zeta}(z,w)$
for $z,w$ such that $\gamma=|B(z,w)|<\tau$.\\
\indent
From the estimate \eqref{LGammaSmallDomainsEstimate} in Lemma~\ref{LTauSmallDomains} we obtain the existence of a constant $C(\lambda)$ satisfying the condition of the lemma
and such that for $\tau>0$ the following estimate holds
\begin{equation}\label{LDeltaEstimateSecond}
\left|B(z,w)^2\cdot L^{\gamma}_j(z,w,\lambda)\right|\leq C(\lambda)
\end{equation}
for the integral in \eqref{LDeltaIntegral} over the domains
$U^{\gamma}_{\zeta}(z^{(i)})$ and $U^{\gamma}_{\zeta}(w)$
for $z,w$ such that $\gamma=|B(z,w)|<\tau$.\\
\indent
Then, combining estimates \eqref{LDeltaEstimateFirst} and \eqref{LDeltaEstimateSecond} we obtain estimate \eqref{LDeltaEstimate}.
\end{proof}

\indent
In the next lemma we estimate the difference between integrals
$L_j(z^{(1)},w,\lambda)$ and $L_j(z^{(2)},w,\lambda)$.

\begin{lemma}\label{LDifferenceLemma}
There exist constants $C(\lambda)$, satisfying $\lim_{\lambda\to\infty}C(\lambda)=0$,
such that the following estimate
\begin{equation}\label{LDifferenceEstimate}
\left|L_j(z^{(1)},w,\lambda)-L_j(z^{(2)},w,\lambda)\right|
\leq C(\lambda)\cdot\frac{\delta}{|B(z^{(1)},w)|^{5/2}}
\end{equation}
holds for an arbitrary fixed $\delta$,
any two points $z^{(1)},z^{(2)}\in V$, such that $|B(z^{(1)},z^{(2)})|=\delta^2$,
$w$ satisfying $\gamma=|B(z^{(1)},w)|>9\delta^2$, for the integral from \eqref{LForm}
\begin{multline*}
L_j(z,w,\lambda)=w_0^{\ell}\cdot\lim_{\epsilon\to 0}
\int_{\Gamma^{\epsilon}_{\zeta}}\bar\zeta_0^{\ell}\zeta_0^{-\ell}\vartheta(\zeta)
e^{\langle\lambda,\zeta/\zeta_0\rangle-\overline{\langle\lambda,\zeta/\zeta_0\rangle}}\\
\times\det\left[\frac{\d{\bar w}}{B(w,\zeta)}\ \frac{\bar{w}}{B(w,\zeta)}\ Q(w,\zeta)\right]
\det\left[\frac{\d z}{B^*(\bar\zeta,\bar{z})}\ \frac{z}{B^*(\bar\zeta,\bar{z})}\
Q(\bar\zeta,\bar{z})\right]
\d\left(\frac{\zeta_j}{\zeta_0}\right)\wedge\frac{\d\bar\zeta}{P(\bar\zeta)}.
\end{multline*}
\end{lemma}
\begin{proof}
Similarly to the proof of Lemma~\ref{NDifferenceLemma} we denote
$\nu=9\delta^2$, consider the domains from \eqref{DifferenceDomains}
\begin{equation*}
U^{\nu}_{\zeta}(z^{(i)})\ \text{for}\ i=1,2;\
U^{\gamma,\nu}_{\zeta}(z)=U^{\gamma}_{\zeta}(z^{(1)}, z^{(2)})
\setminus\left\{\cup_{i=1,2}U^{\nu}_{\zeta}(z^{(i)})\right\};\
U_{\zeta}^{\gamma}(w);\ U_{\zeta}^{\tau,\gamma}(z,w);
\end{equation*}
and estimate the difference of integrals in $L_j(z^{(i)},w,\lambda)$ in those domains.

\indent
For the first two domains: $U^{\nu}_{\zeta}(z^{(i)})$, using estimate \eqref{LKappaEstimate}
from Lemma~\ref{LTauSmallDomains} we obtain the estimate
\begin{multline}\label{LFirstDomains}
\left|B(z^{(i)},w)^2\cdot L^{\nu}_j(z^{(i)},w,\lambda)\right|\\
=\Bigg|B(z^{(i)},w)^2\cdot w_0^{\ell}\cdot\lim_{\epsilon\to 0}
\int_{\Gamma_{\zeta}^{\epsilon} \cap\left\{|B(z^{(i)},\zeta)|<\nu\right\}}
\bar\zeta_0^{\ell}\zeta_0^{-\ell}\vartheta(\zeta)
e^{\langle\lambda,\zeta/\zeta_0\rangle-\overline{\langle\lambda,\zeta/\zeta_0\rangle}}\\
\times\det\left[\frac{\d{\bar w}}{B(w,\zeta)}\ \frac{{\bar w}}{B(w,\zeta)}\ Q(w,\zeta)\right]
\det\left[\frac{\d z}{B^*(\bar\zeta,\bar{z}^{(i)})}\
\frac{z^{(i)}}{B^*(\bar\zeta,\bar{z}^{(i)})}\ Q(\bar\zeta,\bar{z}^{(i)})\right]
\wedge\d\left(\frac{\zeta_j}{\zeta_0}\right)
\wedge\frac{\d\bar\zeta}{P(\bar\zeta)}\Bigg|\\
\leq C(\lambda)\cdot\sqrt{\frac{\nu}{\gamma}}
=C(\lambda)\cdot\frac{\delta}{\sqrt{|B(z^{(i)},w)|}},
\end{multline}
which is similar to estimate \eqref{LzKappaEstimate} and implies estimate
\eqref{LDifferenceEstimate} for each term of the integral over $U^{\nu}_{\zeta}(z^{(i)})$
in the left-hand side of \eqref{LDifferenceEstimate}.

\indent
To estimate the difference of integrals in $L_j(z^{(i)},w,\lambda)$ over
$U^{\gamma,\nu}_{\zeta}(z)$ we will use the notation \eqref{psiNotation} from Lemma~\ref{NDifferenceLemma}
\begin{equation*}
\psi(z,w,\lambda,\zeta)=\bar\zeta_0^{\ell}\zeta_0^{-\ell}\vartheta(\zeta)
e^{\langle\lambda,\zeta/\zeta_0\rangle-\overline{\langle\lambda,\zeta/\zeta_0\rangle}}
S(w,\zeta)\det\left[\d z\ z\ Q(\bar\zeta,\bar{z})\right]
\wedge\left((\d{\bar P}(\zeta)\wedge\d_{\zeta}B^*(\bar\zeta,{\bar z}))
\interior\d\bar\zeta\right).
\end{equation*}
Then, for the integrals over 
$U^{\gamma,\nu}_{\zeta}(z)=U^{\gamma}_{\zeta}(z^{(1)}, z^{(2)})
\setminus\left\{\cup_{i=1,2}U^{\nu}_{\zeta}(z^{(i)})\right\}$ we use the formulas from
Lemma~\ref{NDifferenceLemma} to obtain the following analogue of equality \eqref{NSecondDifference}
\begin{multline}\label{LSecondDifference}
\frac{B(z^{(1)},w)^2}{B(w,\zeta)^2}
\d(\zeta_j/\zeta_0)\wedge\frac{\d{\bar P}(\zeta)}{{\bar P}(\zeta)}
\wedge\left(\frac{\psi(z^{(1)},w,\lambda,\zeta)\wedge\d_{\zeta}B^*(\bar\zeta,{\bar z}^{(1)})}
{B^*({\bar\zeta},{\bar z}^{(1)})^2}
-\frac{\psi(z^{(2)},w,\lambda,\zeta)\wedge\d_{\zeta}B^*(\bar\zeta,{\bar z}^{(2)})}
{B^*({\bar\zeta},{\bar z}^{(2)})^2}\right)\\
=\frac{B(z^{(1)},w)^2}{B(w,\zeta)^2}
\d(\zeta_j/\zeta_0)\wedge\frac{\d{\bar P}(\zeta)}{{\bar P}(\zeta)}\\
\wedge\left(\psi(z^{(1)},w,\lambda,\zeta)\wedge
\d_{\zeta}\left(\frac{1}{B^*({\bar\zeta},{\bar z}^{(1)})}\right)
-\psi(z^{(2)},w,\lambda,\zeta)\wedge
\d_{\zeta}\left(\frac{1}{B^*({\bar\zeta},{\bar z}^{(2)})}\right)\right).
\end{multline}

\indent
As in \eqref{NDifferenceStokes} we apply the Stokes' theorem and obtain
\begin{multline}\label{LDifferenceStokes}
\lim_{\epsilon\to 0}
\int_{\Gamma_{\zeta}^{\epsilon} \cap\left\{\nu<|B(z^{(1)},\zeta)|<\gamma\right\}}
\frac{B(z^{(1)},w)^2}{B(w,\zeta)^2}
\d(\zeta_j/\zeta_0)\wedge\frac{\d{\bar P}(\zeta)}{{\bar P}(\zeta)}\\
\wedge\left(\psi(z^{(1)},w,\lambda,\zeta)\wedge
\d_{\zeta}\left(\frac{1}{B^*({\bar\zeta},{\bar z}^{(1)})}\right)
-\psi(z^{(2)},w,\lambda,\zeta)\wedge
\d_{\zeta}\left(\frac{1}{B^*({\bar\zeta},{\bar z}^{(2)})}\right)\right)\\
=B(z^{(1)},w)^2\int_{\pi^{-1}({\cal C}) \cap\left\{|B(z^{(1)},\zeta|=\nu\right\}}
\left(\frac{\Psi(z^{(1)},w,\lambda,\zeta)}{B^*({\bar\zeta},{\bar z}^{(1)})}
-\frac{\Psi(z^{(2)},w,\lambda,\zeta)}{B^*({\bar\zeta},{\bar z}^{(2)})}\right)\\
-B(z^{(1)},w)^2\int_{\pi^{-1}({\cal C}) \cap\left\{|B(z^{(1)},\zeta|=\gamma\right\}}
\left(\frac{\Psi(z^{(1)},w,\lambda,\zeta)}{B^*({\bar\zeta},{\bar z}^{(1)})}
-\frac{\Psi(z^{(2)},w,\lambda,\zeta)}{B^*({\bar\zeta},{\bar z}^{(2)})}\right)\\
-B(z^{(1)},w)^2\int_{\pi^{-1}({\cal C}) \cap\left\{\nu<|B(z^{(1)},\zeta|<\gamma\right\}}
\left(\frac{\d_{\zeta}\Psi(z^{(1)},w,\lambda,\zeta)}{B^*({\bar\zeta},{\bar z}^{(1)})}
-\frac{\d_{\zeta}\Psi(z^{(2)},w,\lambda,\zeta)}{B^*({\bar\zeta},{\bar z}^{(2)})}\right),
\end{multline}
where
$$\Psi(z^{(i)},w,\lambda,\zeta)=\frac{\d(\zeta_j/\zeta_0)\wedge\psi(z^{(i)},w,\lambda,\zeta)}
{B(w,\zeta)^2}.$$
\indent
For the first integral in the right-hand side of \eqref{LDifferenceStokes} we use estimate \eqref{DeltaAreaEstimate}
\begin{equation*}
A(\nu)=C\cdot\nu^{3/2}
\end{equation*}
of the area of integration $\left\{\pi^{-1}({\cal C})\cap|B(z,\zeta)|=\nu\right\}$ in this integral
and the boundedness of the forms $B(z^{(1)},w)^2\Psi(z^{(i)},w,\lambda,\zeta)$ on
$\left\{\nu<|B(z,\zeta|<\gamma\right\}$, which follows from the estimates
\eqref{GammaInequalities}. Then, we obtain the following estimate
\begin{equation}\label{LDifferenceStokesFirst}
\Bigg|B(z^{(1)},w)^2\int_{\pi^{-1}({\cal C}) \cap\left\{|B(z^{(1)},\zeta|=\nu\right\}}
\frac{\Psi(z^{(i)},w,\lambda,\zeta)}{B^*({\bar\zeta},{\bar z}^{(i)})}\Bigg|
\leq C(\lambda)\cdot\frac{\nu^{3/2}}{\nu}\leq C(\lambda)\cdot\delta,\\
\end{equation}
where the property $\lim_{\lambda\to\infty}C(\lambda)=0$ follows from the Riemann-Lebesgue Lemma (see \cite{23}).\\
\indent
For the second integral in the right-hand side of \eqref{LDifferenceStokes} we use the equality
\begin{multline}\label{LDifferenceStokesSecondRepresentation}
B(z^{(1)},w)^2\int_{\pi^{-1}({\cal C}) \cap\left\{|B(z^{(1)},\zeta|=\gamma\right\}}
\left(\frac{\Psi(z^{(1)},w,\lambda,\zeta)}{B^*({\bar\zeta},{\bar z}^{(1)})}
-\frac{\Psi(z^{(2)},w,\lambda,\zeta)}{B^*({\bar\zeta},{\bar z}^{(2)})}\right)\\
=B(z^{(1)},w)^2\int_{\pi^{-1}({\cal C}) \cap\left\{|B(z^{(1)},\zeta|=\gamma\right\}}
\Psi(z^{(1)},w,\lambda,\zeta)\left(\frac{B^*({\bar\zeta},{\bar z}^{(2)})
-B^*({\bar\zeta},{\bar z}^{(1)})}
{B^*({\bar\zeta},{\bar z}^{(1)})B^*({\bar\zeta},{\bar z}^{(2)})}\right)\\
+B(z^{(1)},w)^2\int_{\pi^{-1}({\cal C}) \cap\left\{|B(z^{(1)},\zeta|=\gamma\right\}}
\left(\frac{\Psi(z^{(1)},w,\lambda,\zeta)-\Psi(z^{(2)},w,\lambda,\zeta)}
{B^*({\bar\zeta},{\bar z}^{(2)})}\right).
\end{multline}
For the first term of the right-hand side of \eqref{LDifferenceStokesSecondRepresentation} using
estimate \eqref{BDifferenceInequality}, the boundedness of the forms
$B(z^{(1)},w)^2\Psi(z^{(i)},w,\lambda,\zeta)$ on $\left\{\nu\leq|B(z,\zeta|\leq\gamma\right\}$
and using coordinates $B(z,\zeta)= it+r^2$ and inequality $|B(w,\zeta)|\geq C\cdot(\gamma+r^2)$
we obtain the estimate
\begin{multline}\label{LDifferenceStokesZSecondFirst}
\Bigg|B(z^{(1)},w)^2\int_{\pi^{-1}({\cal C}) \cap\left\{|B(z^{(1)},\zeta|=\gamma\right\}}
\Psi(z^{(1)},w,\lambda,\zeta)\left(\frac{B^*({\bar\zeta},{\bar z}^{(2)})
-B^*({\bar\zeta},{\bar z}^{(1)})}
{B^*({\bar\zeta},{\bar z}^{(1)})B^*({\bar\zeta},{\bar z}^{(2)})}\right)\Bigg|\\
\leq C(\lambda)\cdot\delta\gamma^2\int_{\nu}^{\gamma}\d t\int_{\sqrt{\nu}}^{\sqrt{\gamma}}
\frac{\d r}{(\gamma+t+r^2)^{3/2}(\gamma+r^2)^2}
\leq C(\lambda)\cdot\delta.
\end{multline}
For the second term of the right-hand side of \eqref{LDifferenceStokesSecondRepresentation}
using estimate
$$\left|\psi(z^{(1)},w,\lambda,\zeta)-\psi(z^{(2)},w,\lambda,\zeta)\right|\leq C\cdot\delta$$
we obtain
\begin{multline}\label{LDifferenceStokesZSecondSecond}
\Bigg|B(z^{(1)},w)^2\int_{\pi^{-1}({\cal C}) \cap\left\{|B(z^{(1)},\zeta|=\gamma\right\}}
\left(\frac{\Psi(z^{(1)},w,\lambda,\zeta)-\Psi(z^{(2)},w,\lambda,\zeta)}
{B^*({\bar\zeta},{\bar z}^{(2)})}\right)\Bigg|\\
\leq C(\lambda)\cdot\delta\gamma^2\int_{\nu}^{\gamma}\d t\int_{\sqrt{\nu}}^{\sqrt{\gamma}}
\frac{\d r}{(\gamma+t+r^2)(\gamma+r^2)^2}
\leq C(\lambda)\cdot\delta.
\end{multline}
\indent
For the third integral in the right-hand side of \eqref{LDifferenceStokes} we have the equality
\begin{multline}\label{LDifferenceStokesThirdRepresentation}
B(z^{(1)},w)^2\int_{\pi^{-1}({\cal C}) \cap\left\{\nu<|B(z^{(1)},\zeta|<\gamma\right\}}
\left(\frac{\d_{\zeta}\Psi(z^{(1)},w,\lambda,\zeta)}{B^*({\bar\zeta},{\bar z}^{(1)})}
-\frac{\d_{\zeta}\Psi(z^{(2)},w,\lambda,\zeta)}{B^*({\bar\zeta},{\bar z}^{(2)})}\right)\\
=B(z^{(1)},w)^2\int_{\pi^{-1}({\cal C}) \cap\left\{\nu<|B(z^{(1)},\zeta|<\gamma\right\}}
\d_{\zeta}\Psi(z^{(1)},w,\lambda,\zeta)\left(\frac{B^*({\bar\zeta},{\bar z}^{(2)})
-B^*({\bar\zeta},{\bar z}^{(1)})}
{B^*({\bar\zeta},{\bar z}^{(1)})B^*({\bar\zeta},{\bar z}^{(2)})}\right)\\
+B(z^{(1)},w)^2\int_{\pi^{-1}({\cal C}) \cap\left\{\nu<|B(z^{(1)},\zeta|<\gamma\right\}}
\left(\frac{\d_{\zeta}\Psi(z^{(1)},w,\lambda,\zeta)-\d_{\zeta}\Psi(z^{(2)},w,\lambda,\zeta)}
{B^*({\bar\zeta},{\bar z}^{(2)})}\right).
\end{multline}
Then, for the first term of the right-hand side of \eqref{LDifferenceStokesThirdRepresentation},
using as in Lemma~\ref{NLzTauDomain} the following analog of the estimate \eqref{dPhi}:
\begin{equation}\label{LAnalogdPhi}
\d_{\zeta}\Psi(z^{(i)},w,\lambda,\zeta)
\sim\d_{\zeta}\left(\frac{\psi(z^{(i)},w,\lambda,\zeta)}{B(w,\zeta)^2}\right)
=\frac{\d_{\zeta}\psi(z^{(i)},w,\lambda,\zeta)}{B(w,\zeta)^2}
+\psi(z^{(i)},w,\lambda,\zeta)\wedge\frac{\d_{\zeta}B(w,\zeta)}{B(w,\zeta)^3}
\end{equation}
coordinates $|\zeta-z|=r$ and $B({\bar\zeta},{\bar z})=it+r^2$, and inequality
$|\zeta-w|<C\sqrt{\gamma}$ for
$\zeta\in \{|B(z,\zeta|<\gamma\}$ from \eqref{GammaInequalities}, and estimate
\eqref{BDifferenceInequality} we obtain
\begin{multline}\label{LDifferenceStokesThirdFirst}
\Bigg|B(z^{(1)},w)^2\int_{\pi^{-1}({\cal C}) \cap\left\{\nu<|B(z^{(1)},\zeta|<\gamma\right\}}
\d_{\zeta}\Psi(z^{(1)},w,\lambda,\zeta)\left(\frac{B^*({\bar\zeta},{\bar z}^{(2)})
-B^*({\bar\zeta},{\bar z}^{(1)})}
{B^*({\bar\zeta},{\bar z}^{(1)})B^*({\bar\zeta},{\bar z}^{(2)})}\right)\Bigg|\\
\leq C(\lambda)\cdot\delta\gamma^2\int_{\nu}^{\gamma}\d t
\int_{\sqrt{\nu}}^{\sqrt{\gamma}}
\frac{r\d r}{(\nu+t+r^2)^{3/2}(\gamma+r^2)^3}
\leq C(\lambda)\cdot\delta\gamma^2\int_{\sqrt{\nu}}^{\sqrt{\gamma}}
\frac{\d r}{(\gamma+r^2)^3}\\
\leq C(\lambda)\cdot \delta\gamma^{-1/2}.
\end{multline}
\indent
For the second term of the right-hand side of \eqref{LDifferenceStokesThirdRepresentation}
also using estimate \eqref{LAnalogdPhi} we obtain
\begin{multline}\label{LDifferenceStokesThirdSecond}
\Bigg|B(z^{(1)},w)^2\int_{\pi^{-1}({\cal C}) \cap\left\{\nu<|B(z^{(1)},\zeta|<\gamma\right\}}
\left(\frac{\d_{\zeta}\Psi(z^{(1)},w,\lambda,\zeta)-\d_{\zeta}\Psi(z^{(2)},w,\lambda,\zeta)}
{B^*({\bar\zeta},{\bar z}^{(2)})}\right)\Bigg|\\
\leq C(\lambda)\cdot\delta\gamma^2\int_{\nu}^{\gamma}\d t\int_{\sqrt{\nu}}^{\sqrt{\gamma}}
\frac{r\d r}{(\nu+t+r^2)(\gamma+r^2)^3}
\leq  C(\lambda)\cdot\delta\gamma^2\int_{\sqrt{\nu}}^{\sqrt{\gamma}}
\frac{r(\log{r})\d r}{(\gamma+r^2)^3}\\
\leq C(\lambda)\cdot \delta\gamma^{-1/2}.
\end{multline}
\indent
Combining the estimates \eqref{LDifferenceStokesFirst}, \eqref{LDifferenceStokesZSecondFirst},
\eqref{LDifferenceStokesZSecondSecond}, \eqref{LDifferenceStokesThirdFirst}, and \eqref{LDifferenceStokesThirdSecond} we obtain the estimate \eqref{LDifferenceEstimate}
for the integrals in $L_j(z^{(i)},w,\lambda)$ over
$U^{\gamma,\nu}_{\zeta}(z)=U^{\gamma}_{\zeta}(z^{(1)}, z^{(2)})
\setminus\left\{\cup_{i=1,2}U^{\nu}_{\zeta}(z^{(i)})\right\}$.

\indent
To estimate the difference of integrals in $L_j(z^{(i)},w,\lambda)$ over $U_{\zeta}^{\gamma}(w)$
we use representation \eqref{pOneDecomposition} from Lemma~\ref{NLwKappaDomain}
\begin{multline*}
\d\bar\zeta=\d{\bar P}(\zeta)\wedge\d_{\zeta}{\bar F}(w,\zeta)
\wedge\left(\sum_{j=0}^2\zeta_j\d{\bar\zeta}_j\right)\\
=\d{\bar P}(\zeta)\wedge\d_{\zeta}{\bar F}(w,\zeta)\wedge\d_{\zeta}B(w,\zeta)
-\d{\bar P}(\zeta)\wedge\d_{\zeta}{\bar F}(w,\zeta)
\wedge\left(\sum_{j=0}^2(\bar\zeta_j-\bar{w}_j)\d\zeta_j\right)
\end{multline*}
and obtain the following equality
\begin{multline}\label{LThirdDifferenceSum}
\frac{B(z^{(1)},w)^2}{B(w,\zeta)^2}
\left(\frac{\phi(z^{(1)},w,\lambda,\zeta)}{B^*({\bar\zeta},{\bar z}^{(1)})^2}
-\frac{\phi(z^{(2)},w,\lambda,\zeta)}{B^*({\bar\zeta},{\bar z}^{(2)})^2}\right)
\wedge\left(\sum_{j=0}^2\zeta_j\d{\bar\zeta}_j\right)\\
=\frac{B(z^{(1)},w)^2}{B(w,\zeta)^2}
\left(\frac{\phi(z^{(1)},w,\lambda,\zeta)}{B^*({\bar\zeta},{\bar z}^{(1)})^2}
-\frac{\phi(z^{(2)},w,\lambda,\zeta)}{B^*({\bar\zeta},{\bar z}^{(2)})^2}\right)
\wedge\d_{\zeta}B(w,\zeta)\\
-\frac{B(z^{(1)},w)^2}{B(w,\zeta)^2}
\left(\frac{\phi(z^{(1)},w,\lambda,\zeta)}{B^*({\bar\zeta},{\bar z}^{(1)})^2}
-\frac{\phi(z^{(2)},w,\lambda,\zeta)}{B^*({\bar\zeta},{\bar z}^{(2)})^2}\right)
\wedge\left(\sum_{j=0}^2(\bar\zeta_j-\bar{w}_j)\d\zeta_j\right),
\end{multline}
where
\begin{multline}\label{phiDifferenceForm}
\phi(z,w,\lambda,\zeta)=\bar\zeta_0^{\ell}\zeta_0^{-\ell}\vartheta(\zeta)
e^{\langle\lambda,\zeta/\zeta_0\rangle-\overline{\langle\lambda,\zeta/\zeta_0\rangle}}
S(w,\zeta)\det\left[\d z\ z\ Q(\bar\zeta,\bar{z})\right]\d(\zeta_j/\zeta_0)\\
\wedge\Big(\big(\d{\bar P}(\zeta)
\wedge\sum_{j=0}^2\zeta_j\d{\bar\zeta}_j\big)\interior\d\bar\zeta\Big).
\end{multline}
Using equality \eqref{LThirdDifferenceSum} we reduce the estimation of the left-hand side
of \eqref{LThirdDifferenceSum} to the estimation of two expressions
in the right-hand side of this equality. For the second term of the right-hand side of
\eqref{LThirdDifferenceSum} using coordinates $|\zeta-w|=r$ and $B(w,\zeta)=it+r^2$
we obtain the following estimate
\begin{multline}\label{LThirdDifferenceTwoEstimate}
\Bigg|\int_{\pi^{-1}({\cal C}) \cap\left\{|B(w,\zeta|<\gamma\right\}}
\frac{B(z^{(1)},w)^2}{B(w,\zeta)^2}
\left(\frac{\phi(z^{(1)},w,\lambda,\zeta)}{B^*({\bar\zeta},{\bar z}^{(1)})^2}
-\frac{\phi(z^{(2)},w,\lambda,\zeta)}{B^*({\bar\zeta},{\bar z}^{(2)})^2}\right)
\wedge\left(\sum_{j=0}^2(\bar\zeta_j-\bar{w}_j)\d\zeta_j\right)\Bigg|\\
\leq\Bigg|\int_{\pi^{-1}({\cal C}) \cap\left\{|B(w,\zeta|<\gamma\right\}}
\frac{B(z^{(1)},w)^2}{B(w,\zeta)^2}
\phi(z^{(1)},w,\lambda,\zeta)\bigg(\frac{B^*({\bar\zeta},{\bar z}^{(2)})^2
-B^*({\bar\zeta},{\bar z}^{(1)})^2}
{B^*({\bar\zeta},{\bar z}^{(1)})^2B^*({\bar\zeta},{\bar z}^{(2)})^2}\bigg)
\wedge\bigg(\sum_{j=0}^2(\bar\zeta_j-\bar{w}_j)\d\zeta_j\bigg)\Bigg|\\
+\Bigg|\int_{\pi^{-1}({\cal C}) \cap\left\{|B(w,\zeta|<\gamma\right\}}
\frac{B(z^{(1)},w)^2}{B(w,\zeta)^2}
\bigg(\frac{\phi(z^{(1)},w,\lambda,\zeta)-\phi(z^{(2)},w,\lambda,\zeta)}
{B^*({\bar\zeta},{\bar z}^{(2)})^2}\bigg)
\wedge\bigg(\sum_{j=0}^2(\bar\zeta_j-\bar{w}_j)\d\zeta_j\bigg)\Bigg|\\
\leq C(\lambda)\cdot\delta\gamma^2\int_0^{\gamma}\d t \int_0^{\sqrt{\gamma}}
\frac{r\d r}{(t+r^2)^{3/2}(\gamma+r^2)^{5/2}}
+C(\lambda)\cdot\delta\gamma^2\int_0^{\gamma}\d t \int_0^{\sqrt{\gamma}}
\frac{r\d r}{(t+r^2)^{3/2}(\gamma+r^2)^2}\\
\leq C(\lambda)\cdot\delta\gamma^2\int_0^{\sqrt{\gamma}}\frac{\d r}{(\gamma+r^2)^{5/2}}
\leq C(\lambda)\cdot\delta,
\end{multline}
which implies estimate \eqref{LDifferenceEstimate} for this term.\\
\indent
To estimate the first term of the right-hand side of \eqref{LThirdDifferenceSum} we use the Stokes'
theorem to obtain
\begin{multline}\label{LDifferenceStokesSecond}
\lim_{\epsilon\to 0}\int_{\pi^{-1}({\cal C}) \cap\left\{|B(w,\zeta|<\gamma\right\}}
B(z^{(1)},w)^2\bigg(\frac{\phi(z^{(1)},w,\lambda,\zeta)}{B^*({\bar\zeta},{\bar z}^{(1)})^2}
-\frac{\phi(z^{(2)},w,\lambda,\zeta)}{B^*({\bar\zeta},{\bar z}^{(2)})^2}\bigg)
\frac{\d_{\zeta}B(w,\zeta)}{B(w,\zeta)^2}\\
=\int_{\pi^{-1}({\cal C}) \cap\left\{|B(w,\zeta|<\gamma\right\}}
B(z^{(1)},w)^2\bigg(\frac{\phi(z^{(1)},w,\lambda,\zeta)}{B^*({\bar\zeta},{\bar z}^{(1)})^2}
-\frac{\phi(z^{(2)},w,\lambda,\zeta)}{B^*({\bar\zeta},{\bar z}^{(2)})^2}\bigg)
\wedge\d_{\zeta}\left(\frac{1}{B(w,\zeta)}\right)\\
=\int_{\pi^{-1}(\cal{C}) \cap\left\{|B(w,\zeta)|=\gamma\right\}}
\frac{B(z^{(1)},w)^2}{B(w,\zeta)}
\bigg(\frac{\phi(z^{(1)},w,\lambda,\zeta)}{B^*({\bar\zeta},{\bar z}^{(1)})^2}
-\frac{\phi(z^{(2)},w,\lambda,\zeta)}{B^*({\bar\zeta},{\bar z}^{(2)})^2}\bigg)\\
-\lim_{\eta\to 0}\int_{\pi^{-1}(\cal{C}) \cap\left\{|B(w,\zeta)|=\eta\right\}}
\frac{B(z^{(1)},w)^2}{B(w,\zeta)}
\bigg(\frac{\phi(z^{(1)},w,\lambda,\zeta)}{B^*({\bar\zeta},{\bar z}^{(1)})^2}
-\frac{\phi(z^{(2)},w,\lambda,\zeta)}{B^*({\bar\zeta},{\bar z}^{(2)})^2}\bigg)\\
-\int_{\pi^{-1}({\cal C}) \cap\left\{|B(w,\zeta|<\gamma\right\}}
\frac{B(z^{(1)},w)^2}{B(w,\zeta)}\bigg(\d_{\zeta}\frac{\phi(z^{(1)},w,\lambda,\zeta)}
{B^*({\bar\zeta},{\bar z}^{(1)})^2}
-\d_{\zeta}\frac{\phi(z^{(2)},w,\lambda,\zeta)}{B^*({\bar\zeta},{\bar z}^{(2)})^2}\bigg).
\end{multline}
\indent
For the first integral of the right-hand side of \eqref{LDifferenceStokesSecond}
we use estimate
\begin{multline*}
\Bigg|\frac{\phi(z^{(1)},w,\lambda,\zeta)}{B^*({\bar\zeta},{\bar z}^{(1)})^2}
-\frac{\phi(z^{(2)},w,\lambda,\zeta)}{B^*({\bar\zeta},{\bar z}^{(2)})^2}\Bigg|\\
\leq C\cdot\left(\left|\frac{\phi(z^{(1)},w,\lambda,\zeta)-\phi(z^{(2)},w,\lambda,\zeta)}
{B^*({\bar\zeta},{\bar z}^{(1)})^2}\right|
+\left|\phi(z^{(2)},w,\lambda,\zeta)
\frac{B^*({\bar\zeta},{\bar z}^{(2)})^2-B^*({\bar\zeta},{\bar z}^{(1)})^2}
{B^*({\bar\zeta},{\bar z}^{(1)})^2B^*({\bar\zeta},{\bar z}^{(2)})^2}\right|\right)
\end{multline*}
and its corollary
\begin{equation}\label{LDifferencePsiEstimate}
\Bigg|\frac{\phi(z^{(1)},w,\lambda,\zeta)}{B^*({\bar\zeta},{\bar z}^{(1)})^2}
-\frac{\phi(z^{(2)},w,\lambda,\zeta)}{B^*({\bar\zeta},{\bar z}^{(2)})^2}
\Bigg|_{|B(w,\zeta)|=\gamma}\\
\leq C\cdot\left(\frac{\delta}{\gamma^2}+\frac{\delta}{\gamma^{5/2}}\right)
\leq C\cdot\frac{\delta}{\gamma^{5/2}}.
\end{equation}
Then, using estimate \eqref{LDifferencePsiEstimate} and estimate \eqref{DeltaAreaEstimate} for
the area of $\left\{|B(w,\zeta)|=\gamma\right\}$ we obtain
\begin{equation}\label{LDifferenceStokesWSecondFirst}
\Bigg|\int_{\pi^{-1}(\cal{C}) \cap\left\{|B(w,\zeta)|=\gamma\right\}}
\frac{B(z^{(1)},w)^2}{B(w,\zeta)}
\bigg(\frac{\phi(z^{(1)},w,\lambda,\zeta)}{B^*({\bar\zeta},{\bar z}^{(1)})^2}
-\frac{\phi(z^{(2)},w,\lambda,\zeta)}{B^*({\bar\zeta},{\bar z}^{(2)})^2}\bigg)\Bigg|\\
\leq C(\lambda)\cdot\delta\gamma^{2-7/2+3/2}=C(\lambda)\delta,
\end{equation}
where the property $\lim_{\lambda\to\infty}C(\lambda)=0$ follows from the Riemann-Lebesgue Lemma (see \cite{23}).

\indent
Similarly, using estimate \eqref{LDifferencePsiEstimate} and estimate \eqref{DeltaAreaEstimate} for
the area of $\left\{|B(w,\zeta)|=\eta\right\}$ we obtain for the second term of the right-hand
side of \eqref{LDifferenceStokesSecond}
the estimate
\begin{equation}\label{LDifferenceStokesWSecondSecond}
\Bigg|\int_{\pi^{-1}(\cal{C}) \cap\left\{|B(w,\zeta)|=\eta\right\}}
\frac{B(z^{(1)},w)^2}{B(w,\zeta)}
\bigg(\frac{\phi(z^{(1)},w,\lambda,\zeta)}{B^*({\bar\zeta},{\bar z}^{(1)})^2}
-\frac{\phi(z^{(2)},w,\lambda,\zeta)}{B^*({\bar\zeta},{\bar z}^{(2)})^2}\bigg)\Bigg|
\leq C\cdot\delta\frac{\eta^{1/2}}{\gamma^{1/2}}\to 0,\\
\end{equation}
as $\eta\to 0$.\\
\indent
For the third term of the right-hand side of \eqref{LDifferenceStokesSecond} using estimate
\eqref{BDifferenceInequality} we obtain the estimate
\begin{multline}\label{LDifferenceStokesSecondThird}
\bigg|\d_{\zeta}\frac{\phi(z^{(1)},w,\lambda,\zeta)}
{B^*({\bar\zeta},{\bar z}^{(1)})^2}
-\d_{\zeta}\frac{\phi(z^{(2)},w,\lambda,\zeta)}{B^*({\bar\zeta},{\bar z}^{(2)})^2}
\bigg|_{|B(w,\zeta)|\leq\gamma}
\leq\bigg|\frac{\d_{\zeta}\phi(z^{(1)},w,\lambda,\zeta)}
{B^*({\bar\zeta},{\bar z}^{(1)})^2}-\frac{\d_{\zeta}\phi(z^{(2)},w,\lambda,\zeta)}
{B^*({\bar\zeta},{\bar z}^{(2)})^2}\bigg|_{|B(w,\zeta)|\leq\gamma}\\
+\bigg|\phi(z^{(1)},w,\lambda,\zeta)\d_{\zeta}\bigg(\frac{1}
{B^*({\bar\zeta},{\bar z}^{(1)})^2}\bigg)
-\phi(z^{(2)},w,\lambda,\zeta)\d_{\zeta}\bigg(\frac{1}
{B^*({\bar\zeta},{\bar z}^{(2)})^2}\bigg)\bigg|_{|B(w,\zeta)|\leq\gamma}\\
\leq\bigg|\d_{\zeta}\phi(z^{(1)},w,\lambda,\zeta)
\bigg(\frac{B^*({\bar\zeta},{\bar z}^{(2)})^2-B^*({\bar\zeta},{\bar z}^{(1)})^2}
{B^*({\bar\zeta},{\bar z}^{(1)})^2B^*({\bar\zeta},{\bar z}^{(2)})^2}\bigg)
\bigg|_{|B(w,\zeta)|\leq\gamma}\\
+\bigg|\frac{\d_{\zeta}\phi(z^{(1)},w,\lambda,\zeta)
-\d_{\zeta}\phi(z^{(2)},w,\lambda,\zeta)}{B^*({\bar\zeta},{\bar z}^{(2)})^2}
\bigg|_{|B(w,\zeta)|\leq\gamma}\\
+\bigg|\bigg(\phi(z^{(1)},w,\lambda,\zeta)-\phi(z^{(2)},w,\lambda,\zeta)\bigg)
\d_{\zeta}\bigg(\frac{1}{B^*({\bar\zeta},{\bar z}^{(1)})^2}\bigg)
\bigg|_{|B(w,\zeta)|\leq\gamma}\\
+\bigg|\phi(z^{(2)},w,\lambda,\zeta)
\bigg[\d_{\zeta}\bigg(\frac{1}{B^*({\bar\zeta},{\bar z}^{(1)})^2}\bigg)
-\d_{\zeta}\bigg(\frac{1}{B^*({\bar\zeta},{\bar z}^{(2)})^2}\bigg)
\bigg]\bigg|_{|B(w,\zeta)|\leq\gamma}\\
\leq C\cdot\delta\bigg(\gamma^{-5/2}+\gamma^{-2}+\gamma^{-3}+\gamma^{-7/2}\bigg)
\leq C\cdot\delta\gamma^{-7/2}.
\end{multline}

\indent
Then, using estimate \eqref{LDifferenceStokesSecondThird} and coordinates $|\zeta-w|=r$
and $B(w,\zeta)=it+r^2$, we obtain the following estimate
\begin{multline}\label{LDifferenceStokesSecondThirdEstimate}
\Bigg|\int_{\pi^{-1}(\cal{C}) \cap\left\{|B(w,\zeta)|<\gamma\right\}}
\frac{B(z^{(1)},w)^2}{B(w,\zeta)}
\bigg(\d_{\zeta}\frac{\phi(z^{(1)},w,\lambda,\zeta)}
{B^*({\bar\zeta},{\bar z}^{(1)})^2}
-\d_{\zeta}\frac{\phi(z^{(2)},w,\lambda,\zeta)}{B^*({\bar\zeta},{\bar z}^{(2)})^2}\bigg)\Bigg|\\
\leq C(\lambda)\cdot\delta\gamma^{2-7/2}
\int_0^{\gamma}\d t\int_0^{\sqrt{\gamma}}\frac{r\d r}{t+r^2}
\leq C(\lambda)\cdot\delta\gamma^{-3/2}\int_0^{\gamma}\int_0^{\gamma}\frac{\d t\d u}{t+u}
\leq C(\lambda)\cdot\delta\gamma^{-1/2}.
\end{multline}
\indent
Combining the estimates \eqref{LThirdDifferenceTwoEstimate},
\eqref{LDifferenceStokesWSecondFirst}, \eqref{LDifferenceStokesWSecondSecond}, and \eqref{LDifferenceStokesSecondThirdEstimate},
we obtain estimate \eqref{LDifferenceEstimate} for the integrals in $L_j(z^{(i)},w,\lambda)$
over the domain $U_{\zeta}^{\gamma}(w)$.

\indent
To estimate the difference of integrals in $L_j(z^{(i)},w,\lambda)$ over the domain
$U_{\zeta}^{\tau,\gamma}(z,w)$ we use the equality
\begin{multline}\label{LFourthDifference}
\frac{B(z^{(1)},w)^2}{B(w,\zeta)^2}
\left(\frac{\phi(z^{(1)},w,\lambda,\zeta)}{B^*({\bar\zeta},{\bar z}^{(1)})^2}
-\frac{\phi(z^{(2)},w,\lambda,\zeta)}{B^*({\bar\zeta},{\bar z}^{(2)})^2}\right)
\wedge\left(\sum_{j=0}^2\zeta_j\d{\bar\zeta}_j\right)\\
=\frac{B(z^{(1)},w)^2}{B(w,\zeta)^2}
\left(\frac{\phi(z^{(1)},w,\lambda,\zeta)-\phi(z^{(2)},w,\lambda,\zeta)}
{B^*({\bar\zeta},{\bar z}^{(1)})^2}\right)
\wedge\left(\sum_{j=0}^2\zeta_j\d{\bar\zeta}_j\right)\\
+\frac{B(z^{(1)},w)^2}{B(w,\zeta)^2}\phi(z^{(2)},w,\lambda,\zeta)
\bigg(\frac{B^*({\bar\zeta},{\bar z}^{(2)})^2-B^*({\bar\zeta},{\bar z}^{(1)})^2}
{B^*({\bar\zeta},{\bar z}^{(1)})^2B^*({\bar\zeta},{\bar z}^{(2)})^2}\bigg)
\wedge\left(\sum_{j=0}^2\zeta_j\d{\bar\zeta}_j\right)\\
\end{multline}
where the form $\phi(z,w,\lambda,\zeta)$ is defined in \eqref{phiDifferenceForm}
\begin{multline*}
\phi(z,w,\lambda,\zeta)=\bar\zeta_0^{\ell}\zeta_0^{-\ell}\vartheta(\zeta)
e^{\langle\lambda,\zeta/\zeta_0\rangle-\overline{\langle\lambda,\zeta/\zeta_0\rangle}}
S(w,\zeta)\det\left[\d z\ z\ Q(\bar\zeta,\bar{z})\right]\d(\zeta_j/\zeta_0)\\
\wedge\Big(\big(\d{\bar P}(\zeta)
\wedge\sum_{j=0}^2\zeta_j\d{\bar\zeta}_j\big)\interior\d\bar\zeta\Big).
\end{multline*}
Then, using the coordinates $|\zeta-w|=r$ and $B(w,\zeta)=it+r^2$, we obtain the following
estimate
\begin{multline}\label{LDifferenceFourthEstimate}
\Bigg|\int_{\pi^{-1}(\cal{C}) \cap\left\{\gamma<|B(w,\zeta)|,|B(z,\zeta|<\tau\right\}}
\frac{B(z^{(1)},w)^2}{B(w,\zeta)^2}
\left(\frac{\phi(z^{(1)},w,\lambda,\zeta)}{B^*({\bar\zeta},{\bar z}^{(1)})^2}
-\frac{\phi(z^{(2)},w,\lambda,\zeta)}{B^*({\bar\zeta},{\bar z}^{(2)})^2}\right)
\wedge\left(\sum_{j=0}^2\zeta_j\d{\bar\zeta}_j\right)\Bigg|\\
\leq C(\lambda)\cdot\delta\gamma^2\bigg(
\int_{\gamma}^{\tau}\d t\int_{\sqrt{\gamma}}^{\sqrt{\tau}}
\frac{r\d r}{(\gamma+t+r^2)^2(\gamma+r^2)^2}
+\int_{\gamma}^{\tau}\d t\int_{\sqrt{\gamma}}^{\sqrt{\tau}}
\frac{r\d r}{(\gamma+t+r^2)^2(\gamma+r^2)^{5/2}}\bigg)\\
\leq C(\lambda)\cdot\delta\gamma^2\bigg(\int_{\sqrt{\gamma}}^{\sqrt{\tau}}
\frac{r\d r}{(\gamma+r^2)^3}+\int_{\sqrt{\gamma}}^{\sqrt{\tau}}
\frac{r\d r}{(\gamma+r^2)^{7/2}}\bigg)\leq C(\lambda)\gamma^{-1/2},
\end{multline}
which implies the estimate \eqref{LDifferenceEstimate} for the difference of integrals in
$L_j(z^{(i)},w,\lambda)$ over the domain $U_{\zeta}^{\tau,\gamma}(z,w)$.\\
\indent
Combining all estimates of the differences of integrals in $L_j(z^{(i)},w,\lambda)$ over all domains
in \eqref{DifferenceDomains} we obtain estimate \eqref{LDifferenceEstimate}.
\end{proof}

\indent
In the next lemma we obtain the estimate of the difference
$L_j(z,w^{(1)},\lambda)-L_j(z,w^{(2)},\lambda)$ to be used in Lemma~\ref{QLemma}.

\begin{lemma}\label{LwDifferenceLemma}
There exist constants $C(\lambda)$, satisfying $\lim_{\lambda\to\infty}C(\lambda)=0$,
such that the following estimate
\begin{equation}\label{LwDifferenceEstimate}
\left|L_j(z,w^{(1)},\lambda)-L_j(z,w^{(2)},\lambda)\right|
\leq C(\lambda)\cdot\frac{\delta}{|B(z,w^{(1)})|^{5/2}}
\end{equation}
holds for an arbitrary fixed $\delta$,
any two points $w^{(1)},w^{(2)}\in V$, such that $|B(w^{(1)},w^{(2)})|=\delta^2$,
$z$ satisfying $\gamma=|B(z,w^{(1)})|>9\delta^2$, for the integral from \eqref{LForm}
\begin{multline*}
L_j(z,w,\lambda)=w_0^{\ell}\cdot\lim_{\epsilon\to 0}
\int_{\Gamma^{\epsilon}_{\zeta}}\bar\zeta_0^{\ell}\zeta_0^{-\ell}\vartheta(\zeta)
e^{\langle\lambda,\zeta/\zeta_0\rangle-\overline{\langle\lambda,\zeta/\zeta_0\rangle}}\\
\times\det\left[\frac{\d{\bar w}}{B(w,\zeta)}\ \frac{\bar{w}}{B(w,\zeta)}\ Q(w,\zeta)\right]
\det\left[\frac{\d z}{B^*(\bar\zeta,\bar{z})}\ \frac{z}{B^*(\bar\zeta,\bar{z})}\
Q(\bar\zeta,\bar{z})\right]
\d\left(\frac{\zeta_j}{\zeta_0}\right)\wedge\frac{\d\bar\zeta}{P(\bar\zeta)}.
\end{multline*}
\end{lemma}
\begin{proof}
For the product of determinants
\begin{equation*}
D(z,w,\zeta)=\det\left[\frac{\bar{w}}{B(w,\zeta)}\ \frac{\d\bar{w}}{B(w,\zeta)}\ Q(w,\zeta)\right]
\wedge\det\left[\frac{\d z}{B^*(\bar\zeta,\bar{z})}\ \frac{z}{B^*(\bar\zeta,\bar{z})}\
Q(\bar\zeta,\bar{z})\right]
\end{equation*}
in the form $L_j(z,w,\lambda)$ we have the following equality
\begin{equation}\label{DzwEquality}
D(z,w,\zeta)=D(\bar{w},\bar{z},\bar\zeta).
\end{equation}
Therefore, the estimate \eqref{LwDifferenceEstimate}, which follows from the estimate of the difference
$$D(z,w^{(1)},\zeta)-D(z,w^{(2)},\zeta)
=D(\bar{w}^{(1)},\bar{z},\zeta)-D(\bar{w}^{(2)},\bar{z},\zeta),$$
is the corollary of the estimate \eqref{LDifferenceEstimate} in Lemma~\ref{LDifferenceLemma}.
\end{proof}

\section{\bf Solvability of equation for $\partial h$.}\label{sec:Solving}

\indent
The goal of this section is the proof of the Fredholm property of equation \eqref{PEquation}, which is the first one of the integral equations mentioned in Theorem~\ref{Main}. This equation is obtained from equality \eqref{hIntegralEquation} with the usage of Lemma~\ref{PartialzLemma}. In the proposition below we prove the Fredholm property of the operator $I+\cal{P}_{\lambda}$ from equation \eqref{PEquation}.

\begin{proposition}\label{PEquationFredholm}
Let
\begin{multline}\label{POperator}
\cal{P}_{\lambda}[g](z,\lambda)=\lim_{\eta\to 0}\Bigg(\lim_{\tau\to 0}
\int_{\Gamma^{\tau}_w\cap \pi^{-1}(bV)\cap\{|B(z,w)|>\eta\}}
g(w,\lambda)\cdot e^{\overline{\left\langle\lambda,w/w_0\right\rangle}
-\langle\lambda,w/w_0\rangle}\partial_z G(z,w,\lambda)\\
+\lim_{\tau\to 0}\int_{\Gamma^{\tau}_w\cap\{|B(z,w)|>\eta\}}
g(w,\lambda)\wedge
e^{\overline{\left\langle\lambda,w/w_0\right\rangle}
-\langle\lambda,w/w_0\rangle}\partial_z\bar\partial_w G(z,w,\lambda)\Bigg)
\end{multline}
be the operator on $\Lambda^{\widehat{\delta},\alpha}_{(1,0)}(V)$.
Then, for $\lambda$ large enough, and $\alpha$ and $\widehat{\delta}$
satisfying conditions
\begin{equation}\label{alphadeltaCondition}
\left\{\begin{aligned}
&\alpha< 1/2,\\
&\widehat{\delta}=|\lambda|^{-2},
\end{aligned}\right.
\end{equation}
operator $I+\cal{P}_{\lambda}$ is a Fredholm-type operator.
\end{proposition}

{\bf Note.}
{\it We note that the choice of $\alpha$ and $\widehat{\delta}$ is not unique. Conditions
\eqref{alphadeltaCondition} are motivated by the inequality \eqref{alphadeltaInequality} and can be
changed with $\alpha\to 1$, $\lambda\to \infty$, $\widehat{\delta}(\lambda)\to 0$.}

\indent
We start with the following lemma, which will be used to transform the formula \eqref{POperator}.

\begin{lemma}\label{Cancellation}
If $g^{(1,0)}(w,\lambda)=\sum_{k=0}^2a_k\d w_k$ is a form with constant coefficients, then
\begin{multline}\label{CancellationEquality}
\lim_{\eta\to 0}\lim_{\tau\to 0}\int_{\Gamma^{\tau}_w\cap bV\cap\{|B(z,w)|>\eta\}}
g^{(1,0)}(w,\lambda)\wedge E(w,\lambda)\partial_z G(z,w,\lambda)\\
+\lim_{\eta\to 0}\lim_{\tau\to 0}\int_{\Gamma^{\tau}_w\cap\{|B(z,w)|>\eta\}}
g^{(1,0)}(w,\lambda)\wedge E(w,\lambda)\bar\partial_w\partial_z G(z,w,\lambda)\\
+\lim_{\eta\to 0}\lim_{\tau\to 0}\int_{\Gamma^{\tau}_w\cap\{|B(z,w)|>\eta\}}
g^{(1,0)}(w,\lambda)\wedge \bar\partial E(w,\lambda)\wedge \partial_z G(z,w,\lambda)=0,
\end{multline}
where
\begin{equation}\label{ENotation}
E(w,\lambda)=e^{\overline{\left\langle\lambda,w/w_0\right\rangle}
-\langle\lambda,w/w_0\rangle}.
\end{equation}
\end{lemma}
\begin{proof}
\indent
Applying the Stokes' theorem and using Lemma~\ref{ZeroLimit} we obtain
\begin{multline*}
\lim_{\eta\to 0}\lim_{\tau\to 0}\int_{\Gamma^{\tau}_w\cap bV\cap\{|B(z,w)|>\eta\}}
g^{(1,0)}(w,\lambda) E(w,\lambda)\wedge\partial_z G(z,w,\lambda)\\
=-\lim_{\eta\to 0}\lim_{\tau\to 0}\int_{\Gamma^{\tau}_w\cap\{|B(z,w)|>\eta\}}
g^{(1,0)}(w,\lambda)\wedge\bar\partial E(w,\lambda)
\wedge\partial_z G(z,w,\lambda)\\
-\lim_{\eta\to 0}\lim_{\tau\to 0}\int_{\Gamma^{\tau}_w\cap\{|B(z,w)|>\eta\}}
g^{(1,0)}(w,\lambda) E(w,\lambda)\wedge\bar\partial_w\partial_z G(z,w,\lambda),
\end{multline*}
which is equivalent to \eqref{CancellationEquality}.
\end{proof}

In our proof of Proposition~\ref{PEquationFredholm} we will be
proving the existence of constants $C(\lambda)$ for $\alpha, \widehat{\delta}$ satisfying
condition \eqref{alphadeltaCondition} for arbitrary
$g\in\Lambda^{\widehat{\delta},\alpha}_{(1,0)}(V)$.
In order to adjust the estimates in Lemmas~\ref{NDeltaDomain}$\div$\ref{NwDifferenceLemma},
and \ref{LDeltaDomain}$\div$\ref{LDifferenceLemma},
which are made for points $z^{(1)},z^{(2)},w^{(1)},w^{(2)}\in \*S^5(1)$, with estimates in the proof of Proposition~\ref{PEquationFredholm}, we will use
the independence of the formulas from Proposition~\ref{dbarSolvability} on the specific choice
of points in $\*S^5(1)$, but only on their projections to $\C\P^2$ (see Theorem 1 from \cite{15}).\\
\indent
An additional step that is needed in the necessary adjustment is provided by the following lemma.

\begin{lemma}\label{FindingZ}
There exist $C_1, C_2$ such that for two points  $u^{(1)},u^{(2)}\in V$ with
$|u^{(1)}-u^{(2)}|=\delta$ there exist points $z^{(i)}\subset \pi^{-1}(V)\ (i=1,2)$ such that
$\pi(z^{(i)})=u^{(i)}$ and
\begin{equation}\label{ZInequality}
C_1\delta^2<|B(z^{(1)},z^{(2)})|<C_2\cdot\delta^2.
\end{equation}
\end{lemma}
\begin{proof}
Without loss of generality we may assume
$u^{(1)}=(0,0)\in \C^2=\C\P^2\setminus \C^{1}_{\infty}$.
Then, in a neighborhood of $u^{(1)}$ we define for $u^{(2)}=(u_1,u_2)$:
\begin{equation*}
\begin{aligned}
&d=\sqrt{1+|u_1|^2+|u_2|^2},\\
&\pi^{-1}(u^{(2)})=(z_0,z_1,z_2)=(1/d,u_1/d,u_2/d),
\end{aligned}
\end{equation*}
so that
$$|z_0|^2+|z_1|^2+|z_2|^2=(1/d)^2+|u_1/d|^2+|u_2/d|^2=\frac{1+|u_1|^2+|u_2|^2}{d^2}=1,$$
and ${\dis u_i=\frac{z_i}{z_0}=\frac{u_i/d}{1/d} }$.\\
\indent
Then, for $\pi^{-1}(u^{(1)})=z^{(1)}=(1,0,0)$ and
$\pi^{-1}(u^{(2)})=z^{(2)}=(1/d,u_1/d,u_2/d)$ we have
$$B(z^{(1)},z^{(2)})=1-\sum_{j=0}^2\bar{z}^{(1)}_jz^{(2)}_j
=1-1\cdot\frac{1}{d}=\frac{d-1}{d}=\frac{\sqrt{1+|u_1|^2+|u_2|^2}-1}
{\sqrt{1+|u_1|^2+|u_2|^2}}$$
$$=\frac{1+|u_1|^2+|u_2|^2-1}
{(\sqrt{1+|u_1|^2+|u_2|^2}+1)\sqrt{1+|u_1|^2+|u_2|^2}}\sim(|u_1|^2+|u_2|^2).$$
\end{proof}

{\it Proof of Proposition~\ref{PEquationFredholm}}
In order to prove the proposition it suffices
to represent operator $\cal{P}_{\lambda}$ as a sum
\begin{equation}\label{PSum}
\cal{P}_{\lambda}[g]={\cal F}_{\lambda}[g]+{\cal Q}_{\lambda}[g],
\end{equation}
where ${\cal Q}_{\lambda}$ is a compact operator on
$\Lambda^{\widehat{\delta},\alpha}_{(1,0)}(V)$ and
$\|{\cal F}_{\lambda}\|_{\Lambda^{\widehat{\delta},\alpha}_{(1,0)}(V)}$ is small enough,
so that operator $I+\cal{F}_{\lambda}$ is invertible on
$\Lambda^{\widehat{\delta},\alpha}_{(1,0)}(V)$ .\\
\indent
Using formula \eqref{POperator} for operator $\cal{P}_{\lambda}[g]$ and the notation
$\widehat{\nu}=9(\widehat{\delta})^2$, we represent operator $\cal{P}_{\lambda}$ as in \eqref{PSum} with
\begin{multline}\label{FOperator}
\cal{F}_{\lambda}[g](z,\lambda)=\lim_{\eta\to 0}\Bigg(\lim_{\tau\to 0}
\int_{\Gamma^{\tau}_w\cap \pi^{-1}(bV)\cap\{|B(z,w)|>\eta\}}g(w)
\cdot E(w,\lambda)\wedge\partial_z G^{\widehat{\nu}}(z,w,\lambda)\\
+\lim_{\tau\to 0}\int_{\Gamma^{\tau}_w\cap\{|B(z,w)|>\eta\}}
g(w)\cdot E(w,\lambda)
\wedge\bar\partial_w \partial_z G^{\widehat{\nu}}(z,w,\lambda)\Bigg),\\
\end{multline}
and
\begin{multline}\label{QOperator}
\cal{Q}_{\lambda}[g](z,\lambda)
=\lim_{\eta\to 0}\Bigg(\lim_{\tau\to 0}
\int_{\Gamma^{\tau}_w\cap \pi^{-1}(bV)\cap\{|B(z,w)|>\eta\}}
g(w)\cdot E(w,\lambda)
\wedge\partial_z \left(G(z,w,\lambda)-G^{\widehat{\nu}}(z,w,\lambda)\right)\\
+\lim_{\tau\to 0}\int_{\Gamma^{\tau}_w\cap\{|B(z,w)|>\eta\}}
g(w)\cdot E(w,\lambda)
\wedge\partial_z\bar\partial_w \left(G(z,w,\lambda)
-G^{\widehat{\nu}}(z,w,\lambda)\right)\Bigg),
\end{multline}
where
\begin{multline}\label{Gnu}
G^{\widehat{\nu}}(z,w,\lambda)\\
=|{\bf{C}}|^2\cdot\bar{z}_0^{-\ell}w_0^{\ell}\cdot\psi^{\widehat{\nu}}(z,w)
\sum_{j=1,2}\Bigg(\lim_{\delta\to 0}\lim_{\epsilon\to 0}
\int_{\Gamma_{\zeta}^{\epsilon}
\setminus\left(U^{\delta}_{\zeta}(w)\cup U^{\delta}_{\zeta}(z)\right)}
\bar\zeta_0^{\ell}\zeta_0^{-\ell}\vartheta(\zeta)
e^{\langle\lambda,\zeta/\zeta_0\rangle-\overline{\langle\lambda,\zeta/\zeta_0\rangle}}\\
\wedge\det\left[\frac{\bar\zeta}{B^*(w,\zeta)}\ \frac{\bar{w}}{B(w,\zeta)}\ Q(w,\zeta)\right]
\det\left[\frac{z}{B^*(\bar\zeta,\bar{z})}\ \frac{\zeta}
{B(\bar\zeta,\bar{z})}\ Q(\bar\zeta,\bar{z})\right]
\wedge\d\left(\frac{\zeta_j}{\zeta_0}\right)
\wedge\frac{\d\bar\zeta}{P(\bar\zeta)}\Bigg)
\wedge\frac{\omega^{(2,0)}_j(w)}{P(w)},
\end{multline}
and $\psi^{\widehat{\nu}}(z,w)\in C^{\infty}(V)$ is a positive function, satisfying:
\begin{equation}\label{psiFunction}
\psi^{\widehat{\nu}}(z,w)
=\begin{cases}
1 & \text{if } |B(z,w)|<\widehat{\nu}/2,\\
0 & \text{if } |B(z,w)|>\widehat{\nu},
\end{cases}
\end{equation}

In the following lemma we prove the \lq\lq smallness\rq\rq\ of the norm
$\|\cal{F}_{\lambda}\|_{C_{(1,0)}(V)}$ for a large enough $\lambda$ and small enough
$\widehat{\nu}$.

\begin{lemma}\label{FCLemma} If $\alpha$ and $\widehat{\delta}$ satisfy conditions \eqref{alphadeltaCondition}, then there exist constants $C(\lambda)\to 0$ as $\lambda\to\infty$ such that the estimate
\begin{equation}\label{F1CEstimate}
\|\cal{F}_{\lambda}[g]\|_{C_{(1,0)}(V)}
\leq C(\lambda)\left\|g\right\|_{\Lambda^{\widehat{\delta},\alpha}_{(1,0)}(V)}
\end{equation}
holds.
\end{lemma}
\begin{proof}
Using equality \eqref{CancellationEquality} from Lemma~\ref{Cancellation} for the form
$$\widehat{g}^{(1,0)}(z,w)=\sum_{k=0}^2g_k(z)\d w_k$$
we transform the operator in \eqref{FOperator} as follows:
\begin{multline}\label{F1DifferenceFormula}
\cal{F}_{\lambda}[g](z,\lambda)=\lim_{\eta\to 0}
\Bigg(\lim_{\tau\to 0}\int_{\Gamma^{\tau}_w\cap \pi^{-1}(bV)\cap\{|B(z,w)|>\eta\}}
g(w)\cdot E(w,\lambda)\wedge\partial_z G^{\widehat{\nu}}(z,w,\lambda)\\
+\lim_{\tau\to 0}\int_{\Gamma^{\tau}_w\cap\{|B(z,w)|>\eta\}}
g(w)\cdot E(w,\lambda)\wedge\partial_z\bar\partial_w G^{\widehat{\nu}}(z,w,\lambda)\Bigg)\\
=\lim_{\eta\to 0}\lim_{\tau\to 0}\int_{\Gamma^{\tau}_w\cap \pi^{-1}(bV)\cap\{|B(z,w)|>\eta\}}
(g(w)-\widehat{g}(z,w))\cdot E(w,\lambda)\wedge\partial_z G^{\widehat{\nu}}(z,w,\lambda)\\
+\lim_{\eta\to 0}\lim_{\tau\to 0}\int_{\Gamma^{\tau}_w\cap\{|B(z,w)|>\eta\}}
(g(w)-\widehat{g}(z,w))\cdot E(w,\lambda)
\wedge\partial_z\bar\partial_w G^{\widehat{\nu}}(z,w,\lambda)\\
-\lim_{\eta\to 0}\lim_{\tau\to 0}\int_{\Gamma^{\tau}_w\cap\{|B(z,w)|>\eta\}}
\widehat{g}(z,w)\wedge\bar\partial_w E(w,\lambda)
\wedge\partial_z G^{\widehat{\nu}}(z,w,\lambda).
\end{multline}
\indent
To prove the \lq\lq smallness\rq\rq\ of the the norm $\|\cal{F}_{\lambda}\|_{C_{(1,0)}(V)}$
for a small enough $\widehat{\nu}$ we estimate separately three terms of the right-hand side of \eqref{F1DifferenceFormula}.

\indent
Then, for the first term of the right-hand side of \eqref{F1DifferenceFormula} using estimate
\eqref{NDeltaEstimate}
\begin{equation*}
\left|N_j(z,w,\lambda)\right|\leq \frac{C(\lambda)}{|B(z,w)|^{3/2}}
\end{equation*}
we obtain
\begin{multline}\label{F1EstimateOne}
\Bigg|\lim_{\eta\to 0}\lim_{\tau\to 0}
\int_{\Gamma^{\tau}_w\cap \pi^{-1}(bV)\cap\{\widehat{\nu}>|B(z,w)|>\eta\}}
(g(w)-\widehat{g}(z,w))\cdot E(w,\lambda)
\wedge\partial_z G^{\widehat{\nu}}(z,w,\lambda)\Bigg|\\
\leq C(\lambda)\cdot\Bigg|\lim_{\eta\to 0}
\int_{bV_w\cap\{\widehat{\nu}>|B(z,w)|>\eta\}}
\frac{(g(w)-\widehat{g}(z,w))E(w,\lambda)}{|B(z,w)|^{3/2}}\Bigg|\\
\leq C(\lambda)\cdot\|g\|_{\Lambda^{\widehat{\delta},\alpha}(V)}
\int_0^{(\widehat{\delta})^2}\d t
\int_0^{ \widehat{\delta}}\frac{r^{\alpha}\d r}{(t+r^2)^{3/2}}
\leq C(\lambda)\cdot\|g\|_{\Lambda^{\widehat{\delta},\alpha}(V)}
\int_0^{ \widehat{\delta}}\frac{\d r}{r^{1-\alpha}}
\leq C(\lambda)\cdot\|g\|_{\Lambda^{\widehat{\delta},\alpha}(V)} (\widehat{\delta})^{\alpha}.\\
\end{multline}

For the second term of the right-hand side of \eqref{F1DifferenceFormula} using estimate
\eqref{LDeltaEstimate}
\begin{equation*}
\left|L_j(z,w,\lambda)\right|\leq \frac{C(\lambda)}{|B(z,w)|^2},
\end{equation*}
we obtain
\begin{multline}\label{F1EstimateTwo}
\Bigg|\lim_{\eta\to 0}\lim_{\tau\to 0}
\int_{\Gamma^{\tau}_w\cap\{\widehat{\nu}>|B(z,w)|>\eta\}}
(g(w)-\widehat{g}(z,w))\cdot E(w,\lambda)\wedge\partial_z
\bar\partial_w G^{\widehat{\nu}}(z,w,\lambda)\Bigg|\\
\leq C(\lambda)\cdot\Bigg|\lim_{\eta\to 0}
\int_{V_w\cap\{ \widehat{\nu}>|B(z,w)|>\eta\}}
\frac{(g(w)-\widehat{g}(z,w))E(w,\lambda)}{|B(z,w)|^2}\Bigg|\\
\leq C(\lambda)\cdot\|g\|_{\Lambda^{\widehat{\delta},\alpha}(V)}\int_0^{ \widehat{\delta}^2}\d t
\int_0^{\widehat{\delta}}\frac{r^{1+\alpha}\d r}{(t+r^2)^2}
\leq C(\lambda)\cdot\|g\|_{\Lambda^{\widehat{\delta},\alpha}(V)}
\int_0^{ \widehat{\delta}}\frac{\d r}{r^{1-\alpha}}
\leq C(\lambda)\cdot\|g\|_{\Lambda^{\widehat{\delta},\alpha}(V)} \widehat{\delta}^{\alpha}.
\end{multline}

\indent
For the last term in the right-hand side of \eqref{F1DifferenceFormula} using estimate
\eqref{NDeltaEstimate} we obtain
\begin{multline}\label{F1EstimateThree}
\Bigg|\lim_{\eta\to 0}\lim_{\tau\to 0}
\int_{\Gamma^{\tau}_w\cap\{\widehat{\nu}>|B(z,w)|>\eta\}}
\widehat{g}(z,w)\wedge\bar\partial_w E(w,\lambda)
\wedge\partial_z G^{\widehat{\nu}}(z,w,\lambda)\Bigg|\\
\leq C(\lambda)\cdot\Bigg|\lim_{\eta\to 0}
\int_{V_w\cap\{\widehat{\nu}>|B(z,w)|>\eta\}}
\frac{\widehat{g}(z,w)\wedge\bar\partial_w E(w,\lambda)}{|B(z,w)|^{3/2}}\Bigg|\\
\leq C(\lambda)|\lambda|\cdot\|g\|_{C_{(1,0)}(V)}\int_0^{ \widehat{\delta}^2}\d t
\int_0^{\widehat{\delta}}\frac{r\d r}{(t+r^2)^{3/2}}
\leq C(\lambda)|\lambda|\widehat{\delta}\cdot\|g\|_{C_{(1,0)}(V)}
\leq C(\lambda)\cdot\|g\|_{C_{(1,0)}(V)}.
\end{multline}
\indent
From the estimates
\eqref{F1EstimateOne}$\div$\eqref{F1EstimateThree} we obtain the estimate \eqref{F1CEstimate}
with a constant $C(\lambda)\to 0$ as $\lambda\to\infty$.
\end{proof}

\indent
In the lemma below we complement Lemma~\ref{FCLemma} in proving the
H$\ddot{\mathrm{\bf o}}$lder-type estimates for the operator $\cal{F}_{\lambda}$
in Proposition~\ref{PEquationFredholm}.

\begin{lemma}\label{FHolder}
If $\alpha$ and $\widehat\delta$ satisfy conditions \eqref{alphadeltaCondition}, then the operator
$g\to \cal{F}_{\lambda}[g]$
from the formula \eqref{FOperator} is a bounded linear operator from the space
$\Lambda^{\widehat{\delta},\alpha}_{(1,0)}(V)$
of $(1,0)$-forms with coefficients in $\Lambda^{\widehat{\delta},\alpha}(V)$ into itself. Moreover,
there exist constants $C(\lambda)\to 0$ as $\lambda\to \infty$ such that
for arbitrary two points $u^{(1)},u^{(2)}\in V$ with $\left|u^{(1)}-u^{(2)}\right|
\leq\delta<\widehat\delta$ and a form $g\in\Lambda^{\widehat{\delta},\alpha}_{(1,0)}(V)$
the following estimate holds
\begin{equation}\label{F1Estimate}
\left|\cal{F}_{\lambda}[g](u^{(1)})-\cal{F}_{\lambda}[g](u^{(2)})\right|
\leq C(\lambda)\cdot\|g\|_{\Lambda^{\widehat{\delta},\alpha}}\cdot\delta^{\alpha}.
\end{equation}
\end{lemma}
\begin{proof}
\indent
Using formula \eqref{F1DifferenceFormula} we reduce the proof of Lemma~\ref{FHolder} to
the proof of estimate \eqref{F1Estimate} for the following expression
\begin{multline*}
\Bigg(\int_{b\Gamma^{\tau}_w\cap U_w^{\widehat{\nu}}(z^{(1)})}(g(w)-\widehat{g}(z^{(1)},w)
\cdot E(w,\lambda)\wedge\partial_z G^{\widehat{\nu}}(z^{(1)},w,\lambda)\\
-\int_{b\Gamma^{\tau}_w\cap U_w^{\widehat{\nu}}(z^{(2)})}(g(w)-\widehat{g}(z^{(2)},w))
\cdot E(w,\lambda)\wedge \partial_z G^{\widehat{\nu}}(z^{(2)},w,\lambda)\Bigg)\\
+\Bigg(\int_{\Gamma^{\tau}_w\cap U_w^{\widehat{\nu}}(z^{(1)})}(g(w)-\widehat{g}(z^{(1)},w))
\cdot E(w,\lambda)\wedge\partial_z\bar\partial_w G^{\widehat{\nu}}(z^{(1)},w,\lambda)\\
-\int_{\Gamma^{\tau}_w\cap U_w^{\widehat{\nu}}(z^{(2)})}(g(w)-\widehat{g}(z^{(2)},w))
\cdot E(w,\lambda)\wedge\partial_z\bar\partial_w G^{\widehat{\nu}}(z^{(2)},w,\lambda)\Bigg)\\
\Bigg(-\int_{\Gamma^{\tau}_w\cap U_w^{\widehat{\nu}}(z^{(1)})}
\widehat{g}(z^{(1)},w)\wedge\bar\partial_w E(w,\lambda)
\wedge\partial_z G^{\widehat{\nu}}(z^{(1)},w,\lambda)\\
+\int_{\Gamma^{\tau}_w\cap U_w^{\widehat{\nu}}(z^{(2)})}
\widehat{g}(z^{(2)},w)\wedge\bar\partial_w E(w,\lambda)
\wedge\partial_z G^{\widehat{\nu}}(z^{(2)},w,\lambda)\Bigg),
\end{multline*}
where $\widehat{\nu}=9(\widehat{\delta})^2$.

For $\delta<\widehat{\delta}$ and $\nu=9\delta^2$ we further transform the expression above
by rewriting it as
\begin{multline}\label{NLSubdivision}
\int_{b\Gamma^{\tau}_w\cap U_w^{\nu}(z^{(1)})}
(g(w)-\widehat{g}(z^{(1)},w))E(w,\lambda)
\wedge\partial_z G^{\widehat{\nu}}(z^{(1)},w,\lambda)\\
-\int_{b\Gamma^{\tau}_w\cap U_w^{\nu}(z^{(2)})}
(g(w)-\widehat{g}(z^{(2)},w))E(w,\lambda)
\wedge\partial_z G^{\widehat{\nu}}(z^{(2)},w,\lambda)\\
+\int_{\{b\Gamma^{\tau}_w\cap U^{\widehat{\nu}}_w(z^{(1)})\}\setminus U^{\nu}_w(z^{(1)})}
(g(w)-\widehat{g}(z^{(1)},w))E(w,\lambda)
\wedge\left(\partial_z G^{\widehat{\nu}}(z^{(1)},w,\lambda)
-\partial_z G^{\widehat{\nu}}(z^{(2)},w,\lambda)\right)\\
+\int_{\Gamma^{\tau}_w\cap U_w^{\nu}(z^{(1)})}
(g(w)-\widehat{g}(z^{(1)},w))E(w,\lambda)\wedge\partial_z
\bar\partial_w G^{\widehat{\nu}}(z^{(1)},w,\lambda)\\
-\int_{\Gamma^{\tau}_w\cap U_w^{\nu}(z^{(2)})}
(g(w)-\widehat{g}(z^{(2)},w))E(w,\lambda)\wedge\partial_z
\bar\partial_w G^{\widehat{\nu}}(z^{(2)},w,\lambda)\\
+\int_{\{\Gamma^{\tau}_w\cap U^{\widehat{\nu}}_w(z^{(1)})\}
\setminus U^{\nu}_w(z^{(1)})}
(g(w)-\widehat{g}(z^{(1)},w))E(w,\lambda)
\wedge\left(\partial_z\bar\partial_w G^{\widehat{\nu}}(z^{(1)},w,\lambda)
-\partial_z\bar\partial_w G^{\widehat{\nu}}(z^{(2)},w,\lambda)\right)\\
+\int_{\{b\Gamma^{\tau}_w\cap U^{\widehat{\nu}}_w(z^{(2)})\}\setminus U^{\nu}_w(z^{(2)})}
(\widehat{g}(z^{(2)},w)-\widehat{g}(z^{(1)},w))E(w,\lambda)
\wedge\partial_z G^{\widehat{\nu}}(z^{(2)},w,\lambda)\\
+\int_{\{\Gamma^{\tau}_w\cap U^{\widehat{\nu}}_w(z^{(2)})\}\setminus U^{\nu}_w(z^{(2)})}
(\widehat{g}(z^{(2)},w)-\widehat{g}(z^{(1)},w))E(w,\lambda)
\wedge\partial_z\bar\partial_w G^{\widehat{\nu}}(z^{(2)},w,\lambda)\\
+\int_{\{\Gamma^{\tau}_w\cap U^{\widehat{\nu}}_w(z^{(2)})\}\setminus U^{\nu}_w(z^{(2)})}
(\widehat{g}(z^{(2)},w)-\widehat{g}(z^{(1)},w))
\wedge\bar\partial_w E(w,\lambda)
\wedge\partial_z G^{\widehat{\nu}}(z^{(2)},w,\lambda)\\
+\int_{\{\Gamma^{\tau}_w\cap U^{\widehat{\nu}}_w(z^{(2)})\}\setminus U^{\nu}_w(z^{(2)})}
\widehat{g}(z^{(1)},w)\wedge\bar\partial_w E(w,\lambda)
\wedge\left(\partial_z G^{\widehat{\nu}}(z^{(2)},w,\lambda)
-\partial_z G^{\widehat{\nu}}(z^{(1)},w,\lambda)\right)\\
-\int_{\Gamma^{\tau}_w\cap U^{\nu}_w(z^{(1)})}
\widehat{g}(z^{(1)},w)\wedge\bar\partial_w E(w,\lambda)
\wedge\partial_z G^{\widehat{\nu}}(z^{(1)},w,\lambda)\\
+\int_{\Gamma^{\tau}_w\cap U^{\nu}_w(z^{(2)})}
\widehat{g}(z^{(2)},w)\wedge\bar\partial_w E(w,\lambda)
\wedge\partial_z G^{\widehat{\nu}}(z^{(2)},w,\lambda).
\end{multline}
\indent
In the estimates below of the integrals in \eqref{NLSubdivision} we use the following estimate:
\begin{equation}\label{HolderCondition}
w\in U^{\nu}_w(z^{(i)})\ \Rightarrow |B(z^{(i)},w)|\leq 9\delta^2 \Rightarrow\
|\text{Re}B(z^{(i)},w)|\leq 9\delta^2 \Rightarrow\ |w-z^{(i)}|\leq C\cdot\delta,
\end{equation}
which implies for $g\in \Lambda^{\widehat{\delta},\alpha}_{(1,0)}$ the inequality
$|g(w)-\widehat{g}(z^{(i)},w)|\leq C\cdot\delta^{\alpha}$.

\indent
For the first two integrals in \eqref{NLSubdivision} using equality
\eqref{PartialzG} and estimate \eqref{NDeltaEstimate} from Lemma~\ref{NDeltaDomain}
we obtain the estimate
\begin{multline}\label{NSmallwEstimate}
\left|\int_{b\Gamma^{\tau}_w\cap U^{\nu}_w(z^{(i)})}\left(g(w)-\widehat{g}(z^{(i)},w)\right)
E(w,\lambda)\wedge N_j(z^{(i)},w,\lambda)\wedge\frac{\omega^{(2,0)}_j(w)}{P(w)}\right|\\
\leq C(\lambda)\cdot\left|\int_{b\Gamma^{\tau}_w\cap U^{\nu}_w(z^{(i)})}
\frac{\left(g(w)-\widehat{g}(z^{(i)},w)\right)}
{B(z^{(i)},w)^{3/2}}\wedge\frac{\omega^{(2,0)}_j(w)}{P(w)}\right|\\
\leq C(\lambda)\cdot\|g\|_{\Lambda^{\widehat{\delta},\alpha}(V)}
\int_0^{\delta^2}\d t\int_0^{\delta}\frac{r^{\alpha}\d r}{(t+r^2)^{3/2}}
\leq C(\lambda)\cdot\|g\|_{\Lambda^{\widehat{\delta},\alpha}(V)}
\int_0^{\delta}\frac{\d r}{r^{1-\alpha}}
\leq C(\lambda)\cdot\|g\|_{\Lambda^{\widehat{\delta},\alpha}(V)}\delta^{\alpha}.
\end{multline}

\indent
For the third integral in \eqref{NLSubdivision} using equality
\eqref{PartialzG} and estimate \eqref{NDifferenceEstimate} in Lemma~\ref{NDifferenceLemma}
we obtain for $w$ such that $9(\widehat{\delta})^2>|B(z^{(i)},w)|>9\delta^2$ the estimate
\begin{multline}\label{NLargewEstimate}
\Bigg|\int_{\{b\Gamma^{\tau}_w\cap U^{\widehat{\nu}}_w(z^{(1)})\}
\setminus U^{\nu}_w(z^{(1)})}
\left(g(w)-\widehat{g}(z^{(1)},w)\right)E(w,\lambda)\\
\wedge\left(N_j(z^{(1)},w,\lambda)-N_j(z^{(2)},w,\lambda)\right)
\wedge\frac{\omega^{(2,0)}_j(w)}{P(w)}\Bigg|\\
\leq C(\lambda)\cdot\|g\|_{\Lambda^{\widehat{\delta},\alpha}(V)}\delta
\int_{\delta^2}^{(\widehat{\delta})^2}\d t
\int_{\delta}^{\widehat{\delta}}\frac{r^{\alpha}\d r}{(t+r^2)^2}
\leq C(\lambda)\cdot\|g\|_{\Lambda^{\widehat{\delta},\alpha}(V)}\delta
\int_{\delta}^{\widehat{\delta}}\frac{r^{\alpha}\d r}{(\delta^2+r^2)}\\
\leq C(\lambda)\cdot\|g\|_{\Lambda^{\widehat{\delta},\alpha}(V)}\delta
\int_{\delta}^{\widehat{\delta}} r^{-2+\alpha}\d r
\leq C(\lambda)\cdot\|g\|_{\Lambda^{\widehat{\delta},\alpha}(V)}\delta\delta^{-1+\alpha}
\leq C(\lambda)\cdot\|g\|_{\Lambda^{\widehat{\delta},\alpha}(V)}\delta^{\alpha}.
\end{multline}

\indent
For the fourth and fifth integrals in \eqref{NLSubdivision}
we use formula \eqref{LForm} and estimate \eqref{LDeltaEstimate} from Lemma~\ref{LDeltaDomain}, to obtain the estimate
\begin{multline}\label{NLSubdivision4and5Estimate}
\Bigg|\int_{\Gamma^{\tau}_w\cap U^{\nu}_w(z^{(i)})}
\left(g(w)-\widehat{g}(z^{(i)},w)\right)
E(w,\lambda)\wedge L^{(2,0)}_j(z^{(i)},w,\lambda)
\wedge\frac{\omega^{(2,0)}_j(w)}{P(w)}\Bigg|\\
\leq C(\lambda)\cdot\left|\int_{\Gamma^{\tau}_w\cap U^{\nu}_w(z^{(i)})}
\frac{\left(g(w)-\widehat{g}(z^{(i)},w)\right)}{B(z^{(i)},w)^2}
\wedge\frac{\omega^{(2,0)}_j(w)}{P(w)}\right|\\
\leq C(\lambda)\cdot\|g\|_{\Lambda^{\widehat{\delta},\alpha}(V)}
\int_0^{\delta^2}\d t\int_0^{\delta}\frac{r^{1+\alpha}\d r}{(t+r^2)^2}
\leq C(\lambda)\cdot\|g\|_{\Lambda^{\widehat{\delta},\alpha}(V)}
\int_0^{\delta}\frac{\d r}{r^{1-\alpha}}\\
\leq C(\lambda)\cdot\|g\|_{\Lambda^{\widehat{\delta},\alpha}(V)}\delta^{\alpha}.
\end{multline}

\indent
For the sixth integral in \eqref{NLSubdivision} using
estimate \eqref{LDifferenceEstimate}
\begin{equation*}
\left|L_j(z^{(1)},w,\lambda)-L_j(z^{(2)},w,\lambda)\right|
\leq C(\lambda)\cdot\frac{\delta}{|B(z^{(1)},w)|^{5/2}},
\end{equation*}
we obtain the estimate
\begin{multline}\label{NLSubdivision6Estimate}
\Bigg|\int_{\{\Gamma^{\tau}_w\cap U^{\widehat{\nu}}_w(z^{(1)})\}
\setminus U^{\nu}_w(z^{(1)})}
\left(g(w)-\widehat{g}(z^{(1)},w)\right)E(w,\lambda)\\
\wedge\left(L^{(2,0)}_j(z^{(1)},w,\lambda)-L^{(2,0)}_j(z^{(2)},w,\lambda)\right)
\wedge\frac{\omega^{(2,0)}_j(w)}{P(w)}\Bigg|\\
\leq C(\lambda)\cdot\|g\|_{\Lambda^{\widehat{\delta},\alpha}(V)}
\delta\int_{\delta^2}^{(\widehat{\delta})^2}\d t
\int_{\delta}^{\widehat{\delta}}\frac{r^{1+\alpha}\d r}{(t+r^2)^{5/2}}
\leq C(\lambda)\cdot\|g\|_{\Lambda^{\widehat{\delta},\alpha}(V)}\delta
\int_{\delta}^{\widehat{\delta}}
\frac{r^{1+\alpha}\d r}{(\delta^2+r^2)^{3/2}}\\
\leq C(\lambda)\cdot\|g\|_{\Lambda^{\widehat{\delta},\alpha}(V)}\delta\int_{\delta}^{\widehat{\delta}} r^{-2+\alpha}\d r
\leq C(\lambda)\cdot\|g\|_{\Lambda^{\widehat{\delta},\alpha}(V)}\delta\delta^{-1+\alpha}
\leq C(\lambda)\cdot\|g\|_{\Lambda^{\widehat{\delta},\alpha}(V)}\delta^{\alpha}.
\end{multline}

\indent
Before considering the rest of the integrals in \eqref{NLSubdivision} we
use the Stokes' theorem for the form 
$$(\widehat{g}(z^{(2)},w,\lambda)-\widehat{g}(z^{(1)},w,\lambda))E(w,\lambda)
\wedge\partial_z G^{\widehat{\nu}}(z^{(2)},w,\lambda)$$
on the domain
$\{\Gamma^{\tau}_w\cap U^{\widehat{\nu}}_w(z^{(i)})\}\setminus U^{\nu}_w(z^{(i)})$
and obtain the following equality
\begin{multline*}
\int_{\{b\Gamma^{\tau}_w\cap U^{\widehat{\nu}}_w(z^{(2)})\}\setminus U^{\nu}_w(z^{(2)})}
(\widehat{g}(z^{(2)},w,\lambda)-\widehat{g}(z^{(1)},w,\lambda))E(w,\lambda)
\wedge\partial_z G^{\widehat{\nu}}(z^{(2)},w,\lambda)\\
-\int_{\Gamma^{\tau}_w\cap bU^{\widehat{\nu}}_w(z^{(2)})}
(\widehat{g}(z^{(2)},w,\lambda)-\widehat{g}(z^{(1)},w,\lambda))E(w,\lambda)
\wedge\partial_z G^{\widehat{\nu}}(z^{(2)},w,\lambda)\\
=\int_{\{\Gamma^{\tau}_w\cap U^{\widehat{\nu}}_w(z^{(2)})\}\setminus U^{\nu}_w(z^{(2)})}
\d_w\left[(\widehat{g}(z^{(2)},w,\lambda)-\widehat{g}(z^{(1)},w,\lambda))E(w,\lambda)
\wedge\partial_z G^{\widehat{\nu}}(z^{(2)},w,\lambda)\right]\\
=-\int_{\{\Gamma^{\tau}_w\cap U^{\widehat{\nu}}_w(z^{(2)})\}\setminus U^{\nu}_w(z^{(2)})}
(\widehat{g}(z^{(2)},w,\lambda)-\widehat{g}(z^{(1)},w,\lambda))
\wedge\bar\partial_w E(w,\lambda)
\wedge\partial_z G^{\widehat{\nu}}(z^{(2)},w,\lambda)\\
-\int_{\{\Gamma^{\tau}_w\cap U^{\widehat{\nu}}_w(z^{(2)})\}\setminus U^{\nu}_w(z^{(2)})}
(\widehat{g}(z^{(2)},w,\lambda)-\widehat{g}(z^{(1)},w,\lambda))E(w,\lambda)
\wedge\partial_z\bar\partial_w G^{\widehat{\nu}}(z^{(2)},w,\lambda),
\end{multline*}
which implies the following equality for the sum of the seventh, eighth and ninth integrals in \eqref{NLSubdivision}
\begin{multline}\label{NLSubdivisionLastTwoSmall}
\int_{\{b\Gamma^{\tau}_w\cap U^{\widehat{\nu}}_w(z^{(2)})\}\setminus U^{\nu}_w(z^{(2)})}
(\widehat{g}(z^{(2)},w,\lambda)-\widehat{g}(z^{(1)},w,\lambda))E(w,\lambda)
\wedge\partial_z G^{\widehat{\nu}}(z^{(2)},w,\lambda)\\
+\int_{\{\Gamma^{\tau}_w\cap U^{\widehat{\nu}}_w(z^{(2)})\}\setminus U^{\nu}_w(z^{(2)})}
(\widehat{g}(z^{(2)},w,\lambda)-\widehat{g}(z^{(1)},w,\lambda))E(w,\lambda)
\wedge\partial_z\bar\partial_w G^{\widehat{\nu}}(z^{(2)},w,\lambda)\\
+\int_{\{\Gamma^{\tau}_w\cap U^{\widehat{\nu}}_w(z^{(2)})\}\setminus U^{\nu}_w(z^{(2)})}
(\widehat{g}(z^{(2)},w,\lambda)-\widehat{g}(z^{(1)},w,\lambda))
\wedge\bar\partial_w E(w,\lambda)\wedge\partial_z G^{\widehat{\nu}}(z^{(2)},w,\lambda)\\
=\int_{\Gamma^{\tau}_w\cap bU^{\widehat{\nu}}_w(z^{(2)})}
(\widehat{g}(z^{(2)},w,\lambda)-\widehat{g}(z^{(1)},w,\lambda))E(w,\lambda)
\wedge\partial_z G^{\widehat{\nu}}(z^{(2)},w,\lambda).
\end{multline}

\indent
Then, using equality \eqref{PartialzG}, estimate \eqref{DeltaAreaEstimate} from Lemma~\ref{DeltaArea}, and estimate \eqref{NDeltaEstimate} from Lemma~\ref{NDeltaDomain} we obtain for the integral in the right-hand side of \eqref{NLSubdivisionLastTwoSmall}
\begin{multline}\label{SmallBoundaryEstimate}
\left|\int_{\Gamma^{\tau}_w\cap bU^{\widehat{\nu}}_w(z^{(2)})}
(\widehat{g}(z^{(2)},w,\lambda)-\widehat{g}(z^{(1)},w,\lambda))E(w,\lambda)
\wedge \partial_z G^{\widehat{\nu}}(z^{(2)},w,\lambda)\right|\\
\leq C\cdot\sum_{j=1,2}\left|\int_{\Gamma^{\tau}_w\cap bU_w^{\widehat{\nu}}(z^{(2)})}
(\widehat{g}(z^{(2)},w,\lambda)-\widehat{g}(z^{(1)},w,\lambda))E(w,\lambda)
\wedge N_j^{(2,0)}(z^{(2)},w,\lambda)\wedge\frac{\omega^{(2,0)}_j(w)}{P(w)}\right|\\
\leq C(\lambda)\cdot \delta^{\alpha}\|g\|_{\Lambda^{\widehat{\delta},\alpha}(V)}
\sum_{j=1,2}\lim_{\tau\to 0}\int_{\{|P(w)|=\tau,|B(z^{(2)},w)|=9(\widehat{\delta})^2\}}
\frac{\omega^{(2,0)}_j(w)}{|B(z^{(2)},w)|^{3/2}P(w)}\\
\leq C(\lambda)\cdot\delta^{\alpha}\|g\|_{\Lambda^{\widehat{\delta},\alpha}(V)},
\end{multline}
where from Lemma~\ref{DeltaArea} we have for
${\cal D}(\eta)=\left\{|w|=1,P(w)=0,\ |B(z,w)|=\eta\right\}$ estimate
\eqref{DeltaAreaEstimate}
\begin{equation*}
\text{Area}\left({\cal D}(\eta)\right)\leq C\cdot\eta^{3/2}
\end{equation*}
with some constant $C$.

\indent
So far, the constant $C(\lambda)$ in estimates
\eqref{NSmallwEstimate}$\div$\eqref{NLSubdivision6Estimate} and \eqref{SmallBoundaryEstimate} that was taken from the estimates in sections \ref{sec:BoundaryEstimates} and \ref{sec:DomainEstimates}
was valid for arbitrary $\alpha<1$ . Therefore, we obtain that the terms of \eqref{NLSubdivision} satisfying those estimates also satisfy the estimate \eqref{F1Estimate} for arbitrary $\alpha<1$
with a constant $C(\lambda)\to 0$ as $\lambda\to\infty$.
The rest of the integrals in \eqref{NLSubdivision}, containing the factor
$\bar\partial_w E(w,\lambda)$, require some additional consideration. Those terms are collected
in the following expression
\begin{multline}\label{TheRest}
\int_{\{\Gamma^{\tau}_w\cap U^{\widehat{\nu}}_w(z^{(2)})\}\setminus U^{\nu}_w(z^{(2)})}
\widehat{g}(z^{(1)},w)\wedge\bar\partial_w E(w,\lambda)
\wedge\left(\partial_z G^{\widehat{\nu}}(z^{(2)},w,\lambda)
-\partial_z G^{\widehat{\nu}}(z^{(1)},w,\lambda)\right)\\
-\int_{\Gamma^{\tau}_w\cap U^{\nu}_w(z^{(1)})}
\widehat{g}(z^{(1)},w)\wedge\bar\partial_w E(w,\lambda)
\wedge\partial_z G^{\widehat{\nu}}(z^{(1)},w,\lambda)\\
+\int_{\Gamma^{\tau}_w\cap U^{\nu}_w(z^{(2)})}
\widehat{g}(z^{(2)},w)\wedge\bar\partial_w E(w,\lambda)
\wedge\partial_z G^{\widehat{\nu}}(z^{(2)},w,\lambda).
\end{multline}

The following estimates show that the integrals in \eqref{TheRest} satisfy
the H$\ddot{\mathrm{\bf o}}$lder condition with any $\alpha\in(0,1)$, but with constants depending
on $\lambda$:
\begin{multline}\label{EDifferenceTerm}
\Bigg|\int_{\{\Gamma^{\tau}_w\cap U^{\widehat{\nu}}_w(z^{(2)})\}
\setminus U^{\nu}_w(z^{(2)})}\widehat{g}(z^{(1)},w)\wedge\bar\partial_w E(w,\lambda)
\wedge\left(\partial_z G^{\widehat{\nu}}(z^{(2)},w,\lambda)
-\partial_z G^{\widehat{\nu}}(z^{(1)},w,\lambda)\right)\Bigg|\\
\leq C\cdot\sum_{j=1,2}\Bigg|\lim_{\tau\to 0}
\int_{\{\Gamma^{\tau}_w\cap U^{\widehat{\nu}}_w(z^{(2)})\}\setminus U^{\nu}_w(z^{(2)})}
\widehat{g}(z^{(1)},w)\wedge\bar\partial_w E(w,\lambda)\\
\wedge\left(N_j(z^{(2)},w,\lambda)-N_j(z^{(1)},w,\lambda)\right)
\wedge\frac{\omega^{(2,0)}_j(w)}{P(w)}\Bigg|\\
\leq C(\lambda)\|g\|_{\Lambda^{\widehat{\delta},\alpha}(V)}\delta\cdot|\lambda|
\int_{\delta^2}^{(\widehat{\delta})^2}\d t
\int_{\delta}^{\widehat{\delta}}\frac{r\d r}{(t+r^2)^2}
\leq C(\lambda)\|g\|_{\Lambda^{\widehat{\delta},\alpha}(V)}
\delta\cdot |\lambda|\int_{\delta}^{\widehat{\delta}}\frac{\d r}{r}\\
\leq C(\lambda)\|g\|_{\Lambda^{\widehat{\delta},\alpha}(V)}\cdot|\lambda|\delta\log{\delta},
\end{multline}
where we used estimate \eqref{NDifferenceEstimate} from the Lemma~\ref{NDifferenceLemma}, and
\begin{multline}\label{SmallNbhdF}
\Bigg|\lim_{\tau\to 0}\int_{\Gamma^{\tau}_w\cap U^{\nu}_w(z^{(i)})}
\widehat{g}(z^{(i)},w)\wedge\bar\partial_w E(w,\lambda)
\wedge\partial_z G^{\widehat{\nu}}(z^{(i)},w,\lambda)\Bigg|\\
\leq C\cdot\sum_{j=1,2}\Bigg|\lim_{\tau\to 0}
\int_{\Gamma^{\tau}_w\cap U^{\nu}_w(z^{(i)})}
\widehat{g}(z^{(i)},w)\wedge\bar\partial_w E(w,\lambda)
\wedge N_j(z^{(i)},w,\lambda)\wedge\frac{\omega^{(2,0)}_j(w)}{P(w)}\Bigg|\\
\leq C(\lambda)\cdot\|g\|_{\Lambda^{\widehat{\delta},\alpha}(V)}|\lambda|\int_0^{\delta^2}\d t
\int_0^{\delta}\frac{r\d r}{(t+r^2)^{3/2}}
\leq C(\lambda)\cdot\|g\|_{\Lambda^{\widehat{\delta},\alpha}(V)}|\lambda|\int_0^{\delta}\d r\\
\leq C(\lambda)\cdot\|g\|_{\Lambda^{\widehat{\delta},\alpha}(V)}|\lambda|\delta,
\end{multline}
where we used estimate \eqref{NDeltaEstimate} from the Lemma~\ref{NDeltaDomain}.\\
\indent
Using estimates \eqref{EDifferenceTerm} and \eqref{SmallNbhdF} and imposing conditions
\eqref{alphadeltaCondition} on the values of  $\alpha$ and $\widehat{\delta}$ we obtain that 
\begin{equation}\label{alphadeltaInequality}
|\lambda|\delta\cdot\log{\delta}<|\lambda|\delta^{\alpha}\delta^{1/2}
<|\lambda|\delta^{\alpha}\widehat{\delta}^{1/2}<\delta^{\alpha},
\end{equation}
and, therefore, the constants $C(\lambda)$ from the estimates \eqref{NDifferenceEstimate} and \eqref{NDeltaEstimate} can be used in the estimate \eqref{F1Estimate}. Then, from the estimates
\eqref{EDifferenceTerm} and \eqref{SmallNbhdF}, we obtain the estimate
\begin{multline}\label{TheRestEstimate}
\Bigg|\int_{\{\Gamma^{\tau}_w\cap U^{\widehat{\nu}}_w(z^{(2)})\}
\setminus U^{\nu}_w(z^{(2)})}
\widehat{g}(z^{(1)},w)\wedge\bar\partial_w E(w,\lambda)
\wedge\left(\partial_z G^{\widehat{\nu}}(z^{(2)},w,\lambda)
-\partial_z G^{\widehat{\nu}}(z^{(1)},w,\lambda)\right)\\
-\int_{\Gamma^{\tau}_w\cap U^{\nu}_w(z^{(1)})}
\widehat{g}(z^{(1)},w)\wedge\bar\partial_w E(w,\lambda)
\wedge\partial_z G^{\widehat{\nu}}(z^{(1)},w,\lambda)\\
+\int_{\Gamma^{\tau}_w\cap U^{\nu}_w(z^{(2)})}
\widehat{g}(z^{(2)},w)\wedge\bar\partial_w E(w,\lambda)
\wedge\partial_z G^{\widehat{\nu}}(z^{(2)},w,\lambda)\Bigg|
\leq C(\lambda)\cdot\|g\|_{\Lambda^{\widehat{\delta},\alpha}(V)}\delta^{\alpha}
\end{multline}
for $\alpha$ and $\widehat{\delta}$ satisfying conditions \eqref{alphadeltaCondition}.

From the estimates \eqref{NSmallwEstimate}$\div$\eqref{NLSubdivision6Estimate},  \eqref{SmallBoundaryEstimate}, and \eqref{TheRestEstimate},
we obtain the statement of the Lemma~\ref{FHolder}.
\end{proof}

\begin{lemma}\label{QLemma} The operator $Q_{\lambda}$ defined in \ref{QOperator} is compact.
\end{lemma}
\begin{proof}
By the construction of the kernels
$$\partial_z \left(G(z,w,\lambda)-G^{\widehat{\nu}}(z,w,\lambda)\right)\ \text{and}\
\partial_z\bar\partial_w \left(G(z,w,\lambda)-G^{\widehat{\nu}}(z,w,\lambda)\right)$$
we obtain that those kernels are supported  in $U=\left\{(z,w): |B(z,w)|> \widehat{\nu}\right\}$. In addition, from Lemmas~\ref{NDifferenceLemma},\ref{NwDifferenceLemma} and
\ref{LDifferenceLemma},\ref{LwDifferenceLemma}, we obtain that those kernels are continuous in this domain. Then, using the Weierstrass' approximation theorem, for an arbitrary $\epsilon>0$ we can construct polynomial kernels $M(z,w,\lambda)$ and $K(z,w,\lambda)$ such that
\begin{equation}\label{Approximation}
\begin{aligned}
&\left\|\partial_z \left(G(z,w,\lambda)-G^{\widehat{\nu}}(z,w,\lambda)\right)-M(z,w,\lambda)
\right\|_{C(U)}\leq \epsilon,\\
&\left\|\partial_z\bar\partial_w \left(G(z,w,\lambda)-G^{\widehat{\nu}}(z,w,\lambda)\right)
-K(z,w,\lambda)\right\|_{C(U)}\leq \epsilon.
\end{aligned}
\end{equation}
Since the operators defined by the kernels $M(z,w,\lambda)$ and $K(z,w,\lambda)$ are
finite-dimensional, we obtain that operator ${\cal Q}_{\lambda}$ can be approximated by finite-dimensional operators, and is therefore compact.
\end{proof}

\indent
From the Lemmas~\ref{FCLemma}$\div$\ref{QLemma} we obtain that operator
$I+\cal{P}_{\lambda}$ on $\Lambda^{\widehat{\delta},\alpha}_{(1,0)}(V)$ admits the representation
$$I+\cal{P}_{\lambda}=\left(I+{\cal F}_{\lambda}\right)+{\cal Q}_{\lambda},$$
where ${\cal Q}_{\lambda}$ is a compact operator and $I+{\cal F}_{\lambda}$ is invertible
for sufficiently large $\lambda$, completing the proof of Proposition~\ref{PEquationFredholm}.\qed

\indent
Now, we intend to transform equality \eqref{hIntegralEquation} into an integral equation
with the operator $I+\cal{P}_{\lambda}$ applied to the form $\partial h$. The natural way of
doing that would be to apply the differentiation $\partial_{z}$ to both parts of equality \eqref{hIntegralEquation}. Then, using notation \eqref{ENotation} and its corollary
$e^{\left\langle\lambda,z/z_0\right\rangle-\overline{\left\langle\lambda,z/z_0\right\rangle}}
=E(z,-\lambda)$, we would obtain the following equation
\begin{multline}\label{FormIntegralEquation}
g(z,\lambda)=E(z,-\lambda)
\left\langle\lambda,\left(\d\frac{z_1}{z_0}+\d\frac{z_2}{z_0}\right)\right\rangle\\
-\lim_{\eta\to 0}\Bigg(\lim_{\tau\to 0}
\int_{\Gamma^{\tau}_w\cap \pi^{-1}(bV)\cap\{|B(z,w)|>\eta\}}
g(w,\lambda)\cdot E(w,\lambda)\wedge\partial_z G(z,w,\lambda)\\
-\lim_{\tau\to 0}\int_{\Gamma^{\tau}_w\cap\{|B(z,w)|>\eta\}}
g(w,\lambda)\cdot E(w,\lambda)\wedge\partial_z\bar\partial_w G(z,w,\lambda)\Bigg)\\
\end{multline}
for the form $g=\partial h$.

\indent
The validity of such step is proved in the following lemma.

\begin{lemma}\label{PartialzLemma}
The following equality holds for $g\in \Lambda^{\alpha}_{(1,0)}(V)$
\begin{multline}\label{PartialzEquality}
\partial_z\lim_{\eta\to 0}\Bigg(\lim_{\tau\to 0}
\int_{\Gamma^{\tau}_w\cap bV\cap\{|B(z,w)|>\eta\}}
g(w)\cdot E(w,\lambda)\wedge G(z,w,\lambda)\\
+\lim_{\tau\to 0}\int_{\Gamma^{\tau}_w\cap\{|B(z,w)|>\eta\}}
g(w)\cdot E(w,\lambda)\wedge\bar\partial_w G(z,w,\lambda)\Bigg)\\
=\lim_{\eta\to 0}\Bigg(\lim_{\tau\to 0}\int_{\Gamma^{\tau}_w\cap bV}
g(w)\cdot E(w,\lambda)\wedge\partial_z G(z,w,\lambda)\\
+\lim_{\tau\to 0}\int_{\Gamma^{\tau}_w\cap\{|B(z,w)|>\eta\}}
g(w)\cdot E(w,\lambda)\wedge\partial_z\bar\partial_w G(z,w,\lambda)\Bigg).
\end{multline}
\end{lemma}
\begin{proof}
\indent
For the first integral in the left-hand side of \eqref{PartialzEquality} considering a smooth form
$\phi^{(0,1)}$ with compact support and using equality
\begin{multline*}
\int_{V_{z}}\phi^{(0,1)}(z)\wedge\partial_z\left(\lim_{\tau\to 0}\int_{\Gamma^{\tau}_w\cap bV}
g(w)\cdot E(w,\lambda)\wedge G(z,w,\lambda)\right)\\
=\int_{V_{z}}\phi^{(0,1)}(z)\wedge\partial_z\lim_{\tau\to 0}\int_{\Gamma^{\tau}_w\cap bV}
g(w)\cdot E(w,\lambda)\wedge\Bigg\{|{\bf{C}}|^2\cdot\bar{z}_0^{-\ell}w_0^{\ell}\\
\bigwedge\sum_{j=1,2}\Bigg(\lim_{\delta\to 0}\lim_{\epsilon\to 0}\int_{\Gamma^{\epsilon}_{\zeta}
\cap\left\{|B(z,\zeta)|>\delta,|B(w,\zeta)|>\delta\right\}}
\bar\zeta_0^{\ell}\zeta_0^{-\ell}\vartheta(\zeta)
e^{\langle\lambda,\zeta/\zeta_0\rangle-\overline{\langle\lambda,\zeta/\zeta_0\rangle}}\\
\times\det\left[\frac{\bar\zeta}{B^*(w,\zeta)}\ \frac{\bar{w}}{B(w,\zeta)}\ Q(w,\zeta)\right]
\det\left[\frac{z}{B^*(\bar\zeta,\bar{z})}\ \frac{\zeta}
{B(\bar\zeta,\bar{z})}\ Q(\bar\zeta,\bar{z})\right]
\d\left(\frac{\zeta_j}{\zeta_0}\right)
\wedge\frac{\d\bar\zeta}{P(\bar\zeta)}\Bigg)
\wedge\frac{\omega^{(2,0)}_j(w)}{P(w)}\Bigg\}\\
=\int_{V_{z}}\partial\phi^{(0,1)}(z)\cdot\lim_{\tau\to 0}\int_{\Gamma^{\tau}_w\cap bV}
g(w)\cdot E(w,\lambda)\wedge\Bigg\{|{\bf{C}}|^2\cdot\bar{z}_0^{-\ell}
\cdot w_0^{\ell}\\
\bigwedge\sum_{j=1,2}\Bigg(\lim_{\delta\to 0}\lim_{\epsilon\to 0}\int_{\Gamma^{\epsilon}_{\zeta}
\cap\left\{|B(z,\zeta)|>\delta,|B(w,\zeta)|>\delta\right\}}
\bar\zeta_0^{\ell}\zeta_0^{-\ell}\vartheta(\zeta)
e^{\langle\lambda,\zeta/\zeta_0\rangle-\overline{\langle\lambda,\zeta/\zeta_0\rangle}}\\
\times\det\left[\frac{\bar\zeta}{B^*(w,\zeta)}\ \frac{\bar{w}}{B(w,\zeta)}\ Q(w,\zeta)\right]
\det\left[\frac{z}{B^*(\bar\zeta,\bar{z})}\ \frac{\zeta}
{B(\bar\zeta,\bar{z})}\ Q(\bar\zeta,\bar{z})\right]
\d\left(\frac{\zeta_j}{\zeta_0}\right)
\wedge\frac{\d\bar\zeta}{P(\bar\zeta)}\Bigg)
\wedge\frac{\omega^{(2,0)}_j(w)}{P(w)}\Bigg\}\\
\end{multline*}
we reduce the problem to the evaluation of the integral
\begin{multline}\label{PartialzBoundaryIntegral}
\int_{V_{z}}\bar{z}_0^{-\ell}\partial\phi^{(0,1)}(z)
\lim_{\tau\to 0}\int_{\Gamma^{\tau}_w\cap bV}g(w)\cdot E(w,\lambda)\wedge\frac{\omega^{(2,0)}_j(w)}{P(w)}\\
\bigwedge\sum_{j=1,2}\Bigg(\lim_{\delta\to 0}\lim_{\epsilon\to 0}\int_{\Gamma^{\epsilon}_{\zeta}
\cap\left\{|B(z,\zeta)|>\delta,|B(w,\zeta)|>\delta\right\}}A(\zeta,w)\\
\det\left[\frac{\bar\zeta}{B^*(w,\zeta)}\ \frac{\bar{w}}{B(w,\zeta)}\ Q(w,\zeta)\right]
\det\left[\frac{z}{B^*(\bar\zeta,\bar{z})}\ \frac{\zeta}
{B(\bar\zeta,\bar{z})}\ Q(\bar\zeta,\bar{z})\right]
\d\left(\frac{\zeta_j}{\zeta_0}\right)
\wedge\frac{\d\bar\zeta}{P(\bar\zeta)}\Bigg),\\
\end{multline}
where $A(\zeta,w)$ is a smooth function of $\zeta$ and $w$.

\indent
Then, dividing the interior integral with respect to $\zeta$ into three parts, we represent the integral in \eqref{PartialzBoundaryIntegral} as
\begin{multline}\label{PartialzBoundaryIntegralDivision}
\int_{V_{z}}\bar{z}_0^{-\ell}\partial\phi^{(0,1)}(z)
\lim_{\eta\to 0}\lim_{\tau\to 0}\int_{\Gamma^{\tau}_w\cap bV\cap\{|B(z,w)|>\eta\}}
g(w)\cdot E(w,\lambda)\wedge\frac{\omega^{(2,0)}_j(w)}{P(w)}\\
\bigwedge\sum_{j=1,2}\Bigg(\lim_{\delta\to 0}\lim_{\epsilon\to 0}\int_{\Gamma^{\epsilon}_{\zeta}
\cap\left\{|B(z,\zeta)|>\delta,|B(w,\zeta)|>\delta\right\}}A(\zeta,w)\\
\times\det\left[\frac{\bar\zeta}{B^*(w,\zeta)}\ \frac{\bar{w}}{B(w,\zeta)}\ Q(w,\zeta)\right]
\det\left[\frac{z}{B^*(\bar\zeta,\bar{z})}\ \frac{\zeta}
{B(\bar\zeta,\bar{z})}\ Q(\bar\zeta,\bar{z})\right]
\d\left(\frac{\zeta_j}{\zeta_0}\right)
\wedge\frac{\d\bar\zeta}{P(\bar\zeta)}\Bigg)\\
=\int_{V_{z}}\bar{z}_0^{-\ell}\partial\phi^{(0,1)}(z)
\lim_{\eta\to 0}\lim_{\tau\to 0}\int_{\Gamma^{\tau}_w\cap bV\cap\{|B(z,w)|>\eta\}}
g(w)\cdot E(w,\lambda)\wedge\frac{\omega^{(2,0)}_j(w)}{P(w)}\\
\bigwedge\sum_{j=1,2}\Bigg(\lim_{\delta\to 0}\lim_{\epsilon\to 0}\int_{\Gamma^{\epsilon}_{\zeta}
\cap\left\{\delta<|B(z,\zeta)|<\eta^6\right\}}A(\zeta,w)\\
\times\det\left[\frac{\bar\zeta}{B^*(w,\zeta)}\ \frac{\bar{w}}{B(w,\zeta)}\ Q(w,\zeta)\right]
\det\left[\frac{z}{B^*(\bar\zeta,\bar{z})}\ \frac{\zeta}
{B(\bar\zeta,\bar{z})}\ Q(\bar\zeta,\bar{z})\right]
\d\left(\frac{\zeta_j}{\zeta_0}\right)
\wedge\frac{\d\bar\zeta}{P(\bar\zeta)}\Bigg)\\
+\int_{V_{z}}\bar{z}_0^{-\ell}\partial\phi^{(0,1)}(z)
\lim_{\eta\to 0}\lim_{\tau\to 0}\int_{\Gamma^{\tau}_w\cap bV\cap\{|B(z,w)|>\eta\}}
g(w)\cdot E(w,\lambda)\wedge\frac{\omega^{(2,0)}_j(w)}{P(w)}\\
\bigwedge\sum_{j=1,2}\Bigg(\lim_{\delta\to 0}\lim_{\epsilon\to 0}\int_{\Gamma^{\epsilon}_{\zeta}
\cap\left\{\delta<|B(w,\zeta)|<\eta^6\right\}}A(\zeta,w)\\
\times\det\left[\frac{\bar\zeta}{B^*(w,\zeta)}\ \frac{\bar{w}}{B(w,\zeta)}\ Q(w,\zeta)\right]
\det\left[\frac{z}{B^*(\bar\zeta,\bar{z})}\ \frac{\zeta}
{B(\bar\zeta,\bar{z})}\ Q(\bar\zeta,\bar{z})\right]
\d\left(\frac{\zeta_j}{\zeta_0}\right)
\wedge\frac{\d\bar\zeta}{P(\bar\zeta)}\Bigg)\\
+\int_{V_{z}}\bar{z}_0^{-\ell}\partial\phi^{(0,1)}(z)
\lim_{\eta\to 0}\lim_{\tau\to 0}\int_{\Gamma^{\tau}_w\cap bV\cap\{|B(z,w)|>\eta\}}
g(w)\cdot E(w,\lambda)\wedge\frac{\omega^{(2,0)}_j(w)}{P(w)}\\
\bigwedge\sum_{j=1,2}\Bigg(\lim_{\epsilon\to 0}\int_{\Gamma^{\epsilon}_{\zeta}
\cap\left\{|B(z,\zeta)|>\eta^6,|B(w,\zeta)|>\eta^6\right\}}A(\zeta,w)\\
\times\det\left[\frac{\bar\zeta}{B^*(w,\zeta)}\ \frac{\bar{w}}{B(w,\zeta)}\ Q(w,\zeta)\right]
\det\left[\frac{z}{B^*(\bar\zeta,\bar{z})}\ \frac{\zeta}
{B(\bar\zeta,\bar{z})}\ Q(\bar\zeta,\bar{z})\right]
\d\left(\frac{\zeta_j}{\zeta_0}\right)
\wedge\frac{\d\bar\zeta}{P(\bar\zeta)}\Bigg).
\end{multline}

\indent
For the first two integrals in the right-hand side of \eqref{PartialzBoundaryIntegralDivision}
using estimates \eqref{GammaInequalities} from Lemma~\ref{GammaNeighborhoods}
we obtain the following estimate
\begin{multline}\label{PartialzBoundaryIntegralFirst}
\Bigg|\int_{V_{z}}\bar{z}_0^{-\ell}\partial\phi^{(0,1)}(z)
\lim_{\tau\to 0}\int_{\Gamma^{\tau}_w\cap bV\cap\{|B(z,w)|>\eta\}}
g(w)\cdot E(w,\lambda)\wedge\frac{\omega^{(2,0)}_j(w)}{P(w)}\\
\bigwedge\sum_{j=1,2}\Bigg(\lim_{\delta\to 0}\lim_{\epsilon\to 0}\int_{\Gamma^{\epsilon}_{\zeta}
\cap\left\{\delta<|B(z,\zeta)|<\eta^6\right\}}A(\zeta,w)\\
\times\det\left[\frac{\bar\zeta}{B^*(w,\zeta)}\ \frac{\bar{w}}{B(w,\zeta)}\ Q(w,\zeta)\right]
\det\left[\frac{z}{B^*(\bar\zeta,\bar{z})}\ \frac{\zeta}
{B(\bar\zeta,\bar{z})}\ Q(\bar\zeta,\bar{z})\right]
\d\left(\frac{\zeta_j}{\zeta_0}\right)
\wedge\frac{\d\bar\zeta}{P(\bar\zeta)}\Bigg)\\
+\int_{V_{z}}\bar{z}_0^{-\ell}\partial\phi^{(0,1)}(z)
\lim_{\tau\to 0}\int_{\Gamma^{\tau}_w\cap bV\cap\{|B(z,w)|>\eta\}}
g(w)\cdot E(w,\lambda)\wedge\frac{\omega^{(2,0)}_j(w)}{P(w)}\\
\bigwedge\sum_{j=1,2}\Bigg(\lim_{\delta\to 0}\lim_{\epsilon\to 0}\int_{\Gamma^{\epsilon}_{\zeta}
\cap\left\{\delta<|B(w,\zeta)|<\eta^6\right\}}A(\zeta,w)\\
\times\det\left[\frac{\bar\zeta}{B^*(w,\zeta)}\ \frac{\bar{w}}{B(w,\zeta)}\ Q(w,\zeta)\right]
\det\left[\frac{z}{B^*(\bar\zeta,\bar{z})}\ \frac{\zeta}
{B(\bar\zeta,\bar{z})}\ Q(\bar\zeta,\bar{z})\right]
\d\left(\frac{\zeta_j}{\zeta_0}\right)
\wedge\frac{\d\bar\zeta}{P(\bar\zeta)}\Bigg)\Bigg|\\
\leq C\cdot\left|\int_{V_{z}}\bar{z}_0^{-\ell}\partial\phi^{(0,1)}(z)
\lim_{\tau\to 0}\int_{\Gamma^{\tau}_w\cap bV}g(w)\cdot E(w,\lambda)\wedge\frac{\omega^{(2,0)}_j(w)}{P(w)}
\int_0^{\eta^6}\d t\int_0^{\eta^3}\frac{r^2\d r}{(t+r^2)^2(\eta+r^2)^2}\right|\\
\leq C\cdot\int_0^{\eta^3}\frac{\d r}{(\eta+r^2)^2}\leq C\cdot\eta\to 0
\end{multline}
as $\eta\to 0$.

\indent
For the third integral in the right-hand side of \eqref{PartialzBoundaryIntegralDivision}
using the smoothness of the functions
\begin{multline*}
\lim_{\epsilon\to 0}\int_{\Gamma^{\epsilon}_{\zeta}
\cap\left\{|B(z,\zeta)|>\eta^6,|B(w,\zeta)|>\eta^6\right\}}A(\zeta,w)\\
\times\det\left[\frac{\bar\zeta}{B^*(w,\zeta)}\ \frac{\bar{w}}{B(w,\zeta)}\ Q(w,\zeta)\right]
\det\left[\frac{z}{B^*(\bar\zeta,\bar{z})}\ \frac{\zeta}
{B(\bar\zeta,\bar{z})}\ Q(\bar\zeta,\bar{z})\right]
\d\left(\frac{\zeta_j}{\zeta_0}\right)
\wedge\frac{\d\bar\zeta}{P(\bar\zeta)}
\end{multline*}
with respect to $z$ and $w$ we obtain the equality
\begin{multline}\label{PartialzBoundaryIntegralSecond}
\int_{V_{z}}\bar{z}_0^{-\ell}\partial\phi^{(0,1)}(z)
\lim_{\eta\to 0}\lim_{\tau\to 0}\int_{\Gamma^{\tau}_w\cap bV\cap\{|B(z,w)|>\eta\}}
g(w)\cdot E(w,\lambda)\wedge\frac{\omega^{(2,0)}_j(w)}{P(w)}\\
\bigwedge\sum_{j=1,2}\Bigg(\lim_{\epsilon\to 0}\int_{\Gamma^{\epsilon}_{\zeta}
\cap\left\{|B(z,\zeta)|>\eta^6,|B(w,\zeta)|>\eta^6\right\}}A(\zeta,w)\\
\times\det\left[\frac{\bar\zeta}{B^*(w,\zeta)}\ \frac{\bar{w}}{B(w,\zeta)}\ Q(w,\zeta)\right]
\det\left[\frac{z}{B^*(\bar\zeta,\bar{z})}\ \frac{\zeta}
{B(\bar\zeta,\bar{z})}\ Q(\bar\zeta,\bar{z})\right]
\d\left(\frac{\zeta_j}{\zeta_0}\right)
\wedge\frac{\d\bar\zeta}{P(\bar\zeta)}\Bigg)\\
=\lim_{\eta\to 0}\int_{V_{z}}\bar{z}_0^{-\ell}\partial\phi^{(0,1)}(z)
\lim_{\tau\to 0}\int_{\Gamma^{\tau}_w\cap bV\cap\{|B(z,w)|>\eta\}}
g(w)\cdot E(w,\lambda)\wedge\frac{\omega^{(2,0)}_j(w)}{P(w)}\\
\bigwedge\sum_{j=1,2}\Bigg(\lim_{\epsilon\to 0}\int_{\Gamma^{\epsilon}_{\zeta}
\cap\left\{|B(z,\zeta)|>\eta^6,|B(w,\zeta)|>\eta^6\right\}}A(\zeta,w)\\
\times\det\left[\frac{\bar\zeta}{B^*(w,\zeta)}\ \frac{\bar{w}}{B(w,\zeta)}\ Q(w,\zeta)\right]
\det\left[\frac{z}{B^*(\bar\zeta,\bar{z})}\ \frac{\zeta}
{B(\bar\zeta,\bar{z})}\ Q(\bar\zeta,\bar{z})\right]
\d\left(\frac{\zeta_j}{\zeta_0}\right)
\wedge\frac{\d\bar\zeta}{P(\bar\zeta)}\Bigg)\\
=\lim_{\eta\to 0}\int_{V_{z}}\bar{z}_0^{-\ell}\phi^{(0,1)}(z)
\lim_{\tau\to 0}\int_{\Gamma^{\tau}_w\cap bV\cap\{|B(z,w)|>\eta\}}
g(w)\cdot E(w,\lambda)\wedge\frac{\omega^{(2,0)}_j(w)}{P(w)}\\
\bigwedge\sum_{j=1,2}\Bigg(\lim_{\epsilon\to 0}\int_{\Gamma^{\epsilon}_{\zeta}
\cap\left\{|B(z,\zeta)|>\eta^6,|B(w,\zeta)|>\eta^6\right\}}A(\zeta,w)\\
\times\det\left[\frac{\bar\zeta}{B^*(w,\zeta)}\ \frac{\bar{w}}{B(w,\zeta)}\ Q(w,\zeta)\right]
\det\left[\partial_z\left(\frac{z}{B^*(\bar\zeta,\bar{z})}\right)\ \frac{\zeta}
{B(\bar\zeta,\bar{z})}\ Q(\bar\zeta,\bar{z})\right]
\d\left(\frac{\zeta_j}{\zeta_0}\right)
\wedge\frac{\d\bar\zeta}{P(\bar\zeta)}\Bigg).
\end{multline}

\indent
Combining estimate \eqref{PartialzBoundaryIntegralFirst} with
equality \eqref{PartialzBoundaryIntegralSecond} we obtain the equality
\begin{multline}\label{PartialzBoundary}
\partial_z\lim_{\eta\to 0}\lim_{\tau\to 0}
\int_{\Gamma^{\tau}_w\cap bV\cap\{|B(z,w)|>\eta\}}
g(w)\cdot E(w,\lambda)\wedge G(z,w,\lambda)\\
=\lim_{\eta\to 0}\lim_{\tau\to 0}\int_{\Gamma^{\tau}_w\cap bV\cap\{|B(z,w)|>\eta\}}
g(w)\cdot E(w,\lambda)\wedge\partial_z G(z,w,\lambda),
\end{multline}
which proves equality \eqref{PartialzEquality} for the boundary integral.

\indent
For the second integral in the left-hand side of equality \eqref{PartialzEquality} using
Lemma~\ref{ChangeTwoDeterminants} we obtain
\begin{multline}\label{PartialzDomainIntegral}
\int_{V_{z}}\phi^{(0,1)}(z)\wedge\partial_z\left(\lim_{\eta\to 0}\lim_{\tau\to 0}
\int_{\Gamma^{\tau}_w\cap\{|B(z,w)|>\eta\}}g(w)\cdot E(w,\lambda)
\wedge\bar\partial_w G(z,w,\lambda)\right)\\
=\int_{V_{z}}\phi^{(0,1)}(z)\wedge\partial_z\lim_{\eta\to 0}\lim_{\tau\to 0}
\int_{\Gamma^{\tau}_w\cap\{|B(z,w)|>\eta\}}g(w)\cdot E(w,\lambda)
\wedge\Bigg\{|{\bf{C}}|^2\cdot\bar{z}_0^{-\ell}\cdot w_0^{\ell}\\
\bigwedge\sum_{j=1,2}\Bigg(\lim_{\delta\to 0}\lim_{\epsilon\to 0}\int_{\Gamma^{\epsilon}_{\zeta}
\cap\left\{|B(z,\zeta)|>\delta,|B(w,\zeta)|>\delta\right\}}
\bar\zeta_0^{\ell}\zeta_0^{-\ell}\vartheta(\zeta)
e^{\langle\lambda,\zeta/\zeta_0\rangle-\overline{\langle\lambda,\zeta/\zeta_0\rangle}}\\
\times\det\left[\frac{\bar{w}}{B(w,\zeta)}\ \frac{\d\bar{w}}{B(w,\zeta)}\ Q(w,\zeta)\right]
\det\left[\frac{z}{B^*(\bar\zeta,\bar{z})}\ \frac{\zeta}
{B(\bar\zeta,\bar{z})}\ Q(\bar\zeta,\bar{z})\right]
\d\left(\frac{\zeta_j}{\zeta_0}\right)
\wedge\frac{\d\bar\zeta}{P(\bar\zeta)}\Bigg)
\wedge\frac{\omega^{(2,0)}_j(w)}{P(w)}\Bigg\}\\
=\int_{V_{z}}\bar{z}_0^{-\ell}\partial\phi^{(0,1)}(z)\lim_{\eta\to 0}\lim_{\tau\to 0}
\int_{\Gamma^{\tau}_w\cap\{|B(z,w)|>\eta\}}
g(w)\cdot E(w,\lambda)\wedge\frac{\omega^{(2,0)}_j(w)}{P(w)}\\
\bigwedge\sum_{j=1,2}\Bigg(\lim_{\delta\to 0}\lim_{\epsilon\to 0}\int_{\Gamma^{\epsilon}_{\zeta}
\cap\left\{|B(z,\zeta)|>\delta,|B(w,\zeta)|>\delta\right\}}A(\zeta,w)\\
\times\det\left[\frac{\bar{w}}{B(w,\zeta)}\ \frac{\d\bar{w}}{B(w,\zeta)}\ Q(w,\zeta)\right]
\det\left[\frac{z}{B^*(\bar\zeta,\bar{z})}\ \frac{\zeta}
{B(\bar\zeta,\bar{z})}\ Q(\bar\zeta,\bar{z})\right]
\d\left(\frac{\zeta_j}{\zeta_0}\right)
\wedge\frac{\d\bar\zeta}{P(\bar\zeta)}\Bigg),\\
\end{multline}
where $A(\zeta,w)$ is a smooth function of $\zeta$ and $w$.

\indent
Then, dividing the interior integral with respect to $\zeta$ into three parts, we represent the integral in the right-hand side of \eqref{PartialzDomainIntegral} as
\begin{multline}\label{PartialzDomainIntegralDivision}
\int_{V_{z}}\bar{z}_0^{-\ell}\partial\phi^{(0,1)}(z)\lim_{\eta\to 0}\lim_{\tau\to 0}
\int_{\Gamma^{\tau}_w\cap\{|B(z,w)|>\eta\}}
g(w)\cdot E(w,\lambda)\wedge\frac{\omega^{(2,0)}_j(w)}{P(w)}\\
\bigwedge\sum_{j=1,2}\Bigg(\lim_{\delta\to 0}\lim_{\epsilon\to 0}\int_{\Gamma^{\epsilon}_{\zeta}
\cap\left\{|B(z,\zeta)|>\delta,|B(w,\zeta)|>\delta\right\}}A(\zeta,w)\\
\times\det\left[\frac{\bar{w}}{B(w,\zeta)}\ \frac{\d\bar{w}}{B(w,\zeta)}\ Q(w,\zeta)\right]
\det\left[\frac{z}{B^*(\bar\zeta,\bar{z})}\ \frac{\zeta}
{B(\bar\zeta,\bar{z})}\ Q(\bar\zeta,\bar{z})\right]
\d\left(\frac{\zeta_j}{\zeta_0}\right)
\wedge\frac{\d\bar\zeta}{P(\bar\zeta)}\Bigg)\\
=\int_{V_{z}}\bar{z}_0^{-\ell}\partial\phi^{(0,1)}(z)
\lim_{\eta\to 0}\lim_{\tau\to 0}\int_{\Gamma^{\tau}_w\cap\{|B(z,w)|>\eta\}}
g(w)\cdot E(w,\lambda)\wedge\frac{\omega^{(2,0)}_j(w)}{P(w)}\\
\bigwedge\sum_{j=1,2}\Bigg(\lim_{\delta\to 0}\lim_{\epsilon\to 0}\int_{\Gamma^{\epsilon}_{\zeta}
\cap\left\{\delta<|B(z,\zeta)|<\eta^6\right\}}A(\zeta,w)\\
\times\det\left[\frac{\bar{w}}{B(w,\zeta)}\ \frac{\d\bar{w}}{B(w,\zeta)}\ Q(w,\zeta)\right]
\det\left[\frac{z}{B^*(\bar\zeta,\bar{z})}\ \frac{\zeta}
{B(\bar\zeta,\bar{z})}\ Q(\bar\zeta,\bar{z})\right]
\d\left(\frac{\zeta_j}{\zeta_0}\right)
\wedge\frac{\d\bar\zeta}{P(\bar\zeta)}\Bigg)\\
+\int_{V_{z}}\bar{z}_0^{-\ell}\partial\phi^{(0,1)}(z)
\lim_{\eta\to 0}\lim_{\tau\to 0}\int_{\Gamma^{\tau}_w\cap\{|B(z,w)|>\eta\}}
g(w)\cdot E(w,\lambda)\wedge\frac{\omega^{(2,0)}_j(w)}{P(w)}\\
\bigwedge\sum_{j=1,2}\Bigg(\lim_{\delta\to 0}\lim_{\epsilon\to 0}\int_{\Gamma^{\epsilon}_{\zeta}
\cap\left\{\delta<|B(w,\zeta)|<\eta^6\right\}}A(\zeta,w)\\
\times\det\left[\frac{\bar{w}}{B(w,\zeta)}\ \frac{\d\bar{w}}{B(w,\zeta)}\ Q(w,\zeta)\right]
\det\left[\frac{z}{B^*(\bar\zeta,\bar{z})}\ \frac{\zeta}
{B(\bar\zeta,\bar{z})}\ Q(\bar\zeta,\bar{z})\right]
\d\left(\frac{\zeta_j}{\zeta_0}\right)
\wedge\frac{\d\bar\zeta}{P(\bar\zeta)}\Bigg)\\
+\int_{V_{z}}\bar{z}_0^{-\ell}\partial\phi^{(0,1)}(z)
\lim_{\eta\to 0}\lim_{\tau\to 0}\int_{\Gamma^{\tau}_w\cap\{|B(z,w)|>\eta\}}
g(w)\cdot E(w,\lambda)\wedge\frac{\omega^{(2,0)}_j(w)}{P(w)}\\
\bigwedge\sum_{j=1,2}\Bigg(\lim_{\epsilon\to 0}\int_{\Gamma^{\epsilon}_{\zeta}
\cap\left\{|B(z,\zeta)|>\eta^6,|B(w,\zeta)|>\eta^6\right\}}A(\zeta,w)\\
\times\det\left[\frac{\bar{w}}{B(w,\zeta)}\ \frac{\d\bar{w}}{B(w,\zeta)}\ Q(w,\zeta)\right]
\det\left[\frac{z}{B^*(\bar\zeta,\bar{z})}\ \frac{\zeta}
{B(\bar\zeta,\bar{z})}\ Q(\bar\zeta,\bar{z})\right]
\d\left(\frac{\zeta_j}{\zeta_0}\right)
\wedge\frac{\d\bar\zeta}{P(\bar\zeta)}\Bigg).
\end{multline}
\indent
For the first integral in the right-hand side of \eqref{PartialzDomainIntegralDivision} 
we have, similarly
to \eqref{PartialzBoundaryIntegralFirst}, the estimate
\begin{multline}\label{PartialzDomainIntegralFirst}
\bigg|\lim_{\tau\to 0}\int_{\Gamma^{\tau}_w\cap\{|B(z,w)|>\eta\}}
g(w)\cdot E(w,\lambda)\wedge\frac{\omega^{(2,0)}_j(w)}{P(w)|B(z,w)|^{3/2}}\\
\bigwedge\sum_{j=1,2}\Bigg(\lim_{\epsilon\to 0}\int_{\Gamma^{\epsilon}_{\zeta}
\cap\left\{|B(z,\zeta)|<\eta^6\right\}}A(\zeta,w)|B(z,w)|^{3/2}\\
\det\left[\frac{\bar{w}}{B(w,\zeta)}\ \frac{\d\bar{w}}{B(w,\zeta)}\ Q(w,\zeta)\right]
\det\left[\frac{z}{B^*(\bar\zeta,\bar{z})}\ \frac{\zeta}
{B(\bar\zeta,\bar{z})}\ Q(\bar\zeta,\bar{z})\right]
\d\left(\frac{\zeta_j}{\zeta_0}\right)
\wedge\frac{\d\bar\zeta}{P(\bar\zeta)}\Bigg)\Bigg|\\
\leq C\cdot\int_{\eta}^A\d v\int_{\sqrt{\eta}}^B\frac{s\d s}{(v+s^2)^{3/2}}
\int_0^{\eta^6}\d t\int_0^{\eta^3}\frac{\gamma^{3/2}r^2\d r}{(t+r^2)^2(\gamma+r^2)^2}
\leq C\cdot\frac{\eta^{3}}{\gamma^{1/2}}\to 0,
\end{multline}
as $\eta\to 0$,  where we denoted $B(z,w)=iv+s^2$, $B(z,\zeta)=it+r^2$, and used estimates
$|B(w,\zeta)|>\gamma+r^2$ in $|B(z,\zeta)|<\eta^6$, and $\gamma=|B(z,w)|>\eta$.

\indent
For the second term of the right-hand side of \eqref{PartialzDomainIntegralDivision},
we use the following variant of Lemma~\ref{NLwKappaDomain} with $p=1$:

\begin{lemma}\label{NLwEtaDomain}
There exists a constant $C$, such that the following estimate holds for
$\kappa\leq\gamma=|B(z,w)|$
\begin{equation}\label{NLwEtaDEstimate}
\left|\int_{\pi^{-1}({\cal C}) \cap\left\{|B(w,\zeta|<\kappa\right\}}
\frac{\Phi(z,w,\lambda,\zeta)\d(\zeta_j/\zeta_0)
\wedge\left(\d{\bar P}(\zeta)\interior\d\bar\zeta\right)}
{B(w,\zeta)^2}\right|\leq C\cdot\frac{\sqrt{\kappa}}{\gamma},
\end{equation}
where
$$\Phi(z,w,\lambda,\zeta)=\psi(z,w,\lambda,\zeta)
\frac{B(z,w)^{3/2}(\zeta-z)}{B^*({\bar\zeta},{\bar z})B(\bar\zeta,\bar{z})}.$$
\end{lemma}
\begin{proof}
As in Lemma~\ref{NLwKappaDomain} we use equality \eqref{pOneDecomposition}
\begin{multline*}
\d\bar\zeta=\d{\bar P}(\zeta)\wedge\d_{\zeta}{\bar F}(z,\zeta)
\wedge\left(\sum_{j=0}^2\zeta_j\d{\bar\zeta}_j\right)\\
=\d{\bar P}(\zeta)\wedge\d_{\zeta}{\bar F}(z,\zeta)\wedge\d_{\zeta}B(w,\zeta)
-\d{\bar P}(\zeta)\wedge\d_{\zeta}{\bar F}(z,\zeta)
\wedge\left(\sum_{j=0}^2(\bar\zeta_j-\bar{w}_j)\d\zeta_j\right)
\end{multline*}
and consider separately both expressions in the right-hand side.\\
\indent
For the second term of the right-hand side of equality above we have
the following estimate
\begin{multline}\label{NLwEtaDomainEstimatepOne}
\Bigg|\int_{\pi^{-1}({\cal C}) \cap\left\{|B(w,\zeta|<\kappa\right\}}
\frac{B(z,w)^{3/2}(\zeta-z)\left(\sum_{j=0}^2(\bar\zeta_j-\bar{w}_j)\d\zeta_j\right)\d(\zeta_j/\zeta_0)
\wedge\left(\d{\bar P}(\zeta)\interior\d\bar\zeta\right)}
{B^*({\bar\zeta},{\bar z})B(\bar\zeta,\bar{z})B(w,\zeta)^2}\Bigg|\\
\leq C(\lambda)\cdot\gamma^{3/2}\int_0^{\kappa}\d t \int_0^{\sqrt{\kappa}}
\frac{r^2\d r}{(\gamma+r^2)^{3/2}(t+r^2)^2}
\leq C(\lambda)\cdot\gamma^{3/2}\int_0^{\sqrt{\kappa}}\frac{\d r}{(\gamma+r^2)^{3/2}}
\leq C(\lambda)\cdot\sqrt{\kappa},
\end{multline}
which implies estimate \eqref{NLwEtaDEstimate} for this term.

For the first term of the right-hand side of \eqref{pOneDecomposition} we
use the Stokes' theorem to obtain the following equality
\begin{multline}\label{NLwEtaDomainEquality}
\int_{\pi^{-1}({\cal C}) \cap\left\{|B(w,\zeta|<\kappa\right\}}
\frac{\Phi(z,w,\lambda,\zeta)}{B(w,\zeta)^2}
\d(\zeta_j/\zeta_0)\wedge\left(\d{\bar P}(\zeta)\interior\d\bar\zeta\right)\\
=\int_{\pi^{-1}({\cal C}) \cap\left\{|B(w,\zeta|<\kappa\right\}}
\Phi(z,w,\lambda,\zeta)\Big(\d_{\zeta}B(w,\zeta)
\interior\Big(\d\left(\zeta_j/\zeta_0\right)\wedge\d\bar\zeta\Big)
\wedge\d_{\zeta}\left(\frac{1}{B(w,\zeta)}\right)\\
=\int_{\pi^{-1}({\cal C}) \cap\left\{|B(w,\zeta|=\kappa\right\}}
\Phi(z,w,\zeta,\lambda)\frac{ \Big(\d_{\zeta}B(w,\zeta)
\interior\Big(\d\left(\zeta_j/\zeta_0\right)\wedge\d\bar\zeta\Big)}{B(w,\zeta)}\\
-\lim_{\nu\to 0}\int_{\pi^{-1}({\cal C}) \cap\left\{|B(w,\zeta|=\nu\right\}}
\Phi(z,w,\zeta,\lambda)\frac{ \Big(\d_{\zeta}B(w,\zeta)
\interior\Big(\d\left(\zeta_j/\zeta_0\right)\wedge\d\bar\zeta\Big)}{B(w,\zeta)}\\
-\int_{\pi^{-1}({\cal C}) \cap\left\{|B(w,\zeta|<\kappa\right\}}
\d_{\zeta}\Phi(z,w,\zeta,\lambda)\wedge\frac{ \Big(\d_{\zeta}B(w,\zeta)
\interior\Big(\d\left(\zeta_j/\zeta_0\right)\wedge\d\bar\zeta\Big)}{B(w,\zeta)}.
\end{multline}
\indent
For the first integral in the right-hand side of \eqref{NLwEtaDomainEquality} we use estimate \eqref{DeltaAreaEstimate}
\begin{equation*}
A(\kappa)\sim C\cdot\kappa^{3/2}
\end{equation*}
of the area of integration $\left\{\pi^{-1}({\cal C})\cap|B(w,\zeta)|=\kappa\right\}$ in this integral
and the boundedness of the function $\Phi(z,w,\lambda,\zeta)$ on
$\left\{\nu<|B(w,\zeta|<\kappa\right\}$, which follows from the estimates
\eqref{GammaInequalities}. Then, we obtain the following estimate
\begin{equation}\label{NLwEtaDomainEstimateFirstOne}
\Bigg|\int_{\pi^{-1}({\cal C}) \cap\left\{|B(w,\zeta|=\kappa\right\}}
\Phi(z,w,\zeta,\lambda)\frac{ \Big(\d_{\zeta}B(w,\zeta)
\interior\Big(\d\left(\zeta_j/\zeta_0\right)\wedge\d\bar\zeta\Big)}{B(w,\zeta)}
\leq C\cdot\sqrt{\kappa}.
\end{equation}
\indent
For the second integral in the right-hand side of \eqref{NLwEtaDomainEquality} using estimates
\eqref{DeltaAreaEstimate} and \eqref{NLwEtaDomainEstimateFirstOne} we obtain the estimate
\begin{equation}\label{NLwEtaDomainEstimateSecondOne}
\Bigg|\int_{\pi^{-1}({\cal C}) \cap\left\{|B(w,\zeta|=\nu\right\}}
\Phi(z,w,\zeta,\lambda)\frac{ \Big(\d_{\zeta}B(w,\zeta)
\interior\Big(\d\left(\zeta_j/\zeta_0\right)\wedge\d\bar\zeta\Big)}{B(w,\zeta)}\Bigg|
\leq C\cdot\sqrt{\nu}\to 0,
\end{equation}
as $\nu\to 0$.\\
\indent
For the third integral in the right-hand side of \eqref{NLwEtaDomainEquality}
using equalities \eqref{DExponent} we obtain the estimate
\begin{multline}\label{PhiOneEtaDecomposition}
\d_{\zeta}\left(\Phi(z,w,\lambda,\zeta)=\psi(z,w,\lambda,\zeta)
\frac{B(z,w)^{3/2}(\zeta-z)}{B^*({\bar\zeta},{\bar z})B(\bar\zeta,\bar{z})}\right)
\frac{1}{B(w,\zeta)}\\
\sim\frac{\gamma^{3/2}}{B(w,\zeta)}\left(\frac{\d\zeta}{B^*({\bar\zeta},{\bar z})B(\bar\zeta,\bar{z})}
+\frac{\d_{\zeta}B^*({\bar\zeta},{\bar z})}
{B^*({\bar\zeta},{\bar z})^2B(\bar\zeta,\bar{z})}
+\frac{\d_{\zeta}B({\bar\zeta},{\bar z})}
{B^*({\bar\zeta},{\bar z})B(\bar\zeta,\bar{z})^2}\right),
\end{multline}
and then, using coordinates $t=\text{Im}B(w,\zeta),\ r=|w-\zeta|$, and inequalities
$$|B(z,\zeta)|>\gamma+|\zeta-w|^2,\ \gamma<|B(z,\zeta)|,\ 
\gamma^{1/2}r\leq \gamma+r^2$$
we obtain
\begin{multline}\label{NLwEtaDomainEstimateThirdOne}
\Bigg|\lim_{\eta\to 0}\int_{\pi^{-1}({\cal C}) \cap\left\{|B(w,\zeta|<\kappa\right\}}
\d_{\zeta}\Phi(z,w,\zeta,\lambda)\wedge\frac{\d\left(\zeta_j/\zeta_0\right)
\wedge\Big(\left(\bar\partial_{\zeta}B^*(w,\zeta)\wedge\d{\bar P}(\zeta)\right)
\interior\d\bar\zeta\Big)}{B(w,\zeta)}\Bigg|\\
\leq C\cdot\gamma^{3/2}\int_0^{\kappa}\d t \int_0^{\sqrt{\kappa}}
\frac{r\d r}{(\gamma+r^2)^3(t+r^2)}
\leq C\cdot\gamma^{-1}\int_0^{\kappa}\d t \int_0^{\sqrt{\kappa}}
\frac{\d r}{t+r^2}\\
\leq C\cdot\gamma^{-1}\int_0^{\kappa}\frac{\d t}{\sqrt{t}}
\int_0^{\sqrt{\frac{\kappa}{t}}}\frac{\d u}{1+u^2}
\leq C\cdot\frac{\sqrt{\kappa}}{\gamma}.
\end{multline}
\indent
Combining estimates \eqref{NLwEtaDomainEstimatepOne},
\eqref{NLwEtaDomainEstimateFirstOne}, \eqref{NLwEtaDomainEstimateSecondOne}, and \eqref{NLwEtaDomainEstimateThirdOne}
we obtain the estimate \eqref{NLwEtaDEstimate}.
\end{proof}

\indent
Using estimate \eqref{NLwEtaDEstimate} with $\kappa=\eta^6$ and $\gamma>\eta$
in the second term of the right-hand side of \eqref{PartialzDomainIntegralFirst} we obtain
\begin{multline}\label{PartialzDomainIntegralSecond}
\bigg|\lim_{\tau\to 0}\int_{\Gamma^{\tau}_w\cap\{|B(z,w)|>\eta\}}g(w)
\wedge\frac{\omega^{(2,0)}_j(w)}{P(w)B(z,w)^{3/2}}
\bigwedge\sum_{j=1,2}\Bigg(\lim_{\epsilon\to 0}\int_{\Gamma^{\epsilon}_{\zeta}
\cap\left\{|B(w,\zeta)|<\eta^6\right\}}A(\zeta,w)B(z,w)^{3/2}\\
\det\left[\frac{\bar{w}}{B(w,\zeta)}\ \frac{\d\bar{w}}{B(w,\zeta)}\ Q(w,\zeta)\right]
\det\left[\frac{z}{B^*(\bar\zeta,\bar{z})}\ \frac{\zeta}
{B(\bar\zeta,\bar{z})}\ Q(\bar\zeta,\bar{z})\right]
\d\left(\frac{\zeta_j}{\zeta_0}\right)
\wedge\frac{\d\bar\zeta}{P(\bar\zeta)}\Bigg)\Bigg|\\
\leq C\cdot\int_{\eta}^A\d v\int_{\sqrt{\eta}}^B\frac{s\d s}{(v+s^2)^{3/2}}
\frac{\eta^{3/2}}{\gamma}\leq C\cdot\frac{\eta^{3/2}}{\eta}\to 0,
\end{multline}
as $\eta\to 0$, where we used the same notations as in \eqref{PartialzDomainIntegralFirst}.\\

\indent
For the third integral in the right-hand side of \eqref{PartialzDomainIntegralDivision} using the smoothness of the function
\begin{multline*}
\lim_{\epsilon\to 0}\int_{\Gamma^{\epsilon}_{\zeta}
\cap\left\{|B(z,\zeta)|>\eta^6,|B(w,\zeta)|>\eta^6\right\}}A(\zeta,w)\\
\times\det\left[\frac{\bar{w}}{B(w,\zeta)}\ \frac{\d\bar{w}}{B(w,\zeta)}\ Q(w,\zeta)\right]
\det\left[\frac{z}{B^*(\bar\zeta,\bar{z})}\ \frac{\zeta}
{B(\bar\zeta,\bar{z})}\ Q(\bar\zeta,\bar{z})\right]
\d\left(\frac{\zeta_j}{\zeta_0}\right)
\wedge\frac{\d\bar\zeta}{P(\bar\zeta)}
\end{multline*}
with respect to $z$ and $w$ we obtain the equality
\begin{multline}\label{PartialzDomainIntegralThird}
\int_{V_{z}}\bar{z}_0^{-\ell}\partial\phi^{(0,1)}(z)
\lim_{\eta\to 0}\lim_{\tau\to 0}\int_{\Gamma^{\tau}_w\cap bV\cap\{|B(z,w)|>\eta\}}
g(w)\cdot E(w,\lambda)\wedge\frac{\omega^{(2,0)}_j(w)}{P(w)}\\
\bigwedge\sum_{j=1,2}\Bigg(\lim_{\epsilon\to 0}\int_{\Gamma^{\epsilon}_{\zeta}
\cap\left\{|B(z,\zeta)|>\eta^6,|B(w,\zeta)|>\eta^6\right\}}A(\zeta,w)\\
\times\det\left[\frac{\bar{w}}{B(w,\zeta)}\ \frac{\d\bar{w}}{B(w,\zeta)}\ Q(w,\zeta)\right]
\det\left[\frac{z}{B^*(\bar\zeta,\bar{z})}\ \frac{\zeta}
{B(\bar\zeta,\bar{z})}\ Q(\bar\zeta,\bar{z})\right]
\d\left(\frac{\zeta_j}{\zeta_0}\right)
\wedge\frac{\d\bar\zeta}{P(\bar\zeta)}\Bigg)\\
=\int_{V_{z}}\bar{z}_0^{-\ell}\phi^{(0,1)}(z)
\lim_{\eta\to 0}\lim_{\tau\to 0}\int_{\Gamma^{\tau}_w\cap bV\cap\{|B(z,w)|>\eta\}}
g(w)\cdot E(w,\lambda)\wedge\frac{\omega^{(2,0)}_j(w)}{P(w)}\\
\bigwedge\sum_{j=1,2}\Bigg(\lim_{\epsilon\to 0}\int_{\Gamma^{\epsilon}_{\zeta}
\cap\left\{|B(z,\zeta)|>\eta^6,|B(w,\zeta)|>\eta^6\right\}}A(\zeta,w)\\
\times\det\left[\frac{\bar{w}}{B(w,\zeta)}\ \frac{\d\bar{w}}{B(w,\zeta)}\ Q(w,\zeta)\right]
\det\left[\partial_z\left(\frac{z}{B^*(\bar\zeta,\bar{z})}\right)\ \frac{\zeta}
{B(\bar\zeta,\bar{z})}\ Q(\bar\zeta,\bar{z})\right]
\d\left(\frac{\zeta_j}{\zeta_0}\right)
\wedge\frac{\d\bar\zeta}{P(\bar\zeta)}\Bigg).
\end{multline}
Combining estimates \eqref{PartialzDomainIntegralFirst}, \eqref{PartialzDomainIntegralSecond}, and \eqref{PartialzDomainIntegralThird} we obtain the equality
\begin{multline}\label{PartialzDomain}
\partial_z\lim_{\eta\to 0}\lim_{\tau\to 0}
\int_{\Gamma^{\tau}_w\cap\{|B(z,w)|>\eta\}}
g(w)\cdot E(w,\lambda)\wedge G(z,w,\lambda)\\
=\lim_{\eta\to 0}\lim_{\tau\to 0}\int_{\Gamma^{\tau}_w\cap\{|B(z,w)|>\eta\}}
g(w)\cdot E(w,\lambda)\wedge\partial_z G(z,w,\lambda).
\end{multline}

\indent
From equalities \eqref{PartialzBoundary} and
\eqref{PartialzDomain} we obtain the equality
\begin{multline}\label{LimPartialzEquality}
\partial_z\lim_{\eta\to 0}\Bigg(\lim_{\tau\to 0}
\int_{\Gamma^{\tau}_w\cap bV\cap\{|B(z,w)|>\eta\}}
g(w)\cdot E(w,\lambda)\wedge G(z,w,\lambda)\\
+\lim_{\tau\to 0}\int_{\Gamma^{\tau}_w\cap\{|B(z,w)|>\eta\}}
g(w)\cdot E(w,\lambda)\wedge\bar\partial_w G(z,w,\lambda)\Bigg)\\
=\lim_{\eta\to 0}\Bigg(\lim_{\tau\to 0}\int_{\Gamma^{\tau}_w\cap bV\cap\{|B(z,w)|>\eta\}}
g(w)\cdot E(w,\lambda)\wedge\partial_z G(z,w,\lambda)\\
+\lim_{\tau\to 0}\int_{\Gamma^{\tau}_w\cap\{|B(z,w)|>\eta\}}
g(w)\cdot E(w,\lambda)\wedge\partial_z\bar\partial_w G(z,w,\lambda)\Bigg).
\end{multline}
\indent
Then, using the existence of the limit in the right-hand side, which follows from the
Proposition~\ref{PEquationFredholm}, we obtain the statement of Lemma~\ref{PartialzLemma}.
\end{proof}

\indent
Now, using Lemma~\ref{PartialzLemma} and Proposition~\ref{PEquationFredholm} we apply operator
$\partial$ to both parts of equation \eqref{hIntegralEquation} and obtain the equation
\eqref{FormIntegralEquation}
\begin{multline*}
g(z,\lambda)=E(z,-\lambda)
\left\langle\lambda,\left(\d\frac{z_1}{z_0}+\d\frac{z_2}{z_0}\right)\right\rangle\\
-\lim_{\eta\to 0}\lim_{\tau\to 0}\int_{\Gamma^{\tau}_w\cap \pi^{-1}(bV)\cap\{|B(z,w)|>\eta\}}
g(w,\lambda)\cdot E(w,\lambda)\wedge\partial_z G(z,w,\lambda)\\
-\lim_{\eta\to 0}\lim_{\tau\to 0}\int_{\Gamma^{\tau}_w\cap\{|B(z,w)|>\eta\}}
g(w,\lambda)\cdot E(w,\lambda)\wedge\partial_z\bar\partial_w G(z,w,\lambda),\\
\end{multline*}
or
\begin{equation}\label{PEquation}
(I+\cal{P}_{\lambda})[g](z,\lambda)=E(z,-\lambda)
\left\langle\lambda,\left(\d\frac{z_1}{z_0}+\d\frac{z_2}{z_0}\right)\right\rangle,
\end{equation}
where $\cal{P}_{\lambda}$ is the operator from \eqref{POperator},
with respect to $g=\partial h$, that should be satisfied by any solution of \eqref{fFormula},
in particular by the solution constructed by the Faddeev-Henkin ansatz.

\indent
Integral equation \eqref{PEquation} is the first equation of the two equations mentioned in
Theorem~\ref{Main}. Solution $g(z,\lambda)$ of this equation will be used in the next section
in the construction of the second Fredholm-type integral equation mentioned in Theorem~\ref{Main}.

\section{Solvability of the boundary value problem.}\label{sec:Solvability}

\indent
The concluding step in our analysis of solvability of the inverse conductivity problem on $V$
is the proof of Theorem~\ref{Main}, leading to a construction of a special solution of equation \eqref{fEquation}. In this construction we will use the form
$g\in \Lambda_{(1,0)}^{\widehat{\delta},\alpha}(V)$ obtained in the previous section with coefficients satisfying the H$\ddot{\mathrm{\bf o}}$lder-type condition
$${\dis \left|\frac{g(z^{(1)})-g(z^{(2)})}{(z^{(1)}-z^{(2)})^{\alpha}}\right|
\leq C\ \text{for}\ z^{(1)}\neq z^{(2)} \in U,\ |z^{(1)}-z^{(2)}|
\leq\widehat{\delta}=|\lambda|^{-2},\ \text{and}\ \alpha<1/2. }$$
Since the form $g\in \Lambda_{(1,0)}^{\widehat{\delta},\alpha}(V)$, satisfying the condition above, has bounded coefficients, and the function
$|z^{(1)}-z^{(2)}|^{-\alpha}$ for $|z^{(1)}-z^{(2)}|>\widehat{\delta}$ is uniformly bounded on $V$, we may assume that $g\in \Lambda_{(1,0)}^{\alpha}(V)$ with $\alpha<1/2$. This condition on $\alpha$ will be assumed below.

\indent
To obtain the sought solution of equation \eqref{fEquation} we consider the following boundary value problem
\begin{equation}\label{hBVProblem}
\left\{\begin{aligned}
&\partial h=g,\\
&h\big|_{bV}=\cal{T}\left(g\big|_{bV}\right),
\end{aligned}\right.
\end{equation}
where the form $g=g_{\lambda}$ is the solution of equation \eqref{PEquation}, and
$\cal{T}:\Lambda^{\alpha}_{(1,0)}(bV)\to \Lambda^{\alpha}(bV)$
is the map from \eqref{TMap}, existence of which we use in the proof of equality \eqref{hEquation}.
Then, using $\partial$-analogues of the formulas and estimates constructed in \cite{22} for operator
$\bar\partial$ we transform the problem \eqref{hBVProblem} into a Fredholm-type integral equation \eqref{hResidualEquality}. Finally,
solution $f$ of equation \eqref{fEquation} can be obtained from a solution $h$ of equation
\eqref{hResidualEquality} via formula \eqref{fhEquality}.\\
\indent
In this section we will use the following tubular domains
\begin{equation}\label{UWepsilon}
\begin{aligned}
&U^{\epsilon}=\left\{\zeta\in \*S^5(1):\ \left|P(\zeta)\right|<\epsilon, \varrho(\zeta)<0\right\},\\
&W^{\epsilon}=\pi(U^{\epsilon})
=\left\{\zeta\in \C\P^2:\ \left|P(\zeta)\right|<\epsilon, \varrho(\zeta)<0\right\},
\end{aligned}
\end{equation}
and use analogues of the formulas (1.8 - 1.11) from Proposition 1.2 in \cite{12} for operator
$\partial$ to obtain the proposition below.

\begin{proposition}\label{OriginalCurveFormula}
The following equality 
\begin{equation}\label{PartialfFormula}
h(z)=J^{\epsilon}[\partial h](z)+K^{\epsilon}\left[h\right](z),
\end{equation}
where
$$J^{\epsilon}[\partial h](z)=-\frac{2}{(2\pi i)^3}\int_{U^{\epsilon}\times[0,1]}
\partial h(\zeta)\wedge \omega_0^{\prime}\left((1-\lambda)
\frac{z}{B^*(\bar\zeta,\bar{z})}
+\lambda\frac{\zeta}{B(\bar\zeta,\bar{z})}\right)\wedge\d\bar\zeta,$$
and
\begin{multline}\label{KFormula}
K^{\epsilon}\left[h\right](z)=\frac{2}{(2\pi i)^3}\int_{U^{\epsilon}}
h(\zeta)\ \omega_0^{\prime}\left(\frac{\zeta}{B(\bar\zeta,\bar{z})}\right)
\wedge\d\bar\zeta\\
-\frac{2}{(2\pi i)^3}\int_{bU^{\epsilon}\times[0,1]}h(\zeta)\
\omega_0^{\prime}\left((1-\lambda)\frac{z}{B^*(\bar\zeta,\bar{z})}
+\lambda\frac{\zeta}{B(\bar\zeta,\bar{z})}\right)\wedge\d\bar\zeta,
\end{multline}
holds for a function $h\in C^{1}(U^{\epsilon})$, where forms $\omega^{\prime}$ are defined
in \eqref{OmegaPrimeForm}.
\end{proposition}
\indent
Following \cite{22}, where the $\bar\partial$-analogue of formula \eqref{PartialfFormula}
was used, we transform formula \eqref{PartialfFormula}
into a Cauchy-Weil-Leray type formula for operator $\partial$ on the domain $U^{\epsilon}$.

\begin{proposition}\label{EpsilonFormula} Let ${\cal C}$ and $V$ be as in
\eqref{ProjectiveCurve} and \eqref{VCurve} respectively,
let $U^{\epsilon}$ and $W^{\epsilon}$ be as in \eqref{UWepsilon}, and
let $h\in C(W^{\epsilon})$, $\partial h\in C_{(1,0)}(W^{\epsilon})$.
Then for an arbitrary $z\in U^{\epsilon}$ the following equality holds
\begin{equation}\label{fEpsilonFormula}
h(z)=I^{\epsilon}[\partial h](z)+M^{\epsilon}\left[ h \right](z),
\end{equation}
where
\begin{multline}\label{IepsilonOperator}
I^{\epsilon}[\partial h](z)=\frac{1}{(2\pi i)^3}\Bigg(-\int_{U^{\epsilon}}
\partial h(\zeta)\wedge \det\left[\frac{z}{B^*(\bar\zeta,\bar{z})}\
\frac{\zeta}{B(\bar\zeta,\bar{z})}\ \partial_{\zeta}\left(\frac{\zeta}{B(\bar\zeta,\bar{z})}\right)\right]
\wedge \d\bar\zeta,\\
-\int_{\Gamma^{\epsilon}_1}
\partial h(\zeta)\wedge\det\left[\frac{Q(\bar\zeta,\bar{z})}{P(\bar\zeta)-P(\bar{z})}\
\frac{z}{B^*(\bar\zeta,\bar{z})}\ \frac{\zeta}{B(\bar\zeta,\bar{z})}\right]
\wedge \d\bar\zeta\Bigg)\\
\end{multline}
and
\begin{multline}\label{MepsilonOperator}
M^{\epsilon}\left[h\right](z)=\frac{1}{(2\pi i)^{3}}\Bigg(
\int_{U^{\epsilon}}h(\zeta)\cdot
\det\left[\frac{\zeta}{B(\bar\zeta,\bar{z})}\
\partial_{\zeta}\left(\frac{\zeta}{B(\bar\zeta,\bar{z})}\right)\
\partial_{\zeta}\left(\frac{\zeta}{B(\bar\zeta,\bar{z})}\right)\right]
\wedge \d\bar\zeta,\\
-\int_{\Gamma^{\epsilon}_1}
h(\zeta)\cdot\det\left[\frac{Q(\bar\zeta,\bar{z})}{P(\bar\zeta)-P(\bar{z})}\
\frac{\zeta}{B(\bar\zeta,\bar{z})}\ \partial_{\zeta}\left(\frac{\zeta}{B(\bar\zeta,\bar{z})}\right)\right]
\wedge \d\bar\zeta\\
+\int_{\Gamma^{\epsilon}_{12}}
h(\zeta)\cdot\det\left[\frac{Q(\bar\zeta,\bar{z})}{P(\bar\zeta)-P(\bar{z})}\
\frac{z}{B^*(\bar\zeta,\bar{z})}\ \frac{\zeta}{B(\bar\zeta,\bar{z})}\right]
\wedge \d\bar\zeta\\
-\int_{\Gamma^{\epsilon}_2}
h(\zeta)\cdot\det\left[\frac{z}{B^*(\bar\zeta,\bar{z})}\
\frac{\zeta}{B(\bar\zeta,\bar{z})}\ \partial_{\zeta}\left(\frac{\zeta}{B(\bar\zeta,\bar{z})}\right)\right]
\wedge \d\bar\zeta\Bigg)\\
\end{multline}
with
$$\Gamma_1^{\epsilon}=\left\{\zeta\in \*S^5(1): |P(\zeta)|=\epsilon,
\varrho(\zeta)< 0\right\},\hspace{0.1in}
\Gamma_2^{\epsilon}=\left\{\zeta\in \*S^5(1): |P(\zeta)|<\epsilon,
\varrho(\zeta)=0\right\},$$
$$\Gamma_{12}^{\epsilon}=\left\{\zeta\in \*S^5(1): 
|P(\zeta)|=\epsilon,  \varrho(\zeta)=0\right\},$$
with coefficients $\left\{Q^i(\zeta,z)\right\}_{i=0}^2$ satisfying conditions
\eqref{QFunctions} from Proposition~\ref{dbarSolvability}.\\
\indent
The function in the right-hand side of \eqref{MepsilonOperator}, which is a priori defined
on $U^{\epsilon}\subset \*S^5(1)$, admits the descent to a function on the neighborhood $W^{\epsilon}$ of $V$ in $\C\P^2$.
\end{proposition}

\indent
Before proving Proposition~\ref{EpsilonFormula} we transform the kernels of integrals in \eqref{PartialfFormula} into the \lq\lq determinantal form\rq\rq.

\begin{lemma}\label{OmegaDet}
The following equality
\begin{equation}\label{OmegaDetEquality}
\int_0^1\omega_0^{\prime}\left((1-\lambda)\frac{z}{B^*(\bar\zeta,\bar{z})}
+\lambda\frac{\zeta}{B(\bar\zeta,\bar{z})}\right)
=\frac{1}{2}\det\left[\frac{z}{B^*(\bar\zeta,\bar{z})}\
\frac{\zeta}{B(\bar\zeta,\bar{z})}\ \frac{\d\zeta}{B(\bar\zeta,\bar{z})}\right]
\end{equation}
holds for $z,\zeta \in \*S(1)$.
\end{lemma}
\begin{proof}
We start with equality
\begin{equation}\label{omegaFormula}
\omega_0^{\prime}(\xi(\zeta,z,\lambda))
=\sum_{k=0}^2(-1)^{k}\xi_k(\zeta,z,\lambda)
\bigwedge_{j\neq k}\d\xi_j(\zeta,z,\lambda)
=\frac{1}{2}\det\left[\begin{tabular}{ccc}
$\xi_0$&$\d\xi_0$&$\d\xi_0$
\vspace{0.05in}\\
$\xi_1$&$\d\xi_1$&$\d\xi_1$\vspace{0.05in}\\
$\xi_2$&$\d\xi_2$&$\d\xi_2$
\end{tabular}\right],
\end{equation}
where the differentials are taken only with respect to variables $\zeta$ and $\lambda$,
and the determinant is defined by the usual formula except the positions of the exterior
differentials are determined by the numbers of their columns.\\
\indent
Then, using in \eqref{omegaFormula} the equality of columns
${\dis \xi(\zeta,z,\lambda)=\xi^{(1)}(\zeta,z,\lambda)
+\xi^{(2)}(\zeta,z,\lambda)}$, where
$$\xi^{(1)}(\zeta,z,\lambda)=(1-\lambda)\frac{z}{B^*(\bar\zeta,\bar{z})}\hspace{0.1in}
\text{and}\hspace{0.1in}
\xi^{(2)}(\zeta,z,\lambda)=\lambda\frac{\zeta}{B(\bar\zeta,\bar{z})},$$
additivity of the determinant, and the fact that only one differential
$\d\lambda$, only one differential $\d\zeta$, and no differentials $\d\bar\zeta$
are allowed in the determinant, we obtain
\begin{multline}\label{Transform}
\det\bigg[\begin{tabular}{ccc}
$\xi(\zeta,z,\lambda),$&$\d_{\zeta,\lambda}\xi(\zeta,z,\lambda),$
&$\d_{\zeta,\lambda}\xi(\zeta,z,\lambda)$
\end{tabular}\bigg]\\
=(1-\lambda)\det\bigg[
\frac{z}{B^*(\bar\zeta,\bar{z})}\ ,\ \partial_{\zeta}\xi(\zeta,z,\lambda)
+\d_{\lambda}\xi(\zeta,z,\lambda)\ ,\
\partial_{\zeta}\xi(\zeta,z,\lambda)
+\d_{\lambda}\xi(\zeta,z,\lambda)\bigg]\\
+\lambda\det\bigg[
\frac{\zeta}{B(\bar\zeta,\bar{z})}\ ,\ \partial_{\zeta}\xi(\zeta,z,\lambda)
+\d_{\lambda}\xi(\zeta,z,\lambda)\ ,\
\partial_{\zeta}\xi(\zeta,z,\lambda)
+\d_{\lambda}\xi(\zeta,z,\lambda)\bigg]\\
=(1-\lambda)\det\left[\frac{z}{B^*(\bar\zeta,\bar{z})}\
\frac{\zeta \d\lambda}{B(\bar\zeta,\bar{z})}\ \frac{\lambda \d\zeta}{B(\bar\zeta,\bar{z})}\right]
+(1-\lambda)\det\left[\frac{z}{B^*(\bar\zeta,\bar{z})}\
\frac{\lambda \d\zeta}{B(\bar\zeta,\bar{z})}\ \frac{\zeta \d\lambda}{B(\bar\zeta,\bar{z})}\right]\\
+\lambda\det\left[\frac{\zeta}{B(\bar\zeta,\bar{z})}\
\frac{-z \d\lambda}{B^*(\bar\zeta,\bar{z})}\ \frac{\lambda \d\zeta}{B(\bar\zeta,\bar{z})}\right]
+\lambda\det\left[\frac{\zeta}{B(\bar\zeta,\bar{z})}\
\frac{\lambda \d\zeta}{B(\bar\zeta,\bar{z})}\
\frac{-z \d\lambda}{B^*(\bar\zeta,\bar{z})}\right]\\
=2\lambda(1-\lambda)\d\lambda\wedge
\det\left[\frac{z}{B^*(\bar\zeta,\bar{z})}\
\frac{\zeta}{B(\bar\zeta,\bar{z})}\ \frac{\d\zeta}{B(\bar\zeta,\bar{z})}\right]
-2\lambda^2\d\lambda\wedge\det\left[\frac{\zeta}{B(\bar\zeta,\bar{z})}\
\frac{z}{B^*(\bar\zeta,\bar{z})}\ \frac{\d\zeta}{B(\bar\zeta,\bar{z})}\right]\\
=2\lambda \d\lambda\wedge\det\left[\frac{z}{B^*(\bar\zeta,\bar{z})}\
\frac{\zeta}{B(\bar\zeta,\bar{z})}\ \frac{\d\zeta}{B(\bar\zeta,\bar{z})}\right],
\end{multline}
where in the second equality above we also used equalities
\begin{equation*}
\partial_{\zeta}\left(\frac{z}{B^*(\bar\zeta,\bar{z})}\right)=0,
\end{equation*}
and
\begin{equation*}
\det\left[\frac{\zeta}{B(\bar\zeta,\bar{z})}\
\frac{\zeta\partial_{\zeta}(B(\bar\zeta,\bar{z}))}{B^2(\bar\zeta,\bar{z})}\
\frac{-zd\lambda}{B^*(\bar\zeta,\bar{z})}\right]
=\frac{\partial_{\zeta}(B(\bar\zeta,\bar{z}))}{B(\bar\zeta,\bar{z})}
\wedge\det\left[\frac{\zeta}{B(\bar\zeta,\bar{z})}\
\frac{\zeta}{B(\bar\zeta,\bar{z})}\
\frac{-z d\lambda}{B^*(\bar\zeta,\bar{z})}\right]=0.
\end{equation*}
\indent
Then, using equalities \eqref{omegaFormula} and \eqref{Transform} we obtain the equality
\begin{multline*}
\int_0^1\omega_0^{\prime}(\xi(\zeta,z,\lambda))
=\int_0^1\lambda\d\lambda\det\left[\frac{z}{B^*(\bar\zeta,\bar{z})}\
\frac{\zeta}{B(\bar\zeta,\bar{z})}\ \frac{\d\zeta}{B(\bar\zeta,\bar{z})}\right]\\
=\frac{1}{2}\det\left[\frac{z}{B^*(\bar\zeta,\bar{z})}\
\frac{\zeta}{B(\bar\zeta,\bar{z})}\ \frac{\d\zeta}{B(\bar\zeta,\bar{z})}\right]
\end{multline*}
which implies equality \eqref{OmegaDetEquality}.
\end{proof}

{\it Proof of Proposition~\ref{EpsilonFormula}}
\indent
Rewriting equality \eqref{PartialfFormula} and taking into account equality \eqref{OmegaDetEquality}
we obtain the equality
\begin{multline}\label{PreliminaryEpsilonFormula}
h(z)=-\frac{1}{(2\pi i)^3}\int_{U^{\epsilon}}
\partial h(\zeta)\wedge\det\left[\frac{z}{B^*(\bar\zeta,\bar{z})}\
\frac{\zeta}{B(\bar\zeta,\bar{z})}\ \partial_{\zeta}\left(\frac{\zeta}{B(\bar\zeta,\bar{z})}\right)\right]\d\bar\zeta\\
+\frac{1}{(2\pi i)^3}\Bigg(\int_{U^{\epsilon}}
h(\zeta) \det\left[\frac{\zeta}{B(\bar\zeta,\bar{z})}\
\partial_{\zeta}\left(\frac{\zeta}{B(\bar\zeta,\bar{z})}\right)\
\partial_{\zeta}\left(\frac{\zeta}{B(\bar\zeta,\bar{z})}\right)\right]\wedge\d\bar\zeta\\
-\sum_{j=1,2}\int_{\Gamma_j^{\epsilon}}h(\zeta)
\det\left[\frac{z}{B^*(\bar\zeta,\bar{z})}\
\frac{\zeta}{B(\bar\zeta,\bar{z})}\
\partial_{\zeta}\left(\frac{\zeta}{B(\bar\zeta,\bar{z})}\right)\right]\wedge\d\bar\zeta\Bigg).
\end{multline}

\indent
To transform the right-hand side of \eqref{PreliminaryEpsilonFormula} into the
right-hand side of \eqref{fEpsilonFormula}
we transform the integral over $\Gamma_1^{\epsilon}$ in the right-hand side of \eqref{PreliminaryEpsilonFormula}.
For that we use the linear dependence of rows and its corollary
\begin{equation}\label{ZeroDeterminant}
\det\left[\begin{tabular}{cccc}
$1$&$1$&$1$&$0$
\vspace{0.05in}\\
$\xi_1(\zeta,z)$&$\xi_2(\zeta,z)$
&${\dis \frac{\zeta}{B(\bar\zeta,\bar{z})} }$
&${\dis \partial_{\zeta}\left(\frac{\zeta}{B(\bar\zeta,\bar{z})}\right) }$
\end{tabular}\right]=0 
\end{equation}
for columns ${\dis \xi_1(\zeta,z)=\frac{Q(\bar\zeta,\bar{z})}{P(\bar\zeta)-P(\bar{z})} }$,
${\dis \xi_2(\zeta,z)=\frac{z}{B^*(\bar\zeta,\bar{z})} }$, to obtain equality
\begin{multline}\label{DeterminantEquality}
\det\left[\frac{z}{B^*(\bar\zeta,\bar{z})}\
\frac{\zeta}{B(\bar\zeta,\bar{z})}\
\partial_{\zeta}\left(\frac{\zeta}{B(\bar\zeta,\bar{z})}\right)\right]\wedge d\bar\zeta\\
=\det\left[\frac{Q(\bar\zeta,\bar{z})}{P(\bar\zeta)-P(\bar{z})}\
\frac{\zeta}{B(\bar\zeta,\bar{z})}\ 
\partial_{\zeta}\left(\frac{\zeta}{B(\bar\zeta,\bar{z})}\right)\right]\wedge d\bar\zeta
-\det\left[\frac{Q(\bar\zeta,\bar{z})}{P(\bar\zeta)-P(\bar{z})}\ \frac{z}{B^*(\bar\zeta,\bar{z})}\
\partial_{\zeta}\left(\frac{\zeta}{B(\bar\zeta,\bar{z})}\right)\right]\wedge d\bar\zeta.
\end{multline}
Then, applying the Stokes' theorem on $\Gamma^{\epsilon}_1$ to the second term of the
right-hand side of \eqref{DeterminantEquality} we obtain the equality
\begin{multline}\label{FirstQIntegral}
\int_{\Gamma^{\epsilon}_1}
h(\zeta)\cdot\det\left[\frac{z}{B^*(\bar\zeta,\bar{z})}\
\frac{\zeta}{B(\bar\zeta,\bar{z})}\ \partial_{\zeta}\left(\frac{\zeta}{B(\bar\zeta,\bar{z})}\right)\right]
\wedge d\bar\zeta\\
=\int_{\Gamma^{\epsilon}_1}
h(\zeta)\cdot\det\left[\frac{Q(\bar\zeta,\bar{z})}{P(\bar\zeta)-P(\bar{z})}\
\frac{\zeta}{B(\bar\zeta,\bar{z})}\ \partial_{\zeta}\left(\frac{\zeta}{B(\bar\zeta,\bar{z})}\right)\right]
\wedge d\bar\zeta\\
-\int_{\Gamma^{\epsilon}_{12}}
h(\zeta)\cdot\det\left[\frac{Q(\bar\zeta,\bar{z})}{P(\bar\zeta)-P(\bar{z})}\
\frac{z}{B^*(\bar\zeta,\bar{z})}\ \frac{\zeta}{B(\bar\zeta,\bar{z})}\right]
\wedge d\bar\zeta\\
+\int_{\Gamma^{\epsilon}_1}
\partial h(\zeta)\wedge\det\left[\frac{Q(\bar\zeta,\bar{z})}{P(\bar\zeta)-P(\bar{z})}\
\frac{z}{B^*(\bar\zeta,\bar{z})}\ \frac{\zeta}{B(\bar\zeta,\bar{z})}\right]
\wedge d\bar\zeta.
\end{multline}

\indent
Substituting the expression from equality \eqref{FirstQIntegral} into
\eqref{PreliminaryEpsilonFormula} we obtain the following equality
\begin{multline*}
h(z)=\frac{1}{(2\pi i)^3}\Bigg(
-\int_{U^{\epsilon}}\partial h(\zeta)\wedge\det\left[\frac{z}{B^*(\bar\zeta,\bar{z})}\
\frac{\zeta}{B(\bar\zeta,\bar{z})}\ \partial_{\zeta}\left(\frac{\zeta}{B(\bar\zeta,\bar{z})}\right)\right]\d\bar\zeta\\
-\int_{\Gamma^{\epsilon}_1}
\partial h(\zeta)\wedge\det\left[\frac{Q(\bar\zeta,\bar{z})}{P(\bar\zeta)-P(\bar{z})}\
\frac{z}{B^*(\bar\zeta,\bar{z})}\ \frac{\zeta}{B(\bar\zeta,\bar{z})}\right]
\wedge d\bar\zeta\Bigg)\\
+\frac{1}{(2\pi i)^3}\Bigg(\int_{U^{\epsilon}}h(\zeta) \det\left[\frac{\zeta}{B(\bar\zeta,\bar{z})}\
\partial_{\zeta}\left(\frac{\zeta}{B(\bar\zeta,\bar{z})}\right)\
\partial_{\zeta}\left(\frac{\zeta}{B(\bar\zeta,\bar{z})}\right)\right]\wedge\d\bar\zeta\\
-\int_{\Gamma^{\epsilon}_1}
h(\zeta)\cdot\det\left[\frac{Q(\bar\zeta,\bar{z})}{P(\bar\zeta)-P(\bar{z})}\
\frac{\zeta}{B(\bar\zeta,\bar{z})}\ \partial_{\zeta}\left(\frac{\zeta}{B(\bar\zeta,\bar{z})}\right)\right]
\wedge d\bar\zeta\\
+\int_{\Gamma^{\epsilon}_{12}}
h(\zeta)\cdot\det\left[\frac{Q(\bar\zeta,\bar{z})}{P(\bar\zeta)-P(\bar{z})}\
\frac{z}{B^*(\bar\zeta,\bar{z})}\ \frac{\zeta}{B(\bar\zeta,\bar{z})}\right]
\wedge d\bar\zeta\\
-\int_{\Gamma_2^{\epsilon}}h(\zeta)\det\left[\frac{z}{B^*(\bar\zeta,\bar{z})}\
\frac{\zeta}{B(\bar\zeta,\bar{z})}\
\partial_{\zeta}\left(\frac{\zeta}{B(\bar\zeta,\bar{z})}\right)\right]\wedge\d\bar\zeta\Bigg),
\end{multline*}
which is equivalent to \eqref{fEpsilonFormula}.
\qed

\indent
In the next three lemmas we assume that $z\in \pi^{-1}(V)$, i.e. $P(z)=0$, and simplify the limits
of the right-hand sides of equalities \eqref{IepsilonOperator} and \eqref{MepsilonOperator}
as $\epsilon\to 0$.

\begin{lemma}\label{epsilonVolume}
For $\partial h\in C_{(1,0)}(U^{\epsilon})$ and $z\in \pi^{-1}(V)$ the following equality holds
\begin{equation}\label{epsilonVolumeEquality}
\lim_{\epsilon\to 0}\int_{U^{\epsilon}}
\partial h(\zeta)\wedge \det\left[\frac{z}{B^*(\bar\zeta,\bar{z})}\
\frac{\zeta}{B(\bar\zeta,\bar{z})}\ \partial_{\zeta}\left(\frac{\zeta}{B(\bar\zeta,\bar{z})}\right)\right]
\wedge d\bar\zeta=0.
\end{equation}
\end{lemma}
\begin{proof}
For a fixed small enough neighborhood ${\widetilde U}^{\epsilon}$ of $z\in \pi^{-1}(V)\subset \*S^5$
we consider the coordinates
\begin{equation}\label{LocalCoordinates}
t=\text{Im}B(\zeta,z),\ \rho(\zeta)=|P(\zeta)|,\ \theta(\zeta)=\text{Arg}\ P(\zeta),\ F(z,\zeta)=u+iv,
\end{equation}
where $F(z,\zeta)$ is a local complex coordinate in $\C\P^2$ complementary to $P(\zeta)$
in $\pi({\widetilde U}^{\epsilon})$. Then we have
\begin{multline*}
\left|\int_{{\widetilde U}^{\epsilon}}\partial h(\zeta)\wedge\det\left[\frac{z}{B^*(\bar\zeta,\bar{z})}\
\frac{\zeta}{B(\bar\zeta,\bar{z})}\
\frac{\d\zeta}{B(\bar\zeta,\bar{z})}\right]\wedge\d\bar\zeta\right|\\
\leq C\cdot\|\partial h\|_{C_{(1,0)}}
\left|\int_0^A\d t\int_0^{\epsilon}\rho\ \d\rho\int_0^B\frac{(u^2+v^2)^{1/2}\d u\wedge\d v}
{(t+\rho^2+u^2+v^2)^3}\right|\\
\leq C\cdot\|\partial h\|_{C_{(1,0)}}\int_0^{\epsilon}\rho\d\rho\int_0^B\frac{r^2\d r}{(\rho^2+r^2)^2}
\leq C\cdot\|\partial h\|_{C_{(1,0)}}\int_0^{\epsilon}\rho\d\rho\int_0^B\frac{\d r}{(\rho^2+r^2)}\\
\leq C\cdot\|\partial h\|_{C_{(1,0)}}\int_0^{\epsilon}\d\rho\int_0^{\infty}\frac{\d s}{(1+s^2)}
\leq C\cdot\|\partial h\|_{C_{(1,0)}}\cdot\epsilon\to 0
\end{multline*}
as $\epsilon\to 0$.\\
\indent
Same estimate holds for the integral over $U^{\epsilon}\setminus {\widetilde U}^{\epsilon}$, since
function $\text{Re}B(\bar\zeta,\bar{z})$ will be strictly separated from zero.
\end{proof}

\begin{lemma}\label{epsilonBoundary}
For  $h\in C(U^{\epsilon})$ and $z\in \pi^{-1}(V)$ the following equality
\begin{equation}\label{epsilonBoundaryEquality}
\lim_{\epsilon\to 0}\int_{\Gamma^{\epsilon}_2}
h(\zeta)\cdot\det\left[\frac{z}{B^*(\bar\zeta,\bar{z})}\
\frac{\zeta}{B(\bar\zeta,\bar{z})}\ \partial_{\zeta}\left(\frac{\zeta}{B(\bar\zeta,\bar{z})}\right)\right]
\wedge \d\bar\zeta=0
\end{equation}
holds.
\end{lemma}
\begin{proof}
Since the functions ${\dis \frac{\zeta}{B(\bar\zeta,\bar{z})},\ \frac{z}{B^*(\bar\zeta,\bar{z})} }$ are
uniformly bounded for $z\in \pi^{-1}(V)$ and $\zeta\in \Gamma^{\epsilon}_2$ we obtain
\begin{equation*}
\Bigg|\int_{\Gamma^{\epsilon}_2}
h(\zeta)\cdot\det\left[\frac{z}{B^*(\bar\zeta,\bar{z})}\
\frac{\zeta}{B(\bar\zeta,\bar{z})}\ \partial_{\zeta}\left(\frac{\zeta}{B(\bar\zeta,\bar{z})}\right)\right]
\wedge \d\bar\zeta\Bigg|\leq C(z)\cdot\|h\|_{C}\cdot\epsilon\to 0
\end{equation*}
as $\epsilon\to 0$.
\end{proof}

\begin{lemma}\label{ZeroU}Let $h\in \Lambda^{\alpha}(U^{\epsilon})$. Then
for $z\in \pi^{-1}(V)$ the following equality
\begin{equation}\label{ZeroUEquality}
\lim_{\epsilon\to 0}\int_{U^{\epsilon}}h(\zeta)\cdot
\det\left[\frac{\zeta}{B(\bar\zeta,\bar{z})}\
\partial_{\zeta}\left(\frac{\zeta}{B(\bar\zeta,\bar{z})}\right)\
\partial_{\zeta}\left(\frac{\zeta}{B(\bar\zeta,\bar{z})}\right)\right]
\wedge \d\bar\zeta=0
\end{equation}
holds.
\end{lemma}
\begin{proof}
To prove equality \eqref{ZeroUEquality} we represent the integral in the left-hand side as
\begin{multline}\label{SumRepresentation}
\int_{U^{\epsilon}}h(\zeta)\cdot
\det\left[\frac{\zeta}{B(\bar\zeta,\bar{z})}\
\partial_{\zeta}\left(\frac{\zeta}{B(\bar\zeta,\bar{z})}\right)\
\partial_{\zeta}\left(\frac{\zeta}{B(\bar\zeta,\bar{z})}\right)\right]
\wedge \d\bar\zeta\\
=\int_{U^{\epsilon}\cap\left\{|B(\zeta,z)|>\epsilon\right\}}h(\zeta)\cdot
\det\left[\frac{\zeta}{B(\bar\zeta,\bar{z})}\
\partial_{\zeta}\left(\frac{\zeta}{B(\bar\zeta,\bar{z})}\right)\
\partial_{\zeta}\left(\frac{\zeta}{B(\bar\zeta,\bar{z})}\right)\right]
\wedge \d\bar\zeta\\
+\int_{U^{\epsilon}\cap\left\{|B(\zeta,z)|<\epsilon\right\}}
(h(\zeta)-h(z))\cdot
\det\left[\frac{\zeta}{B(\bar\zeta,\bar{z})}\
\partial_{\zeta}\left(\frac{\zeta}{B(\bar\zeta,\bar{z})}\right)\
\partial_{\zeta}\left(\frac{\zeta}{B(\bar\zeta,\bar{z})}\right)\right]
\wedge \d\bar\zeta\\
+h(z)\int_{U^{\epsilon}\cap\left\{|B(\zeta,z)|<\epsilon\right\}}
\det\left[\frac{\zeta}{B(\bar\zeta,\bar{z})}\
\partial_{\zeta}\left(\frac{\zeta}{B(\bar\zeta,\bar{z})}\right)\
\partial_{\zeta}\left(\frac{\zeta}{B(\bar\zeta,\bar{z})}\right)\right]
\wedge \d\bar\zeta.
\end{multline}
\indent
Then we obtain the following estimate for the first integral in the right-hand side
of \eqref{SumRepresentation}
\begin{multline}\label{AbsoluteEstimateOne}
\left|\int_{U^{\epsilon}\cap\left\{|B(\zeta,z)|>\epsilon\right\}}h(\zeta)\cdot
\det\left[\frac{\zeta}{B(\bar\zeta,\bar{z})}\
\partial_{\zeta}\left(\frac{\zeta}{B(\bar\zeta,\bar{z})}\right)\
\partial_{\zeta}\left(\frac{\zeta}{B(\bar\zeta,\bar{z})}\right)\right]
\wedge d\bar\zeta\right|\\
\leq C\cdot\|h\|_{C}\cdot\int_0^Adt\int_0^{\epsilon}du\int_0^{\epsilon}dv
\int_0^A\frac{r\d r}{(\epsilon+t+u^2+v^2+r^2)^3}
\leq C\cdot\|h\|_{C}\cdot\int_0^{\epsilon}\frac{\rho \d\rho}{\epsilon+\rho^2}\\
\leq C\cdot\|h\|_{C}\cdot\int_0^{\epsilon^2}\frac{\d s}{\epsilon+s}
\leq C\cdot\|h\|_{C}\cdot\epsilon,
\end{multline}
where we denoted $|P(\zeta)|=\rho=\sqrt{u^2+v^2}$, $s=\rho^2$.

\indent
For the second integral in the right-hand side of \eqref{SumRepresentation}, using estimate
$|h(\zeta)-h(z)|\leq C\cdot|\zeta-z|^{\alpha}$, we have
\begin{multline}\label{AbsoluteEstimateTwo}
\left|\int_{U^{\epsilon}\cap\left\{|B(\zeta,z)|<\epsilon\right\}}(h(\zeta)-h(z))\cdot
\det\left[\frac{\zeta}{B(\bar\zeta,\bar{z})}\
\partial_{\zeta}\left(\frac{\zeta}{B(\bar\zeta,\bar{z})}\right)\
\partial_{\zeta}\left(\frac{\zeta}{B(\bar\zeta,\bar{z})}\right)\right]
\wedge \d\bar\zeta\right|\\
\leq C\cdot\|h\|_{\Lambda^{\alpha}}\cdot\int_0^{\epsilon}\d t\int_0^{\epsilon}\d u\int_0^{\epsilon}\d v
\int_0^A\frac{r\d r}{(t+u^2+v^2+r^2)^{3-\alpha/2}}\\
\leq C\cdot\|h\|_{\Lambda^{\alpha}} \int_0^{\epsilon}\frac{\rho\d\rho}{\rho^{2(1-\alpha/2)}}
\leq C\cdot\|h\|_{\Lambda^{\alpha}}\int_0^{\epsilon^2}\frac{\d s}{s^{1-\alpha/2}}
\leq C\cdot\|h\|_{\Lambda^{\alpha}}\cdot\epsilon^{\alpha}.
\end{multline}
\indent
Using estimates \eqref{AbsoluteEstimateOne}, \eqref{AbsoluteEstimateTwo} and
equality
\begin{equation}\label{ZeroSingular}
\lim_{\epsilon\to 0}\int_{U^{\epsilon}\cap\left\{|B(\zeta,z)|<\epsilon\right\}}
\det\left[\frac{\zeta}{B(\bar\zeta,\bar{z})}\
\partial_{\zeta}\left(\frac{\zeta}{B(\bar\zeta,\bar{z})}\right)\
\partial_{\zeta}\left(\frac{\zeta}{B(\bar\zeta,\bar{z})}\right)\right]
\wedge \d\bar\zeta=0
\end{equation}
we obtain equality \eqref{ZeroUEquality}.
\end{proof}

Formulas in the propositions and lemmas above were proved for a function
$h\in \Lambda^{\alpha}(U^{\epsilon})$. In what follows we will be using those formulas for
a function $h\in \Lambda^{\alpha}(V)$. In those applications we assume that function
$h\in \Lambda^{\alpha}(V)$ is extended into
$$W^{\epsilon}=\pi(U^{\epsilon})
=\left\{\zeta\in \C\P^2:\ \left|P(\zeta)\right|<\epsilon, \varrho(\zeta)<0\right\}$$
identically along the fibers of the normal bundle over $V$ in $\C\P^2$, and then into
$$U^{\epsilon}=\left\{\zeta\in \*S^5(1):\ \left|P(\zeta)\right|<\epsilon, \varrho(\zeta)<0\right\}$$
identically along the fibers of the map $\pi$.

Now, considering the limit of equality \eqref{fEpsilonFormula} for an arbitrary $z\in \pi^{-1}(V)$
as $\epsilon\to 0$ and using Lemmas~\ref{epsilonVolume}, \ref{epsilonBoundary},
\ref{ZeroU} we obtain the following proposition.

\begin{proposition}\label{SecondStep}
The following equality holds for $z\in V$ and a function $h\in\Lambda^{\alpha}(V)$ such that
$\partial h\in C_{(1,0)}(V)$
\begin{equation}\label{hResidualFormula}
h(z)=I[\partial h](z)+M[h](z),
\end{equation}
where
\begin{equation}\label{IOperator}
I[\partial h](z)=-\frac{1}{(2\pi i)^3}\lim_{\epsilon\to 0}\int_{\Gamma^{\epsilon}_1}
\partial h(\zeta)\wedge\det\left[\frac{Q(\bar\zeta,\bar{z})}{P(\bar\zeta)}\
\frac{z}{B^*(\bar\zeta,\bar{z})}\ \frac{\zeta}{B(\bar\zeta,\bar{z})}\right]
\wedge d\bar\zeta,
\end{equation}
and
\begin{multline}\label{MOperator}
M[h](z)=\frac{1}{(2\pi i)^{3}}\Bigg(-\lim_{\epsilon\to 0}\int_{\Gamma^{\epsilon}_1}
h(\zeta)\cdot\det\left[\frac{Q(\bar\zeta,\bar{z})}{P(\bar\zeta)}\
\frac{\zeta}{B(\bar\zeta,\bar{z})}\ \partial_{\zeta}\left(\frac{\zeta}{B(\bar\zeta,\bar{z})}\right)\right]
\wedge d\bar\zeta\\
+\lim_{\epsilon\to 0}\int_{\Gamma^{\epsilon}_{12}}
h(\zeta)\cdot\det\left[\frac{Q(\bar\zeta,\bar{z})}{P(\bar\zeta)}\
\frac{z}{B^*(\bar\zeta,\bar{z})}\ \frac{\zeta}{B(\bar\zeta,\bar{z})}\right]
\wedge d\bar\zeta\Bigg).
\end{multline}
\end{proposition}
\qed

\indent
From equality \eqref{hResidualFormula} we conclude that any solution $h\in \Lambda^{\alpha}(V)$
of the boundary value problem \eqref{hBVProblem} such that
$\partial h \in \Lambda^{\alpha}_{(1,0)}(V)$, in particular the one that can be found through the Faddeev-Henkin exponential ansatz in Proposition~\ref{Solvability}, must satisfy equality \eqref{hResidualFormula}.\\
\indent
We transform equality \eqref{hResidualFormula} into the following equality for
$h\in \Lambda^{\alpha}(V)$ and $z\in V$ by rewriting it as follows
\begin{multline}\label{hResidualEquality}
h(z)+\frac{1}{(2\pi i)^3}\Bigg(\lim_{\epsilon\to 0}\int_{\Gamma^{\epsilon}}
h(\zeta)\cdot\det\left[\frac{Q(\bar\zeta,\bar{z})}{P(\bar\zeta)}\
\frac{\zeta}{B(\bar\zeta,\bar{z})}\ \partial_{\zeta}\left(\frac{\zeta}{B(\bar\zeta,\bar{z})}\right)\right]\wedge d\bar\zeta\\
-\lim_{\epsilon\to 0}\int_{b\Gamma^{\epsilon}}
h(\zeta)\cdot\det\left[\frac{Q(\bar\zeta,\bar{z})}{P(\bar\zeta)}\
\frac{z}{B^*(\bar\zeta,\bar{z})}\ \frac{\zeta}{B(\bar\zeta,\bar{z})}\right]
\wedge d\bar\zeta\Bigg)\\
=-\frac{1}{(2\pi i)^{3}}\lim_{\epsilon\to 0}\int_{\Gamma^{\epsilon}}
g(\zeta)\wedge\det\left[\frac{Q(\bar\zeta,\bar{z})}{P(\bar\zeta)}\
\frac{z}{B^*(\bar\zeta,\bar{z})}\ \frac{\zeta}{B(\bar\zeta,\bar{z})}\right]
\wedge d\bar\zeta,
\end{multline}
where
\begin{equation*}
\begin{aligned}
&\Gamma^{\epsilon}=\left\{\zeta\in \*S^5(1): |P(\zeta)|=\epsilon, \varrho(\zeta)<0\right\},\\
&b\Gamma^{\epsilon}=\left\{\zeta\in \*S^5(1): |P(\zeta)|=\epsilon,\varrho(\zeta)=0\right\},
\end{aligned}
\end{equation*}
and the form $g$ in the right-hand side of \eqref{hResidualEquality} is
the solution of equation \eqref{PEquation}.

\indent
To further transform equality \eqref{hResidualEquality} we assume the existence of a linear operator, which we call the $\partial$-to-$\bar\partial$ map:
\begin{equation}\label{chiDefinition}
\chi:\Lambda_{(1,0)}^{\alpha}(bU)\to \Lambda_{(0,1)}^{\alpha}(bU),
\end{equation}
such that for any solution $f$ of equation \eqref{fEquation} the following equality holds
$$\chi(\partial f)=\bar\partial f.$$
\indent
To substantiate the existence of such operator we assume that solution $f$ of equation \eqref{PEquation} is constructed on some neighborhood
$\widetilde{V}=\left\{\zeta\in {\cal C}: \varrho(\zeta)<\gamma\right\}$
of domain $V$, where $\varrho$ is a smooth function on ${\cal C}$ from \eqref{VCurve}.
Since function $f$ satisfies equation \eqref{fEquation} on the domain $\widetilde{V}\setminus V$,
we have $\partial\bar\partial f\Big|_{\widetilde{V}\setminus V}=0$, with the forms
$$\partial_{\zeta}f(\zeta)\hspace{0.1in}\text{and}\hspace{0.1in}\bar\partial_{\zeta}f(\zeta)$$
being, respectively, holomorphic and anti-holomorphic forms on $\widetilde{V}\setminus V$, uniquely
defined by the function $f$.

Then, we consider the form
\begin{equation*}
g(\zeta)=\partial_{\zeta}h(\zeta)
=\partial_{\zeta}\left(e^{-\overline{\left\langle\lambda,\zeta/\zeta_0\right\rangle}}
f(\zeta,\lambda)\right)
=e^{-\overline{\left\langle\lambda,\zeta/\zeta_0\right\rangle}}\partial_{\zeta}f(\zeta,\lambda),
\end{equation*}
where
\begin{equation*}
\partial_{\zeta}f(\zeta,\lambda)=e^{\overline{\left\langle\lambda,\zeta/\zeta_0\right\rangle}}
\partial_{\zeta} h(\zeta,\lambda)
\end{equation*}
is a holomorphic form on this neighborhood.

Using this equality and formula \eqref{fhEquality} we obtain the formula
\begin{equation*}
f(z,\lambda)\Big|_{bV}=\int_{bV}\partial_{\zeta}f(\zeta)+\int_{bV}\bar\partial_{\zeta}f(\zeta)
=\int_{bV}e^{\overline{\left\langle\lambda,\zeta/\zeta_0\right\rangle}}g(\zeta)
+\int_{bV}\chi\left(e^{\overline{\left\langle\lambda,\zeta/\zeta_0\right\rangle}}g(\zeta)\right),
\end{equation*}
and its corollary
\begin{equation}\label{TMap}
h(z,\lambda)\big|_{bV}={\cal T}(g\big|_{bV})
=e^{-\overline{\left\langle\lambda,z/z_0\right\rangle}}
\left[\int_{bV}e^{\overline{\left\langle\lambda,\zeta/\zeta_0\right\rangle}}g(\zeta)
+\int_{bV}\chi\left(e^{\overline{\left\langle\lambda,\zeta/\zeta_0\right\rangle}}
g(\zeta)\right)\right].
\end{equation}

\indent
Then, we transform equality \eqref{hResidualFormula} into the following integral equation
for $h\in \Lambda^{\alpha}(V)$ and $z\in V$ by rewriting it as follows
\begin{equation}\label{hEquation}
h(z)+S[h](z)=v(z),
\end{equation}
where
\begin{equation}\label{SOperator}
S[h](z)=\frac{1}{(2\pi i)^3}\lim_{\epsilon\to 0}\int_{\Gamma^{\epsilon}}
h(\zeta)\cdot\det\left[\frac{Q(\bar\zeta,\bar{z})}{P(\bar\zeta)}\
\frac{\zeta}{B(\bar\zeta,\bar{z})}\ \partial_{\zeta}\left(\frac{\zeta}{B(\bar\zeta,\bar{z})}\right)\right]\wedge d\bar\zeta\\
\end{equation}
and
\begin{multline}\label{vFunction}
v(z)=-\frac{1}{(2\pi i)^{3}}\Bigg(\lim_{\epsilon\to 0}\int_{\Gamma^{\epsilon}}
g(\zeta)\wedge\det\left[\frac{Q(\bar\zeta,\bar{z})}{P(\bar\zeta)}\
\frac{z}{B^*(\bar\zeta,\bar{z})}\ \frac{\zeta}{B(\bar\zeta,\bar{z})}\right]
\wedge d\bar\zeta\\
-\lim_{\epsilon\to 0}\int_{b\Gamma^{\epsilon}}
h(\zeta)\cdot\det\left[\frac{Q(\bar\zeta,\bar{z})}{P(\bar\zeta)}\
\frac{z}{B^*(\bar\zeta,\bar{z})}\ \frac{\zeta}{B(\bar\zeta,\bar{z})}\right]
\wedge d\bar\zeta\Bigg),
\end{multline}
where the value of $h$ on the boundary $bV$ is determined from equality \eqref{TMap}.

\indent
Our next goal is to prove that integral equation \eqref{hEquation} is a Fredholm-type
equation on $\Lambda^{\alpha}(V)$. In the Propositions~\ref{gEstimate}
and \ref{SecondRight} below we prove  necessary estimates for the integrals in \eqref{vFunction}.
Then, in the Proposition~\ref{Compactness} we prove the Fredholm property of equation
\eqref{hEquation}.

\begin{proposition}\label{gEstimate}
If $g\in C_{(1,0)}(U^{\tau})$ for some $\tau>0$, then the first integral
in the right-hand side of \eqref{vFunction} defines a function in $\Lambda^{\alpha}(V)$
for any $\alpha<1$.
\end{proposition}
\begin{proof}
We obtain the statement of the proposition as a corollary of the following two estimates
\begin{equation}\label{gSmallEstimate}
\Bigg|\lim_{\epsilon\to 0}\int_{\Gamma^{\epsilon}\cap\{|B(z,\zeta)|<\delta^2\}}
g(\zeta)\wedge\det\left[\frac{Q(\bar\zeta,\bar{z})}{P(\bar\zeta)}\
\frac{z}{B^*(\bar\zeta,\bar{z})}\ \frac{\zeta}{B(\bar\zeta,\bar{z})}\right]
\wedge d\bar\zeta\Bigg|\leq C\|g\|_{C}\cdot\delta,
\end{equation}
and
\begin{multline}\label{gLargeEstimate}
\Bigg|\lim_{\epsilon\to 0}\int_{\Gamma^{\epsilon}\cap\{|B(z^{(1,2)},\zeta)|>\delta^2\}}
g(\zeta)\wedge\det\left[\frac{Q(\bar\zeta,\bar{z}^{(1)})}{P(\bar\zeta)}\
\frac{z^{(1)}}{B^*(\bar\zeta,\bar{z}^{(1)})}\ \frac{\zeta}{B(\bar\zeta,\bar{z}^{(1)})}\right]
\wedge d\bar\zeta\\
-\lim_{\epsilon\to 0}\int_{\Gamma^{\epsilon}\cap\{|B(z^{(1,2)},\zeta)|>\delta^2\}}
g(\zeta)\wedge\det\left[\frac{Q(\bar\zeta,\bar{z}^{(2)})}{P(\bar\zeta)}\
\frac{z^{(2)}}{B^*(\bar\zeta,\bar{z}^{(2)})}\ \frac{\zeta}{B(\bar\zeta,\bar{z}^{(2)})}\right]
\wedge d\bar\zeta\Bigg|\leq C\|g\|_{C}\cdot\delta^{\alpha},
\end{multline}
where points $z^{(1)},z^{(2)}\in \pi^{-1}(V)$ satisfy condition $|\pi(z^{(1)})-\pi(z^{(2)})|=\delta$
and the condition of Lemma~\ref{FindingZ}, i.e. $|B(z^{(1)},z^{(2)})|=\delta^2$.\\
\indent
For the integral in \eqref{gSmallEstimate} we have
\begin{multline*}
\Bigg|\lim_{\epsilon\to 0}\int_{\Gamma^{\epsilon}\cap\{|B(z,\zeta)|<\delta^2\}}
g(\zeta)\wedge\det\left[\frac{Q(\bar\zeta,\bar{z})}{P(\bar\zeta)}\
\frac{z}{B^*(\bar\zeta,\bar{z})}\ \frac{\zeta}{B(\bar\zeta,\bar{z})}\right]
\wedge d\bar\zeta\Bigg|\\
\leq C\cdot\|g\|_{C}\int_0^{\delta^2}\d t\int_0^{\delta}\frac{r^2\d r}{(t+r^2)^2}
\leq C\cdot\|g\|_{C}\int_0^{\delta}\d r\leq C\|g\|_{C}\cdot\delta.
\end{multline*}

\indent
For the integrals in \eqref{gLargeEstimate} we have
\begin{multline*}
\Bigg|\lim_{\epsilon\to 0}\int_{\Gamma^{\epsilon}\cap\{|B(z^{(1,2)},\zeta)|>\delta^2\}}
g(\zeta)\wedge\det\left[\frac{Q(\bar\zeta,\bar{z}^{(1)})}{P(\bar\zeta)}\
\frac{z^{(1)}}{B^*(\bar\zeta,\bar{z}^{(1)})}\ \frac{\zeta}{B(\bar\zeta,\bar{z}^{(1)})}\right]
\wedge d\bar\zeta\\
-\lim_{\epsilon\to 0}\int_{\Gamma^{\epsilon}\cap\{|B(z^{(1,2)},\zeta)|>\delta^2\}}
g(\zeta)\wedge\det\left[\frac{Q(\bar\zeta,\bar{z}^{(2)})}{P(\bar\zeta)}\
\frac{z^{(2)}}{B^*(\bar\zeta,\bar{z}^{(2)})}\ \frac{\zeta}{B(\bar\zeta,\bar{z}^{(2)})}\right]
\wedge d\bar\zeta\Bigg|\\
\leq \Bigg|\lim_{\epsilon\to 0}\int_{\Gamma^{\epsilon}\cap\{|B(z^{(1,2)},\zeta)|>\delta^2\}}
g(\zeta)\wedge\det\left[\frac{Q(\bar\zeta,\bar{z}^{(1)})-Q(\bar\zeta,\bar{z}^{(2)})}{P(\bar\zeta)}\
\frac{z^{(1)}}{B^*(\bar\zeta,\bar{z}^{(1)})}\ \frac{\zeta}{B(\bar\zeta,\bar{z}^{(1)})}\right]
\wedge d\bar\zeta\Bigg|\\
+\Bigg|\lim_{\epsilon\to 0}\int_{\Gamma^{\epsilon}\cap\{|B(z^{(1,2)},\zeta)|>\delta^2\}}
g(\zeta)\wedge\det\left[\frac{Q(\bar\zeta,\bar{z}^{(2)})}{P(\bar\zeta)}\
\frac{z^{(1)}-z^{(2)}}{B^*(\bar\zeta,\bar{z}^{(1)})}\ \frac{\zeta}{B(\bar\zeta,\bar{z}^{(1)})}\right]
\wedge d\bar\zeta\Bigg|\\
+\Bigg|\lim_{\epsilon\to 0}\int_{\Gamma^{\epsilon}\cap\{|B(z^{(1,2)},\zeta)|>\delta^2\}}
g(\zeta)\wedge\det\left[\frac{Q(\bar\zeta,\bar{z}^{(2)})}{P(\bar\zeta)}\
z^{(2)}\left(\frac{B^*(\bar\zeta,\bar{z}^{(2)})-B^*(\bar\zeta,\bar{z}^{(1)})}
{B^*(\bar\zeta,\bar{z}^{(1)})B^*(\bar\zeta,\bar{z}^{(2)})}\right)\
\frac{\zeta}{B(\bar\zeta,\bar{z}^{(1)})}\right]
\wedge d\bar\zeta\Bigg|\\
+\Bigg|\lim_{\epsilon\to 0}\int_{\Gamma^{\epsilon}\cap\{|B(z^{(1,2)},\zeta)|>\delta^2\}}
g(\zeta)\wedge\det\left[\frac{Q(\bar\zeta,\bar{z}^{(2)})}{P(\bar\zeta)}\
\frac{z^{(2)}}{B^*(\bar\zeta,\bar{z}^{(2)})}\
\zeta\left(\frac{B(\bar\zeta,\bar{z}^{(2)})-B(\bar\zeta,\bar{z}^{(1)})}
{B(\bar\zeta,\bar{z}^{(1)})B(\bar\zeta,\bar{z}^{(2)})}\right)\right]
\wedge d\bar\zeta\Bigg|\\
\leq C\|g\|_{C}\cdot\delta\int_{\delta^2}^A\d t\int_{\delta}^B\frac{r\d r}{(t+r^2)^2}
+C\|g\|_{C}\cdot\int_{\delta^2}^A\d t\int_{\delta}^B\frac{\delta r^3\d r}{(t+r^2)^3}\\
\leq C\|g\|_{C}\cdot\delta\log{\delta}< C\|g\|_{C}\cdot\delta^{\alpha}\ \text{for any}\ \alpha<1,
\end{multline*}
where in the first two integrals of \eqref{gLargeEstimate} we used inequality
$|z^{(1)}-z^{(2)}|\leq \delta$, and in the last two integrals we used the inequality \eqref{BDifferenceInequality}.

\indent
From the estimates \eqref{gSmallEstimate} and \eqref{gLargeEstimate} we obtain
the statement of the proposition, which implies that the first integral from the right-hand side
of \eqref{vFunction} is in $\Lambda^{\alpha}(V)$.
\end{proof}

\indent
The next lemma will be used in the estimates for the second integral in the right-hand
side of \eqref{vFunction} and for the integral in the right-hand side of \eqref{SOperator}.

\begin{lemma}\label{fConstant} 
For any fixed $\tau>0$ and $W^{\tau}=\left\{z\in \C\P^2:\ \left|P(z)\right|<\tau,
\varrho(z)<0\right\}$ from \eqref{UWepsilon} the function
\begin{equation}\label{fConstantFunction}
\lim_{\epsilon\to 0}\int_{\Gamma^{\epsilon}}
\det\left[\frac{Q(\bar\zeta,\bar{z})}{P(\bar\zeta)}\
\frac{\zeta}{B(\bar\zeta,\bar{z})}\ \partial_{\zeta}\left(\frac{\zeta}{B(\bar\zeta,\bar{z})}\right)\right]
\wedge\d\bar\zeta\\
\end{equation}
is a well defined anti-holomorphic function in $W^{\tau}$.
\end{lemma}
\begin{proof}
\indent
We notice that the function in \eqref{fConstantFunction} is well defined for
$z\in W^{\tau}\setminus V$.
Also, since the function $B(\bar\zeta,\bar{z})=1-\sum_{j=0}^2\zeta_j\bar{z}_j$ is anti-holomorphic
with respect to $z$, all terms in the determinant are anti-holomorphic with respect to $z$,
as is the function defined by the integral.\\
\indent
To prove that the function in \eqref{fConstantFunction} is well defined and anti-holomorphic
with respect to $z$ in the whole $W^{\tau}$ we represent it as follows
\begin{multline}\label{fConstantEquality}
\lim_{\epsilon\to 0}\int_{\Gamma^{\epsilon}}
\det\left[\frac{Q(\bar\zeta,\bar{z})}{P(\bar\zeta)}\
\frac{\zeta}{B(\bar\zeta,\bar{z})}\ \partial_{\zeta}\left(\frac{\zeta}{B(\bar\zeta,\bar{z})}\right)\right]
\wedge\d\bar\zeta\\
=\int_{\pi^{-1}(V)}\frac{S(\zeta,\bar{z})\d\zeta}{B(\bar\zeta,\bar{z})^2}
\wedge\left(\d P(\bar\zeta)\interior\d\bar\zeta\right)
=\int_{\pi^{-1}(V)}S(\zeta,\bar{z})\d\zeta
\wedge\frac{\d_{\zeta}\bar{B}(\bar\zeta,\bar{z})}{B(\bar\zeta,\bar{z})^2}
\wedge\big(\left(\d_{\zeta}\bar{B}(\bar\zeta,\bar{z})\wedge\d P(\bar\zeta)\right)
\interior\d\bar\zeta\big)\\
=\Bigg(\int_{\pi^{-1}(V)}S(\zeta,\bar{z})\d\zeta
\wedge\frac{\d_{\zeta}\left(B(\bar\zeta,\bar{z})+\bar{B}(\bar\zeta,\bar{z})\right)}{B(\bar\zeta,\bar{z})^2}
\wedge\big(\left(\d_{\zeta}\bar{B}(\bar\zeta,\bar{z})\wedge\d P(\bar\zeta)\right)
\interior\d\bar\zeta\big)\\
-\int_{\pi^{-1}(V)}S(\zeta,\bar{z})\d\zeta
\wedge\frac{\d_{\zeta}B(\bar\zeta,\bar{z})}{B(\bar\zeta,\bar{z})^2}
\wedge\big(\left(\d_{\zeta}\bar{B}(\bar\zeta,\bar{z})\wedge\d P(\bar\zeta)\right)
\interior\d\bar\zeta\big)\Bigg),
\end{multline}
where $S(\zeta,\bar{z})$ is a smooth function, anti-holomorphic with respect to $z$,
and where we used the equalities
$$\d_{\zeta}\bar{B}(\bar\zeta,\bar{z})=-\sum_{j=0}^2 z_j\d\bar\zeta_j
=-\left\langle z,\d\bar\zeta\right\rangle,$$
and
$$\d\bar\zeta=\left\langle z,\d\bar\zeta\right\rangle
\wedge\left(\left\langle z,\d\bar\zeta\right\rangle\interior\d\bar\zeta\right).$$
\indent
For the first integral in the right-hand side of \eqref{fConstantEquality} we have
\begin{multline}\label{fConstantFirst}
\Bigg|\int_{\pi^{-1}(V)\cap\{|B(z,\zeta)|<\eta\}}S(\zeta,\bar{z})\d\zeta
\wedge\frac{\d_{\zeta}\left(B(\bar\zeta,\bar{z})+\bar{B}(\bar\zeta,\bar{z})\right)}{B(\bar\zeta,\bar{z})^2}
\wedge\big(\left(\d_{\zeta}\bar{B}(\bar\zeta,\bar{z})\wedge\d P(\bar\zeta)\right)
\interior\d\bar\zeta\big)\Bigg|\\
\leq C\cdot\int_0^{\eta}\d t\int_0^{\sqrt{\eta}}\frac{r^2\d r}{(it+r^2)^2}
\leq C\cdot\int_0^{\sqrt{\eta}}\d r\leq C\cdot\sqrt{\eta},
\end{multline}
where in the right-hand side of the estimate above we used the equality
$$\d_{\zeta}\left(B(\bar\zeta,\bar{z})+\bar{B}(\bar\zeta,\bar{z})\right)
=2\d_{\zeta}\text{Re}B(\bar\zeta,\bar{z})
=\d_{\zeta}\left(\sum_{j=0}^2(\zeta_j-z_j)(\bar\zeta_j-\bar{z}_j)\right).$$

\indent
For the second integral in the right-hand side of \eqref{fConstantEquality} using the
Stokes' theorem we obtain
\begin{multline}\label{fConstantSecond}
\Bigg|\int_{\pi^{-1}(V)\cap\{|B(z,\zeta)|<\eta\}}S(\zeta,\bar{z})\d\zeta
\wedge\frac{\d_{\zeta}B(\bar\zeta,\bar{z})}{B(\bar\zeta,\bar{z})^2}
\wedge\big(\left(\d_{\zeta}\bar{B}(\bar\zeta,\bar{z})\wedge\d P(\bar\zeta)\right)
\interior\d\bar\zeta\big)\Bigg|\\
\leq C\cdot\Bigg|\int_{\pi^{-1}(V)\cap\{|B(z,\zeta)|<\eta\}}\d_{\zeta}
S(\zeta,\bar{z})\wedge\frac{\d\zeta}{B(\bar\zeta,\bar{z})}
\wedge\big(\left(\d_{\zeta}\bar{B}(\bar\zeta,\bar{z})\wedge\d P(\bar\zeta)\right)
\interior\d\bar\zeta\big)\Bigg|\\
+ C\cdot\Bigg|\int_{\pi^{-1}(V)\cap\{|B(z,\zeta)|=\eta\}}
S(\zeta,\bar{z})\wedge\frac{\d\zeta}{B(\bar\zeta,\bar{z})}
\wedge\big(\left(\d_{\zeta}\bar{B}(\bar\zeta,\bar{z})\wedge\d P(\bar\zeta)\right)
\interior\d\bar\zeta\big)\Bigg|\\
\leq C\cdot\int_0^{\eta}\d t\int_0^{\sqrt{\eta}}\frac{r\d r}{t+r^2}
+C\cdot\frac{\eta^{3/2}}{\eta}\leq C\cdot\sqrt{\eta},
\end{multline}
where in the second term of the right-hand side we used estimate \eqref{DeltaAreaEstimate}
from Lemma~\ref{DeltaArea}.\\
\indent
From the estimates \eqref{fConstantFirst} and \eqref{fConstantSecond} we obtain that the function
defined by the integral in \eqref{fConstantFunction} is uniformly bounded in $W^{\tau}\setminus V$.
Combined with the fact that it is anti-holomorphic in $W^{\tau}\setminus V$ it shows that this function
is anti-holomorphic in $W^{\tau}$.
\end{proof}

In the next proposition, using Lemma~\ref{fConstant}, we prove  necessary estimates for the
second integral in \eqref{vFunction}.

\begin{proposition}\label{SecondRight}
If a function $h\in \Lambda^{\alpha}(bV)$ for some $\alpha\in (0,1)$, then the second integral
in the right-hand side of \eqref{vFunction} defines a function in $\Lambda^{\alpha}(V)$.
\end{proposition}
\begin{proof}
Using the estimate $|h(\zeta)-h(z)|\leq C\cdot|\zeta-z|^{\alpha}$ for $\zeta$ satisfying
$\{|B(z,\zeta)|<\delta^2\}$ we obtain the following estimate for the second integral in the right-hand side of \eqref{vFunction}:
\begin{multline}\label{SmallSecondRight}
\lim_{\epsilon\to 0}\int_{b\Gamma^{\epsilon}\cap\{|B(z,\zeta)|<\delta^2\}}
\left(h(\zeta)-h(z)\right)\det\left[\frac{Q(\bar\zeta,\bar{z})}{P(\bar\zeta)}\
\frac{z}{B^*(\bar\zeta,\bar{z})}\ \frac{\zeta}{B(\bar\zeta,\bar{z})}\right]
\wedge d\bar\zeta\\
\leq C\cdot\|h\|_{\Lambda^{\alpha}}\int_0^{\delta^2}\d t\int_0^{\delta}\frac{r^{1+\alpha}\d r}{(t+r^2)^2}
\leq C\cdot\|h\|_{\Lambda^{\alpha}}\int_0^{\delta}\frac{\d r}{r^{1-\alpha}}
\leq C\cdot\|h\|_{\Lambda^{\alpha}}\cdot r^{\alpha}\Bigg|_0^{\delta}
\leq C\|h\|_{\Lambda^{\alpha}}\cdot\delta^{\alpha}.
\end{multline}

\indent
For the complementary integral, denoting $\nu=9\delta^2$, we obtain the following
representation
\begin{multline}\label{LargeSecondRight}
\lim_{\epsilon\to 0}\int_{b\Gamma^{\epsilon}
\setminus U^{\nu}_{\zeta}(z^{(1)},z^{(2)})}
\left(h(\zeta)-h(z^{(1)})\right)\det\left[\frac{Q(\bar\zeta,\bar{z}^{(1)})}{P(\bar\zeta)}\
\frac{z^{(1)}}{B^*(\bar\zeta,\bar{z}^{(1)})}\ \frac{\zeta}{B(\bar\zeta,\bar{z}^{(1)})}\right]
\wedge d\bar\zeta\\
-\lim_{\epsilon\to 0}\int_{b\Gamma^{\epsilon}
\setminus U^{\nu}_{\zeta}(z^{(1)},z^{(2)})}
\left(h(\zeta)-h(z^{(2)})\right)\det\left[\frac{Q(\bar\zeta,\bar{z}^{(2)})}{P(\bar\zeta)}\
\frac{z^{(2)}}{B^*(\bar\zeta,\bar{z}^{(2)})}\ \frac{\zeta}{B(\bar\zeta,\bar{z}^{(2)})}\right]
\wedge d\bar\zeta\\
=\lim_{\epsilon\to 0}\int_{b\Gamma^{\epsilon}
\setminus U^{\nu}_{\zeta}(z^{(1)},z^{(2)})}
\left(h(z^{(2)})-h(z^{(1)})\right)\det\left[\frac{Q(\bar\zeta,\bar{z}^{(1)})}{P(\bar\zeta)}\
\frac{z^{(1)}}{B^*(\bar\zeta,\bar{z}^{(1)})}\ \frac{\zeta}{B(\bar\zeta,\bar{z}^{(1)})}\right]
\wedge d\bar\zeta\\
+\lim_{\epsilon\to 0}\int_{b\Gamma^{\epsilon}
\setminus U^{\nu}_{\zeta}(z^{(1)},z^{(2)})}
\left(h(\zeta)-h(z^{(2)})\right)
\det\left[\frac{Q(\bar\zeta,\bar{z}^{(1)})-Q(\bar\zeta,\bar{z}^{(2)})}{P(\bar\zeta)}\
\frac{z^{(1)}}{B^*(\bar\zeta,\bar{z}^{(1)})}\ \frac{\zeta}{B(\bar\zeta,\bar{z}^{(1)})}\right]
\wedge d\bar\zeta\\
+\lim_{\epsilon\to 0}\int_{b\Gamma^{\epsilon}
\setminus U^{\nu}_{\zeta}(z^{(1)},z^{(2)})}
\left(h(\zeta)-h(z^{(2)})\right)\det\left[\frac{Q(\bar\zeta,\bar{z}^{(2)})}{P(\bar\zeta)}\
\frac{z^{(1)}-z^{(2)}}{B^*(\bar\zeta,\bar{z}^{(1)})}\ \frac{\zeta}{B(\bar\zeta,\bar{z}^{(1)})}\right]
\wedge d\bar\zeta\\
+\lim_{\epsilon\to 0}\int_{b\Gamma^{\epsilon}
\setminus U^{\nu}_{\zeta}(z^{(1)},z^{(2)})}
\left(h(\zeta)-h(z^{(2)})\right)\\
\det\left[\frac{Q(\bar\zeta,\bar{z}^{(2)})}{P(\bar\zeta)}\
\left(\frac{z^{(2)}}{B^*(\bar\zeta,\bar{z}^{(1)})}-\frac{z^{(2)}}{B^*(\bar\zeta,\bar{z}^{(2)})}\right)\
\frac{\zeta}{B(\bar\zeta,\bar{z}^{(1)})}\right]
\wedge d\bar\zeta\\
+\lim_{\epsilon\to 0}\int_{b\Gamma^{\epsilon}
\setminus U^{\nu}_{\zeta}(z^{(1)},z^{(2)})}
\left(h(\zeta)-h(z^{(2)})\right)\\
\det\left[\frac{Q(\bar\zeta,\bar{z}^{(2)})}{P(\bar\zeta)}\
\frac{z^{(2)}}{B^*(\bar\zeta,\bar{z}^{(2)})}\
\left(\frac{\zeta}{B(\bar\zeta,\bar{z}^{(1)})}-\frac{\zeta}{B(\bar\zeta,\bar{z}^{(2)})}\right)\right]
\wedge d\bar\zeta.
\end{multline}
\indent
For the second integral in the right-hand side of \eqref{LargeSecondRight} we have the following estimate
\begin{multline}\label{LargeSecond}
\Bigg|\lim_{\epsilon\to 0}\int_{b\Gamma^{\epsilon}
\setminus U^{\nu}_{\zeta}(z^{(1)},z^{(2)})}
\left(h(\zeta)-h(z^{(2)})\right)
\det\left[\frac{Q(\bar\zeta,\bar{z}^{(1)})-Q(\bar\zeta,\bar{z}^{(2)})}{P(\bar\zeta)}\
\frac{z^{(1)}}{B^*(\bar\zeta,\bar{z}^{(1)})}\ \frac{\zeta}{B(\bar\zeta,\bar{z}^{(1)})}\right]
\wedge d\bar\zeta\Bigg|\\
\leq C\cdot\delta\|h\|_{\Lambda^{\alpha}}
\int_{\delta^2}^A\d t\int_{\delta}^B\frac{r^{\alpha}\d r}{(t+r^2)^2}
\leq C\|h\|_{\Lambda^{\alpha}}\cdot\delta^{\alpha}.
\end{multline}
\indent
For the third integral in the right-hand side of \eqref{LargeSecondRight} we have the following estimate
\begin{multline}\label{LargeThird}
\Bigg|\lim_{\epsilon\to 0}\int_{b\Gamma^{\epsilon}
\setminus U^{\nu}_{\zeta}(z^{(1)},z^{(2)})}
\left(h(\zeta)-h(z^{(2)})\right)\det\left[\frac{Q(\bar\zeta,\bar{z}^{(2)})}{P(\bar\zeta)}\
\frac{z^{(1)}-z^{(2)}}{B^*(\bar\zeta,\bar{z}^{(1)})}\ \frac{\zeta}{B(\bar\zeta,\bar{z}^{(1)})}\right]
\wedge d\bar\zeta\Bigg|\\
\leq C\cdot\delta\|h\|_{\Lambda^{\alpha}}
\int_{\delta^2}^A\d t\int_{\delta}^B\frac{r^{\alpha}\d r}{(t+r^2)^2}
\leq C\|h\|_{\Lambda^{\alpha}}\cdot\delta^{\alpha}.
\end{multline}
\indent
For the fourth integral in the right-hand side of \eqref{LargeSecondRight} using the estimate
\eqref{BDifferenceInequality} for points
$z^{(1)}, z^{(2)}, w$ with $|B(z^{(1)},z^{(2)})|=\delta^2$, and $|B(z^{(i)},\zeta)|>9\delta^2$ for $i=1,2$,
we obtain the following estimate
\begin{multline}\label{LargeFourth}
\lim_{\epsilon\to 0}\int_{b\Gamma^{\epsilon}
\setminus U^{\nu}_{\zeta}(z^{(1)},z^{(2)})}
\left(h(\zeta)-h(z^{(2)})\right)\\
\times\det\left[\frac{Q(\bar\zeta,\bar{z}^{(2)})}{P(\bar\zeta)}\
\frac{z^{(2)}\left(B^*(\bar\zeta,\bar{z}^{(2)})-B^*(\bar\zeta,\bar{z}^{(1)})\right)}
{B^*(\bar\zeta,\bar{z}^{(1)})B^*(\bar\zeta,\bar{z}^{(2)})}\
\frac{\zeta}{B(\bar\zeta,\bar{z}^{(1)})}\right]
\wedge d\bar\zeta\\
\leq C\cdot\delta\|h\|_{\Lambda^{\alpha}}
\int_{\delta^2}^A\d t\int_{\delta}^B\frac{r^{2+\alpha}\d r}{(t+r^2)^3}
\leq C\cdot\delta\|h\|_{\Lambda^{\alpha}}\int_{\delta}^B r^{-2+\alpha}\d r
\leq C\|h\|_{\Lambda^{\alpha}}\cdot\delta^{\alpha}.
\end{multline}
\indent
Similar estimate holds for the fifth integral in the right-hand side of \eqref{LargeSecondRight}
\begin{multline}\label{LargeFifth}
\lim_{\epsilon\to 0}\int_{b\Gamma^{\epsilon}
\setminus U^{\nu}_{\zeta}(z^{(1)},z^{(2)})}
\left(h(\zeta)-h(z^{(2)})\right)\\
\times\det\left[\frac{Q(\bar\zeta,\bar{z}^{(2)})}{P(\bar\zeta)}\
\frac{z^{(2)}}{B^*(\bar\zeta,\bar{z}^{(2)})}\
\frac{\zeta(B(\bar\zeta,\bar{z}^{(2)})-B(\bar\zeta,\bar{z}^{(1)})}
{B(\bar\zeta,\bar{z}^{(1)})B(\bar\zeta,\bar{z}^{(2)})}\right]
\wedge d\bar\zeta\\
\leq C\|h\|_{\Lambda^{\alpha}}\cdot\delta^{\alpha}.
\end{multline}

\indent
To prove that the first term of the right-hand side of \eqref{LargeSecondRight} satisfies the
estimate similar to estimates \eqref{LargeSecond}$\div$\eqref{LargeFifth}
we use the fact that $h\in \Lambda^{\alpha}(V)$, and therefore it suffices
to prove the uniform boundedness with respect to $z$ of the integral
\begin{equation}\label{LargeFirstIntegral}
\lim_{\epsilon\to 0}\int_{b\Gamma^{\epsilon}\setminus U^{\nu}_{\zeta}(z)}
\det\left[\frac{Q(\bar\zeta,\bar{z})}{P(\bar\zeta)}\
\frac{z}{B^*(\bar\zeta,\bar{z})}\ \frac{\zeta}{B(\bar\zeta,\bar{z})}\right]
\wedge d\bar\zeta.
\end{equation}
On the first step we use the Stokes' theorem to obtain the equality
\begin{multline}\label{StokesLarge}
\lim_{\epsilon\to 0}\int_{b\Gamma^{\epsilon}\setminus U^{\nu}_{\zeta}(z)}
\det\left[\frac{Q(\bar\zeta,\bar{z})}{P(\bar\zeta)}\
\frac{z}{B^*(\bar\zeta,\bar{z})}\ \frac{\zeta}{B(\bar\zeta,\bar{z})}\right]
\wedge d\bar\zeta\\
=\lim_{\epsilon\to 0}\int_{\Gamma^{\epsilon}\setminus U^{\nu}_{\zeta}(z)}
\det\left[\frac{Q(\bar\zeta,\bar{z})}{P(\bar\zeta)}\
\frac{z}{B^*(\bar\zeta,\bar{z})}\ \partial_{\zeta}\left(\frac{\zeta}{B(\bar\zeta,\bar{z})}\right)\right]
\wedge d\bar\zeta\\
-\lim_{\epsilon\to 0}\int_{\Gamma^{\epsilon}\cap b\{|B(z,\zeta)|<9\delta^2\}}
\det\left[\frac{Q(\bar\zeta,\bar{z})}{P(\bar\zeta)}\
\frac{z}{B^*(\bar\zeta,\bar{z})}\ \frac{\zeta}{B(\bar\zeta,\bar{z})}\right]
\wedge d\bar\zeta.
\end{multline}
\indent
For the first integral in the right-hand side of \eqref{StokesLarge} we use equality \eqref{DeterminantEquality} for columns
$$\xi_1(\zeta,z)=\frac{Q(\bar\zeta,\bar{z})}{P(\bar\zeta)-P(\bar{z})},\
\xi_2(\zeta,z)=\frac{z}{B^*(\bar\zeta,\bar{z})},$$
and obtain equality
\begin{multline*}
\det\left[\frac{Q(\bar\zeta,\bar{z})}{P(\bar\zeta)}\ \frac{z}{B^*(\bar\zeta,\bar{z})}\
\partial_{\zeta}\left(\frac{\zeta}{B(\bar\zeta,\bar{z})}\right)\right]\wedge d\bar\zeta
=\det\left[\frac{Q(\bar\zeta,\bar{z})}{P(\bar\zeta)}\ \frac{\zeta}{B(\bar\zeta,\bar{z})}\ 
\partial_{\zeta}\left(\frac{\zeta}{B(\bar\zeta,\bar{z})}\right)\right]\wedge d\bar\zeta\\
-\det\left[\frac{z}{B^*(\bar\zeta,\bar{z})}\ \frac{\zeta}{B(\bar\zeta,\bar{z})}\
\partial_{\zeta}\left(\frac{\zeta}{B(\bar\zeta,\bar{z})}\right)\right]\wedge d\bar\zeta.
\end{multline*}
Then, using the absence of $P(\bar\zeta)$ in the denominator of the second term of the right-hand
side of equality above, we obtain the equality
\begin{multline}\label{AnotherKernel}
\lim_{\epsilon\to 0}\int_{\Gamma^{\epsilon}\setminus\{|B(z,\zeta)|<9\delta^2\}}
\det\left[\frac{Q(\bar\zeta,\bar{z})}{P(\bar\zeta)}\
\frac{z}{B^*(\bar\zeta,\bar{z})}\
\partial_{\zeta}\left(\frac{\zeta}{B(\bar\zeta,\bar{z})}\right)\right]\wedge d\bar\zeta\\
=\lim_{\epsilon\to 0}\int_{\Gamma^{\epsilon}\setminus\{|B(z,\zeta)|<9\delta^2\}}
\det\left[\frac{Q(\bar\zeta,\bar{z})}{P(\bar\zeta)}\ \frac{\zeta}{B(\bar\zeta,\bar{z})}\ 
\partial_{\zeta}\left(\frac{\zeta}{B(\bar\zeta,\bar{z})}\right)\right]\wedge d\bar\zeta.
\end{multline}

Using Lemma~\ref{fConstant} for the right-hand side of equality \eqref{AnotherKernel} we obtain
that for the first integral in the right-hand side of \eqref{StokesLarge}
\begin{equation}\label{LargeFirstFirst}
\left|\lim_{\epsilon\to 0}\int_{\Gamma^{\epsilon}\setminus\{|B(z,\zeta)|<9\delta^2\}}
\det\left[\frac{Q(\bar\zeta,\bar{z})}{P(\bar\zeta)}\
\frac{z}{B^*(\bar\zeta,\bar{z})}\
\partial_{\zeta}\left(\frac{\zeta}{B(\bar\zeta,\bar{z})}\right)\right]\wedge d\bar\zeta\right|
\leq C
\end{equation}
uniformly with respect to $z$.\\
\indent
For the second integral in the right-hand side of \eqref{StokesLarge} we have the following estimate
\begin{multline}\label{LargeFirstSecond}
\lim_{\epsilon\to 0}\int_{\Gamma^{\epsilon}\cap b\{|B(z,\zeta)|<9\delta^2\}}
\det\left[\frac{Q(\bar\zeta,\bar{z})}{P(\bar\zeta)}\
\frac{z}{B^*(\bar\zeta,\bar{z})}\ \frac{\zeta}{B(\bar\zeta,\bar{z})}\right]
\wedge d\bar\zeta\\
=\lim_{\epsilon\to 0}\int_{\Gamma^{\epsilon}\cap\{|B(z,\zeta)|=9\delta^2\}}
\det\left[\frac{Q(\bar\zeta,\bar{z})}{P(\bar\zeta)}\
\frac{z}{B^*(\bar\zeta,\bar{z})}\ \frac{\zeta}{B(\bar\zeta,\bar{z})}\right]
\wedge d\bar\zeta\\
\leq C\cdot\frac{\delta\cdot\text{Area}\{|B(z,\zeta)|=9\delta^2\}}{\delta^4}\leq C,
\end{multline}
where in the last inequality we used estimate \eqref{DeltaAreaEstimate} from Lemma~\ref{DeltaArea}.

\indent
Combining estimates \eqref{LargeSecond}$\div$\eqref{LargeFifth}, \eqref{LargeFirstFirst}, and
\eqref{LargeFirstSecond}, we obtain the statement of Proposition~\ref{SecondRight}.
\end{proof}

\indent
From the Propositions~\ref{gEstimate} and \ref{SecondRight} we obtain that the right-hand side
of the integral equation \eqref{hEquation} is in $\Lambda^{\alpha}(V)$.

\indent
In the next proposition we prove the compactness of the integral operator defined by the
integral in the right-hand side of \eqref{SOperator}.

\begin{proposition}\label{Compactness}
The following integral
\begin{multline}\label{CompactOperator}
\lim_{\epsilon\to 0}\int_{\Gamma^{\epsilon}}h(\zeta)
\det\left[\frac{Q(\bar\zeta,\bar{z})}{P(\bar\zeta)}\
\frac{\zeta}{B(\bar\zeta,\bar{z})}\ \partial_{\zeta}\left(\frac{\zeta}{B(\bar\zeta,\bar{z})}\right)\right]
\wedge\d\bar\zeta\\
=\int_{\pi^{-1}(V)}h(\zeta)\det\left[Q(\bar\zeta,\bar{z})\
\frac{\zeta}{B(\bar\zeta,\bar{z})}\ \frac{\d\zeta}{B(\bar\zeta,\bar{z})}\right]
\wedge\left(\d P(\bar\zeta)\interior\d\bar\zeta\right)\\
\end{multline}
defines a compact operator from $\Lambda^{\alpha}(V)$ into itself for any $\alpha<1$.
\end{proposition}
\begin{proof}
We represent the integral operator in \ref{CompactOperator} as follows
\begin{multline}\label{OperatorRepresentation}
\lim_{\epsilon\to 0}\int_{\Gamma^{\epsilon}}h(\zeta)
\det\left[\frac{Q(\bar\zeta,\bar{z})}{P(\bar\zeta)}\
\frac{\zeta}{B(\bar\zeta,\bar{z})}\ \frac{\d\zeta}{B(\bar\zeta,\bar{z})}\right]
\wedge\d\bar\zeta\\
=\lim_{\epsilon\to 0}\int_{\Gamma^{\epsilon}}(h(\zeta)-h(z))\det\left[Q(\bar\zeta,\bar{z})\
\frac{\zeta}{B(\bar\zeta,\bar{z})}\ \frac{\d\zeta}{B(\bar\zeta,\bar{z})}\right]
\wedge\left(\d P(\bar\zeta)\interior\d\bar\zeta\right)\\
+h(z)\cdot\lim_{\epsilon\to 0}\int_{\Gamma^{\epsilon}}\det\left[Q(\bar\zeta,\bar{z})\
\frac{\zeta}{B(\bar\zeta,\bar{z})}\ \frac{\d\zeta}{B(\bar\zeta,\bar{z})}\right]
\wedge\left(\d P(\bar\zeta)\interior\d\bar\zeta\right).
\end{multline}

\indent
In the following lemma we prove the necessary estimates for the first integral in the
right-hand side of \eqref{OperatorRepresentation}.

\begin{lemma}\label{FirstRightBoundedness}
The first integral in the right-hand side of \eqref{OperatorRepresentation} defines a linear bounded operator from $\Lambda^{\alpha}(V)$ into itself.
\end{lemma}
\begin{proof}
To estimate the first term in the right-hand side of \eqref{OperatorRepresentation} we choose
$\delta>0$ and points $z=z^{(1)}, z^{(2)}$ such that $\pi(z^{(i)})=u^{(i)}$ and
$|B(z^{(1)},z^{(2)})|<C\cdot\delta^2$, which exist according to Lemma~\ref{FindingZ}. Then we obtain for $z=z^{(i)}, i=1,2$ the following estimate
\begin{multline}\label{FirstOperatorSmallEstimate}
\left|\lim_{\epsilon\to 0}\int_{\Gamma^{\epsilon}\cap\{|B(z,\zeta)|<\delta^2\}}
(h(\zeta)-h(z))\det\left[Q(\bar\zeta,\bar{z})\
\frac{\zeta}{B(\bar\zeta,\bar{z})}\ \frac{\d\zeta}{B(\bar\zeta,\bar{z})}\right]
\wedge \d\bar\zeta\right|\\
=\left|\int_{\pi^{-1}(V)\cap\{|B(z,\zeta)|<\delta^2\}}
(h(\zeta)-h(z))\frac{S(\zeta,\bar{z})}{B(\bar\zeta,\bar{z})^2}
\d\zeta\wedge\left(\d P(\bar\zeta)\interior\d\bar\zeta\right)\right|\\
\leq C\|h\|_{\Lambda^{\alpha}}\cdot\int_0^{\delta^2}\d t\int_0^{\delta}\frac{r^{1+\alpha}\d r}{(t+r^2)^2}
\leq C\|h\|_{\Lambda^{\alpha}}\cdot\int_0^{\delta}r^{-1+\alpha}\d r
\leq C\|h\|_{\Lambda^{\alpha}}\cdot\delta^{\alpha}.
\end{multline}

\indent
For the estimate of the same integral in
$\pi^{-1}(V)\cap\{|B(z^{(1)},\zeta)|,|B(z^{(2)},\zeta)|>9\delta^2\}$
we use the following representation
\begin{multline}\label{FirstOperatorRepresentation}
\int_{\pi^{-1}(V)\cap\{|B(z^{(i)},\zeta)|>9\delta^2\}}
(h(\zeta)-h(z^{(1)})\frac{S(\zeta,\bar{z}^{(1)})}{B(\bar\zeta,\bar{z}^{(1)})^2}\d\zeta
\wedge\left(\d P(\bar\zeta)\interior\d\bar\zeta\right)\\
-\int_{\pi^{-1}(V)\cap\{|B(z^{(i)},\zeta)|>9\delta^2\}}
(h(\zeta)-h(z^{(2)})\frac{S(\zeta,\bar{z}^{(2)})}{B(\bar\zeta,\bar{z}^{(2)})^2}\d\zeta
\wedge\left(\d P(\bar\zeta)\interior\d\bar\zeta\right)\\
=\int_{\pi^{-1}(V)\cap\{B(z^{(i)},\zeta)|>9\delta^2\}}
(h(z^{(2)})-h(z^{(1)}))\frac{S(\zeta,\bar{z}^{(1)})}{B(\bar\zeta,\bar{z}^{(1)})^2}
\d\zeta\wedge\left(\d P(\bar\zeta)\interior\d\bar\zeta\right)\\
+\int_{\pi^{-1}(V)\cap\{|B(z^{(i)},\zeta)|>9\delta^2\}}
(h(\zeta)-h(z^{(2)}))
\frac{S(\zeta,\bar{z}^{(1)})-S(\zeta,\bar{z}^{(2)})}{B(\bar\zeta,\bar{z}^{(1)})^2}
\d\zeta\wedge\left(\d P(\bar\zeta)\interior\d\bar\zeta\right)\\
+\int_{\pi^{-1}(V)\cap\{|B(z^{(i)},\zeta)|>9\delta^2\}}
(h(\zeta)-h(z^{(2)}))\left(\frac{S(\zeta,\bar{z}^{(2)})}{B(\bar\zeta,\bar{z}^{(1)})^2}
-\frac{S(\zeta,\bar{z}^{(2)})}{B(\bar\zeta,\bar{z}^{(2)})^2}\right)
\d\zeta\wedge\left(\d P(\bar\zeta)\interior\d\bar\zeta\right).
\end{multline}

\indent
For the first term of the right-hand side of \eqref{FirstOperatorRepresentation} we use
Lemma~\ref{fConstant} to obtain the estimate
\begin{equation}\label{FirstTermSmallDeltaEstimate}
|h(z^{(2)})-h(z^{(1)})|\cdot\left|\int_{\pi^{-1}(V)\cap\{|B(z^{(i)},\zeta)|>9\delta^2\}}
\frac{S(\zeta,\bar{z}^{(1)})}{B(\bar\zeta,\bar{z}^{(1)})^2}
\d\zeta\wedge\left(\d P(\bar\zeta)\interior\d\bar\zeta\right)\right|
\leq C\|h\|_{\Lambda^{\alpha}}\cdot\delta^{\alpha}.
\end{equation}

\indent
For the second term of the right-hand side of \eqref{FirstOperatorRepresentation} we have
the following estimate
\begin{multline}\label{SecondTermSmallDeltaEstimate}
\Bigg|\int_{\pi^{-1}(V)\cap\{|B(z^{(i)},\zeta)|>9\delta^2\}}(h(\zeta)-h(z^{(2)}))
\frac{S(\zeta,\bar{z}^{(1)})-S(\zeta,\bar{z}^{(2)})}{B(\bar\zeta,\bar{z}^{(1)})^2}
\d\zeta\wedge\left(\d P(\bar\zeta)\interior\d\bar\zeta\right)\Bigg|\\
\leq C\|h\|_{\Lambda^{\alpha}}\cdot\delta\int_{\delta^2}^A\d t\int_{\delta}^B
\frac{r^{1+\alpha}\d r}{(t+r^2)^2}\leq C\|h\|_{\Lambda^{\alpha}}\cdot\delta\int_{\delta}^B
\frac{r^{1+\alpha}\d r}{(\delta+r)^2}=C\|h\|_{\Lambda^{\alpha}}\cdot\delta^{1+\alpha}.
\end{multline}

\indent
For the third term of the right-hand side of \eqref{FirstOperatorRepresentation}
using estimate \eqref{BDifferenceInequality}
we obtain the following estimate
\begin{multline}\label{ThirdTermSmallDeltaEstimate}
\Bigg|\int_{\pi^{-1}(V)\cap\{B(z^{(i)},\zeta)|>9\delta^2\}}
(h(\zeta)-h(z^{(2)}))S(\zeta,\bar{z}^{(2)})
\frac{B(\bar\zeta,\bar{z}^{(2)})-B(\bar\zeta,\bar{z}^{(1)})}
{B(\bar\zeta,\bar{z}^{(1)})^2B(\bar\zeta,\bar{z}^{(2)})}\Bigg|\\
\leq C\|h\|_{\Lambda^{\alpha}}\cdot\delta
\int_{\delta^2}^A\d t\int_{\delta}^B\frac{r^{2+\alpha}\d r}{(t+r^2)^3}
\leq C\|h\|_{\Lambda^{\alpha}}\cdot\delta\int_{\delta}^B\frac{r^{2+\alpha}\d r}{(\delta+r)^4}\\
\leq C\|h\|_{\Lambda^{\alpha}}\cdot\delta\int_{\delta}^Br^{-2+\alpha}\d r
=C\|h\|_{\Lambda^{\alpha}}\cdot\delta(r^{-1+\alpha})\Big|_{\delta}^B
\leq C\|h\|_{\Lambda^{\alpha}}\cdot\delta^{\alpha}.
\end{multline}
\indent
From the estimates \eqref{FirstOperatorSmallEstimate} and
\eqref{FirstTermSmallDeltaEstimate}$\div$\eqref{ThirdTermSmallDeltaEstimate}
we obtain that the first term of the right-hand side \eqref{OperatorRepresentation} is a linear
operator from $\Lambda^{\alpha}(V)$ into itself. To obtain the same property for the second term of the right-hand side \eqref{OperatorRepresentation} we apply Lemma~\ref{fConstant}
as in the proof of estimate \eqref{FirstTermSmallDeltaEstimate}.\end{proof}

\indent
From the Lemma~\ref{FirstRightBoundedness} we obtain that the operator in the left-hand side of \eqref{OperatorRepresentation} is an operator from $\Lambda^{\alpha}(V)$ into itself. In addition,
from the formula for this operator it follows that it defines an anti-holomorphic function in
$W^{\tau}\setminus V$ for any $\tau>0$ small enough. From these two properties we conclude that the operator
\begin{equation*}
T[h](z)=\lim_{\epsilon\to 0}\int_{\Gamma^{\epsilon}}h(\zeta)
\det\left[\frac{Q(\bar\zeta,\bar{z})}{P(\bar\zeta)}\
\frac{\zeta}{B(\bar\zeta,\bar{z})}\ \frac{\d\zeta}{B(\bar\zeta,\bar{z})}\right]
\wedge\d\bar\zeta
\end{equation*}
maps $\Lambda^{\alpha}(V)$ into the space of anti-holomorphic functions on $W^{\tau}$ for any
$\tau>0$ small enough. Therefore, the function defined by the formula above can be also
defined by the Cauchy
integral formula over the boundary of $W^{\tau}$ for some $\tau>0$, i.e.
\begin{equation}\label{BoundaryFormula}
T[h](z)=\int_{|\bar{w}-\bar{z}|=\tau}\frac{\d\bar{w}}{\bar{w}-\bar{z}}
\left(\lim_{\epsilon\to 0}\int_{\Gamma^{\epsilon}_{\zeta}}h(\zeta)
\det\left[\frac{Q(\bar\zeta,\bar{w})}{P(\bar\zeta)}\
\frac{\zeta}{B(\bar\zeta,\bar{w})}\ \frac{\d\zeta}{B(\bar\zeta,\bar{w})}\right]
\wedge\d\bar\zeta\right),
\end{equation}
where we use the triviality of the normal bundle of $V\subset \C\P^2$ (see \cite{7})
and use the coordinate $w$ in every disk $D(z)=\left\{w\in W^{\tau}:\text{pr}(w)=z\right\}$,
which is projected onto $z\in V$.\\
\indent
Then, we approximate the function $1/(\bar{w}-\bar{z})$ by the polynomials
$$S_{\delta}(\bar{z},\bar{w})=\sum_{j=1}^Ns_j(\bar{z})r_j(\bar{w})$$
on $bW^{\tau}$ so that
\begin{equation*}
\left\|\frac{1}{(\bar{z}-\bar{w})}-S_{\delta}(\bar{z},\bar{w})\right\|_{C^1(bW^{\tau})}\leq \delta,
\end{equation*}
and obtain the approximation of the operator in \eqref{BoundaryFormula} by finite-dimensional
operators
\begin{equation}\label{TDeltaOperator}
T_{\delta}[h](z)=\sum_{j=1}^Ns_j(\bar{z})\int_{|\bar{w}-\bar{z}|=\tau}
r_j(\bar{w})\d\bar{w}\left(\lim_{\epsilon\to 0}\int_{\Gamma^{\epsilon}_{\zeta}}h(\zeta)
\det\left[\frac{Q(\bar\zeta,\bar{w})}{P(\bar\zeta)}\
\frac{\zeta}{B(\bar\zeta,\bar{w})}\ \frac{\d\zeta}{B(\bar\zeta,\bar{w})}\right]
\wedge\d\bar\zeta\right),
\end{equation}
satisfying
\begin{equation}\label{TDeltaEstimate}
\left\|T[h](z)-T_{\delta}[h](z)\right\|_{C^1(V)}\leq C\cdot\delta.
\end{equation}
Since the operator $T$ can be approximated arbitrarily well by finite-dimensional operators,
it is compact \cite{23}.
\end{proof}

Finally, we obtain the following result:

{\it Proof of Theorem~\ref{Main}}.
From the Propositions~\ref{PEquationFredholm}, \ref{gEstimate}, \ref{SecondRight},
and \ref{Compactness} we obtain that function $h$ can be obtained as a solution of two
Fredholm-type integral equations:
\begin{itemize}
\item[(i)]
equation \eqref{PEquation} for $\lambda$ large enough,
from which we obtain $g=\partial h$,\vspace{0.1in}
\item[(ii)]
equation \eqref{hEquation} with the usage of Propositions~\ref{gEstimate}, \ref{SecondRight},
and \ref{Compactness} from which we obtain $h$.
\end{itemize}
Solution $f$ of equation \eqref{fEquation} can be determined from equality \eqref{fhEquality}
$$f(z,\lambda)=e^{\overline{\left\langle\lambda,z/z_0\right\rangle}}h(z,\lambda),$$
and the conductivity from the formula ${\dis q=\frac{\partial\bar\partial f}{f} }$.
\qed


\end{document}